%% file: borda_memoirs_3.tex
\documentclass[12pt,reqno]{memo-l}
\usepackage{amsmath}
\usepackage{xfrac}
\usepackage{stmaryrd,mathrsfs,bm,amsthm,mathtools,yfonts,amssymb,color}
\usepackage{xcolor}
\usepackage{epsfig,color,graphicx,graphics}
\usepackage{hyperref}
\usepackage{imakeidx}
\usepackage{todonotes}
\presetkeys{todonotes}{fancyline, color=blue!30, size=\tiny}{}

\begingroup
\newtheorem{theorem}{Theorem}[chapter]
\newtheorem{lemma}[theorem]{Lemma}
\newtheorem{proposition}[theorem]{Proposition}
\newtheorem{corollary}[theorem]{Corollary}
\newtheorem{conjecture}[theorem]{Conjecture}
\newtheorem{question}[theorem]{Question}
\endgroup

\theoremstyle{definition}
\newtheorem{definition}[theorem]{Definition}
\newtheorem{remark}[theorem]{Remark}
\newtheorem{ipotesi}[theorem]{Assumption}
\newtheorem{example}[theorem]{Example}

\numberwithin{equation}{chapter}

\newcommand\supp{{\rm spt}}
\newcommand\res{\mathop{\hbox{\vrule height 7pt width .3pt depth 0pt\vrule height .3pt width 5pt depth 0pt}}\nolimits}
\newcommand{\cH}{{\mathcal{H}}}
\def\a#1{\left\llbracket{#1}\right\rrbracket}
\newcommand{\breg}{{\rm Reg_b}}
\newcommand\bsing{{\rm Sing_b}}
\newcommand\spi{{\rm Spine}}
\newcommand\bdim{{\rm Bdim}}
\newcommand{\cone}{{\times\hspace{-0.6em}\times\,}}
\newcommand\sS{{\mathscr S}}
\newcommand\sC{{\mathscr C}}
\newcommand\sW{{\mathscr W}}

\newcommand\N{{\mathbb N}}
\newcommand\R{{\mathbb R}}

\newcommand{\eps}{{\varepsilon}}
\newcommand{\allin}{\alpha_\be,\alpha_\bh,M_0,N_0,C_\be,C_\bh}

\newcommand{\Lip}{{\rm {Lip}}}
\newcommand{\diam}{{\rm {diam}}}
\newcommand{\dist}{{\rm {dist}}}
\newcommand{\bB}{{\mathbf B}}

\newcommand\weaks{{\stackrel{*}{\rightharpoonup}}\,}
\newcommand\weak{{\rightharpoonup}\,}

\newcommand{\D}{\textup{Dir}}
\newcommand{\xii}{{\bm{\xi}}}
 
\newcommand{\etaa}{{\bm{\eta}}}
\def\Is#1{{\mathcal{A}}_{#1} (\R^{n})}
\newcommand{\Iqs}{{\mathcal{A}}_Q(\R^{n})}
\newcommand{\Iqqs}{{\mathcal{A}}_{Q-1}(\R^{n})}
\newcommand{\Iq}{{\mathcal{A}}_Q}
\newcommand{\qhalf}{\left(Q-{\textstyle \frac12}\right)}

\newcommand{\abs}[1]{\lvert#1\rvert}
\newcommand{\norm}[1]{\left\lVert#1\right\rVert}

\newcommand{\G}{\mathcal{G}}

\newcommand{\bA}{{\mathbf{A}}}
\newcommand{\ba}{{\mathbf{a}}}
\newcommand{\bp}{{\mathbf{p}}}
\newcommand{\bE}{{\mathbf{E}}}
\newcommand{\me}{{\mathbf{me}}}
\newcommand{\be}{{\mathbf{e}}}
\newcommand{\bef}{{\mathbf{f}}}
\newcommand{\beg}{{\mathbf{g}}}
\newcommand{\bC}{{\mathbf{C}}}
\newcommand{\bh}{{\mathbf{h}}}
\newcommand{\bG}{{\mathbf{G}}}
\newcommand{\bL}{{\mathbf{L}}}
\newcommand{\bM}{{\mathbf{M}}}
\newcommand{\bmm}{{\mathbf{m}}}

\DeclareMathAlphabet{\mathpzc}{OT1}{pzc}{m}{it}
\newcommand\fancyN{{\mathpzc{N}}}
\newcommand\fancyD{{\mathpzc{D}}}
\newcommand\fancyE{{\mathpzc{E}}}

\newcommand{\gammaup}{\Gamma}
\newcommand{\gammado}{\gamma}

\newcommand{\sigmaexp}{\sigma}
\newcommand{\sigmaexpcm}{{\alpha_\bL}}

\newcommand{\InV}{{\mathrm InV}}

\newcommand{\gr}{{\mathrm{Gr}}}

\def\Xint#1{\mathchoice
{\XXint\displaystyle\textstyle{#1}}%
{\XXint\textstyle\scriptstyle{#1}}%
{\XXint\scriptstyle\scriptscriptstyle{#1}}%
{\XXint\scriptscriptstyle\scriptscriptstyle{#1}}%
\!\int}
\def\XXint#1#2#3{{\setbox0=\hbox{$#1{#2#3}{\int}$ }
\vcenter{\hbox{$#2#3$ }}\kern-.6\wd0}}
\def\mint{\Xint-}

\makeindex

\begin{document}

\frontmatter

\title[Boundary behavior of mass minimizers]{On the boundary behavior of mass-minimizing integral currents}

\author{Camillo De Lellis$^1$}
\address{$^1$ School of Mathematics,Institute for Advanced Study, 1 Einstein Dr., Princeton NJ 05840, USA}
\email{camillo.delellis@math.ias.edu}

\author{Guido De Philippis$^2$}
\address{$^2$ Courant Institute,,
New York University,
251 Mercer Street, New York, NY10012, USA}
\email{guido@cims.nyu.edu}

\author{Jonas Hirsch$^3$}
\address{$^3$ Mathematisches Institut, Universit\"at Leipzig, Augustus Platz 10, D04109 Leipzig, Germany}
\email{hirsch@math.uni-leipzig.de}

\author{Annalisa Massaccesi$^4$}
\address{$^4$ DTG, Universit\`a di Padova, Stradella San Nicola, 3, I36100 Vicenza, Italy}
\email{annalisa.massaccesi@unipd.it}

\date{Nov. 19th 2018}

\subjclass[2020]{49Q20,35B65}

\keywords{Boundary regularity, area-minimizing currents, minimal surfaces, calculus of variations,geometric measure theory}

\begin{abstract}
Let $\Sigma$ be a smooth Riemannian manifold, $\gammaup \subset \Sigma$ a smooth closed oriented submanifold of codimension higher than $2$ and $T$ an integral area-minimizing current in $\Sigma$ which bounds $\gammaup$. We prove that the set of regular points of $T$ at the boundary is dense in $\gammaup$. Prior to our theorem the existence of any regular point was not known, except for some special choice of $\Sigma$ and $\Gamma$. As a corollary of our theorem
\begin{itemize}
\item we answer to a question of Almgren (cf. \cite{Alm}) showing that, if $\gammaup$ is connected, then $T$ has at least one point $p$ of multiplicity $\frac{1}{2}$, namely there is a neighborhood of the point $p$ where $T$ is a classical submanifold with boundary $\gammaup$;
\item we generalize Almgren's connectivity theorem showing that the support of $T$ is always connected if $\gammaup$ is connected;
\item we conclude a structural result on $T$ when $\gammaup$ consists of more than one connected component, generalizing a previous theorem  proved by Hardt and Simon in \cite{HS} when $\Sigma = \mathbb R^{m+1}$ and $T$ is $m$-dimensional. 
\end{itemize}
\end{abstract}

\maketitle

\tableofcontents

\mainmatter

\chapter{Introduction}

Consider a smooth complete Riemannian manifold $\Sigma$\index{aags\Sigma@$\Sigma$} of dimension $m+\bar n$ and a smooth closed oriented submanifold $\gammaup\subset \Sigma$\index{aagc\gammaup@$\gammaup$} of dimension $m-1$ which is a boundary in integral homology. Since the  work of Federer and Fleming (cf. \cite{FF}) we know that $\gammaup$ bounds an integer rectifiable current $T$\index{aaltT@$T$} in $\Sigma$ which is mass minimizing. 


\medskip

Starting with the pioneering work of De Giorgi (see \cite{DG}) and thanks to the efforts of several mathematicians in 
the sixties and the seventies (see \cite{Fleming, DeGiorgi5,Alm2,Simons}), it is known that, if $\Sigma$ is of class $C^{2,a}$ for some $a>0$, in codimension $1$ (i.e., when $\bar n =1$) and away from the boundary $\gammaup$, $T$ is a smooth submanifold except for a relatively closed set of Hausdorff dimension at most $m-7$. Such set, which from now on we will call {\em interior singular set}, is indeed $(m-7)$-rectifiable (cf. \cite{Simon2}) and it has been recently proved that it must have locally finite Hausdorff $(m-7)$-dimensional measure (see \cite{NV}).\index{Interior singular set}

In higher codimension, namely when $\bar n \ge 2$, Almgren proved in a monumental work (known as Almgren's Big regularity paper \cite{Alm}) that, if $\Sigma$ is of class $C^5$, then the interior singular set has Hausdorff dimension at most $m-2$. Subsequently Chang proved in \cite{Chang} that such set is indeed discrete when $m=2$. In fact Chang's paper is missing one substantial step of the proof, which was completed only recently by the first author in a series of joint works with Emanuele Spadaro and Luca Spolaor, cf. \cite{DSS1,DSS2,DSS3,DSS4}. The latter papers are based on a  revisitation of Almgren's theory, due to the first author and Emanuele Spadaro (cf. \cite{DS1,DS2,DS3,DS4,DS5}), which  simplifies Almgren's proof introducing several new ideas. The latter works are indeed one of the starting points of this paper.

Both in codimension one and in higher codimension the interior regularity theory described above is, in terms of dimensional bounds for the singular set, optimal: 
\begin{itemize}
\item The celebrated paper by Bombieri, De Giorgi and Giusti \cite{BDG} (see \cite{Guido} for a very short proof) shows that Simons' cone $\{x_1^2+x_2^2 +x_3^2 + x_4^2 = x_5^2 + x_6^2 + x_7^2 + x_8^2\}$ is an area-minimizing current of dimension $7$ in $\mathbb R^8$ with an isolated singularity.
\item Federer's calibration theorem shows that any holomorphic subvariety of a K\"ahler manifold induces an area-minimizing current: in particular the holomorphic curve $\{(z,w)\in \mathbb C^2: z^2 = w^3\}$ is a $2$-dimensional area-minimizing current in $\mathbb R^4$ with an isolated singularity. 
\end{itemize} 

\medskip

The main purpose of this paper is to study the regularity of the minimizers at the boundary. In the rest of the note we will always assume that such boundary is the integer rectifiable current naturally induced by some oriented submanifold $\Gamma$ and we will use the notation $\a{\gammaup}$ for it. As it is customary in the literature, we take advantage of Nash's isometric embedding theorem and we consider $\Sigma$ as a submanifold of some Euclidean space $\R^{m+n}$. In particular we can regard any integer rectifiable current $T$ in $\Sigma$ as an integer rectifiable current in the Euclidean space whose support $\supp (T)$ is contained in $\Sigma$: hence $T$ minimizes the mass among all currents $S$ which are supported in $\Sigma$ and such that $\partial S = \a{\gammaup}$.

\begin{definition}\label{def:reg_points}
A point $x\in\gammaup$ is a \emph{boundary regular point}\index{Regular point} for $T$ if there exist a neighborhood $U\ni x$ and a regular $m$-dimensional submanifold $\Xi\subset U\cap\Sigma$  as in Definition \ref{def:reg_points} (without boundary in $U$)\index{aagO@$\Xi$} such that $\supp (T)\cap U\subset \Xi$. The set of such points will be denoted by $\breg(T)$\index{aalr\breg@$\breg(T)$} and its complement in $\gammaup$ will be denoted by $\bsing(T)$\index{aals\bsing@$\bsing(T)$}.

Analogously, the set of interior regular points and interior singular points will be denoted by ${\rm Reg_i} (T)$\index{aalr{\rm Reg_i}@${\rm Reg_i}(T)$} and ${\rm Sing_i} (T)$\index{aals\bsingi@${\rm Sing_i}(T)$}. 

 We further subdvide $\bsing (T)$ into two categories. We will say that $x\in \bsing (T)$ is of {\em crossing type} if there is a neighborhood $U$ of $x$ and two currents $T_1$ and $T_2$ in $U$ with the properties that:
\begin{itemize}
\item $T_1+T_2 = T$ and $\partial T_1=0$;
\item $x\in \breg (T_2)$.
\end{itemize}
If $x\in \bsing (T)$ is not of crossing type, we will then say that $x$ is a {\em genuine boundary singularity} of $T$.
\end{definition}

\begin{remark}\label{rmk:constancy}
Notice that $\bsing(T)$ is closed in $\gammaup$. Moreover, the Constancy Lemma has the following simple consequence. Let $p\in \gammaup$ be a regular point and $\Xi$. Assume the neighborhood $U$ is sufficiently small, so that $U\cap \Xi$ is diffeomorphic to an $m$-dimensional disk. Then the following holds:
\begin{itemize}
\item $\gammaup\cap U$ is necessarily contained in $\Xi$ and divides it in two disjoint regular submanifolds $\Xi^+$ and $\Xi^-$\index{aagO\pm@$\Xi^+$ (resp. $\Xi^-$)} of $U$ with boundaries $\pm\gammaup$;
\item there is a positive 
$Q\in\N$\index{aalqQ@$Q$}  such that 	
\[
T\res U=Q\a{\Xi^+}+(Q-1)\a{\Xi^-}.
\] 
\end{itemize}
We define the \emph{density}\index{Density of a point} of such points $p$ in $\gammaup\cap U$ as $Q-\frac 12$ and we denote it by $\Theta(T,p)=Q-\frac 12$\index{aagh\Theta(T,p)@$\Theta(T,p)$}. Later (in Definition \ref{def:density}) we will define, as customary, the density at every boundary point $p$ as the limit, as $r\downarrow 0$, of the ratio between the mass of the current in a ball of radius $r$ (denoted by $\|T\| (\bB_r (p))$) and the $m$-dimensional volume of an $m$-dimensional disk of radius $r$ (denoted by $\omega_m r^m$)\index{aagz\omega_m@$\omega_m$}. The two definitions clearly agree on regular points. \index{aalb\bB_r (p)@$\bB_r (p)$}
\end{remark}

Of particular interest are those regular points where $Q = 1$: at such points there is a neighborhood $U$ where the current $T$ is a classical submanifold with multiplicity $1$ and with boundary $\gammaup \cap U$. Such points will be called in the rest of the note \emph{density $\frac{1}{2}$ points}\index{Density $\frac{1}{2}$ points} or \emph{one-sided points}\index{One-sided points}. In contrast, the regular points where $Q>1$ will be called \emph{two-sided}\index{Two-sided points}. Note that, when $p$ is a one-sided point only $\Xi^+\cap U$ is determined (and coincides, in fact, with the support of the current in $U$): $\Xi^-\cap U$ can be chosen to be any ``smooth continuation'' of $\Xi^+\cap U$ across the boundary $\Gamma\cap U$. On the other hand when $p$ is two-sided then the whole submanifold $\Xi\cap U$ is determined by the current $T$
and coincides with its support in $U$. 

\medskip

The first boundary regularity result is due to Allard who, in his Ph.D. thesis (cf. \cite{AllPhD}), proved that, if $\Sigma = \mathbb R^{m+\bar n}$ and $\gammaup$ is lying on the boundary of a uniformly convex set, then every point $p\in \gammaup$ is regular and has multiplicity $\frac{1}{2}$. In his later paper \cite{AllB} Allard developed a more general boundary regularity theory from which he concluded the above result as a simpler corollary. In particular Allard's theory establishes, among other things, the following two facts: 
\begin{itemize}
\item[(a)] if $p\in \gammaup$ is a point where the density $\Theta (T, p)$, defined as 
$\lim_{r\downarrow 0} \frac{\|T\| (\bB_r (p))}{\omega_m r^m}$, equals $\frac{1}{2}$, then $p$ belongs to $\breg (T)$;
\item[(b)] if there is some wedge $W$ of opening angle smaller than $\pi$ whose tip contains $p$ and such that $\supp (T)\subset W$ then $\Theta (T,p) =\frac{1}{2}$  and  thus \(p\in \breg (T)\).
\footnote{A wedge $W\subset \mathbb R^{m+\bar n}$ with opening angle $\vartheta$ is a set which can be mapped via a suitable rigid motion to   $\{(x, y)\in \mathbb R^m \times \mathbb R^{\bar n} :  |y| \leq x_1 \tan \frac{\vartheta}{2} \}$;
the tip of $W$ is the set $\{(x,y): |y|=x_1=0\}$.} 
\end{itemize} 
In contrast to (b), a boundary point $p\in \gammaup$ with density $Q+ \frac{1}{2}$ for some $Q\in \mathbb N \setminus \{0\}$ is not necessarily a regular point.

\medskip

Suitable generalizations of (a) and (b) can be proved in more general ambient manifolds $\Sigma$ and they imply full boundary regularity under geometrically interesting assumptions: a simple example is given when $\gammaup$ lies on the boundary of a geodesic ball of sufficiently small radius. 
However, even when $\Sigma = \mathbb R^{m+\bar n}$, Allard's theory implies the existence of relatively few boundary regular points for general submanifolds $\gammaup$; in particular (b) above can be guaranteed for an appropriate subset of those points where \(\Gamma\) coincides with its convex envelope, for the proof see \cite{Hardt}.

In the codimension one case Hardt and Simon proved later in \cite{HS} that the set of boundary singular points is empty, hence solving  the boundary regularity problem when $\bar n =1$ (although the paper \cite{HS} deals only with the case $\Sigma = \mathbb R^{m+\bar n}$, its extension to a general Riemannian ambient manifold should not cause real issues). A major problem that Hardt and Simon have to face compared to Allard is that under their assumption two-sided boundary points may occur, as it is witnessed by the following example. 

\begin{example}\label{ex:due_cerchi} Let $\gammaup$ be the union of two concentric circles $\gammaup_1$ and $\gammaup_2$ contained in a given $2$-dimensional plane $\pi_0\subset \mathbb R^{2+\bar n}$ and having the same orientation. Then the area-minimizing current $T$ in $\mathbb R^{2+\bar n}$ which bounds $\gammaup$ is unique and it is the sum of the two disks bounded by $\gammaup_1$ and $\gammaup_2$ in $\pi_0$. In particular $T$ has density $\frac{3}{2}$ at every point $p$ which belongs to the inner circle, see Figure\ref{fig:sided}. 
\end{example}

\begin{figure}[htbp]
\begin{center}\label{fig:sided}
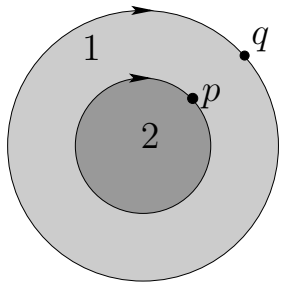
\end{center}
\caption{\(p\) is a {\em two-sided} point while \(q\) is a {\em one-sided} point.}
\end{figure}

Nonetheless, an outcome of the Hardt-Simon boundary regularity theorem is that, if $\gammaup$ contains a two-sided point $p$, then the connected component $\gammaup'$ which contains $p$ arises from a situation like the one described in Example \ref{ex:due_cerchi}. Therefore the presence of regular two-sided points is very rare: for instance, when $\Sigma = \mathbb R^{m+1}$, we can immediately exclude it if we know that no connected component of $\gammaup$ can be included in the interior of a real analytic hypersurface. 

\medskip 

According to the results described so far, in higher codimension and for a general ambient manifold $\Sigma$ we cannot even exclude that the set of boundary regular points is empty.
In particular, in the last remark of the last section of his Big regularity paper, cf. \cite[Section 5.23, p. 835]{Alm}, Almgren states the following open problem:

\begin{question}[Almgren]\label{q:alm}
``I do not know if it is possible that the set of density $\frac{1}{2}$ points is empty when $\gammaup$ is connected.''
\end{question}

We will see in the next chapter that such question is equivalent to ask the existence of at least one regular boundary point. 

The interest of Almgren in Question \ref{q:alm} is motivated by an important geometric conclusion: in \cite[Section 5.23]{Alm} he shows that, if there is at least one density $\frac{1}{2}$ point and $\gammaup$ is connected, then 
$\supp (T)$ is as well connected and the current $T$ has (therefore) multiplicity $1$ almost everywhere, in other words
the mass of $T$ coincides with the Hausdorff $m$-dimensional measure of its interior regular set.

\medskip

In this note we fill the aforementioned  gap in the literature, proving the first general boundary regularity theorem without any restrictions on the codimension, on the ambient manifold $\Sigma$ or on the geometry of $\gammaup$. Since it will be used repeatedly throughout the paper, we isolate the assumptions of our main theorem for further reference. 

\begin{ipotesi}\label{ass:main}
Let $a_0\in ]0,1]$.  Consider a $C^{3,a_0}$\index{aala a_0@$a_0$} complete Riemannian submanifold $\Sigma\subset\R^{m+n}$\index{aags\Sigma@$\Sigma$} with dimension $m+\overline n$ and $\gammaup\subset\Sigma$\index{aagc\gammaup@$\gammaup$} a $C^{3,a_0}$ oriented submanifold without boundary. Let $T$ be an integral $m$-dimensional area-minimizing current in $\bB_2 \cap \Sigma$ with boundary $\partial T \res \bB_2 =\a{\gammaup\cap \bB_2}$, namely such that
\begin{itemize}
\item[(AM)] $\mathbf{M} (T') \geq \mathbf{M} (T)$ for every integer rectifiable current $T'$ with $\partial (T-T') \res \bB_2 =0$ and $\supp (T-T')\subset \Sigma \cap \bB_2$. 
\end{itemize}
\end{ipotesi}

\begin{theorem}\label{thm:main}\label{THM:MAIN}
Let $T,\Sigma,\gammaup$ be as in \ref{ass:main}. Then $\breg(T)$ is dense in $\gammaup\cap \bB_2$.
\end{theorem}

Of course by rescaling and translating, the ball of radius $2$ centered at $0$ can be replaced by any ball $\bB_r (p)$.

 It can be easily shown that boundary singular points can occur when $\gammaup$ is a $C^k$ curve in $\mathbb R^4$ for any $k$, cf. \cite{White_branching}. Such examples are isolated and can be both of crossing type or genuine boundary singularities. A typical construction of the latter goes as follows. We identify $\mathbb R^4$ with $\mathbb C^2$, we take a holomorphic subvariety with a singularity, as for instance $\Lambda := \{(z,w)\in \mathbb C^2: w^3 = z^{3k+1} \}$, and then we consider a suitable $C^k$ closed (real) curve $\Gamma$ lying in $\Lambda$ and passing through the singularity of $\Lambda$. In the specific case $\{(z,w)\in \mathbb C^2: w^3 = z^{3k+1}\}$, a $\Gamma$ of interest is defined so that:
\begin{itemize}
\item its projection on the plane $\pi = \{w=0\}$ contains an open segment $\sigma = \{w=0, {\rm Im}\, z =0, -r<{\rm Re}\, z< r\}$;
\item it bounds a disk $D\subset \Lambda$;
\item the intersection of $D$ with the cylinder $\{|z|<r\}$ covers once the half disk $\{w=0, {\rm Im}\, z <0, |z|<r\}$ and twice the half disk $\{w=0, {\rm Im}\, >0, |z|<r\}$. 
\end{itemize}
$T:= \a{D}$ is then the unique area-minimizing current which bounds $\a{\Gamma}$, while $0$ is an isolated genuine boundary singular point. 

Below we will show examples where $\bsing (T)$ has the same (Hausdorff) dimension of the boundary. Nonetheless the theorem above does not seem optimal from at least two points of view: first of all our example leaves open the possibility that $\bsing (T)$ has zero $(m-1)$-dimensional measure; secondly the singularities of the example are all of crossing type. Indeed it is tempting to advance the following conjecture, which in view of the examples known so far seems rather reasonable. 

\begin{conjecture}\label{conj:sczz}
Let $T,\Sigma,\gammaup$ be as in \ref{ass:main}. The Hausdorff dimension of the set of genuine singular points is at most $m-2$.
\end{conjecture}

When $m=2$ we cannot however expect that genuine singular points are isolated. 

\begin{theorem}\label{thm:example}\label{THM:EXAMPLE}
There are:
\begin{itemize}
\item[(a)] A smooth closed simple curve $\Gamma \subset \mathbb R^4$ and a mass minimizing current $T$ in $\mathbb R^4$ such that $\partial T = \a{\Gamma}$ and $\bsing(T)$ has a genuine boundary singularity which is an accumulation point.
\item[(b)] A smooth $1$-dimensional closed submanifold $\Gamma_1 \subset \mathbb R^4$ (consisting of two disjoint simple curves) and a mass minimizing current $T_1$ in $\mathbb R^4$ such that $\partial T_1 = \a{\Gamma_1}$ and $\bsing(T_1)$ has Hausdorff dimension $1$.
\end{itemize}
\end{theorem}

Moreover the proof of (a) can be easily modified to provide an example of a two dimensional mass minimizing current for which there exists a sequence of {\em interior} singular points accumulating towards the boundary. This shows that the (interior) regularity results for two dimensional mass minimizing currents  in \cite{Chang,DS1,DS2,DS3,DS4,DS5} are actually optimal, see Remark \ref{rmk:intsing}. The proof of (b) is essentially contained in \cite{Jonas2}.

The example of Theorem \ref{thm:example} is related to a previous one of Gulliver\footnote{Gulliver's
example is a minimal immersed disk in the 3-dimensional space. It is obviously not a
minimizer as a current, but it is not known whether it is a solution of the Douglas-Rad\'o problem.} given in \cite{Gulliver}. In both examples there is a boundary branch point where the surface has an infinite order of contact with a plane. In view of
Gulliver's surface, White in \cite{White_branching} stated that ``Proving partial regularity for integral currents at $C^\infty$-boundaries seems to be much harder''. In the case of real analytic curves White proved in \cite{White_branching} that there is no branching boundary point for any solution of the Douglas-Rad\'o problem. In view of this he conjectured that the topology of any area minimizing $2$-dimensional integral current is finite if its boundary is a real analytic curve: combined with his result, White's conjecture would then imply that for real analytic curves both the boundary singular points and the interior singular points are isolated and that the boundary singular points can only be of ``crossing'' type,  i.e. there is no genuine boundary singularity.

\medskip 

 Even though at the moment we cannot progress further in a finer analysis of the singularities, as a corollary of Theorem \ref{thm:main} we can reduce it to the analysis of one-sided boundaries.

\begin{theorem}\label{thm:main2}\label{THM:MAIN2}
Let $\Sigma$ and $\gammaup$ be as in Assumption \ref{ass:main}. Assume $\gammaup$ is closed and $T$ is an area-minimizing integral current in $\Sigma$ with $\partial T = \a{\gammaup}$. Let $\gammaup'\subset\gammaup$ be a connected component of $\gammaup$. If $\gammaup'\cap \breg(T)$ contains a point $p$ with multiplicity $\Theta(T,p)>\frac 12$, then
\begin{itemize}
\item[(a)] the Hausdorff dimension of $\bsing(T)\cap\gammaup'$ is at most $m-2$;
\item[(b)] if $m=2$, then $\bsing(T)\cap\gammaup'$ consists of finitely many points.
\end{itemize}
\end{theorem}

Theorem \ref{thm:main2} is a consequence of a suitable decomposition of the current $T$, which will be stated in the next chapter (cf. Theorem \ref{thm:primes}). One consequence of the latter result is that the two-sided components of $\gammaup$ are, in a suitable sense, ``internal to the current'', as in Example \ref{ex:due_cerchi}. So, even if Theorem \ref{thm:main} is not a full regularity statement as the one in \cite{HS}, it is still powerful enough to yield a similar description of the current $T$ in a neighborhood of the two-sided connected components of $\gammaup$. Moreover, the decomposition Theorem \ref{thm:primes} leads easily to a full answer to Question \ref{q:alm} and in particular we can show the connectedness of the support of any minimizer $T$ whose boundary $\gammaup$ is connected. 

\begin{corollary}\label{c:almgren_problem}
Let $\Sigma, \gammaup$ and $T$ be as in Theorem \ref{thm:main2} and assume in addition that $\gammaup$ is connected and that both $\Gamma$ and $\supp (T)$ are compactly contained contained in $\bB_2$. Then, 
\begin{itemize}
\item[(a)] $\breg (T)$ coincides with the set of density $\frac{1}{2}$ points;
\item[(b)] the set of interior regular points ${\rm Reg_i} (T)$ is connected;
\item[(c)] $\Theta (T,p)=1$ for all $p\in {\rm Reg_i} (T)$ and $\bM (T) = \cH^m ({\rm Reg_i} (T)) = \cH^m (\supp (T))$. 
\end{itemize}
\end{corollary}

While Theorem \ref{thm:primes}, Theorem \ref{thm:main2} and Corollary \ref{c:almgren_problem} are rather straightforward consequences of Theorem \ref{thm:main} and of the interior regularity theory via well-established techniques in geometric measure theory, 
the proof of Theorem \ref{thm:main} is very long and will occupy essentially all the rest of the note.  In a nutshell we will develop a suitable counterpart of Almgren's interior regularity theory at the boundary in order to prove it. Such task poses many additional difficulties and in order to overcome them we introduce several new ideas and tools, some of which might be useful even for the interior regularity theory.

\medskip

Our work would have not been possible without the new insight provided by the papers \cite{DS1,DS2,DS3,DS4,DS5} and by the Ph.D. thesis of the third author, cf. \cite{Jonas, Jonas2}. In particular the latter contains two fundamental starting points: a suitable boundary regularity theory for ${\rm Dir}$-minimizing multiple valued map and a fruitful discussion on how the frequency function estimate of Almgren might fail at the boundary. Such discussion has been essential to identify the key ``estimate'' which underlies the present work. 

In Section \ref{s:on_the_road} we will give
a road map to the proof of Theorem \ref{thm:main}, we will discuss the most important ideas which 
enter into it and we will point out their relations with Almgren's big regularity paper \cite{Alm}, with the works \cite{DS1,DS2,DS3,DS4,DS5} and with \cite{Jonas}. 

\section{Acknowledgments}

The research of Camillo De Lellis has been supported by the ERC grant agreement RAM (Regularity for Area Minimizing currents), ERC 306247.

The research of Guido De Philippis has been supported by the ERC grant agreement RAM (Regularity for Area Minimizing currents), ERC 306247 and by the MIUR SIR-grant ``Geometric Variational Problems''.

The research of Jonas Hirsch has been supported by the MIUR SIR-grant ``Geometric Variational Problems''.

The research of Annalisa Massaccesi has been supported by European Union's Horizon 2020 research and innovation programme under the grant agreement No.752018 (CuMiN).

\chapter{Corollaries, open problems and plan of the paper}

\section{Indecomposable components of $T$}

We start this chapter by stating and proving our main structure theorem as corollary of Theorem \ref{thm:main}. 

\begin{theorem}\label{thm:primes}
Let  $\Sigma,\gammaup, T$ be as in Assumption \ref{ass:main} and assume in addition that \(\gammaup\) and $\supp (T)$ are compactly contained in $\bB_2$. Let us  denote by $\gammaup_1,\ldots,\gammaup_N$ the connected components of $\gammaup$. Then there exist a natural number \(\overline{N}\in \mathbb N \),  integer multiplicities \(Q_{j}\in \mathbb N\setminus \{0\}\) and currents \(T_{j}\) such that 
\begin{equation}\label{e:dec_indec}
T=\sum_{j=1}^{\overline N} Q_jT_j\,, 
\end{equation}
where:
\begin{itemize}
\item[(a)]  For every $j=1,\ldots,\overline N$, $T_j$ is an integral current with $\partial T_j=\sum_{i=1}^N\sigma_{ij}\a{\gammaup_i}$ and $\sigma_{ij}\in\{-1,0,1\}$.
\item[(b)] For every $j=1,\ldots,\overline N$, $T_j$ is an area-minimizing current and $T_j=\cH^{m}\res \Lambda_j$, where $\Lambda_1,\ldots,\Lambda_{\overline N}$ are the connected components of $\supp(T)\setminus(\gammaup\cup {\rm Sing}_i (T)) = {\rm Reg}_i (T)$.
\item[(c)] Each $\gammaup_i$ is 
\begin{itemize} 
\item either \emph{one-sided}, which means that there is one index $o(i)$ such that $\sigma_{io(i)}=1$ and $\sigma_{ij}=0$ $\forall j\neq o(i)$;
\item or \emph{two-sided}, which means that:
\begin{itemize}
\item there is one $j=p(i)$ such that $\sigma_{ip(i)}=1$,
\item there is one $j=n(i)$ such that $\sigma_{in(i)}=-1$,
\item all other $\sigma_{ij}$ equal $0$.
\end{itemize}
\end{itemize}
\item[(d)] If $\gammaup_i$ is one-sided, then $Q_{o(i)}=1$ and all points in $\gammaup_i\cap \breg{T}$ have multiplicity $\frac 12$. 
\item[(e)] If $\gammaup_i$ is two-sided, then $Q_{n(i)}=Q_{p(i)}-1$, all points in $\gammaup_i\cap\breg{T}$ have multiplicity $Q_{p(i)}-\frac 12$ and $T_{p(i)}+T_{n(i)}$ is area minimizing.
\end{itemize}
\end{theorem}
\begin{proof} Let $\Lambda$ be a connected component of 
\[
\supp(T)\setminus (\gammaup\cup {\rm Sing}_i (T)) = {\rm Reg}_i (T)\, .
\]
Since $\Lambda$ is smooth and connected, by the Constancy Theorem the multiplicity of $T$ is a constant $Q\in \mathbb N\setminus \{0\}$ on $\Lambda$. Let $S:=Q\a{\Lambda\cap {\rm Reg}_i (T)}$, where we orient $\Lambda$ so that $S=T$ in every sufficiently small neighborhood of every point $p\in \Lambda$. Observe that $\supp (\partial S) \subset \Gamma \cup {\rm Sing}_i (T)$. Since $\mathcal{H}^{m-1} ({\rm Sing}_i (T)) =0$, from \cite[Theorem 4.1.20]{Fed} we then conclude that $\partial S =0$ on $\mathbb R^{m+n}\setminus \gammaup$. Thus $\supp (\partial S) \subset \gammaup$. Let now $\gammaup_i$ be a connected component of $\gammaup$ and let $\bp$ be a retraction of a neighborhood \(U\) of $\gammaup_i$ onto $\gammaup_i$. Since $\partial S$ is a flat chain supported in $\gammaup_i$, Federer's flatness theorem, cf. \cite[Section 4.1.15]{Fed}, implies that $R := \bp_\sharp (\partial S\res U) = \partial S\res U$. On the other hand, since $\partial (\partial S\res U) =0$, we also have $\partial R=0$ and we conclude from the Constancy theorem, cf. \cite[Section 4.1.7]{Fed}, that $R = c \a{\gammaup_i}$ for some $c\in \R$. Thus $\partial S = \sum_{i=1}^N c_i\a{\gammaup_i}$. 

From Theorem \ref{thm:main} there is at least one point $p\in\breg(T)\cap\gammaup_i$.  In a sufficiently small neighborhood $V$ of $p$,  the set $\supp(T)\setminus\gammaup_i$ consists of at most two connected components which are regular submanifolds and which we call $\Xi^+$ and $\Xi^-$, consistently with the notation of Definition \ref{def:reg_points} and Remark \ref{rmk:constancy}. Since $\Lambda$ is connected, we have the following three alternatives:
\begin{itemize}
\item[(i)] $p\not \in \overline\Lambda$;
\item[(ii)] $\Lambda$ contains only one of the two components $\Xi^\pm$;
\item[(iii)] $\Lambda$ contains both $\Xi^+$ and $\Xi^-$.
\end{itemize}
However, by the Constancy Lemma, the density of $T$ on $\Lambda$ must be constant, whereas, according to Remark \ref{rmk:constancy}, it differs on the two surfaces $\Xi^+$ and $\Xi-$. For this reason we can exclude the alternative (iii) and in particular,
\begin{itemize}
\item either $\partial S \res V = 0$,
\item or $\partial S \res V =(\Theta (p, T) +\frac{1}{2})  \a{\gammaup_i} \res V = Q \a{\gammaup_i}\res V$,
\item or $\partial S \res V = -(\Theta (T,p) - \frac{1}{2})\a{\gammaup_i}\res V = - Q\a{\gammaup_i} \res V$.
\end{itemize}
If we consider the connected components of ${\rm Reg}_i (T)$ we obtain a decomposition as in \eqref{e:dec_indec} with property (a), except that we have not yet shown that the number of connected components is finite (they might be countably infinite).
First observe that
\begin{equation}\label{e:dec_masses}
\bM(T)=\sum_{j\ge 1} Q_j\bM(T_j)\,,
\end{equation}
and hence we easily see that each $T_j$ must be area-minimizing. 
Next observe that each connected component $\Lambda_j$ must contain a point at a fixed positive distance from $\gammaup$ (otherwise we could retract $T_j$ on $\gammaup$). By the monotonicity formula the mass of each $T_j$ can be bounded from below with a constant independent of $j$. Thus from \eqref{e:dec_masses} we conclude that the number of $T_j$'s must be finite.

We now prove (c), (d) and (e): fix $\gammaup_i$ and fix a regular point $p\in\breg(T)\cap\gammaup_i$. If $\Theta(T,p)=\frac 12$, then in a suitable neighborhood $V$ of $p$ the set $(\supp(T)\setminus\gammaup)\cap V$ coincides with ${\rm Reg}_i (T) \cap V$ and consists of only one connected component, so there is one and only one $\sigma_{ij}\neq 0$. Moreover, for that particular $j=:o(i)$, $Q_{o(i)}=1$. In particular, $\breg(T)\cap\gammaup_i\cap\supp(T_j)=\emptyset$ for every $j\neq o(i)$, which proves (d) and the first part of (c).

Analogously, if $\Theta(T,p)>\frac 12$, then $V\cap \supp(T)\setminus\gammaup$ consists of exactly two connected components with two different multiplicities in the current $T$, namely there must be exactly $\Lambda_{j^+}$ and $\Lambda_{j^-}$ from which the two connected components of  $\supp(T)\setminus\gammaup\cap V = {\rm Reg}_i (T) \cap V$ arise. Moreover the difference of the two multiplicities $Q_{j^+}-Q_{j^-}$ must necessarily be $1$. As above, since all other $\sigma_{ij}$ are equal to $0$, at any other point $q\in \gammaup_i\cap \breg(T)$ there is a neighborhood $V$ which intersects only $\Lambda_{j^+}$ and $\Lambda_{j^-}$. On the other hand it must intersect at least one of them (otherwise $\partial T\res V=0$) and therefore it must intersect both of them (otherwise either $\partial T\res V=Q_{j^+}\a{\gammaup_i\cap V}$ or $\partial T\res V=- Q_{j^-}\a{\gammaup_i\cap V}$, which is not possible because $Q_{j^+}\ge 2$ and $Q_{j^-}\ge 1$). This completes the proof of (c) and shows the first part of (e).

In order to complete the proof of (e), consider a $\gammaup_i$ which is two-sided. Denote by $S$ the current $T_{p(i)}+T_{n(i)}$. Notice that
\[
\bM(T)=Q_{n(i)}\bM(S)+\bM(T_{p(i)})+\sum_{n(i)\neq j\neq p(i)}Q_j\bM(T_j).
\] 
From this it follows easily that $S$ must be area-minimizing. 
\end{proof}

\section{Almgren's question and proof of Theorem \ref{thm:main2}}

We can now use Theorem \ref{thm:primes} to prove Corollary \ref{c:almgren_problem} and Theorem \ref{thm:main2}.

\begin{proof}[Proof of Corollary \ref{c:almgren_problem}] 

When $\gammaup$ is connected the decomposition in \eqref{e:dec_indec} consists necessarily of at most two currents because of Theorem \ref{thm:primes}(c), depending on whether $\gammaup$ is one-sided or two-sided. On the other hand, if $\gammaup$ were two-sided, the decomposition \eqref{e:dec_indec} would consist of two currents $T_1$ and $T_2$ with $Q_1 = Q_2+1\geq 2$. Thus $T_1$ would have boundary $\a{\gammaup}$ and strictly less mass than $T$, contradicting the minimality of $T$. 
\end{proof}

\begin{proof}[Proof of Theorem \ref{thm:main2}] 
Consider $\gammaup'$ and $p$ as in the statement and apply Theorem \ref{thm:primes}. Without loss of generality assume $\gammaup'=\gammaup_1$. By point (d) of Theorem \ref{thm:primes}, $\gammaup_1$ is necessarily two-sided, therefore $S:=T_{p(1)}+T_{n(1)}$ is area-minimizing. Since all points of $\gammaup_1$ are interior points of $S$, we know from the interior regularity theory that $S$ is regular at $p$ in $\gammaup_1$, except for a set of points of dimension $m-2$ (which is finite if $m=2$). At any point $p$ where $S$ is regular, the boundary regularity of $T_{p(1)}$ and $T_{n(1)}$ follows easily from the Constancy Theorem \cite[Section 4.1.7]{Fed}.
\end{proof}

\begin{remark}
It is clear from the proof of Theorem \ref{thm:primes} and of Corollary \ref{c:almgren_problem} that the requirement that  \(\gammaup\) and $\supp (T)$ are compactly contained in $\bB_2$ can be somehow relaxed, and that suitably local versions of these results are true. Since however the proof will follow the same arguments described above, we leave these generalizations  to the interested reader.
\end{remark}

\section{Proof of Theorem \ref{thm:example}}

First of all consider the complex halfplane $\mathbb H := \{z\in \mathbb C: {\rm Re}\, z>0\}$
over which we fix the following determination of the complex logarithm: 
\[
{\rm Log}\, z = \log |z| + i \arctan \frac{{\rm Im}\, z}{{\rm Re}\, z}\, .
\]
(where $\arctan: \mathbb R \to (-\frac{\pi}{2}, \frac{\pi}{2})$ is the usual inverse trigonometric function on the real axis).
Correspondingly we define (again on $\mathbb H$) the functions $z^{-\alpha} = \exp (-\alpha {\rm Log}\, z)$ for
$\alpha\in (0,1)$ and
\[
f_k (z) = \exp (- z^{-\alpha}) \sin \left({\rm Log}\, z + \frac{3-2k}{6} \pi i \right)\,  \qquad \mbox{for $k=0,1,2,3$.}
\]
Observe that:
\begin{itemize}
\item[(i)] Each $f_k$ can be extended smoothly to a $C^\infty$ function on $\overline{\mathbb H}$. Indeed, observe first that there is an holomorphic extension of $f_k$ to $\mathbb C \setminus \{z\in \mathbb R: {\rm Im}\, z =0, {\rm Re}\, z\leq 0\}$, which, with a slight abuse of notation, we keep denoting by $f_k$. Such extension is thus defined on $\overline{\mathbb H}\setminus \{0\}$. Hence, in order to prove our claim it suffices to show that any partial derivative (of any order) of $f_k$ can be extended continuously from $\overline{\mathbb H}\setminus \{0\}$ to the origin. We claim in particular that such extension can be achieved by setting it $0$ at the origin. Since $\partial_{\overline z} f_k =0$ (on $\overline{\mathbb H}\setminus \{0\}$), it suffices to show our claim for any partial derivative $\partial_z^\ell f$. For the latter we easily have the inequality
\begin{equation}\label{e:stima_jonas}
|\partial^\ell_z f_k (z)|\leq C (\alpha, \ell) |z|^{-N (\alpha, \ell)} e^{ - {\rm Re}\, z^{-\alpha}} \leq C (\alpha, \ell) |z|^{-N (\ell,\alpha)} e^{- c (\alpha) |z|^{-\alpha}}\, ,
\end{equation}
where $N (\alpha, \ell), C(\alpha, \ell)$ and $c (\alpha)=\cos(\alpha \frac{\pi}{2})$ are positive constants. 
\item[(ii)] Since $\exp (- z^{-\alpha})$ does not vanish on $\overline{\mathbb H}\setminus \{0\}$, the zero set $Z_k$ of $f_k$ in $\overline{\mathbb H}\setminus \{0\}$ is given
by 
\[
Z_k= \left\{z\in \overline{\mathbb H}: {\rm Log}\, z + \frac{3-2k}{6} \pi i \in \pi \mathbb Z\right\}\, ,
\]
namely by 
\begin{equation}\label{e:formula_Z_k}
Z_k = \left\{\exp \left(n \pi + i\frac{2k-3}{6}\pi \right): n \in \mathbb Z\right\}\, . 
\end{equation}
\end{itemize}
Consider next the function
\[
g(z) = \prod_{k=0}^3 f_k (z)\, .
\]
We then conclude that $g$ is holomorphic on $\mathbb H$, it is $C^\infty$ on $\overline{\mathbb H}$ and its zero set, which we denote by $Z$, is given by
\[
Z = \{0\} \cup \bigcup_{k=0}^3 Z_k\, .
\]

Define now the map $G: \overline{\mathbb H}\to \mathbb C^2$ by $G(z) = (z^3, g(z))$. We consider a smooth simple
curve $\gamma\subset \overline{\mathbb H}$  which contains a nontrivial segment 
\begin{equation}\label{e:def_imaginary_segment}
\sigma = [-\tau i, \tau i]
\end{equation}
on the imaginary axis and we let $D\subset \mathbb H$ be the open disk bounded by $\gamma$. The current $T := G_\sharp \a{D}$ is integer rectifiable and 
\[
\partial T = G_\sharp \partial \a{D} = G_\sharp \a{\gamma}\, .
\] 
Observe that $G (D)$ is an holomorphic curve of $\mathbb C^2$, which carries a natural orientation. If $\a{G(D)}$ denotes the corresponding integer rectifiable current, we then have $T = \Theta \a{G(D)}$, where $\Theta$ is the integer-valued function which at $\mathcal{H}^m$-a.e. point $p\in G(D)$ counts the number of preimages in $D$, namely $\Theta (p) = \sharp \{z\in D: G(z) = p\}$ (indeed our argument below will show that $\Theta$ equals $1$ except for a countable number of points).  It follows from a classical result of Federer (cf. \cite{Fed}) that $T$ is an area-minimizing current.

We then claim that
\begin{itemize}
\item[(a)] for an appropriate choice of $\gamma$, $G_\sharp \a{\gamma} = \a{G (\gamma)}$ and $G (\gamma) \subset \mathbb C^2 = \mathbb R^4$ is a smooth embedded curve;
\item[(b)] $\sigma \cap G (Z)$ is contained in ${\rm Sing_b} (T)$.
\end{itemize}
Since 
\[
G(Z) = \{0\} \cup \bigcup_{k=0}^3 G (Z_k) = \{0\} \cup \{(\pm i e^{3 n\pi},0)\in \mathbb C^2 = \mathbb R^4: n \in \mathbb Z\}\, ,
\] 
we conclude from (b) that ${\rm Sing_b} (T)$ has an accumulation point at the origin. Thus, because of (a), $\Gamma = G (\gamma)$ is a closed curve which satisfies the claims of the theorem. 

\medskip

In order to show (a) and (b) consider first that the map $G$ is a local smooth embedding at every point $z\in \overline{\mathbb H}$ which is not the origin, because the differential of $z\mapsto z^3$ has full rank everywhere except at the origin. We next claim that
\begin{itemize}
\item[(c)] There is a discrete subset $W\subset \overline{\mathbb H}\setminus \{0\}$ such that the map $G$ is injective when restricted onto $\overline{\mathbb H}\setminus (W\cup\{0\})$.
\end{itemize} 
In order to show (c) consider first that, if $G(z) = G(w)$, then $z^3=w^3$. Thus our claim reduces to showing that
the map $\lambda (z) := g(z)- g(e^{2\pi i/3} z)$ has a discrete set of zeros on the domain 
\[
\Lambda := \left\{z\neq 0 : z\in \overline{\mathbb H} \;\;\mbox{and}\;\; e^{2\pi i/3} z \in \overline{\mathbb H}\right\}\, .
\]
By the holomorphicity of $\lambda$ and the connectedness of $\Lambda$, it suffices to show that $\lambda$ does not vanish identically on $\Lambda$. On the other hand, if it were $\lambda \equiv 0$, then we could extend $g$ holomorphically to a function $\tilde{g}$ on $\mathbb C^2\setminus \{0\}$ with the property that $\tilde{g} (z) = \tilde{g} (e^{2\pi i/3}z)$ for every $z$. From the discussion above it follows easily that such a map $\tilde{g}$ could be extended continuously at the origin and it would thus be holomorphic on the entire complex plane. On the other hand $\tilde{g}$ has a sequence of zeros which accumulate to the origin and thus it would be forced to vanish identically. In particular we would conclude that $g$ vanishes identically and that one of the $f_k$'s must vanish identically too. By the very definition of $f_k$ this is obviously false.

\medskip

Having proved (c) we now show the existence of $\gamma$ as in (a). First we show that $\gamma$ can be chosen so that $G|_\gamma$ is injective. As a preliminary remark, the only point of $\overline{\mathbb H}$ which $G$ maps to the origin $(0,0)$ of $\mathbb C^2$ is the origin $0$ of $\mathbb C$, so we just need to show the
injectivity of $G$ on $\gamma \setminus \{0\}$. 
Observe that, by (c), we can assume that both $G (\tau i)$ and $G (-\tau i)$ have exactly one preimage in $\overline{\mathbb H}$. Since $G$ is an immersion on $\overline{\mathbb H}\setminus \{0\}$, we can choose $\tau$ so that there are two neighborhoods $U_1$ and $U_2$ of, respectively, the endpoints $\tau i$ and $-\tau i$ of the segment $\sigma$ with the property that $G(z)$ has exactly one counterimage in $\overline{\mathbb H}$ for every $z\in (U_1\cup U_2)\cap \overline{\mathbb H}$. Moreover, a generic $\gamma$ will avoid the set $W$, which is discrete, and thus we have shown that
$G$ is injective on $\gamma\setminus \sigma$. Furthermore, we can ensure that all points $z$ in $\gamma\setminus \sigma$ have modulus strictly larger than $\tau$. Since $G(z) = G(w)$ implies $z^3 = w^3$ and hence $|z|=|w|$, such a choice enforces that $G (\gamma\setminus \sigma) \cap G (\sigma) = \emptyset$. It remains to show that $G$ is injective on $\sigma$, but this is easy because, if $z,w\in \sigma$, then both $z$ and $w$ are purely imaginary and the equation $z^3= w^3$ implies $z=w$.

We next wish to show that $G (\gamma)$ is a smooth curve. As already observed, $G$ is an immersion when restricted to $\overline{\mathbb H}\setminus \{0\}$. Thus we only have to show that $G (\gamma)$ is smooth in a neighborhood of $(0,0) = G (0)$. Observe that, in such a neighborhood $G(\gamma)$ is given by the points $\{(-is^3, g(is)): s \in ]-\delta, \delta[\}$, which we can rewrite as 
$\{(-is, g (i s^{\frac{1}{3}})): s \in ]-\delta^3, \delta^3[\}$. We thus have to show that the map
\[
\mathbb R \ni s\mapsto h(s) = g (i s^{\frac{1}{3}})\in \mathbb C
\]
is smooth in a neighborhood of the origin and we will then conclude that $G (\gamma)$ is indeed a smooth embedded curve. In fact the map $h$ is certainly smooth on $(-1,0)\cup (0,1)$. Computing its derivatives we conclude easily that
\[
|h^{(\ell)} (s)|\leq C (\ell) |s|^{-N (\ell)} \sum_{0\leq k \leq \ell} |D^k g (is^{\frac{1}{3}})|
\leq C (\ell, \alpha) |s|^{-N (\ell)} e^{- c(\alpha) |s|^{-\alpha/3}}\, ,
\]
where we have used the estimate \eqref{e:stima_jonas}. In particular
\[
\lim_{s\to 0} h^{(\ell)} (s) = 0\, 
\]
for every $\ell \in \mathbb N$. This shows the smoothness of $g$ in $0$. 

\medskip

We finally come to (b). We just have to show that every point $p\in G(Z)$ is singular: since the origin is an accumulation point of $G(Z)$ and ${\rm Sing_b} (T)$ is closed, the origin will be a singular point as well. Let  \(p\)  be in \(G(Z)\setminus \{0\}\), then \(p=(\pm i e^{3n\pi}, 0)\) for some \(n\in \mathbb Z\). Let us assume that \(p=(i e^{3n\pi},0)\) (the other case being analogous) and  note that $p$  has exactly two preimages in $\overline{\mathbb H}$ through $G$, namely 
\[
z_1=\exp \left(n \pi - i\frac{\pi}{2} \right)\qquad z_2=\exp \left(n \pi + i\frac{\pi}{6} \right)=e^{2\pi i/3}z_1.
\]
Since, as already observed, \(d G_{z_i}\) has full rank for \(i=1,2\), there are small neighborhoods \(U_1\) and \(U_2\) of \(z_1\) and \(z_2\) such that \(G|_{U_1}\) and \(G|_{U_2}\) are embeddings.  Since we have already shown that the set \(\{z: g(z)= g(e^{2\pi i/3}z)\}\) is discrete in $\overline{\mathbb H}\setminus \{0\}$, up to making the neighborhoods smaller we have that \(G(U_1) \cap G(U_2)=\{p\}\). This shows that around \(p\),  \(G(D)\) is an immersed surface with boundary and with a ``double point'' at $p$. Thus $p$ belongs to ${\rm Sing_b} (T)$. 

\begin{remark}\label{rmk:intsing}
Note that the curve \(\gamma\) in the above Theorem can be slightly  modified in order to have that \(G(\gamma)\) is still a smooth curve and that \(\gamma\) bounds a smooth connected  open disk \(\tilde D\) with \(0\in \partial \tilde D\) and \(\sigma = (-\tau i, \tau i)\setminus\{0\}\subset \tilde D\). In particular there is a sequence of points in \(Z\) which are in the interior of \(\tilde D\) and that accumulates towards \(\{0\}\). \(G(Z)\) now consist of interior singular points for \(\tilde T := G_\sharp \llbracket\tilde D\rrbracket\) which accumulate towards the boundary.
\end{remark}

%
\begin{remark}
It is not difficult to see that, in the example above, at any singular point $p\in G (Z)$ the tangent cone consists of one two-dimensional plane $\a{\pi (p)}$ and a two-dimensional half-plane $\a{\pi^+ (p)}$, which intersect only at the origin. By slightly modifying the example, namely by considering the map $G (z) = (z^3, (g (z))^2)$, we can easily ensure that the tangent cone at every $p\in G (Z)$ is contained in a single two-dimensional plane $\pi (p)$. In particular the tangent line to the boundary curve splits such planes in two halves $\pi^- (p)$ and $\pi^+ (p)$: the tangent cone is then 
$2\a{\pi^+ (p)}+ \a{\pi^- (p)}$. On the other hand we do not know whether it is possible to have a sequence of boundary branching singularities which accumulate somewhere. 
\end{remark}

\subsection{Proof of Statement (b)} We now turn to the proof of statement (b) in Theorem \ref{thm:example}.
The starting point is the following fact, proved by the third author in \cite{Jonas2},  where we keep using the notation 
\[
\mathbb H = \{z\in \mathbb C: {\rm Re}\, z>0\}
\]
for the complex halfplane.
\begin{lemma}[{\cite[Lemma 0.1]{Jonas2}}]\label{lm:jonas}
There exists a holomorphic function  \(g: \mathbb H\to \mathbb C\) which extends to a smooth function $F\in C^\infty (\overline{\mathbb H})$ and such that the set 
\[
E := \{F =0\}\cap \partial \mathbb H
\]
is contained in the segment $\sigma := \partial \mathbb H \cap \{{\rm Im}\, z\in [-\frac{1}{2}, \frac{1}{2}]\}$ and has Hausdorff dimension $\dim_{\cH} (E)$ equal to 1.
\end{lemma}

Let now \(\gamma\) be a smooth curve contained in \(\overline{\mathbb {H}}\cap \{|z|\le 1\}\) such that 
\begin{itemize}
\item[(a)] $\sigma \subset \gamma$;
\item[(b)] \(\gamma \cap\{z \in \mathbb H : g(z)=0\}=\emptyset\).
\end{itemize}
Note that this is possible since \(\{g=0\} \cap \mathbb H\) is at most countable. We denote by  \(\mathbb D_+\subset \mathbb {H}\) the disk bounded by \(\gamma\). We let  
\[
G(z)=(z,F(z)) 
\]
and \(S=G_\sharp \a{D_+}\). Note that \(G(\gamma)\) is a smooth curve.  Arguing as in the proof of Statement (a),
we get 
\[
\partial S=G_\sharp \a{\gamma}=\a{G(\gamma)}.
\]
We furthermore let \(\mathbb D=\{|z|\le 1\}\) be the unit disk and \(R=\iota_\sharp\a{\mathbb D}\), \(\iota: \mathbb C\to \mathbb C^2\), \(\iota(z)=(z,0)\). Note that 
\[
\supp \partial S\cap \supp \partial R=\emptyset.
\]
Now the  current  \(T_1=R+S\) satisfies the conclusion of the first of the claim. Indeed 
\[
\partial T_1=\a{\gammaup_1}\qquad\mbox{with}\qquad \gammaup_1= G (\gamma)\cup \{(z, w): |z|=1, w=0\}
\]
and, since the latter union is disjoint, \(\Gamma_1\) is a smooth \(1\)-dimensional manifold. Furthermore, since  both \(R\) and \(S\) are calibrated by the K\"ahler form,  so is \(T_1\), implying that it is the only mass minimizing current spanned by \(\a{\gammaup_1}\). Finally
\[
\bsing(T_1)\supset E\times\{0\},
\]
from which we conclude that $\dim_{\cH} (\bsing (T_1))=1$.

\begin{remark}
In fact it is easy to see that $\bsing (T_1) = E \times \{0\}$, therefore, even though the latter set has Hausdorff dimension $1$, it is a $\cH^1$-null set. Note also that around points in \(E\), the current \(S\) can be represented by a smooth graph, and thus these are crossing singularities.

Eventually we remark that  by the F. and M. Riesz' Theorem,~\cite{Riesz}, the conclusion of \cite{Jonas2} is optimal, meaning that the set \(E\)  in Lemma \ref{lm:jonas} cannot have positive measure. Hence the above construction cannot give an example of a $2$-dimensional mass minimizing current which bounds a smooth submanifold and has a boundary singular set of positive $\cH^1$-measure.
\end{remark}

\section{Plan of the proof of Theorem \ref{THM:MAIN}}\label{s:on_the_road}

In this section we outline the long road which will take us finally to the proof of Theorem
\ref{thm:main}. We fix therefore $\Sigma, \gammaup$ and $T$ as in Assumption \ref{ass:main}.

\medskip

{\bf Reduction to  collapsed points.} We start in Chapter \ref{chap:monot} by recalling Allard's monotonicity formula at the boundary. First of all, combining it with a suitable variant of Almgren's stratification theorem, we conclude that, except for a set of Hausdorff dimension at most $m-2$, at any boundary point $p$ there is a tangent cone which is ``flat'', namely which is contained in an $m$-dimensional plane $\pi \supset T_0 \gammaup$. Secondly, using a classical upper semicontinuity argument, we will focus our attention on `` collapsed points'', cf. Definition \ref{def:cc_point}: additionally to the existence of a flat tangent cone, at such points $p$ we know that there is a sufficiently small neighborhood $U$ where $\Theta (T,q)\geq \Theta (T,p)$ for all $q\in \gammaup \cap U$. 
In particular we will reduce the proof of Theorem \ref{thm:main} to proving that any  collapsed point is regular, cf. Theorem \ref{thm:step0} and Theorem \ref{thm:real_main}. 

\medskip

{\bf The ``linear'' theory.} Assume next that $0\in \gammaup$ is a  collapsed point and let $Q-\frac{1}{2}$ be its density. Note that by Allard's regularity theory we know a priori that $0$ is a regular point if $Q=1$ and thus we can assume, without loss of generality, that $Q \geq 2$. Fix a flat tangent cone $S$ to $T$ at $0$ and assume, up to rotations, that it is supported in the plane $\pi_0 = \mathbb R^m \times \{0\}$ and that $T_0 \gammaup = \{x_1 =0\} \cap \pi_0$. Denote by $\pi_0^\pm$ the two half-planes $\pi_0^\pm := \{\pm x_1 > 0\} \cap \pi_0$, with the assumption that $S = (Q-1) \a{\pi_0^-}+ Q \a{\pi_0^+}$. It is reasonable to expect that, at suitably chosen small scales, the current $T$ is formed by $Q$ sheets over $\pi_0^+$ and $Q-1$ sheets over $\pi_0^-$, respectively. Taken all together such sheets form the current $T$ and have boundary $\a{\gammaup}$. Moreover, by a simple linearization argument such sheets can be expected to be almost harmonic. 

Having this picture in mind, it is natural to develop a theory of $\qhalf$-valued functions minimizing the Dirichlet energy. Their domain of definition is an open subset $\Omega$ of $\mathbb R^m$ which is divided into two halves $\Omega^\pm$ by some smooth $(m-1)$-dimensional surface $\gamma \subset \Omega$. A $\qhalf$-valued map consists then of a pair $(f^+, f^-)$ where $f^-$ is a $(Q-1)$-valued map over $\Omega^-$ (in the sense of Almgren, cf. \cite{DS1}) and $f^+$ is a $Q$-valued map over $\Omega^+$. Such pairs are required to satisfy an additional assumption: the trace of $f^+$ over $\gamma$ is obtained from that of $f^-$ by adding a classical single valued map $\varphi$, which is called the ``interface'', 
cf. Definition \ref{def:Qhalf} for the precise statement. The relevant problem is then that of minimizing the sum of the Dirichlet energies of the two maps subject to the constraint that their boundary values on $\partial \Omega$ and the interface $\varphi$ are both kept fixed. In Chapter \ref{chap:Qhalf} we develop a suitable existence theory for such objects, cf. Theorem \ref{thm:ex_dm}. Concerning their interior structure, we can apply all the conclusions of Almgren's theory (indeed in this paper we
will take advantage of the point of view developed in \cite{DS1}). 

The correct counterpart of the  collapsed situation in Theorem \ref{thm:real_main} must assume, however, that all the $2Q-1$ sheets meet at the interface $\varphi$; under such assumption we say that the $\qhalf$ ${\rm Dir}$-minimizer  collapses at the interface, cf. Definition \ref{def:totally_collapsed}. The core of Chapter \ref{chap:Qhalf} is a suitable regularity theory for minimizers which collapse at the interface. First of all their H\"older continuity follows directly from the Ph.D. thesis of the third author, cf. \cite{Jonas}. Secondly, the most important conclusion of our analysis is that a minimizer collapses at the interface  only if it consists of a single harmonic sheet ``passing through'' the interface, counted therefore with multiplicity $Q$ on one side and with multiplicity $Q-1$ on the other side, cf. Theorem \ref{thm:collasso}. 

Theorem \ref{thm:collasso} is ultimately the \emph{deus ex machina} of the entire argument leading
to Theorem \ref{thm:main}. The underlying reason for its validity is that a monotonicity formula for a suitable variant of Almgren's frequency function holds, cf. Theorem \ref{thm:limit_ff}. Given the discussion of \cite{Jonas2}, such monotonicity can only be hoped in the  collapsed situation and, remarkably, this suffices to carry on our program. 

The validity of the monotonicity formula is clear when the  collapsed interface is flat.
When we have a curved boundary a subtle yet important point becomes crucial: we cannot hope in general for the exact first variation identities which led Almgren to his monotonicity formula, but we can replace them with suitable inequalities. However the latter can be achieved only if we adapt the frequency function by integrating a suitable weight, cf. Definition \ref{def:ff}. The idea of ``smoothing'' Almgren's frequency function with a suitable weight is indeed already present in \cite{DS5} and in this paper we need to push it much further, distorting substantially the geometry of the domain.

\medskip

{\bf First Lipschitz approximation.} In Chapter \ref{chap:Lip1} we use the linear theory for approximating the current with the graph of a Lipschitz $\qhalf$-valued map and we then show that such approximation is close to be ${\rm Dir}$-minimizing, cf. Theorem \ref{t:Lipschitz_1} and Theorem \ref{t:harm_1}. The approximation algorithm is a suitable adaptation of the one developed in \cite{DS3} for interior points. In particular, after adding an ``artificial sheet'', we can directly use the Jerrard-Soner modified BV estimates of \cite{DS3} to give a rather accurate Lipschitz approximation: the subtle point is to engineer the approximation so that it collapses at the interface. 

\medskip

{\bf Height bound and excess decay.} In Chapter \ref{chap:decay} we use the Lipschitz approximation of Chapter \ref{chap:Lip1} together with the regularity theory of Chapter \ref{chap:Qhalf} to establish a power-law decay of the excess \emph{\`a la} De Giorgi in a neighborhood of a  collapsed point, cf. Theorem \ref{thm:decay_and_uniq}. The effect of such theorem is that the tangent cone is flat and unique at every point $p\in \gammaup$ in a suitable neighborhood of a  collapsed point $0\in \gammaup$. Correspondingly, the plane $\pi (p)$ which contains such tangent cone is H\"older continuous in the variable $p\in \gammaup$ and the current is contained in a suitable horned neighborhood of the union of such $\pi (p)$, cf. Corollary \ref{c:cone_cut}. 

An important ingredient of our argument is an accurate height bound in a neighborhood of any collapsed point in terms of the spherical excess, cf. Theorem \ref{t:height_bound}. The argument follows an important idea of Hardt and Simon in \cite{HS} and takes advantage of an appropriate 
variant of Moser's iteration on varifolds, due to Allard, combined with a crucial use of the remainder in the monotonicity formula. The same argument has been also used by Spolaor in a similar context in \cite{Spolaor}, where he combines it with the decay of the energy for ${\rm Dir}$-minimizers, cf. \cite[Proposition 5.1 \& Lemma 5.2]{Spolaor}.

\medskip

{\bf Second Lipschitz approximation.} The decay of the excess proved in Chapter \ref{chap:decay} is used in Chapter \ref{chap:Lip2} to improve the accuracy of the Lipschitz approximation of Theorem \ref{t:harm_1}, cf. Theorem \ref{thm:second_lip}.  In particular, by suitably decomposing the domain of the approximating map in a Whitney-type cubical decomposition which refines towards the boundary, we can take advantage of the interior approximation theorem of \cite{DS3} on each cube and then patch the corresponding graphs together. 

As in the case of the interior regularity, this new Lipschitz approximation is of key importance since it coincides with the current up to an error which is superlinear in the excess.

\medskip

{\bf Left and right center manifolds.} In Chapter \ref{chap:center_manifolds} we use the approximation Theorem \ref{thm:second_lip} and a careful smoothing and patching argument to construct a ``left'' and a ``right'' center manifold $\mathcal{M}^+$ and $\mathcal{M}^-$, cf. Theorem \ref{thm:center_manifold}. The $\mathcal{M}^\pm$ are $C^{3,\kappa}$ submanifolds of $\Sigma$ with boundary $\gammaup$ and they provide a good approximation of the ``average of the sheets'' on both sides of $\gammaup$ in a neighborhood of the  collapsed point $0\in \gammaup$. They can be glued together to form a $C^{1,1}$ submanifold $\mathcal{M}$ which ``passes through $\gammaup$'': each portion has $C^{3,\kappa}$ estimates {\em up to the boundary}, but we only know that the tangent spaces at the boundary coincide, whereas we have a priori no information on the higher derivatives (it must be noted though that, at the end of the argument for Theorem \ref{thm:main}, we will conclude that the center manifolds and the current coincide and that the latter is regular: a posteriori we will then conclude that $\mathcal{M}$ is indeed $C^{3,\kappa}$). The construction algorithm follows closely that of \cite{DS4} for the interior, but some estimates must be carefully adapted in order to ensure the needed boundary regularity. 

The center manifolds are coupled with two suitable approximating maps $N^\pm$, cf. Theorem \ref{thm:cm_app}. The latter take values on the normal bundles of $\mathcal{M}^\pm$ and provide an accurate approximation of the current $T$. Their construction is a minor variant of the one in \cite{DS4}. 

\medskip

{\bf Monotonicity of the frequency function.} In Chapter \ref{chap:frequency} we use a suitable Taylor expansion of the area functional to show that the monotonicity of the frequency function holds for the approximating maps $N^\pm$ as well, cf. Theorem \ref{thm:ff_estimate_current}. In particular we use the first variations of the current along suitably chosen vector fields in order to derive the same inequalities which allow to prove Theorem \ref{thm:limit_ff}. Such inequalities contain however several additional error terms which must be estimated with high accuracy: our proof follows crucially some ideas of \cite{DS5}. Moreover, the ``adapted'' frequency function introduced in Chapter \ref{chap:Qhalf} plays a central role in the estimate of Theorem \ref{thm:ff_estimate_current}. 

\medskip

{\bf Final blow-up argument.} In Chapter \ref{chap:blowup} we then complete the proof of Theorem \ref{thm:main}: in particular we show that, if $0$ were a singular  collapsed point, suitable rescalings of the approximating maps $N^\pm$ would produce, in the limit, a $\qhalf$ ${\rm Dir}$-minimizer violating the regularity Theorem \ref{thm:collasso}. On the one hand the estimate on the frequency function of Chapter \ref{chap:monot} plays a primary role in showing that the limiting map is nontrivial. On the other hand the properties of the center manifolds $\mathcal{M}^\pm$ enter in a fundamental way in showing that the average of the sheets of the limiting $\qhalf$ map is zero on both sides.

\section{Open problems}

Clearly, since the size of the boundary singular set in all known examples is much smaller than what proved in Theorem \ref{thm:main}, the most central open question is whether one can improve the ``generic boundary regularity'' proved in this paper. As already mentioned in the introduction, the most daring conjecture compatible with the examples known so far is the following:

 \begin{conjecture}\label{conj:sczz2}
Let $T,\Sigma,\gammaup$ be as in Assumption \ref{ass:main}. The Hausdorff dimension of the set of genuine boundary singularities is at most $m-2$.
\end{conjecture}

A somewhat milder statement, which would still give a substantial improvement of Theorem \ref{thm:main} is instead

\begin{conjecture}\label{conj:sczz3}
Let $T,\Sigma,\gammaup$ be as in Assumption \ref{ass:main}. Then $\mathcal{H}^{m-1} (\bsing(T)) =0$.
\end{conjecture}

The ``linearized problem'' discussed in Chapter \ref{chap:Qhalf} enjoys a regularity theorem which is analogous to Theorem \ref{thm:main}. 

\begin{definition}
Let $(g^+,g^-)$ be a $\qhalf$-valued function with interface $(\gamma, \varphi)$ as defined in Chapter \ref{chap:Qhalf}. A point $p\in \gamma$ is regular if there are a ball $B_r (p)$, $Q-1$ functions $u_2, \ldots, u_Q : B_r (p) \to \mathbb R^n$ and a function $u_1: B_r^+ (p) \to \mathbb R^n$ such that
\begin{itemize}
\item[(i)] $g^+ = \sum_{i=1}^Q \a{u_i}$ on $B_r^+ (p)$ and $g^- = \sum_{i=2}^Q \a{u_i}$ on $B_r^- (p)$;
\item[(ii)] For any pair $i, j \geq 2$ either the graphs of $u_i$ and $u_j$ are disjoint or they coincide;
\item[(iii)] For any $i\geq 2$ either the graphs of $u_1$ and $u_i$ are disjoint or the graph of $u_1$ is contained in that of $u_i$. 
\end{itemize}
The complement of the regular points in $\gamma$ is called the {\em set of boundary singular points}.  If at a boundary singular point there are maps $u_j$'s which satisfy (i) and (ii) (but not (iii)), then the singular point will be called of {\em crossing type}. Singular points which are not of crossing type will be called {\em genuine boundary singularities}.

A point $p\in \Omega\setminus \gamma$ is {\em regular} if it is an interior regular point for either the $Q$-valued map $f^+$ or the $(Q-1)$-valued map $f^-$ (cf. the introduction of \cite{DS1} for the precise definition). The complement, in $\Omega\setminus \gamma$, is the set of {\em interior singular points}. The union of interior singular points and boundary singular points will be called the {\em singular set}. 

\end{definition}

\begin{theorem}\label{thm:linearized}
Let $(g^+,g^-)$ be a $\qhalf$-valued function with $C^3$ interface $(\gamma, \varphi)$ defined over a domain $\Omega$ and assume that it minimizes the Dirichlet energy in $\Omega \subset \mathbb R^m$. Then the set of boundary singular points is meager.
\end{theorem} 

We do not give a proof of Theorem \ref{thm:linearized}: using the tools developed in Chapter \ref{chap:Qhalf}, the argument is a simple adaptation of the interior regularity theory for $Q$-valued maps, cf. \cite{DS1}. The conjectures corresponding to \ref{conj:sczz2} and \ref{conj:sczz3} are then open in the linearized case as well:

 \begin{conjecture}\label{conj:sczz4}
Let $(g^+, g^-)$ be as in Theorem \ref{thm:linearized}. The Hausdorff dimension of the set of genuine singularities is then at most $m-2$.
\end{conjecture}

\begin{conjecture}\label{conj:sczz5}
Let $(g^+, g^-)$ be as in Theorem \ref{thm:linearized}. The boundary singular set is then a $\mathcal{H}^{m-1}$-null set. 
\end{conjecture}

Recently, in \cite{DeLellisZhao} the first author, together with Z.~Zhao, proved   that for $m=2$ and real analytic boundary data, the set of boundary singularities is discrete even though there is one example of \emph{genuine boundary singularity}. Note that the examples (a) and (b) of Theorem \ref{thm:example},  combined with a routine adjustment of the arguments given in
\cite{Emanuele}, see also \cite[Corollary 3.5]{Jonas2},  to the $\qhalf$-valued setting, gives a $\varphi$ which is not real analytic for which the above conclusions are indeed false.

\begin{theorem}\label{thm:linearized_false}
There is a real analytic\footnote{ In fact $\gamma$ is a segment, in our example.}  $\gamma \subset B_1 \subset \mathbb R^2$ passing through the origin, a $C^\infty$ function $\varphi: \gamma \to \mathbb R^2$ and a 
$\frac{3}{2}$-map $(g^+, g^-)$ with interface $(\gamma, \varphi)$ which is ${\rm Dir}$-minimizing on $B_1$ and whose singular set has Hausdorff dimension $1$.  
\end{theorem}

 Conjecture \ref{conj:sczz2}  is  widely open also for real analytic boundary data.  As we already mentioned, the ``linear''  \(2\)-dimensional case of the conjecture is addressed in \cite{DeLellisZhao}. On the other hand, the $2$-dimensional  ``fully non-linear'' counterpart of \cite{DeLellisZhao} is a well-known conjecture of White, cf. \cite{White_branching}:

\begin{conjecture}\label{conj:sczz7}
Let $T,\Sigma,\gammaup$ be as in \ref{ass:main}, let $m=2$ and assume $\Sigma$ and $\gammaup$ are real analytic. Then the union of the boundary and of the interior singular sets is discrete. 
\end{conjecture}

Again such conjecture is widely open and in \cite{DDH} the first three authors have shown that the conclusion of the conjecture is false when $\Sigma$ and $\gammaup$ are just $C^\infty$. A first step in the positive direction is given in the paper \cite{HM} where the third author and Marini prove the uniqueness of tangent cones 
at any point $p\in \Gamma$ when the latter is merely $C^{1,\alpha}$. 

Coming back to the case of $C^\infty$ boundaries $\Gamma$, the example (a) in Theorem \ref{thm:example} shows that Conjecture \ref{conj:sczz2} must be taken with a grain of salt. One reason why Conjecture \ref{conj:sczz2} might still be correct is that, while the accumulation singular point in the example of Theorem \ref{thm:example}(b) is a boundary branch point, the singularities accumulating to it are of ``crossing type'', namely points where the minimizer is in fact an immersed surface. If it were possible to produce an example with an accumulating sequence of branch points, one could conceive to modify the construction to produce a Cantor-like set of genuine boundary singularities, possibly disproving Conjecture \ref{conj:sczz2}. The following question seems thus a very relevant one:

\begin{question}\label{ques:sczz9}
Is it possible to produce an example as in Theorem \ref{thm:example} with a boundary singular point which is an accumulation of boundary branch points?
\end{question}


\chapter{Stratification and reduction to collapsed points}\label{chap:monot}

\section{First variation and monotonicity formula}

Here and in the sequel we will denote by $A_\Sigma$\index{aala\A_\Sigma@$A_\Sigma$} and $A_\gammaup$\index{aala\A_\gammaup@$A_\gammaup$} the second fundamental forms  of $\Sigma$ and $\gammaup$ and we will assume that  $T$ is as in Assumption \ref{ass:main}.

As usual, given a vector field $X \in C^1_c (\bB_2)$ we let $\bB_2 \times \R \ni (x, t) \to \Phi_t (x)$ be the flow generated by $X$, namely
each curve $\eta_x (t) := \Phi_t (x)$ satisfies the ODE $\dot \eta_x (t) = X (\eta_x (t))$ subject to the initial condition $\eta_x (0) = x$. We
then define the first variation\index{First variation}\index{aagd\delta T(X)@$\delta T(X)$} of $T$ along $X$ as
\[
\delta T (X) := \left. \frac{d}{dt}\right|_0 \bM ((\Phi_t)_\sharp T)\, .
\]
If the vector field $X$ is tangent to $\supp (\partial T) = \gammaup$ and is tangent to the manifold $\Sigma$, we then know that $\delta T (X) = 0$. Moreover, it is well known that if $X$ vanishes on $\supp (\partial T)$ but it is not tangent to $\Sigma$, then 
\[
\delta T (X) = - \int_{\bB_2} X \cdot \vec{H}_T (x)\, d\|T\| (x)
\]
where the mean curvature vector $\vec{H}_T$\index{aalh\vec{H}_T@$\vec{H}_T$} can be explicitly computed from the second fundamental form $A_\Sigma$. More precisely,
if $\vec{T} (x) = v_1 \wedge \ldots \wedge v_m$ and $v_i$ are orthonormal, then
\begin{equation}\label{e:mean_curv}
\vec{H}_T (x) = \sum_{i=1}^m A_\Sigma (v_i, v_i)\, 
\end{equation}
(see for instance \cite{Sim}).
In this section we derive a similar formula for variations along general vector fields $X$, namely not necessarily vanishing on the boundary. 
As a consequence we also get Allard's monotonicity formula at the boundary, with precise error terms. We summarize all these conclusions
in the next theorem. These
are in fact classical facts, under our assumption. Since however it is not easy to pin-point precise references for our statements in the literature,
we include a short derivation from similar (more general) statements proved in other articles.

\begin{definition}\label{def:density}
For every point $p\in \bB_2$, the \emph{density}\index{Density of $T$ at some point}\index{aagh\Theta(T,p)@$\Theta(T,p)$} of $T$ at $p$ is defined as
\[
\Theta(T,p):=\lim_{r\downarrow 0}\frac{\|T\|(\bB_r(p))}{\omega_m r^m}\, ,
\]
whenever the latter limit exists.
\end{definition}

We then consider the functions\index{aagh\Theta_{\rm i} (T,p,r)@$\Theta_{\rm i} (T,p,r)$}\index{aagh\Theta_{\rm b}@$\Theta_{\rm b} (T,p,r)$}
\begin{align}
\Theta_{\rm i} (T,p,r):=\; &\exp \left( C_0 \|A_\Sigma\|_0 r\right)\frac{\|T\|(\bB_r(p))}{\omega_m r^m}\,, \\
\Theta_{\rm b} (T, p,r):=\; &\exp \left( C_0 (\|A_\Sigma\|_0 + \|A_\gammaup\|_0) r\right)\frac{\|T\|(\bB_r(p))}{\omega_m r^m}\, ,
\end{align}
where $C_0=C_0 (m,n, \bar{n})$ is a suitably large constant.

\begin{theorem}\label{thm:allard}\label{THM:ALLARD}\index{Monotonicity formula}
Let $T$ be as in Assumption \ref{ass:main}.
\begin{itemize}
\item[(a)] If $p\in \bB_2 \setminus \gammaup$, then $r\mapsto \Theta_{\rm i} (T,p,r)$ is monotone on the interval $(0,\min \{\dist (p, \gammaup), 2-|p|\})$;
\item[(b)] if $p\in \bB_2 \cap \gammaup$, then $r\mapsto\Theta_{\rm b} (T,p,r)$ is monotone on $(0,2-|p|)$.
\end{itemize}
Thus  the density exists at every point. Moreover, the restrictions of the map $p\mapsto \Theta (T,p)$ to $\Gamma\cap \bB_2$ and to $\bB_2\setminus \Gamma$ are both upper semicontinuous. 

If $X\in C^1_c (\bB_2, \R^n)$, then we have 
\begin{equation}\label{e:first_var}
\delta T (X) = - \int_{\bB_2} X \cdot \vec{H}_T (x)\, d\|T\| (x) + \int_\gammaup X\cdot \vec{n} (x)\, d\cH^{m-1} (x)\,  
\end{equation}
where $\vec{H}_T$ is the vector field in \eqref{e:mean_curv} and $\vec n$ is a Borel unit vector field orthogonal to $\gammaup$\index{aaln\vec n@$\vec n$}. 

Moreover, if $p\in \gammaup$ and $0<s<r < 2-|p|$, we then have the following precise monotonicity identity
\begin{align}
&\;r^{-m} \|T\| (\bB_r (p)) - s^{-m} \|T\| (\bB_s (p)) - \int_{\bB_r (p)\setminus \bB_s (p)} \frac{|(x-p)^\perp|^2}{|x-p|^{m+2}}\, d\|T\| (x)\nonumber\\
= &\; \int_s^r \rho^{-m-1} \Bigg[\int_{\bB_\rho (p)} (x-p)^\perp \cdot \vec{H}_T (x) d\|T\| (x)\nonumber\\
&\qquad\qquad\qquad +   \int_{\gammaup \cap \bB_\rho (p)} (x-p)\cdot \vec{n} (x)\, d\cH^{m-1} (x)\Bigg]\, d\rho\, ,
\label{e:monot_identity}
\end{align}
where $Y^\perp (x)$ denotes the component of the vector $Y (x)$ orthogonal to the tangent plane of $T$ at $x$ (which is oriented by $\vec{T} (x)$). 
\end{theorem}

In this chapter we in fact only need (a) and (b), which are proved in \cite{All} and \cite{AllB}, and some consequences of the monotonicity formula
for which less precise versions are sufficient: in particular many of the statements needed can be easily derived from \cite{AllB} and for this reason we postpone the proof of Theorem \ref{thm:allard} to the last section. 

Note  that at any $p\in \breg (T)$ the density equals $Q-\frac{1}{2}$, where the positive integer $Q$ is as in Remark \ref{rmk:constancy}. Moreover we recall the following
\begin{theorem}[{cf. \cite[Theorem 3.5 (2)]{AllB}}]\label{t:Allard>1/2}
$\Theta(T,p)\ge \frac 12$ for every $p\in\gammaup$.
\end{theorem}

\begin{definition}\label{def:tc}
Fix a point $p\in\supp(T)$ and define\index{aagi\iota_{p,r}@$\iota_{p,r}$}
\[
\iota_{p,r}(q):=\frac{q-p}{r}\quad\forall\,r>0\,.
\]
We denote by $T_{p,r}$ the currents\index{aaltT_{p,r}@$T_{p,r}$}
\[
T_{p,r}:=(\iota_{p,r})_\sharp T\quad\forall\,r>0\,.
\]
\end{definition}
We recall the following consequence of the Allard's monotonicity formula, cf. \cite{AllB}. From now on, given any smooth oriented submanifold of $\mathbb R^{m+n}$ like $\Gamma$ and $\Sigma$, we will use the notation $T_p \Gamma$ and $T_p \Sigma$ for the tangent space to the manifold at the point $p$ (which will be always identified with a linear {\em oriented} subspace of $\mathbb R^{m+n}$)\index{aaltT_p \Sigma@$T_p \Sigma$}\index{aaltT_p \Gamma@$T_p \Gamma$}. 

\begin{theorem}\label{thm:ex_tc}
Take $p\in\supp(T)$ and any sequence $r_k\downarrow 0$. Up to subsequences $T_{p,r_k}$ is converging locally to an area-minimizing integral current $T_0$ supported in $T_p\Sigma$ such that
\begin{itemize}
\item[(a)] $T_0$ is a cone with vertex $0$ and $\|T\|(\bB_1(0))=\omega_m\Theta(T,p)$;
\item[(b)] if $p\in\supp\, (T)\setminus\gammaup$, then $\partial T_0=0$;
\item[(c)] if $p\in\gammaup$, then $\partial T_0=\a{T_p\gammaup}$.
\end{itemize}
Moreover $\|T_{p,r_k}\|$ converges, in the sense of measures, to $\|T_0\|$. 
\end{theorem}
\begin{definition}
Any cone $T_0$ as in Theorem \ref{thm:ex_tc} will be called \emph{a tangent cone to $T$ at $p$}\index{Tangent cone}. A tangent cone $T_0$ will be called \emph{flat}\index{Flat tangent cone} if $\supp(T_0)$ is contained in an $m$-dimensional plane.
\end{definition}

Note that a flat tangent cone at a point $p\in \supp (T)\setminus \gammaup$ is necessarily a positive integer multiple of $\a{\pi}$ for some $m$-dimensional plane $\pi$ contained in $T_{p} \Sigma$: this is a consequence of the Constancy Theorem and of (b) above. For $p\in \gammaup$ a flat tangent cone has instead the form $Q \a{\pi^+} + (Q-1) \a{\pi^-}$, where $Q\geq 1$ is an integer, $\pi = \pi^+ \cup \pi^-$ is an $m$-dimensional plane contained in $T_p \Sigma$ and $\partial \a{\pi^+}  = \a{T_p\gammaup} = - \partial\a{ \pi^-}$. The latter is again a consequence of the Constancy Theorem taking into account that, by (b), $\partial T_0 = \a{T_p \Gamma}$. 

\begin{definition}\label{def:cc_point}
A point $p\in \gammaup$ will be called a \emph{collapsed point}\index{Collapsed point} if
\begin{itemize}
\item[(i)] there exists a flat tangent cone to $T$ at $p$;
\item[(ii)] there exists a neighborhood $U$ of $p$ such that $\Theta (T,q)\geq \Theta (T,p)$ at every $q\in \gammaup\cap U$. 
\end{itemize}
\end{definition}

The first main point of this chapter is to show how standard regularity theory implies that

\begin{theorem}\label{thm:step0}\label{THM:STEP0}
If $\breg (T)$ is not dense in $\gammaup$ then there exists a  collapsed singular point.
\end{theorem}

The proof of Theorem \ref{thm:main} will then be reduced to the following statement: 

\begin{theorem}\label{thm:real_main}
A  collapsed point is always a regular point.
\end{theorem}

All the remaining chapters will in fact be devoted to prove it. 

\medskip

Observe that at collapsed points the density $\Theta (T,p)$ equals $Q-\frac{1}{2}$ for some positive integer $Q$. The case $Q=1$ of the above theorem is indeed a consequence of Allard's boundary regularity theorem for varifolds. Moreover, if $p$ is a point where $\Theta (T,p)=\frac{1}{2}$, then by Theorem \ref{t:Allard>1/2} assumption (ii) in Definition \ref{def:cc_point} is automatically satisfied and in fact the theory of \cite{AllB} shows that even (i) holds necessarily. Therefore, multiplicity $\frac{1}{2}$ points are always regular:

\begin{theorem}[Allard's boundary regularity theorem]\label{t:AllBfinal}
All points $p\in \gammaup$ with $\Theta (T, p) = \frac{1}{2}$ are regular points. 
\end{theorem}

Finally, it is worth noticing the following two consequences of our analysis, which we will also prove in the last section of this chapter:

\begin{corollary}\label{c:uniqueness}\label{C:UNIQUENESS}
For every $\alpha>0$ at $\cH^{m-2+\alpha}$-a.e. $p\in \gammaup$ there is a flat tangent cone, and hence $Q =\Theta (T,p)+\frac{1}{2}$ is a positive integer. 
At $\cH^{m-1}$-a.e. $p\in \gammaup$ any flat tangent cone takes the form  $Q \a{\pi^+} + (Q-1)\a{\pi^-}$, where the plane $\pi$ is the unique plane containing $T_p \gammaup$ and the vector $\vec{n} (x)$ appearing in \eqref{e:first_var} (with the natural orientation). 
\end{corollary}

Finally, by the very same arguments of \cite[Theorem 35.3 (1)]{Sim} and a simple analysis of two dimensional tangent cones at the boundary,  one of the conclusions of the above corollary can be strengthened as follows. 

\begin{corollary}\label{c:3dim}\label{C:3DIM} For
every $\alpha>0$ and $\cH^{m-3+\alpha}$-a.a. $p\in \Gamma$, $\Theta (T,p)+\frac{1}{2}$ is a positive integer.
 \end{corollary}

\section{Stratification}

\begin{definition}
Let $p\in\gammaup$ and $T_0$ be a tangent cone at $p$. The \emph{spine}\index{Spine of a cone} $\spi(T_0)$\index{aals\spi(T_0)@$\spi(T_0)$} is the set of vectors $v\in T_p\gammaup$ such that $(\tau_v)_\sharp T_0=T_0$, where $\tau_v(q):=q+v$. \index{aagt\tau_v@$\tau_v$}
\end{definition}
We recall that the following conclusions are simple consequences of the monotonicity formula, cf. for instance \cite[Sections 3 \& 5]{White97}.
\begin{lemma}\label{lm:above}
$\spi(T_0)$ is a vector space and we have the following characterizations:
\begin{itemize}
\item[(a)] $v\in\spi(T_0)$ if and only if $\Theta(T_0,0)=\Theta(T_0,v)$;
\item[(b)] $v\in\spi(T_0)$ if and only if $(\iota_{v,r})_\sharp T_0=T_0$ for every $r>0$.
\end{itemize}
\end{lemma}
\begin{definition}\label{def:bdim}
Given a point $p\in\gammaup$, an area-minimizing current $T$ with boundary $\partial T=\gammaup$ and a tangent cone $T_0$ of $T$ at p, the \emph{building dimension}\index{Building dimension} $\bdim(T_0)$\index{aalb\bdim(T_0)@$\bdim(T_0)$} is the dimension of $\spi(T_0)$. We stratify the boundary $\gammaup$ according to the maximum of the building dimension of the tangent cones at the given point:\index{Strata}
\index{aals\sS_j (T, \gammaup)@$\sS_j (T, \gammaup)$}
\[
\sS_j(T,\gammaup):=\left\{p\in\gammaup:\,\bdim(T_0)\le j\text{ for every tangent cone }T_0\text{ at }p\right\}\,.
\]
\end{definition}

The following stratification result holds, cf. \cite[Theorem 5]{White97} (note that by definition $\spi(T_0) \subset T_p\gammaup$).
\begin{theorem}\label{thm:stratification}
$\sS_0(T,\gammaup)$ is at most countable, the Hausdorff dimension of each stratum $\sS_j(T,\gammaup)$ is at most $j$ and
\[
\sS_0(T,\gammaup)\subset\sS_1(T,\gammaup)\subset\ldots\subset\sS_{m-1}(T,\gammaup) = \gammaup\,.
\]
\end{theorem}
We close this section proving the following elementary but useful lemma.
\begin{lemma}\label{lem:flat_cone}
If $\bdim(T_0)=m-1$ then $T_0$ is flat.
\end{lemma}
\begin{proof}
Fix a tangent cone $T_0$ to $T$ at $p$ of maximal building dimension $m-1$ and observe that $\spi(T_0)=T_p\gammaup$. By a well-known result of Federer (cf. \cite[Section 5.4.8]{Fed}) there exists a one-dimensional area-minimizing current $S$  in $(T_p\gammaup)^\perp$ such that $T_0=\a{T_p\gammaup}\times S$. Note in particular that $\partial S=\a{0}$ and there exist $\ell_1^+,\ldots,  \ell_{Q-1}^+,\ell_Q^+$ and $\ell_1^-,\ldots,\ell_{Q-1}^-$ oriented half lines with endpoint at $0$ such that
\[
\partial\a{\ell_j^\pm}=\pm \a{0}\, ,
\]
\begin{equation}\label{e:decomp_1}
S=\sum_{i=1}^Q\a{\ell_i^+}+\sum_{j=1}^{Q-1}\a{\ell^-_j}\,
\end{equation}
and
\begin{equation}\label{e:decomp_2}
\|S\| = \sum_{i=1}^{Q} \left\|\a{\ell_i^+}\right\|+\sum_{j=1}^{Q-1}\left\|\a{\ell^-_j}\right\|\, , 
\end{equation}
cf. Figure \ref{fig:star}.

\begin{figure}[htbp]
\begin{center}\label{fig:star}
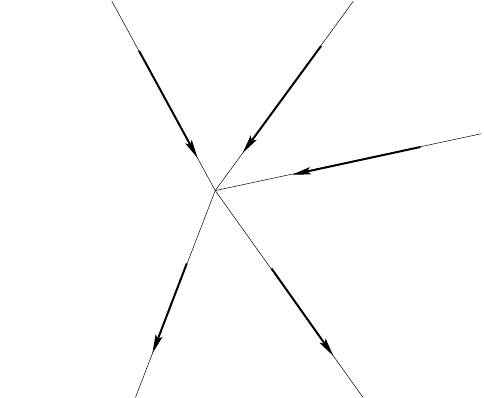
\end{center}
\caption{An example of current $S$ and oriented lines $\ell^\pm_j$ when $Q=4$: the arrows represent the oriented tangent to the lines. Note that pairs of lines $\ell^+_j, \ell_k^+$ and $\ell^-_j, \ell_j^+$ might coincide: in the example we have $\ell^+_1=\ell^+_2$ and $\ell^-_1=\ell^-_2$. However the support of any line $\ell^+_j$ can intersect the support of any line $\ell^-_k$ only at the origin, otherwise \eqref{e:decomp_2} would be violated.}
\end{figure}

In particular $\a{\ell_i^+}+\a{\ell^-_j}$ is an area-minimizing current without boundary for every $i,j$. But then we conclude the existence of a single one-dimensional vector space $\ell_{ij}$ such that $\supp(\a{\ell^+_i}+\a{\ell^-_j}) = \ell_{ij}$. Since this has to be valid for any choice of $(i,j)$, we then also conclude that the $\ell_{ij}$ coincide all with a single line $\ell$. Hence $\supp(T_0)\subset T_p\gammaup+\ell$, which shows the flatness of $T_0$. 
\end{proof}

\section{Proof of Theorem \ref{thm:step0}}

Fix an area minimizing current $T$ with boundary $\partial T=\a{\gammaup}$ and assume that $\bsing(T)$ has nonempty interior, which we denote by $G$. Define 
\[
C_i:=\left\{p\in\gammaup\colon \Theta(T,p)\ge i-\textstyle{\frac 12}\right\}\cap G\,.
\]
Recall that, by upper semicontinuity of the density, $C_i$ is relatively closed in $G$. Let $D_i$ be the interior of $C_i$ and $E_i:=D_i\setminus C_{i+1}$. If $p$ is not in $\bigcup_{i\ge 1}E_i$, then fix the natural number $i\ge 1$ such that
\[
i-\textstyle{\frac 12}\le \Theta(T,p)<i+\textstyle{\frac 12}
\]
and observe that therefore $p\in C_i\setminus D_i$. The latter is a relatively closed meager subset of $G$ and thus we conclude that $G\setminus \bigcup_i E_i$ is the union of countably many closed meager subsets of $G$. By the Baire Category Theorem $\bigcup_{i\ge 1} E_i$ cannot be empty. 

This means that at least one $E_i$ is not empty and, being relatively open in $\gammaup$, by the stratification Theorem \ref{thm:stratification} we conclude that $E_i$ contains a point $p\notin\sS_{m-2}$. By the Lemma \ref{lem:flat_cone} there is at least one flat tangent cone $T_0$ at $p$, which in turn implies the existence of a positive integer $Q$ such that $\Theta(T_0,p)=Q-\frac 12$. Observe that $p\in E_i\subset C_i\setminus C_{i+1}$ and, hence, $Q=i$. Being $E_i$ relatively open in $\gammaup$, there is a neighborhood $U$ of $p$ such that $U\cap\gammaup\subset E_i\subset C_i$. Therefore $\Theta(T,q)\ge \Theta(T,p)$ for every $q\in U\cap\gammaup$. Thus $p$ is a  collapsed point. On the other hand $p\in G$, namely it is a singular point.   \qed

\section{Proofs of Theorem \ref{thm:allard} and Corollaries \ref{c:uniqueness} and \ref{c:3dim}}\label{sec:prufs}

Statement (a)  is the classical monotonicity formula, which in fact holds in a much more general situation, see for instance \cite[Theorem 5.1(1)]{All}. Statement (b) follows from Allard's monotonicity formula at the boundary for varifolds, see \cite[Theorem 3.4(2)]{AllB}\footnote{For an alternative approach, similar to the one used for proving Theorem \ref{thm:limit_ff} we refer the reader to \cite[Section 4]{Matrix}}. The upper semicontinuity of the restriction of the density on the two sets $\gammaup$ and $\bB_2\setminus \gammaup$ is then a standard consequence, see for instance \cite[Corollary 17.8]{Sim}. 

Since $T$ is stationary with respect to variations which vanish on $\gammaup$ and are tangential to $\Sigma$, we have the usual identity
\[
\delta T (X) = - \int_{\bB_2} X \cdot \vec{H}_T (x)\, d\|T\| (x) \qquad \mbox{for all $X\in C^1_c (\bB_2 \setminus \gammaup)$},
\]
cf. for instance \cite[Lemma 9.6]{Sim}. 
Thus we can apply \cite[Lemma 3.1]{AllB} to the integer rectifiable varifold naturally induced by $T$ to conclude $\delta T = \vec{H}_T \|T\| + \delta T_s$ where $\delta T_s$ is a singular Radon measure supported in $\gammaup$. By the Radon-Nikod\'ym decomposition, if we denote by $\|\delta T_s\|$ the total variation of $\delta T_s$ we conclude the existence of a unit Borel vector field $\vec{n}$ such that
\begin{equation}\label{e:reg_and_sing}
\delta T (X) = - \int_{\bB_2} X \cdot \vec{H}_T (x)\, d\|T\| (x) + \int_\gammaup X\cdot \vec{n} (x)\, d \|\delta T_s\| (x)\, 
\end{equation}
for all $X\in C^1_c (\bB_2)$.
 Note next that, by the explicit formula for $\vec{H}_T$ in \eqref{e:mean_curv}, $\vec{H}_T (x)$ is orthogonal to $T_x \Sigma$ and in particular it is orthogonal to the tangent plane to $T$ at $x$. Thus in the first integral of the right hand side of \eqref{e:reg_and_sing} we can certainly substitute $X$ with $X^\perp$.

Moreover, according to \cite[Section 3.1]{AllB}, $\|\delta T_s\|$ satisfies the following upper bound for any positive $\psi \in C_c (\bB_2)$:
\[
\int_{\Gamma} \psi\,  d \|\delta T_s\| \leq \lim_{h\to 0} \frac{1}{h}\int_{\{x:  \dist (x, \gammaup) <h\}} \psi (x) d\|T\| (x).
\]
Hence it follows easily from the existence and boundedness of the density $\Theta_{\rm b} (T,p)$ that $\|\delta T_s\| = \theta \cH^{m-1}\res \gammaup$ for a locally
bounded Borel function $\theta$ with $0\leq \theta (p) \leq C(m) \Theta_{\rm b} (T, p)$ 

Now, we know from the previous sections that at $\cH^{m-1}$-a.e. $p$ there exists a flat tangent cone $S_p = Q \a{\pi^+} + (Q-1) \a{\pi^{-}}$, 
where $\pi$ contains $T_p \gammaup$. On the other hand we know from the convergence of the currents together with the convergence of the respective total variations  that the varifolds induced by $(\iota_{p,r})_\sharp T$ converge to the varifold induced by $S_p$. Thus, by continuity of the first variation, we conclude that
\[
\delta S_p (X) = \lim_{r\downarrow 0} \delta (\iota_{p,r})_\sharp T (X)\, .
\]
On the one hand simple computations lead to the identity
\[
\delta S_p (X) = \int_{T_p \gammaup} \nu \cdot X\, d\cH^{m-1}\, ,
\]
where $\nu$ is the unique unit vector contained in $\pi$ which is orthogonal to $T_p \gammaup$ and is compatible with the orientations of $\pi$ and $T_p \gammaup$. On the other hand, by a simple rescaling argument
\begin{equation}\label{e:limit_without_subs}
\lim_{r\to 0}  \delta (\iota_{p,r})_\sharp T (X) = \int_{T_p \gammaup} \theta (p) \vec{n} (p) \cdot X d\cH^{m-1}
\end{equation}
at $\cH^{m-1}$-a.e. $p$. We thus conclude $\vec{n} (p) = \nu$, and $\theta =1$. This argument proves the identity \eqref{e:first_var}, but it shows as well the validity of the last conclusion of Corollary \ref{c:uniqueness}: if we fix a point $p$ where \eqref{e:limit_without_subs} holds, we have actually shown that, for any flat tangent cone $Q \a{\pi^+} + (Q-1)\a{\pi^-}$ at that point, the vector $\vec{n} (p)$ must belong to $\pi^-$, which uniquely determines the pair $(\pi^+, \pi^-)$. Since $Q$ is uniquely determined as $\Theta (T, p) + \frac{1}{2}$, we conclude that any flat tangent cone at $p$ is determined by $\vec{n} (p)$.  
The identity of \eqref{e:monot_identity} is then a consequence of \cite[Eq. (31)]{Theodora}. Finally, the first assertion of Corollary \ref{c:uniqueness} is a consequence of Theorem \ref{thm:stratification} and of Lemma \ref{lem:flat_cone}.

\medskip

To prove Corollary \ref{c:3dim}, by Theorem \ref{thm:stratification} it suffices to show that the density is a half integer at every point $p\in \sS_{m-2} (T, \Gamma)$: the latter claim follows if we can show that every boundary area-minimizing cone $T_0$ with building dimension $m-2$ satisfies the property that $\Theta (T_0, 0)$ is a half-integer. The latter property is in effect of the following characterization.

\begin{lemma}[Characterization of 2 dimensional area minimizing cones with boundary]
Let $T_0$ be an integral 2-dimensional locally area-minimizing current in $\R^{2+k}$ with $(\iota_{0,r})_\sharp T_0 = T_0$ for every $r>0$ and $\partial T_0 =  \a{\Gamma_0}$, where $\Gamma_0 = \{(x,y)\in \R^2\times \R^{k}: x_1=|y|=0\}$, Then 
\[
T_0 = \a{\pi^+} + \sum_{i=1}^N \theta_i \a{\pi_i}
\] 
where 
\begin{itemize}
\item[(a)] $\pi^+$ is a closed oriented half-plane;
\item[(b)] the $\pi_i$'s are all oriented $2$-dimensional planes which can only meet at the origin;
\item[(c)] the coefficients $\theta_i$'s are all natural numbers;
\item[(d)] if $\pi^+\cap \pi_i\neq \{0\}$, then $\pi^+\subset \pi_i$ and they have the same orientation.
\end{itemize}
\end{lemma}

\begin{proof}
Let $|\cdot| : \mathbb R^{2+k}\to \mathbb R^+$ be the Lipschitz map $(x,y) \mapsto |(x,y)|$ and
consider the 1-dimensional integral current $ S:=\langle T_0, |\cdot|, 1\rangle$. Recall that, since $T_0$ is a cone, 
\begin{align*}
T_0\res \bB_1 &= S \cone \a{0}\, ,\\
T_0 & = \lim_{r\uparrow \infty} (\iota_{0,r})_\sharp\, (S\cone \a{0})\, ,
\end{align*}
Note moreover that, by the usual formula on the boundary of slices, 
\begin{equation}\label{e:tardi_2}
\partial S =  \langle \partial T_0, |\cdot|, 1\rangle = \a{e_1} - \a{-e_1}\, ,
\end{equation} 
where $e_1 = (1, 0, \ldots, 0)$. By \cite[4.2.25]{Fed} we have 
\[ S = \sum_{j=0}^N \theta_j \a{\gamma_j}, \]
where $\gamma_j$ is a simple Lipschitz curve, $\theta_j \in \N$ and $\gamma_j \neq \gamma_i$ for $i \neq j$ and 
\begin{equation}\label{e:tardi} 
\mathbf{M} (S) = \sum_{j=0}^N \theta_j \mathbf{M} (\a{\gamma_j}), \quad \mathbf{M} (\partial S) = \sum_{j=0}^N \theta_j \mathbf{M} (\partial \a{\gamma_j})\, .
\end{equation}
From the second identity in \eqref{e:tardi} and from \eqref{e:tardi_2} we conclude that there is precisely one $i$ for which $\pm \partial \a{\gamma_i} =\a{e_1}- \a{-e_1}$, whereas all the other curves $\gamma_j$'s are closed. Without loss of generality we assume that such $i$ is $0$ and note that
$\theta_0=1$, so that we can write
\begin{equation}\label{e:tardi_3}
S = \a{\gamma_0} + \sum_{j=1}^N \theta_j \a{\gamma_j}\, .
\end{equation} 
Consider now the currents $Z_j = \lim_{r \uparrow \infty} (\iota_{0, r})_\sharp (\theta_j \a{\gamma_j} \cone \a{0})$ and observe that:
\begin{align}
T_0  &= Z_0 + \sum_{i=1}^N Z_i, \quad \mathbf{M} (T_0\res \bB_R)\nonumber\\
&= \mathbf{M} (Z_0\res \bB_R) + \sum_{i=1}^N \mathbf{M} (Z_i \res \bB_R)\qquad \forall R>0\, .\label{e:tardi_4}
\end{align} 
In addition ${\rm Sing}_i (T_0)$ must be empty, otherwise it would have dimension at least $1$. Thus all the $\gamma_j$'s are disjoint great circles for $j=1, \ldots, N$ and $\gamma_0$ is half of a great circle. This gives (a), (b) and (c),
where we let $\pi^+$ be the half-plane containing $\gamma_0$ and $\pi_j$ be the plane containing $\gamma_j$. Note next that if $\pi^+ \cap \pi_j$ contains one point $p$ besides the origin, then
\begin{itemize}
\item If $p\not\in \Gamma_0$, then $\pi^+$ must be a subset of $\pi_j$ because otherwise $p$ would be an interior singular point of $T_0$;
\item If $p\in \Gamma_0$, then $S_0 + S_j$ is, by \eqref{e:tardi}, an area minimizing $2$-dim. cone with boundary $\a{\Gamma_0}$ and it has building dimension $1$; thus by Lemma \ref{lem:flat_cone} we have again $\pi^+\subset \pi_j$.
\end{itemize}
We thus conclude that $\pi^+\subset \pi_i$. The fact that both have the same orientation follows finally from the second identity in \eqref{e:tardi_4}. 
\end{proof}

\chapter{Regularity for $\qhalf$ $\D$-minimizers}\label{chap:Qhalf}

As explained in the introduction  the second important step in the proof of  Theorem \ref{thm:main} is the understanding of its  ``linearized'' version. This requires the study of the boundary regularity of $\D$-minimizers \(Q\)-valued map subject to a particular type of boundary condition, see Definition \ref{def:Qhalf} and Remark \ref{rmk:bcond} below.

\medskip

We assume the reader to be familiar with the theory of \(Q\) valued maps as it is presented in \cite{DS1,DS2,Jonas}. We just recall here that a \(Q\)-valued map is a map \(u:\Omega \subset \mathbb R^m \to \Iqs\)  where \index{aala\Iqs@$\Iqs$}
\[
\Iqs:=\left\{\sum_{i=1}^Q\a{P_i}:\,P_i\in\R^n,\,\forall\,i=1,\ldots,Q\right\}
\]
can be thought as the set of \(Q\)-tuples of unordered points in $\mathbb R^n$. \(\Iqs\) can be easily given the structure of a metric space via the following definition: given  $F_1,F_2\in \Iqs$ with $F_1=\sum_i \a{P_i}$ and $F_2=\sum_i \a{S_i}$ we define their distance\index{aalg\G@$\G$} as
\[
\G(F_1,F_2):=\min_{\sigma\in{\mathscr P}_Q}\sqrt{\sum_{i=1}^Q \left|P_i-S_{\sigma(i)}\right|^2}\,,
\]
where ${\mathscr P}_Q$ denotes the group of permutations of $Q$ items.

Throughout all the chapter we will consider an open set $\Omega\subset\R^m$ together with a hypersurface $\gammado$\index{aagc\gammado@$\gammado$} dividing $\Omega$ in two disjoint open sets $\Omega^+$ and $\Omega^-$\index{aagz\Omega^\pm@$\Omega^+$ (resp. $\Omega^-$)}.

\begin{definition}\label{def:Qhalf}
Let $\varphi\in H^\frac 12(\gammado,\R^n)$ be given. A \emph{$(Q-\frac 12)$-valued function with interface $(\gammado,\varphi)$}\index{Q@$(Q-\frac 12)$-valued function with interface $(\gammado,\varphi)$} consists of a pair $(f^+,f^-)$ \index{aalf(f^+,f^-)@$(f^+,f^-)$} with the following properties:
\begin{itemize}
\item[(i)] $f^+\in W^{1,2}(\Omega^+,\Iqs)$ and $f^-\in W^{1,2}(\Omega^-,\Is{Q-1})$; 
\item[(ii)] $f^+|_{\gammado}=f^-|_{\gammado}+\a{\varphi}$.
\end{itemize}
Its Dirichlet energy\index{Dirichlet energy of a $(Q-\frac 12)$-valued function} is defined to be the sum of the Dirichlet energies of $f^+$ and $f^-$.

Such a pair will be called \emph{$\D$-minimizing}\index{Dir-minimizing@$\D$-minimizing $\left(Q-\frac 12\right)$-valued function with interface $(\gammado,\varphi)$} if any other $\left(Q-\frac 12\right)$-valued function with interface $(\gammado,\varphi)$ which agrees with $(f^+,f^-)$ outside of a compact set $K\subset\Omega$ has bigger or equal Dirichlet energy.
\end{definition} 

\begin{figure}[htbp]
\begin{center}\label{fig:Q-1/2}
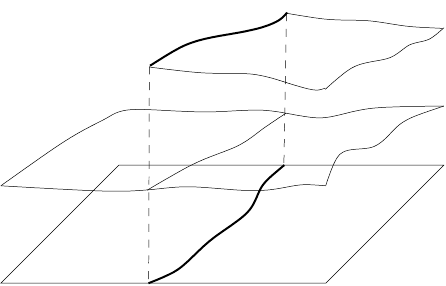
\end{center}
\caption{A \(\frac 32\)-valued function with interface $(\gammado, \varphi)$: the function $f^+$ is the $2$-valued map $\a{f^+_1}+\a{f^+_2}$ and $f^-$ coincides with the (classical) single-valued $f_1^-$.}
\end{figure}

Although the definition makes sense also for $Q=1$, notice that, in that case, the pair $(f^+,f^-)$ consists of a single-valued function $f^+$ and its $\D$-minimality is equivalent to the harmonicity of $f^+$. In this chapter we will focus on the nontrivial case $Q\ge 2$.

The first result of this chapter is a ``soft'' existence theorem for $\qhalf$-valued $\D$-minimizers.

\begin{theorem}\label{thm:ex_dm}\label{THM:EX_DM}
Given a $\left(Q-\frac 12\right)$-valued function $(g^+,g^-)$ with interface $(\gammado,\varphi)$ on a bounded Lipschitz domain $\Omega$, there exists a $\qhalf$ $\D$-minimizer $(f^+,f^-)$ with interface $(\gammado,\varphi)$ such that $f^+=g^+$ on $\partial\Omega^+\setminus\gammado$ and $f^-=g^-$ on $\partial\Omega^-\setminus\gammado$.
\end{theorem}

A particular class of $\left(Q-\frac 12\right)$-valued functions with interface $(\gammado,\varphi)$ are the ones with  collapsed interface.

\begin{definition}\label{def:totally_collapsed}
A $\left(Q-\frac 12\right)$-valued function with interface $(\gammado,\varphi)$ is said to \emph{collapse at the interface}\index{Collapsed at the interface} if $f^+|_\gammado= Q\a{\varphi}$.
\end{definition}
\begin{remark}
Observe that $(f^+,f^-)$ collapses at the interface  if and only if $f^-|_\gammado=(Q-1)\a{\varphi}$.
\end{remark}

\begin{figure}[htbp]
\begin{center}\label{fig:collapsed}
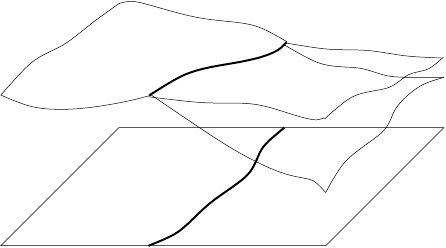
\end{center}
\caption{A  $\frac{3}{2}$-valued function which collapses at the interface $(\gammado, \varphi)$.}
\end{figure}

The main theorem of this chapter is the following:
\begin{theorem}\label{thm:collasso}\label{THM:COLLASSO}
Let $\varphi:\gammado\to\R^n$ be of class $C^{1,\alpha}$, $\gammado$ be of class $C^3$, $Q\ge 2$ and $(f^+,f^-)$  be a $\left(Q-\frac 12\right)$-valued $\D$-minimizer with interface $(\gammado,\varphi)$. If $(f^+,f^-)$ collapses at the interface, then there is a single-valued harmonic function $h:\Omega\to\R^n$ such that $f^+=Q\a{h|_{\Omega^+}}$ and $f^-=(Q-1)\a{h|_{\Omega^-}}$.   
\end{theorem}

Note that the above theorem is the ``linearized'' version of Theorem \ref{thm:real_main}. Note also that we are requiring \(C^3\) regularity of \(\gamma\), this seems to be due to our method of proof more then to a serious technical obstruction, see Section \ref{sss:gv} below. However Theorem \ref{thm:collasso} is enough for our purposes because the boundary data $\Gamma$ is assumed to be of class $C^{3,a_0}$ in Assumption \ref{ass:main}.

\section{Preliminaries and proof of Theorem \ref{thm:ex_dm}}

In this Section we prove  existence of $\D$-minimizing $(Q-\frac 12)$-valued functions.

\begin{proof}[Proof of Theorem \ref{thm:ex_dm}]
Take a minimizing sequence $(f_k^+,f_k^-)$ with interface $(\gammado,\varphi)$ and $f^{\pm}_k=g^\pm$ on $\partial\Omega^\pm\setminus\gammado$. It is simple to see that $f^\pm_k$ enjoy a uniform bound in $L^2(\Omega^\pm)$. For instance, consider the bi-Lipschitz embeddings \index{aago\xii_{Q}@ $\xii_{Q}, \xii_{Q-1}$}
 \[
 \xii_{Q}:\Iqs\to\R^{N(Q,n)}, \qquad \xii_{Q-1}:\Iqqs\to\R^{N(Q-1,n)}
 \]
 of \cite[Theorem 2.1]{DS1}. Then it suffices to bound the $L^2$ norm of $\xii_{Q}\circ f_k^+$, $\xii_{Q-1}\circ f_k^-$ and the latter bounds are a simple consequence of the classical Poincar\'e inequality using the uniform $H^\frac 12$-bound for the restriction of $\xii\circ f^\pm_k$ to $\partial\Omega^\pm\setminus\gammado$.

By \cite[Proposition 2.11]{DS1} we can extract a subsequence (not relabeled) such that $f_k^+$ and $f_k^-$ converge strongly in $L^2$ to $W^{1,2}$ functions $f^+$ and $f^-$, respectively. By continuity of the trace operator (cf. \cite[Proposition 2.10]{DS1}) the pair $(f^+,f^-)$ has interface $(\gammado,\varphi)$ and coincides with $(g^+,g^-)$ on the boundary of $\Omega$. By lower semicontinuity of the Dirichlet energy (cf. \cite[Section 2.3.2]{DS1}),
\[
\D(f^+,\Omega^+)+\D(f^-,\Omega^-)\le\liminf_{k\to +\infty}\left(\D(f_k^+,\Omega^+)+\D(f_k^-,\Omega^-)\right)\,.
\]
This obviously implies that $(f^+,f^-)$ is one of the sought minimizers.
\end{proof}

Next we record the following continuity property for $\qhalf$ $\D$-minimizers which collapse at the interface. The property is a direct consequence of the main result in \cite{Jonas}. Note that, from now on, for every metric space $(X,d)$ and any map $f: \Omega \to X$ we will use the notation $[f]_{\beta,K}$ for the H\"older seminorm of the restriction of $f$ to the subset $K\subset \Omega$, more precisely
\index{H\"older seminorm}\index{aalf[f]_{\beta,K}@$[f]_{\beta,K}$}
\[
[f]_{\beta, K} := \sup_{x,y\in K, x\neq y} \frac{d (f(x), f(y))}{|x-y|^\beta}\, .
\]

\begin{theorem}\label{thm:cont_dm}
If $\gammado$ is of class $C^1$ and $\varphi$  of class $C^{0,\beta}$, with $\beta>\frac 12$, then there exist a positive constant $C=C(m,n,\gammado,Q)$ and a positive constant $\alpha=\alpha(m,n,Q,\beta)$ with the following property. Consider a $\qhalf$ $\D$-minimizer which collapses at the interface $(\gammado, \varphi)$.
Then the following estimates hold for every $x\in\Omega^+\cup\gammado$, respectively $x\in\Omega^-\cup\gammado$, and every $0<2\rho<\dist(x,\partial\Omega)$:
\begin{align*}
[f^\pm]_{\alpha,B_\rho(x)\cap\Omega^\pm}&\le C \rho^{1-\frac n2-\alpha}\left(\D(f^\pm,B_{2\rho}(x)\cap\Omega^\pm)\right)^\frac 12\\
&\qquad\qquad +C \rho^{\beta-\alpha}[\varphi]_{\beta,\gammado\cap B_{2\rho}(x)}\,.
\end{align*}
\end{theorem}

An  outcome of the proof of Theorem \ref{thm:cont_dm} in \cite{Jonas} is the following compactness statement:
\begin{lemma}\label{lem:cpt_dm}
Let $(f_k^+,f^-_k)$ be a sequence of $\qhalf$ $\D$-minimizers in $\Omega$ which collapse at the interfaces $(\gammado_k,\varphi_k)$ and satisfy the following assumptions:
\begin{itemize}
\item[(i)] $\limsup_{k\to +\infty}\left(\D(f_k^+)+\D(f_k^-)\right)<\infty$;
\item[(ii)] $\gammado_k$ is converging in $C^1$ to a hyperplane $\gammado$;
\item[(iii)] $\varphi_k$ is converging\footnote{By this we mean that for every \(k\)  there is a  \(C^{0,\beta}\) extension \(\tilde \varphi_k\) of \(\varphi_k\big|_{ \gammado_k}\) to the whole  \(\R^m\) such that the sequence \(\{\tilde \varphi_k\}\) converges  to  a constant function} in $C^{0,\beta}$ to a constant function $\varphi$ for some $\beta>\frac{1}{2}$.
\end{itemize} 
Then there exists a subsequence (not relabeled) and a $\qhalf$-valued function $(f^+,f^-)$ with interface $(\gammado, \varphi)$ such that
\begin{itemize}
\item[(a)] $f_k^\pm\to f^\pm$ in $L^2(K)$ for every compact set $K\subset\Omega^\pm$.
\item[(b)] $\D(f^\pm,\Omega^\pm\cap \Omega')=\lim_k \D(f_k^\pm,\Omega^\pm_k\cap\Omega')$ for every $\Omega'\subset\subset\Omega$, where $\Omega^\pm_k$ denote the two open domains in which $\Omega$ is subdivided by $\gammado_k$;
\item[(c)] $f^+$ is $\D$-minimizing in $\Omega^+$ and $f^-$ is $\D$-minimizing in $\Omega^-$.  
\end{itemize} 
\end{lemma}

In turn we can take advantage of a standard blow-up argument to upgrade Lemma \ref{lem:cpt_dm} to the following more general statement, where the convergence in (c) is to a general hypersurface $\gamma$ and we conclude additionally that the limiting $(f^+, f^-)$ is $\D$-minimizing as a $\qhalf$ map.

\begin{theorem}\label{thm:cpt_dm}
Let $\Omega$ be bounded and let $(f_k^+,f^-_k)$ be a sequence of $\qhalf$ $\D$-minimizers in $\Omega$ which collapse at the interfaces $(\gammado_k,\varphi_k)$ and satisfy the following assumptions:
\begin{itemize}
\item[(i)] $\limsup_{k\to +\infty}\left(\D(f_k^+)+\D(f_k^-)\right)<\infty$;
\item[(ii)] $\gammado_k$ is converging in $C^1$ to a hypersurface $\gammado$;
\item[(iii)] $\varphi_k$ is converging in $C^{0,\beta}$ to a function $\varphi$ for some $\beta>\frac{1}{2}$.
\end{itemize} 
Then there exist a subsequence (not relabeled) and a $\qhalf$-valued function $(f^+,f^-)$ with interface $(\gammado, \varphi)$ such that the conclusions (a) and (b) of Lemma \ref{lem:cpt_dm} apply. Moreover $(f^+, f^-)$ is a $\qhalf$ $\D$-minimizer which collapses at the interface. 
\end{theorem}

Before coming to the proof of the latter theorem we need two important technical ingredients.

\subsection{Interpolation lemma}

The following technical lemma allows to ``glue'' together two different functions and will be instrumental to several  proofs:

\begin{lemma}[Interpolation]\label{l:interpolation}
Let $U \subset \R^m$ be a domain with smooth boundary $\partial U$ and let  $\gammado\subset \R^m$ be a smooth interface that intersects $\partial U$ transversally and divides \(U\) into two subdomains \(U^\pm\).  Then for every  compact subset $K \subset U$ there exist constants $C, \lambda_0>0$ depending on
\begin{itemize}
\item $m, Q$, $K$,
\item the \(C^2\) regularity of \(U\) and  \(\gammado$,
\item  and \(\min \{|T_x\partial U - T_x \gammado| : x\in \gammado\cap \partial U\}\),
\end{itemize}
such that the following holds.

Let $(f^+,f^-), (g^+,g^-)$ be two $\qhalf$-valued maps in $U$ with interface $(\gammado, \varphi|_\gammado)$ for some $\varphi \in W^{1,2}(U)$. Additionally we assume that $(f^+, f^-)$ collapses at the interface.  Then for every $0<\lambda <\lambda_0$ there exist open sets $K \subset V_\lambda \subset W_\lambda \subset U$ and a $\qhalf$-valued map $(\zeta^+, \zeta^-)$ in $W_\lambda \setminus V_\lambda$ 
 with the following properties:
 \begin{itemize}
 \item[(a)] $\zeta^\pm(x)= \begin{cases} f^\pm(x), &\text{ if } x \in \partial W_\lambda^\pm \\ g^\pm(x), &\text{ if } x \in \partial V^\pm_\lambda \end{cases}$;
 \item[(b)] $\zeta \text{ has interface } (\gammado, \varphi|_\gammado )$;
 \item[(c)] the following estimate holds
 \begin{align}\label{e:raccordo_Dir}
 \int_{W_\lambda^\pm \setminus V_\lambda} \abs{D\zeta^\pm}^2 &\le  C \lambda \int_{U^\pm \setminus K} \left(\abs{D f^\pm}^2+ \abs{D g^\pm}^2 + Q \abs{D\varphi}^2\right)\nonumber\\
&\qquad+ \frac{C}{\lambda} \int_{U^\pm\setminus K} \G(f^\pm,g^\pm)^2.
 \end{align}
 \end{itemize}
If in addition $f$ and $g$ are Lipschitz then $\zeta$ can be chosen to satisfy
\begin{equation}\label{e:lip approx}
\operatorname{Lip}(\zeta^\pm) \le C \left( \operatorname{Lip}(f^\pm) + \operatorname{Lip}(g^\pm) + \frac{1}{\lambda} \sup_{x \in U \setminus K} \G(f^\pm, g^\pm)(x)  \right) .
\end{equation}
\end{lemma}

\begin{remark}
	If $U=B_1 \subset \R^m$, we can take any $\lambda_0 \leq \frac{1}{4}$ and we may assume that $V_\lambda = B_{s-\lambda}$ and $W_\lambda= B_{s}$ for some $s \in ]1-\lambda_0, 1[$, while the constant $C$ in the estimates depends only on $m,n,Q$. Furthermore, with an obvious scaling and translation argument, we can get a corresponding statement for $U= B_r (x)$. 
\end{remark}

\begin{proof}
We divide the proof in some steps:

\medskip

\emph{Step 1: Choice of "cylindrical" coordinates around $\partial U$:} We may assume that there is a smooth function $d$ such that:
\begin{itemize}
\item $U=\{ d >0 \}$;
\item $0$ is a regular value of $d$.
\end{itemize}
In particular there is $\eta>0$ such that 
\begin{equation}\label{e:regular_value}
\abs{ \nabla d(x) } >\eta  \quad \mbox{in a neighborhood of $U'$ of $\partial U$.}
\end{equation} 
As it will be customary in the sequel, we will use the symbol $\bp_\pi$\index{aalp\bp_\pi@$\bp_\pi$} to denote the orthogonal projection onto a plane $\pi$.  
By assumption $\gammado$ intersects $\partial U$ transversally:  hence, possibly choosing $\eta>0$ and $U'$ smaller, we can also assume
\begin{equation}\label{e:transversal}
\left|\bp_{T_x \gamma} (\nabla d (x))\right|\geq \eta\qquad \forall x\in \gammado \cap U'\, .
\end{equation} 
In order to simplify our notation from now on we will set $(\nabla d (x))^T = \bp_{T_x \gamma} (\nabla d (x))$.

The inequalities above imply that we can define a smooth vectorfield  $X$ in a neighborhood $V$ of $\partial U$ with the following properties:
\begin{itemize}
	\item[(A)] $\abs{X}=1$  and $\langle \nabla d(x), X(x) \rangle > \frac{\eta}{2}$ for all $x \in V$;
	\item[(B)] $X= \frac{ (\nabla d(x))^T }{\abs{(\nabla d(x))^T}}$ for all $x \in V \cap \gammado$. 
\end{itemize}
Let $\psi: V \times [-t_0, t_0] \to \R^m$ be the flow generated by $X$. Hence the map
\[ (y, t) \in \partial U \times  [-t_0, t_0]  \mapsto \psi(y,t) \]
gives a parametrization of a neighborhood $V'$ of $\partial U$ with the additional property that 
\begin{equation}\label{eq:straight interface}
	\psi(y,t) \in \gammado \text{ for all } (y,t) \in \gammado \cap \partial U \times [0, t_0]. 
\end{equation}
Possibly decreasing $t_0$, we may assume that $\psi( \partial U \times ]0, t_0[ ) \subset U \setminus K$. 

\medskip

\emph{Step 2: Reduction to $\varphi=0$.} Instead of considering $f,g$ directly, we look first at the two functions
\[ \tilde{f}^\pm:= \sum_{i} \a{f^\pm_i - \varphi}, \qquad \tilde{g}^\pm:= \sum_{i} \a{g^\pm_i - \varphi}. 
\]
Note that they satisfy  the same assumptions of  $f$ and $g$ but with  interface $(\gammado, 0)$. Furthermore, one readily checks that 
\begin{equation}\label{e:energy shift by varphi} \abs{D\tilde{f}^\pm}^2(x) \le 2 \abs{Df^\pm}^2(x) + 2Q \abs{D\varphi}^2(x) \end{equation}
and similarly for $\tilde{g}$. Additionally we have that
\[ \G(\tilde{f}^\pm, \tilde{g}^\pm) = \G(f^\pm, g^\pm).
\]

\medskip

\emph{Step 3: Choice of $V_\lambda\subset W_\lambda$ and definition of $\tilde{\zeta}$ for $\tilde{f}, \tilde{g}$.} Define next
\begin{align*}
\bar f^\pm (y,t) &:= \tilde{f}^\pm (\psi (y,t))\\
\bar g^\pm (y,t) &:= \tilde{g}^\pm (\psi (y,t))\quad
\mbox{and}\\
\bar\varphi (y,t) &:= \varphi (\psi (y,t) \, .
\end{align*}
Set now $\lambda_0 := t_0$, let $\lambda$ be a positive number smaller than $\lambda_0$ and select the natural number $N$ such that $N \lambda \leq t_0 < (N+1) \lambda$. For our purposes, by making $t_0$ slightly smaller, from now on we can assume $\lambda = \frac{t_0}{N}$. Consider the disjoint intervals $I_j:=[(j-1) \frac{t_0}{N}, j \frac{t_0}{N}[$ for $j=1, \dotsc, N$. Then there must be at least one $j \in \{1, \dots, N-1 \}$ such that 
\begin{align*} 
\int_{(\partial U)^\pm \times I_j}  \abs{D \bar{f}^\pm}^2+ \abs{D \bar{g}^\pm}^2  &\le 8 \lambda \int_{(\partial U)^\pm \times [0, t_0]} \abs{D \bar{f}^\pm}^2+ \abs{D \bar{g}^\pm}^2 \,\\
\int_{(\partial U)^\pm \times I_j}  \G(\bar{f}^\pm, \bar{g}^\pm)^2  &\le 8 \lambda \int_{(\partial U)^\pm \times [0, t_0]} \G(\bar{f}^\pm, \bar{g}^\pm)^2 \,.
\end{align*}
If $\varphi \neq 0$ we require additionally that 
\begin{equation}\label{e.slice for phi} 
\int_{(\partial U)^\pm \times I_j} \abs{D\bar \varphi}^2 \le 8 \lambda \int_{(\partial U)^\pm \times [0,t_0]} \abs{D\bar \varphi}^2\,. 
\end{equation}
Fix such a $j$ and define 
\[
V_\lambda:= U \setminus \psi\Big( \partial U \times [0, j t_0/N]\Big)
\qquad
W_\lambda:= U \setminus \psi\Big( \partial U \times [0, (j-1)t_0/N]\Big),
\]
so that 
 \[
 W_\lambda \setminus V_\lambda = \psi\Big( \partial U \times ](j-1) t_0/N, j t_0/N]\Big).
 \]
We consider the Almgren  embedding  $\xii_{Q}: \Iqs \to \R^{N(Q,n)}$ (resp. $\xii_{Q-1}: \Iqqs\to \R^{N(Q-1,n)}$) and the retraction   \(\boldsymbol{\rho}_{Q}: \R^{N(Q,n)}\to \xii_Q(\Iqs)\) (resp. \(\boldsymbol{\rho}_{Q-1}\)) as in   \cite[Theorem 2.1]{DS1}.  We then define the functions  $\bar{\zeta}^+$ as
\[ \bar{\zeta}^+(y,t) = \xii_{Q}^{-1}\circ \boldsymbol{\rho}_{Q} \left( \frac{ j \lambda - t}{\lambda} \boldsymbol{\xi}_{Q}(\bar{f}^+(y,t)) + \frac{ t- (j-1)\lambda}{\lambda} \boldsymbol{\xi}_{Q}(\bar{g}^+(y,t)) \right).\]
and analogously for \(\bar{\zeta}^-\). Finally, we set $\tilde \zeta (x) := \zeta (\psi^{-1} (x))$. The estimates \eqref{e:raccordo_Dir} and \eqref{e:lip approx} are then routine calculations for the case $\varphi =0$. Hence, it remains to check that $(\tilde{\zeta}^+, \tilde{\zeta}^-)$  has interface $(\gammado, 0 )$ , namely  that  
\[
\bar{\zeta}^+ (y,t) = \bar{\zeta}^- (y,t) + \a{0} \qquad \mbox{whenever $x = \psi (y,t) \in \gammado$.}
\]

Fix thus $(y, t)\in \partial U  \times ](j-1)\lambda , j\lambda]$ such that $x = \psi (y,t) \in \gammado$ and observe that, since $\bar f^+ (y,t) = \tilde{f}^+ (x) = Q \a{0}$, $\bar f^- (y,t) = \tilde{f}^- (x) = (Q-1) \a{0}$,  and $\boldsymbol{\xi}_{Q}(Q\a{0}) =0$, we have
\[
\bar{\zeta}^+ (y, t) = \boldsymbol{\xi}_{Q}^{-1}\circ \boldsymbol{\rho}_Q \left(\frac{ t- (j-1)\lambda}{\lambda} \boldsymbol{\xi}(\bar{g}^+(y,t)) \right)\, .
\]
and the same for \(\bar{\zeta}^-\).
Note next that $\boldsymbol{\xi}_Q (\Iqs)$ is a cone and in fact
\[
\boldsymbol{\xi}_Q \left(\sum_i \a{\lambda T_i}\right) = \lambda \boldsymbol{\xi}_Q \left(\sum \a{T_i}\right)\, .
\]
We therefore conclude
\[
\bar{\zeta}^+ (y, t) = \sum_i \a{ \frac{ t- (j-1)\lambda}{\lambda} (\bar{g}^{+})_i (y, t)} \, .
\]
and the same for \(\bar{\zeta}^- (y, t)\). Since $\bar{g}^+ (y, t) = \bar{g}^- (y, t) + \a{0}$ we conclude as well that
$\bar{\zeta}^+ (y, t) = \bar{\zeta}^- (y, t) + \a{0}$.

\emph{Step 4: The general case}. 
To conclude the proof we finally define 
\[ \zeta^\pm(x):= \sum_i \a{ \tilde{\zeta}^\pm_i(x) + \varphi(x) } .\]
One readily checks that $\zeta$ satisfies the claimed boundary values and has interface $(\gammado, \varphi)$. Using once again \eqref{e:energy shift by varphi} for $\zeta$ and exploiting also \eqref{e.slice for phi}, we obtain the estimates \eqref{e:raccordo_Dir} and \eqref{e:lip approx}.
\end{proof}

\subsection{A simple measure theoretical lemma}
The second technical ingredient is the following simple measure theoretic fact.

\begin{lemma}\label{lm:mt1}Let $\mu$ be a Radon measure supported in a $C^1$ $k$-dimensional submanifold $M$ of some Euclidean space. Set
\[
A:= \left\{ x \in \supp(\mu) \colon \liminf_{r \to 0} \frac{ \mu(B_r(x))}{ r^{k} } > 0 \right\}\, 
\]
and 
\[ 
B:=  \left\{ x\in \supp(\mu) \colon \limsup_{r \to 0 } \frac{ \mu(B_{r}(x))}{\mu(B_{2r}(x))} \ge 2^{-k}  \right\}\, .
\]
Then $\mu(M\setminus A) = 0=\mu(M\setminus B)$.
\end{lemma}

\begin{proof}
Since the statements can be easily localized and a \(C^1\) change of variable would not affect them, we can assume w.l.o.g. that \(M=\mathbb R^{k}\). By Radon-Nikod\'ym Theorem we can decompose  $\mu$ as  
\[
\mu_a + \mu_s = f dx + \mu_s\, \,
\] 
where $dx$ is the \(k\)-dimensional Lebesgue measure, $f$ is a nonnegative $L^1$ function and $\mu_s$ is a singular measure with respect to Lebesgue. Moreover, for $\mu_s$-a.e. $x$ we have
\[
 \lim_{r \to 0} \frac{ \mu(B_r(x))}{\omega_k r^k } = \infty
\]
and for $\mu_a$-a.e. $x$ we have
\[
 \lim_{r \to 0} \frac{ \mu(B_r(x))}{\omega_k r^k } = f(x) >0\, .
\]
Combining the above facts one immediately gets that \(\mu(A^c)=0\).

To prove the second claim assume by contradiction that there exists \(\varepsilon_0>0\) such that the set 
\[
B^{\varepsilon_0}=  \left\{ x\in \supp(\mu) \colon \limsup_{r \to 0 } \frac{ \mu(B_{r}(x))}{\mu(B_{2r}(x))} \le 2^{-k}(1-2 \varepsilon_0)  \right\}
\]
has positive measure. Since for all \(x_0\in B^{\varepsilon_0}\) there exists \(r_0\) such that
\[
\mu(B_{r}(x_0))\le 2^{-k}(1-\varepsilon_0) \mu(B_{2r}(x_0)) \qquad \textrm{for all \(r\in (0,r_0]\)},
\] 
one easily get that, for all \(\j\ge 1\)
\[
\frac{\mu(B_{2^{-j}r_{0}}(x_0))}{2^{-kj}r_0^k}\le (1-\varepsilon_0)^l\frac{\mu(B_{r_{0}}(x_0))}{r_0^k}.
\]
Hence, letting \(j\to \infty\), \(B^{\varepsilon_0}\subset A\), a contradiction with \(\mu(B^{\varepsilon_0})>0\).
\end{proof}

\begin{remark}\label{r:little_measure_theory}
Note that, as a consequence of the above Lemma, for \(\mu\)-a.e. \(x\) there exists a vanishing sequence $\{r_j\}$ such that 
\[
\lim_{j \to \infty} \frac{ \mu(B_{r_j}(x))}{\mu(B_{2r_j}(x))} \ge 2^{-k}.
\]
Recall moreover that $\mu(\partial B_{s}(y)) \neq 0$ for only countably many radii $s$. Since
\[
\lim_{s\uparrow r} \mu (B_s (x)) = \mu (B_r (x))\, ,
\]
we can choose $s_j < r_j$ so close to $r_l$ to ensure
\[
\lim_{j \to \infty} \frac{ \mu(B_{s_j}(x))}{\mu(B_{2s_j}(x))} =
\lim_{j \to \infty} \frac{ \mu(B_{r_j}(x))}{\mu(B_{2r_j}(x))} \ge 2^{-k}.
\]
and at the same time enforce the additional property $ \mu( \partial B_{2s_j}(x)) = 0 = \mu( \partial B_{s_j}(x))$.
\end{remark}

\subsection{Proof of Theorem \ref{thm:cpt_dm}: Compactness}:  Let $(f_k^+, f_k^-)$ be a sequence of $\qhalf$- $\D$-minimizers satisfying the assumption of the theorem. As in the proof of Theorem \ref{thm:ex_dm}, we can extract a subsequence such that $f^\pm_k $ converges strongly in $L^2$ to a $W^{1,2}$ function $f^\pm$ with $\D(f^\pm, \Omega^\pm) \le \liminf_k \D(f_k^\pm, \Omega_k^\pm) $. It remains to prove that, when $\Omega' \subset \Omega$ we actually have
\[
\D (f^\pm, \Omega^\pm \cap \Omega') = \lim_{k\to\infty} \D (f_k^\pm, \Omega_k^\pm \cap \Omega')\, .
\]
The argument is the same for $f^+$ and $f^-$ and for simplicity we focus on $f^+$.

Possibly passing to a further subsequence, we may assume that the sequence of Radon measures $\mu_k$ defined by $\mu_k(A):= \D(f_k^+, A\cap \Omega^+_k)$ converges, weakly$^\star$ in the sense of measures, to some $\mu$. By lower semicontinuity of the Dirichlet energy there is then a nonnegative ``defect measure $\nu$'' such that
\[
\mu(A) = \D( f^+, A\cap \Omega^+) + \nu (A)\qquad \mbox{for all Borel $A \subset\subset \Omega$.}
\] 
The goal is to show that $\nu=0$ and we therefore assume, by contradiction, that $\nu>0$. Observe that $\nu$ must be supported in $\gammado$, because in the interior of $\Omega^+$ we can appeal to \cite[Proposition 3.20]{DS1}. We can then apply  Lemma  \ref{lm:mt1} (with \(M=\gammado\)) and the Remark \ref{r:little_measure_theory} to find that at $\nu$-a.e. point $x_0 \in \supp(\nu)$ there is a sequence $r_j \downarrow 0$ such that: 
\begin{equation}\label{e:to_recall_alpha} 
\begin{split}
\liminf_{l\to \infty} \frac{ \nu(B_{r_j}(x_0))}{\omega_{m-1} r_{l}^{m-1}}& \geq \alpha > 0\\ 
 \nu(B_{r_j}(x_0)) & \leq  (2^{m-1}+o(1)) \nu( B_{r_j/2}(x_0)),
\\
\nu(\partial B_{r_j}(x_0)) &=0 =\nu( \partial B_{r_j/2}(x_0)).
\end{split}
\end{equation}
Moreover, since $\nu$ is singular with respect to the Lebesgue measure,  we also have 
\[
\frac{ \mu(B_{r_j}(x_0))}{\nu(B_{r_j}(x_0))} = 1 + o(1)
\]
for $\nu$-a.e. $x_0$.

\medskip

We thus fix an $x_0$ and a sequence $r_j$ with the properties above and also assume, after applying a suitable rotation, that  the blow up $\iota_{x_0, r_j}(\gammado)$ converges to the hyperplane $\gammado_0 = \{ x_m =0 \}$.
We next consider the sequences\footnote{In order to simplify our formulas, we will use the following abuse of notation: if $f = \sum_i \a{f_i}$ is a multivalued map and $\lambda$ is a classical real valued function, we will denote by $\lambda f$ the map $x\mapsto \sum_i \a{\lambda f_i (x)}$.}
\[ 
g_j(x) = \frac{ f^+(x_0 + r_j x) }{ \left(r_j^{m-2} \nu(B_{r_j}(x_0)\right)^{\frac 12}}\quad \text{ and }\quad  h_j(x) = \frac{ f_{k(j)}^+(x_0 + r_j x) }{ \left(r_j^{m-2} \nu(B_{r_j}(x_0)\right)^{\frac 12}},
 \]
where we have chosen $k(j)$ sufficient large such that
\begin{align}
 &\max\{ | \mu_{k(j)}(B_{r}(x_0)) - \mu(B_r(x_0))| \colon  r = r_j, r_j/2 \} \le 2^{-l} r_j^{m-2}\nu(B_{r_j}(x_0) )\, ;\nonumber\\
&\int_{B_{r_j}(x_0) \cap \Omega_{k(j)}^+ \cap \Omega^+} \G(f_{k(l)}^+, f^+)^2 \le 2^{-l} r_j^{m-2}\nu(B_{r_j}(x_0))\, .\nonumber
\end{align}
Furthermore the choice of $k(j)$ ensures that 
\[
\D(h_j, \Omega_{k(j)}^+ \cap B_1) = \frac{ \mu_{k(l)}( B_{r_j}(x_0))}{ \nu(B_{r_j}(x_0))}= 1 + o(1)
\] 
and
\[
\int_{B_1 \cap \{x_m >0 \}} \G(g_j, h_j)^2 \le 2^{-j}\, .
\] 
Note that $h_j$ and $g_j$ are $\qhalf$ $\D$ minimizers which collapse at their   interfaces $(\tilde{\gammado}_j, \tilde{\varphi}_j)$ and $(\hat{\gammado}_j, \hat{\varphi}_j)$, respectively, where   $ \tilde{\gammado}_j := \iota_{x_0, r_j}(\gammado) $,  $ \hat{\gammado}_j := \iota_{x_0, r_j}(\gammado_{k(l)}) $ and
 \[ \tilde{\varphi}_j(x) = \frac{ \varphi(x_0 + r_j x) }{ \left(r_j^{m-2} \nu(B_{r_j}(x_0)\right)^{\frac 12}} \quad\text{ and } \quad \hat{\varphi}_j(x) = \frac{ \varphi_{k(l)}(x_0 + r_j x) }{ \left(r_j^{m-2} \nu(B_{r_j}(x_0)\right)^{\frac 12}}\, . \]
Note that, as $l\to\infty$, $\tilde{\gammado}_j, \hat{\gammado}_j \to \gammado_0$ in $C^1$. Moreover $\tilde{\varphi}_j , \hat{\varphi}_j \to \varphi(x_0)$ in $C^\beta$, since,  thanks to \eqref{e:to_recall_alpha},
\[ [ \hat{\varphi}_j]_{\beta, \hat{\gammado_j}\cap B_1}= \frac{ r_j^\beta [\varphi_{k(l)}]_{\beta, \gammado_{k(l)} \cap B_{r_j}(x_0)} }{ \left(r_j^{m-2} \nu(B_{r_j}(x_0))\right)^{\frac 12}}  \le \frac{ r_j^\beta }{\alpha r_j^{\frac 12}}   [\varphi_{k(l)}]_{\beta, \gammado_{k(l)} \cap B_{r_j}(x_0)}\]
and $\beta > \frac{1}{2}$ (and  similarly for $\tilde{\varphi}$). 

We are therefore in the situation of Lemma \ref{lem:cpt_dm} and thus we can find functions \(h\) and \(g\) such that, passing to a subsequence, $h_j \to h$ and $g_j \to g$. Furthermore, by  condition (B) above,  $h = g$.

Let us show that this is a contradiction and thus conclude the proof. Indeed,  on the one hand,
\[ \D(g, B_1 \cap \{ x_m > 0 \}) \le \liminf_{ l \to \infty} \frac{ \D( f^+, B_{r_j}( x_0))}{\nu(B_{r_j}(x_0))}  = 0 \]
and, on the other hand, due to the conclusions of Lemma \ref{lem:cpt_dm}, 
\begin{align*}
&\D(h, B_{\frac 12}\cap\{ x_m >0 \}) = \lim_j \D( h_j , B_{\frac 12}\cap  \iota_{x_0, r_j}(\Omega_{k(j)}^+))\\&= \lim_{j \to \infty}  \frac{ \mu_{k(j)}( B_{r_j/2}(x_0))}{\nu( B_{r_j}(x_0))}
 = \lim_{j \to \infty} \frac{ \mu(B_{r_j/2}(x_0))}{\nu(B_{r_j}(x_0))}\ge 2^{-(m-1)}\, . 
\end{align*}

\subsection{Proof of Theorem \ref{thm:cpt_dm}: Minimality}\label{s:gluing} We now come to the second part of the theorem, namely to the claim that $(f^+, f^-)$ is a $\qhalf$ $\D$-minimizer. This requires a suitable modification of the same argument given in
\cite[Proposition 3.20]{DS1}. We assume by contradiction that $(f^+, f^-)$ is not a minimizer and let $(g^+, g^-)$ be a suitable competitor, which
coincides with $(f^+, f^-)$ outside of a compact set $K$. First of all we notice that we may assume that, by Sard Lemma,  we can find an open set $U \subset \Omega$ that contains $K$ and intersects $\gammado$ transversally. 

Thus we have that $(g^+, g^-)= (f^+, f^-)$ on $\partial U$, that $g^+|_\gammado = \a{\varphi} + g^-|_\gammado$ and that
\[
\D (g^+) + \D (g^-) \leq \D (f^+) + \D (f^-) - 4c
\]
for some positive $c$. For each $k$ we let $\Phi_k$ be a diffeomorphism which maps $U$ onto itself and $\gammado_k\cap U$ onto $\gammado\cap U$. Clearly this can be done so that $\|\Phi_k - \Phi\|_{C^1} \to 0$, where $\Phi$ is the identity map. Thus, from the convergence in energy of $(f_k^+, f_k^-)$ to $(f^+, f^-)$ we conclude that, for a  sufficiently large $k$,
\[
\D (g^+ \circ \Phi_k) + \D (g^-\circ \Phi_k) \leq \D (f_k^+) + \D (f_k^-) - 3c\, .
\]
Observe that each pair $(g^+\circ \Phi_k, g^-\circ \Phi_k)$ has interface $(\gammado_k, \varphi\circ \Phi_k)$, where $\|\varphi\circ \Phi_k - \varphi_k\|_{C^{0, \beta}} \to 0$. 

In particular, since $\beta>\frac{1}{2}$, we can fix first $\tilde{\varphi} \in W^{1,2}(U)$ such that $\tilde{\varphi}|_\gammado= \varphi$. Furthermore, since $\|\varphi\circ \Phi_k - \varphi_k\|_{H^{1/2} (\gammado_k)} \to 0$, there is a sequence of classical $W^{1,2}$ functions $\varkappa_k$ on $U$ such that
\begin{itemize}
\item $\varkappa_k = \varphi\circ \Phi_k - \varphi_k$ on $\gammado_k$;
\item $\|\varkappa_k\|_{W^{1,2}}\to 0$.
\end{itemize}
This implies that $\int_{U} \abs{D (\tilde{\varphi}\circ \Phi_k - \varkappa_k)}^2$ is uniformly bounded. 
We consider the maps 
\[
h_k^\pm := \sum_i \a{g^\pm_i \circ \Phi_k - \varkappa_k}\, .
\]
Observe that $(h_k^+, h_k^-)$ have interfaces $(\gammado_k, \varphi_k)$, that $\G(f_k^\pm, h_k^\pm) \to 0$ strongly in $L^2(U^\pm\setminus K)$ and that, for $k$ large enough,
\[
\D (h_k^+) + \D (h_k^-) \leq \D (f_k^+) + \D (f_k^-) - 2c\, .
\]
We apply the interpolation Lemma \ref{l:interpolation} to the maps $(f^+_k, f^-_k)$, $(h^+_k, h^-_k)$ and the set $K\subset U$. We obtain, for each $\lambda>0$, interpolation maps $(\zeta^+_k, \zeta^-_k)$ defined on $K \subset V_\lambda^k \subset W_\lambda^k \subset U$. We can now define competitors to $(f^+_k, f^-_k)$ on $W_\lambda^k$ by
\[ u^\pm_k := \begin{cases}
 \zeta^\pm_k &\text{ on } (W^k_\lambda)^+ \setminus V^k_\lambda\\
 h^\pm_k &\text{ on } (V^k_\lambda)^+.
 \end{cases}
\]
Using \eqref{e:raccordo_Dir} one readily checks that, for $k$ sufficiently large and $\lambda >0$ sufficiently small, 
\begin{align*} 
\D(u_k^+)+ \D(u_k^-) &\le \D(h_k^+) + \D(h_k^-) + \D(\zeta_k^+) + \D(\zeta_k^-)\\
&\le 	\D(f_k^+) + \D(f_k^-) -2 c +  \D(\zeta_k^+) + \D(\zeta_k^-)\\
&\le \D(f_k^+) + \D(f_k^-) -c.
\end{align*}
This contradicts the minimality of $(f^+_k, f^-_k)$.

\section{The main frequency function estimate}

We start this section by introducing the frequency function and deriving the main analytical estimate of the entire chapter.

\begin{definition}\label{def:ff}
Consider $f\in W^{1,2}_{loc}(\Omega,\Iqs)$ and fix any cut-off $\phi: [0,\infty[ \to [0, \infty[$ which equals $1$ in a neighborhood of $0$, it is non increasing and  equals $0$ on $[1, \infty[$. We next fix a function $d: \mathbb R^m \to \mathbb R^+$ which is $C^2$ on the punctured space $\mathbb R^m\setminus \{0\}$ and satisfies the following properties:
\begin{itemize}
\item[(i)] $d (x) = |x| + O (|x|^2)$;
\item[(ii)] $\nabla d (x) = \frac{x}{|x|} + O (|x|)$;
\item[(iii)] $D^2 d = |x|^{-1} ({\rm Id} - |x|^{-2} x\otimes x) + O (1)$.
\end{itemize}

We define the following quantities\index{aald D_{\phi,d}@$D_{\phi,d}$}\index{aalh H_{\phi,d}@$H_{\phi,d}$}:
\begin{align*}
D_{\phi,d} (f,r) & := \hphantom{-}\int_\Omega\phi\left(\frac{d(x)}{r}\right)|Df|^2(x)\,dx\\
H_{\phi,d} (f,r) & :=-\int_\Omega\phi'\left(\frac{d(x)}{r}\right) |\nabla d (x)|^2 \frac{|f(x)|^2}{d(x)}\,dx\,.
\end{align*}
The \emph{frequency function}\index{Frequency function} is then the ratio\index{aali I_{\phi,d}@$I_{\phi,d}$}
\[
I_{\phi,d} (f,r) :=\frac{rD_{\phi,d} (f,r)}{H_{\phi,d} (f,r)}\,.
\]
\end{definition}

$H$ obviously makes sense when $\phi$ is Lipschitz. When $\phi'$ is just a measure we understand $H$ as an integral with respect to the measure $\phi'$ in the variable $d(x)/r$ and this also makes sense because the integrand is bounded and continuous on the support of $\phi'$. Of particular interest is the case when $\phi$ is the indicator function of $[0,1[$ and $d(x) = |x|$: then $D (r)$ is the Dirichlet energy on $B_r (0)$, $H (r)$ is the integral $\int_{\partial B_r} |f|^2$ and $I$ is the usual frequency function defined by Almgren. 
In the sequel,
if we do not specify $\phi$ and $d$, we then drop the subscripts and understand that the claims hold for {\em all} cut-off functions $\phi$ and all $d$  as in Definition \ref{def:ff}. If instead we require some more assumptions on $\phi$ or $d$ (for instance a certain regularity) we then leave the cut-off $\phi$ or the function $d$ in the subscripts.

\begin{remark}\label{rmk:zoom}
Note that  if a function \(d\) satisfies (i), (ii) and (iii) in Definition \ref{def:ff} with certain implicit constants, than the function \(d_r(x)=d(rx)/r\) satisfies the same assumptions with the same constants (actually smaller). Moreover \(d_r(x)\to |x|\) in \(C_{\rm loc}^2(\mathbb R^m\setminus\{0\}) \cap C_{\rm loc}^0(\mathbb R^m)\).
\end{remark}

\begin{figure}[htbp]
\begin{center}\label{fig:Omega}
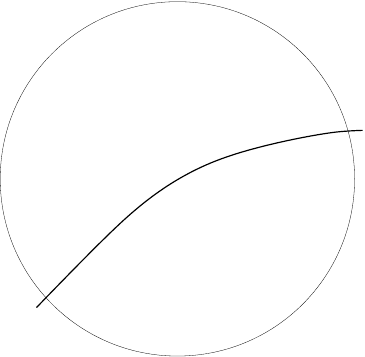
\end{center}
\caption{The domain $\Omega$. $f$ in Theorem \ref{thm:limit_ff} collapses to $Q \a{0}$ on $\partial \Omega$.}
\end{figure}

\begin{theorem}\label{thm:limit_ff}
Let $\Omega\subset\R^m$ be an open set of class $C^3$, with $0\in\partial\Omega$. Then there is a function $d$ satisfying the requirements of Definition \ref{def:ff} such that the following holds for every $\phi$ as in the same definition.

If $f\in W^{1,2}(\Omega\cap B_1,\Iqs)$ satisfies
\begin{itemize}
\item[(i)] $f|_{\partial\Omega\cap B_1}\equiv Q\a{0}$;
\item[(ii)]$\D (f)\le\D (g)$ for every $g\in W^{1,2}(\Omega\cap B_1,\Iqs)$ such that $g|_{\partial(\Omega\cap B_1)}=f|_{\partial(\Omega\cap B_1)}$;
\end{itemize}
then, either $f\equiv Q \a{0}$ in a neighborhood of $0$, or the limit 
\[
\lim_{r\downarrow 0} I_{\phi,d} (f,r
\] 
exists and it is a positive finite number. 
\end{theorem}

\begin{remark}
In fact the conclusion of Theorem \ref{thm:limit_ff} holds for every $d$ which, additionally to the requirements of Definition \ref{def:ff}, has the property that $\nabla d$ is tangent to $\partial \Omega$. The existence of such a $d$ is then guaranteed by a simple geometric lemma,
cf. Lemma \ref{l:good_vector_field}.
\end{remark}

\begin{remark}
Note that if $(f^+, f^-)$ is a $\qhalf$-function  which collapses at its interface $(\partial\Omega\cap B_1,0)$, then $f^+$ satisfies the assumptions of Theorem \ref{thm:limit_ff}.
\end{remark}

\subsection{$H'$ and $D'$} In this section we compute  $H'$ and $D'$. Since there is no possibility of misunderstanding, we omit to specify the dependence of $D,H,I$ on $f$.

\begin{proposition}\label{p:H'_maiala}
Let $\phi$ and $d$ be as in Definition \ref{def:ff}, assume in addition that $\phi$ is Lipschitz and 
let $\Omega$ be as in Theorem \ref{thm:limit_ff}. If $f\in W^{1,2}(\Omega\cap B_1,\Iqs)$ satisfies
condition (i) of Theorem \ref{thm:limit_ff}, then the following identities hold for every $r\in ]0,1[$:
\begin{equation}\label{der_D}
D'(r)=-\int\phi'\left(\frac{|d (x)|}{r}\right)\frac{|d (x)|}{r^2}|Df|^2\,dx\,;
\end{equation}
\begin{equation}\label{der_H}
H' (r) = \left(\frac{m-1}{r} + O(1)\right) H (r) + 2 E(r)\,,
\end{equation}
where
\begin{equation}\label{e:E}
 E(r) := - \frac{1}{r} \int \phi' \left(\frac{d(x)}{r}\right) \sum_i f_i (x) \cdot (Df_i (x) \cdot \nabla d (x))\, dx\, 
\end{equation}
and the constant $O(1)$ appearing in \eqref{der_H} depends on the function $d$ but not on $\phi$.
\end{proposition}
\begin{remark}
It is possible to make sense of the identities above even when $\phi$ is not Lipschitz. In that case, using the coarea formula appropriately, it is possible to see that the right hand sides of the two identities \eqref{der_D} and \eqref{der_H} are in fact well-defined for a.e. $r$ and that both $D$ and $H$ are absolutely continuous. Hence, if formulated appropriately, the proposition is valid for every $d$ and $\phi$ as in Definition \ref{def:ff}, without any additional regularity requirement on $\phi$. This will, however, not be needed in the sequel. 
\end{remark}
\begin{proof}
The identity \eqref{der_D} is an obvious computation. In order to compute $H'$ we first use the coarea formula to write
\begin{align}
H(r) & = -\int_0^\infty \int_{\{d=\rho\}} \rho^{-1} \phi' \left(\frac{\rho}{r}\right) |\nabla d (x)| |f|^2 (x)\, d\mathcal{H}^{m-1} (x)\, d\rho\nonumber\\
&=- \int_0^\infty \frac{\phi' (\sigma)}{\sigma} \underbrace{\int_{\{d= r \sigma\}} |\nabla d (x)| |f|^2 (x)\, d\mathcal{H}^{m-1} (x)}_{=: h (r\sigma)}\, d\sigma\, .  \label{e:H_coarea}
\end{align}
In order to compute $h' (t)$ we  note that $\nu (x) = \frac{\nabla d (x)}{|\nabla d (x)|}$ is orthogonal to the 
level sets of $d$ and we use the divergence theorem to obtain
\begin{align}
h(t+\varepsilon) - h(t) &= \int_{\{d=t+\varepsilon\}} |f|^2 \nabla d \cdot \nu
d\mathcal{H}^{m-1} - \int_{\{d=t\}}|f|^2 \nabla d \cdot \nu
d\mathcal{H}^{m-1}\nonumber\\
&= \int_{\{t<d<t+\varepsilon\}} {\rm div}\, (|f|^2 \nabla d (x))\, dx\label{e:per_parti_perche_0}\\
&= \int_{\{t<d<t+\varepsilon\}} 2 \sum_i f_i (x) \cdot (Df_i (x) \cdot \nabla d (x))\, dx\nonumber\\
&\qquad + \int_{\{t<d<t+\varepsilon\}} |f|^2 \Delta d (x)\, dx\nonumber
\end{align}
Dividing by $\varepsilon$, taking the limit (and using again the coarea formula) we conclude
\begin{equation}\label{e:der_h_piccola}
h' (t) = \int_{\{d=t\}} |\nabla d|^{-1} \left(2 \sum_i f_i \cdot (Df_i\cdot \nabla d)\, + |f|^2 \Delta d\right)\, d\mathcal{H}^{m-1}\, .  
\end{equation}
By  the properties of $d$, we have that
\[
\Delta d = \frac{m-1}{d(x)} + O(1).
\]
Differentiating \eqref{e:H_coarea} in $r$, inserting \eqref{e:der_h_piccola}  and using that if \(\phi(d/r)\ne 0\) then \(d=O(r)\) we conclude
\begin{align}
&\quad H'(r)\nonumber\\
&=  - \int_0^\infty \phi' (\sigma) \int_{\{d=\sigma r\}} \frac{1}{|\nabla d|} \Big(2 \sum_i f_i \cdot (Df_i\cdot \nabla d)\, + |f|^2 \Delta d\Big)\, d\mathcal{H}^{m-1}\, d\sigma\nonumber\\
&= 2E (r) - \frac{1}{r} \int \phi' \left(\frac{d(x)}{r}\right) |f|^2 \Delta d (x)\, dx\nonumber\\
&= 2E(r) - \frac{1}{r} \int \phi' \left(\frac{d(x)}{r}\right) |f|^2 \left(\frac{(m-1) + O(r)}{d (x)}\right)\, dx\\
&= 2E (r)+\left(\frac{m-1}{r}+ O(1)\right) H(r)\, .\qedhere
\end{align}
\end{proof}

\begin{remark}
Observe that the assumption $f = Q \a{0}$ on $\partial \Omega$ has been used only in deriving \eqref{e:per_parti_perche_0}: without
that condition we would have the additional term
\[
- \int_{\partial \Omega \cap \{t<d<t+\varepsilon\}} |f|^2 \nabla d \cdot n
\]
where $n$ is the outward unit normal to $\partial \Omega$. Note in particular that we could drop the assumption $f= Q \a{0}$ and add
instead the requirement that $\nabla d$ is tangent to $\partial \Omega$. 
\end{remark}

\subsection{Lower bound on $H$}

\begin{lemma}\label{lem:lb}
Assume $\phi$ is identically $1$ on some interval $[0, \rho[$. 
Under the assumption of Theorem \ref{thm:limit_ff} there exist constants $C_0$ and \(r_0\), depending only on the $C^1$-regularity of $\Omega$, on $\rho$ and on $d$ (but not on $\phi$), such that
\begin{equation}\label{lb_ff}
H(r)\le C_0 rD(r) \qquad \mbox{for all $r\leq r_0 $.}
\end{equation}
\end{lemma}
\begin{proof} If we introduce the usual scaling 
\[
f_r (x) := f (rx) \qquad \mbox{and}\qquad d_r (x) = r^{-1} d (rx)\, ,
\] 
then 
\[
H_{\phi, d_r} (f_r, 1) = r^{m-1} H_{\phi,d} (f,r)\quad\mbox{and}\quad
D_{\phi, d_r} (f_r, 1) = r^{m-2} D_{\phi,d} (f, r)\, .
\] 
Observe also that for $r\leq 1$ the $C^1$ regularity of the boundary of $\Omega_r:=\left\{x/r:\,x\in\Omega\right\}$ improves compared to that of $\Omega$ and $d_r$ satisfies the same properties of $d$ with better bounds on the errors, see Remark \ref{rmk:zoom}. By taking \(r_0\) sufficiently small we can assume that 
\begin{equation}\label{e:comp}
B_{\varrho r/2}\subset  \{d_{r}<\varrho\} \subset B_{2\varrho r}\qquad \textrm{for all \(r\le r_0\) and \(\varrho\le 1\)}. 
\end{equation}
Let us assume without loss of generality that \(r_0=1\). If we define the  ``distorted balls'' 
\[
B^*_\rho :=\{x: d(x) < \rho\},
\]
the inclusions above imply that they are comparable to  the Euclidean ones up and thus we can transfer most estimates of the last sections to these new balls. Let us now extend \(f\) to be  identically $0$ outside on $\Omega\setminus B^*_1$ so that we can consider the integrals in the definitions of $H (1)$ and $D (1)$ as taken over the whole $B^*_1 $.

%

By a standard approximation procedure we can assume that $\phi$ is smooth. 
Let $0<\bar \rho < \frac{1}{4}$ be such that $\phi$ is identically $1$ on $[0, \bar \rho]$. Then, as a particular case of Theorem \ref{thm:cont_dm} we have
\[
[f]_{\alpha,B_{\bar \rho}^{*}\cap \Omega}\le C \D(f,B_{4\bar \rho}^{+}\cap \Omega)^\frac 12\le C D(1)^\frac 12\,,
\]
where $\alpha=\alpha(m,n,Q)$ and $C=C(m,n,Q, \bar \rho)$ and in the last inequality we have also used \eqref{e:comp}. Of course the same estimate extends trivially to $B_{\bar \rho} \setminus \Omega$, where the function vanishes identically. Thus
\begin{equation}\label{up_cont}
\int_{\partial B_{\bar \rho}^{*}}|\nabla d (x)| |f|^2(x)\,dx=\int_{\partial B_{\bar \rho}^{*}} |\nabla d (x)| \G(f(x),f(0))^2\le C D(1)\, .
\end{equation}
On the other hand, using the coarea formula
\begin{equation}\label{ub_H1}
H(1)=- \int_{\bar \rho}^1\frac{\phi' (r)}{r} \int_{\{d=r\}} |\nabla d (x)| |f|^2(x')\,dx'\,dr 
= - \int_{\bar \rho}^1 \frac{\phi'(r)}{r} h(r)\, dr\, ,
\end{equation}
where $h\ge 0$ is as in \eqref{e:H_coarea}.
Integrating by parts we get
\begin{align}
H(1) \leq\  & C \int_{\partial B^*_{\bar \rho}} |f|^2 + \int_{\bar \rho}^1 \phi (r) (r^{-1} h'(r) - r^{-2} h(r))\nonumber\\
\leq & C D(1) + \int_{\bar \rho}^1 \phi (r) \frac{h'(r)}{r}\, dr\nonumber\\
\stackrel{\eqref{e:der_h_piccola}}{=} & C D(1) + 
C\int_{B^*_1\setminus B^*_{\bar \rho}} \frac{\phi(d(x))}{d(x)} \left(|Df|^2 + |f|^2\right)\leq C D(1)\nonumber\\
&\quad + C \int_{B_1^*\setminus B_{\bar \rho}^*} \phi (d (x)) |f|^2 (x)\, dx\, .\label{e:inserendo}
\end{align}
where the constants depend only on ${\bar \rho}$ and $d$, but not on $\phi$.  The proof will be concluded if we can show that 
\begin{equation}\label{e:corona_est}
 \int_{B_1^*\setminus B_{\bar \rho}^*} \phi (d (x)) |f|^2 (x)\le C D(1)
\end{equation}
To this end note that for  $\bar \rho\le r\le 1$ the function \(|f|^2\) vanishes on a non trivial part of \(B_r^*\) (namely \( B_r^*\setminus \Omega\)). Hence by the \((m-1)\)-dimensional Poincar\'e inequality on \(\partial B^*_{r}\)  
\[
\int_{\partial B^*_{r}} |f|^2\le C \int_{\partial B^*_{r}} |D |f|^2|\le C \int_{\partial B^*_{r}} |f||D f|.
\]
Hence, the function \(h'\) defined in \eqref{e:der_h_piccola} satisfies:
\[
|h'(r)|\le C\int_{\partial B^*_{r}} |f||D f|
\]
Since \(\phi(t)\ge \phi(r)\) for $\bar \rho\le t \le r\le 1$, using again the coarea formula we can now  estimate
\begin{align*}
\phi (r) h(r) &\leq \phi (r) h (\bar \rho) + \phi (r) \int_{\bar \rho}^r |h'(t)|\, dt\\ 
&\leq
C D(1) + \int_{\bar \rho}^r \phi (t) |h'(t)|\, dt\\
& \leq C D(1) + C \int_{B_1^*\setminus B_{\bar \rho}^*} \phi (d(x)) |f||Df| (x)\, dx\, .
\end{align*}
Integrating in $r$ and using Young's inequality we obtain
\begin{align*}
& \int_{B_1^*\setminus B_{\bar \rho}^*} \phi (d (x)) |f|^2 (x)\, dx\\
 \leq & C D(1) 
+ C \int_{B_1^*\setminus B_{\bar \rho}^*} \phi (d(x)) |f||Df| (x)\, dx\\
\leq & C D(1) + \frac{C}{\varepsilon} D(1) + C \varepsilon  \int_{B_1^*\setminus B_{\bar \rho}^*} \phi (d(x)) |f|^2 (x)\, dx\, .
\end{align*}
Choosing $\varepsilon$ appropriately we get \eqref{e:corona_est} and thus  we conclude the proof.
\end{proof}

\begin{corollary}\label{cor:lb}
Assume $\phi$ is identically $1$ on some interval $[0, \rho[$. 
Unless $f\equiv Q\a{0}$ in a neighborhood of $0$, the following lower bound for the frequency function holds:
\[
\liminf_{r\downarrow 0} I(r)\ge C_0>0\, ,
\]
where $C_0$ depends only on the $C^1$ regularity of $\Omega$, on $\rho$ and on $d$. 
\end{corollary}

\subsection{Outer variations} We now derive the first interesting identity relating $D$ and $E$, which is proved variationally using a perturbation of the map in the target. 

\begin{lemma}[Outer variation]\label{l:outer}
Let $\Omega$ be open and $f\in W^{1,2}(\Omega\cap B_1,\Iqs)$ be as in Theorem \ref{thm:limit_ff}. Then $D(r)=E(r)$ for every $0<r<1$, where \(E(r)\) is defined in \eqref{e:E}.
\end{lemma}
\begin{proof}
We first assume $\phi$ to be Lipschitz. 
Consider the family 
\[
g_\varepsilon(x):=\sum_i \a{f_i (x)+\varepsilon\phi\left({\textstyle{\frac{d (x)}{r}}}\right)f_i(x)}
\] 
and observe that on $\partial \Omega$ we have $f(x)= Q \a{0}$ and so $g_\varepsilon(x)=Q\a{0}$. Therefore each $g_\varepsilon$ is a  competitor and we conclude 
\[
\frac{d}{d\varepsilon}\Big|_{\varepsilon=0}\int_{\Omega\cap B_1}|Dg_\varepsilon|^2=0\,.
\]
Hence
\begin{align*}
0 &=\int\phi\left(\frac{d (x)}{r}\right)|Df(x)|^2\,dx\\
&\quad +\frac{1}{r}\int\phi'\left(\frac{d (x)}{r}\right)\sum_i\left( Df_i(x): \nabla d (x)\otimes f_i(x)\right)\,dx\\
&=D(r)-E(r)\,.
\end{align*}
For a general $\phi$ it suffices to use a standard approximation argument.
\end{proof}

\subsection{Inner variations} 

We now derive the second key identity, which uses perturbations of the domain. To this end consider a compactly supported vector field \(Y\) which is tangent to  \(\partial \Omega\) (i.e. such that such that \(Y (x)\cdot \nu (x)=0\) for all \(x\in \partial \Omega\), where $\nu$ denotes the outward unit normal to $\partial \Omega$). Let $\Phi_t$ the one-parameter family of diffeomorphisms generated by $Y$, namely $\Phi_t (x) = \Phi (x,t)$ where
\[
\begin{cases}
\partial_{t} \Phi (x,t)=Y (\Phi (x,t)) \\
\Phi (x,0)=x\, .
\end{cases}
\] 
Obviously $\Phi_t$ maps $\Omega$ into itself and, more importantly, maps $\partial \Omega$ into itself. 
In particular we have the following lemma.

\begin{lemma}[Inner variation]\label{l:inner}
Consider a modified distance function  $d$ as in Definition \ref{def:ff} such that $\nabla d (x) \cdot \nu (x) = 0$ for every $x\in \partial \Omega\cap B_1$, where $\nu$ denotes the outward unit normal to $\Omega$ and fix a Lipschitz $\phi$ as in the same same definition. Let
\[
Y (x) = \phi \left(\frac{d(x)}{r} \right) \frac{d (x) \nabla d (x)}{|\nabla d (x)|^2}\, .
\]
and let \(\Phi_t\) be the flow generated by \(Y\) . Then 
\begin{equation}\label{e:inner_maiala}
\InV:=\left. \frac{d}{dt}\right|_{t =0} \int |D (f (\Phi_t(x))|^2 = 0\, .
\end{equation}
In particular, if we define
\[
G(r):=-\frac{1}{r}\int\phi'\left(\frac{d (x)}{r}\right) \frac{d(x)}{r\abs{\nabla d(x)}^2} \sum_i\left|Df_i(x)\cdot\nabla d (x)\right|^2\,dx\, ,
\]
we conclude
\begin{equation}\label{lb_iv}
D' (r) - \left(\frac{m-2}{r}- O(1)\right) D(r)-2G(r) = \frac{\InV}{r} = 0\, \, ,
\end{equation}
where the constant $O(1)$ depends on $d$ and $\Omega$ but not on $\phi$. In particular the latter identity holds even for a general $\phi$ as in Definition \ref{def:ff}. 
\end{lemma}
\begin{proof} \eqref{e:inner_maiala} is obvious  by the minimality of \(f\), because \(\Phi_t(\partial \Omega)=\partial \Omega\). We thus  just need to prove the identity between the left hand side of \eqref{lb_iv} and $\InV$ in \eqref{e:inner_maiala}. Note that, by standard computations (cf. \cite{DS1}) 
\begin{equation}\label{e:fatevelo_voi}
\InV = 2 \int \sum_i Df_i : Df_i DY - \int |Df|^2 {\rm div}\, Y\, .
\end{equation}
Hence, by the properties of $d$, we compute
\begin{align*}
D Y &= \phi' \left(\frac{d}{r} \right) \frac{d}{r} |\nabla d|^{-2} \nabla d \otimes \nabla d + \phi \left(\frac{d}{r}\right) D \left(|\nabla d|^{-2} d \nabla d\right)\\
&= \phi' \left(\frac{d}{r}\right) \frac{d}{r} |\nabla d|^{-2} \nabla d\otimes \nabla d + \phi \left(\frac{d}{r}\right) ({\rm Id} + O (d))\,\\
&= \phi' \left(\frac{d}{r}\right) \frac{d}{r} |\nabla d|^{-2} \nabla d\otimes \nabla d + \phi \left(\frac{d}{r}\right) ({\rm Id} + O (r))\, ,
\end{align*}
and
\[
{\rm div}\, Y = \phi' \left(\frac{d}{r}\right) \frac{d}{r} + \phi \left(\frac{d}{r}\right) (m + O (r))\, .
\]
Plugging the latter identities in \eqref{e:fatevelo_voi} and recalling the formula \eqref{der_D} for $D'$, we conclude the proof.
\end{proof}

\subsection{A good function $d$}\label{sss:gv} 
In this section, relying on the  \(C^3\) regularity of \(\partial \Omega\) we construct a modified distance function whose gradient is tangent to \(\partial \Omega\). We believe that the same result can be achieved with less regularity of $\partial \Omega$, namely \(C^2\), however since we will not need this in the sequel, we stick to \(C^3\) regularity, where the proof is rather straightforward.

\begin{lemma}\label{l:good_vector_field}
Let $\Omega$ be a $C^3$ domain such that $0\in \Omega$ and $T_0 \partial \Omega = \{x_m=0\}$. Then there is a continuous function $d: \Omega \to \mathbb R^+$ which belongs to $C^2 (\Omega\setminus \{0\})$ and such that
\begin{itemize}
\item[(a)] $\partial_J d (x) = \partial_J |x| + O (|x|^{2-|J|})$ for every multiindex $J$ with $|J|\leq 2$;
\item[(b)] $\nabla d$ is tangent to $\partial \Omega$.
\end{itemize}
\end{lemma}
\begin{proof}
Consider normal coordinates on a sufficiently small tubular neighborhood $U_\delta$ of $\partial \Omega$ and construct a diffeomorphism between $U_\delta$ and a tubular neighborhood $V_\delta$ of a suitable subset of $\R^{m-1}\times \{0\}$ with the properties that:
\begin{itemize}
\item $\Phi\in C^2$, $\Phi (0) =0$ and $D\Phi|_0 = {\rm Id}$;
\item $\Phi (\partial \Omega)\subset \R^{m-1}\times \{0\}$;
\item For every $p\in \partial \Omega$ and every vector $\nu$ normal to \(\partial \Omega\) at $p$, $D\Phi|_p (\nu)$ is normal to $\R^{m-1}\times \{0\}$.
\end{itemize}
The existence of such diffeomorphism follows easily from our assumptions. Define then $d(x) := |\Phi (x)|$. It is obvious that $d(x) = |x| + O (|x|^2)$. Computing the first and second derivatives we get, using Einstein's summation convention,
\begin{align}
\partial_{i} d  &= \frac{\Phi^{k} \partial_{i} \Phi^k}{|\Phi|} = \frac{x_i}{|x|} + O (|x|)\label{e:nabla_d}\\
\partial_{ij}^2 d &= \frac{\partial_{j}\Phi^{k}\partial_{i} \Phi^k}{|\Phi|}+\frac{\Phi^{k} \partial_{ij} \Phi^k}{|\Phi|}-\frac{\Phi^{k} \partial_{i} \Phi^k\Phi^{l} \partial_{j} \Phi^l}{|\Phi |^{3}} 
\nonumber\\
&= |x|^{-1} \delta_{ij} - |x|^{-3} x_ix_j + O (1)\, .
\end{align}
In particular (a) follows easily. 

Next, consider a vector $v$ orthogonal to $\partial \Omega$ at $p\neq 0$, let $z= \Phi (p)$. Let $\langle\cdot,\cdot \rangle$ be the standard Euclidean scalar product and observe that, from the first equality in \eqref{e:nabla_d}, we get
\begin{align}
\langle \nabla d (p), v\rangle &= |z|^{-1} \langle z, D\Phi|_p (v)\rangle\, .
\end{align}
On the other hand, since $z = \Phi (p)\in \R^{m-1}\times \{0\}$ and \(D\Phi|_p (v) \in (\R^{m-1}\times \{0\})^\perp$ by the assumptions on $\Phi$ above, we clearly have 
\[
\langle \nabla d (p), v\rangle =0\, .
\] 
We conclude that $\nabla d$ is orthogonal to any vector field normal to $\partial \Omega$ and thus it must be tangent to $\partial \Omega$. 
\end{proof}

\subsection{Proof of Theorem \ref{thm:limit_ff}} 
Assume that $\phi$ and $d$ have the properties of Definition \ref{def:ff}. As a consequence of Lemma \ref{l:good_vector_field} we may assume that $\nabla d \cdot \nu = 0$ on $B_{r_0}(0)$. This implies that the conditions of Proposition \ref{p:H'_maiala}, Lemma \ref{l:outer}, \ref{l:inner} are satisfied. Hence,
\begin{align*}
-\frac{d}{dr} \ln(I(r)) = \quad\, & \frac{H'(r)}{H(r)} - \frac{D'(r)}{D(r)} - \frac{1}{r} \\ \overset{\eqref{der_H},\eqref{der_D}}{=} &\frac{2E(r)}{H(r)} - \frac{2 G(r)}{D(r)} + O(1) 
\end{align*}
Furthermore due to \eqref{l:outer} we have 
\begin{align*} 
& \hphantom{-} H(r) E(r) \left(\frac{E(r)}{H(r)} - \frac{G(r)}{D(r)}\right) =  \left( E(r)^2 - H(r) G(r) \right)\\
=&\hphantom{-}  \left( \frac{1}{r} \int \phi'\left(\frac{d}{r}\right) \sum_i f_i \cdot (Df_i \cdot\nabla d) \right)^2 \\
& - \left(\int \phi'\left(\frac{d}{r}\right) \frac{\abs{\nabla d}^2}{d} \abs{f}^2 \right)\left( \frac{1}{r} \int \phi'\left(\frac{d}{r}\right) \frac{d}{r} \frac{1}{\abs{\nabla d}^2} \sum_i (Df_i\cdot \nabla d)^2 \right)\\
 \le & 0, 
\end{align*}
due to the Cauchy--Schwarz inequality. Moreover the equality holds if and only if there is a function $\alpha_r$ such that 
\begin{equation}\label{eq: equality in frequency}
	f_i =\alpha_r \frac{d}{\abs{\nabla d}^2} (Df_i \cdot \nabla d)
\end{equation}
Finally we deduce, that 
\begin{equation}\label{e:mo}
-\frac{d}{dr} \ln(I(r)) \le O(1) 
\end{equation}
and therefore we deduce that, for $r<r_0$,
\[ r \mapsto e^{Cr} I(r) \]
is monotone. 
This directly implies that $\lim_{r\searrow 0} e^{Cr}I(r) = I_0$ exists. Moreover, by Corollary \ref{cor:lb}, we have $ I_0\ge C_0 >0$.

\section{Further consequences of the frequency estimate}

As a further consequence of the almost monotonicity of the frequency we obtain  the following result, compare \cite[Corollary 3.16]{DS1}.
\begin{corollary}\label{cor:consequence of monotone frequency for H and D}
	Under the assumptions of Theorem \ref{thm:limit_ff} there exists a constant \(C\) such that setting $I(0)=I_0>0$ for every $\lambda >1$ there exists $r_1\le r_0$ for which the following estimates hold true
	\begin{itemize}
	\item[(a)] $\lambda^{-1} I_0 \le I(r) \le \lambda I_0$ for all $r<r_1$;
	\item[(b)] for all $0\le s \le t \le r_1$ \begin{equation}\label{eq:H comparison}
 e^{-C(t-s)} \left( \frac{t}{s} \right)^{m-1+2 \lambda^{-1} I_0} \le \frac{H(t)}{H(s)} \le e^{C(t-s)} \left( \frac{t}{s} \right)^{m-1+2 \lambda I_0};
 \end{equation}
 \item[(c)] for all $0\le s \le t \le r_1$ \begin{equation}\label{eq: D comparison}
 	\lambda^{-2} e^{-C(t-s)} \left( \frac{t}{s} \right)^{m-2+2 \lambda^{-1} I_0} \le \frac{D(t)}{D(s)} \le \lambda^2 e^{C(t-s)} \left( \frac{t}{s} \right)^{m-2+2 \lambda I_0}\, .
 \end{equation}
	\end{itemize}
\end{corollary}

\begin{proof}
Point (a) is an immediate consequence of the almost monotonicity of the frequency, \eqref{e:mo}

Concerning point (b), using \eqref{der_H} and Lemma \ref{l:outer}, we compute
\begin{align*}
	\frac{d}{dr} \ln\left( \frac{H(r)}{r^{m-1}} \right) = \frac{H'(r)}{H(r)} - \frac{m-1}{r} = \frac{2}{r} I(r) + O(1)\,.
\end{align*}
Integrating the above identity  between $0\le s\le t \le r_1$ and using point (a), we obtain the estimate \ref{eq:H comparison}.

To prove  (c), we have only to note that 
\[ \frac{D(t)}{D(s)} = \frac{I(t)}{I(s)} \left( \frac{t}{s} \right)^{-1} \frac{ H(t)}{H(s)} \]
and appeal to points (a) and (b).
\end{proof}

\begin{corollary}\label{cor:Dirichlet comparison}
Under the assumptions of Theorem \ref{thm:limit_ff} with $I_0 = I(0)$, there are constants $\lambda>1$ (depending only on $\phi$), $\bar C>1$ (depending on $\phi, d$ and $I_0$) and $r_1 >0$ such that the following estimate holds for all $0 < \lambda^2 s < t < r_1$:
\begin{equation}\label{eq:Dirichlet comparison}
	\bar C^{-1} \left(\frac{t}{s}\right)^{m-2+2\lambda^{-1} I_0} \le \frac{\int_{\Omega \cap B_t} \abs{Df}^2}{\int_{\Omega \cap B_s} \abs{Df}^2} \le \bar C \left(\frac{t}{s}\right)^{m-2+2\lambda I_0} .
\end{equation}
When $\phi = {\mathbf 1}_{[0,1]}$, we can choose both $\lambda$ and $\bar C$ arbitrarily close to $1$, provided $r_1$ is small enough. 
\end{corollary}
\begin{proof}  Recall that $\phi \equiv 1$ on some interval $[0, \bar{\rho}[$. 
By the assumptions on $d$, for any $\lambda > \bar{\rho}^{-1}$ there is then a positive $r_1$ such that
\[ \mathbf{1}_{B_{\lambda^{-1}r}} (x) \le \phi \left(\frac{d(x)}{r}\right) \le \mathbf{1}_{B_{\lambda r }} (x) \qquad \forall r< r_1, \forall x\in \mathbb R^m\, .\]
Hence we deduce that 
\[ D(\lambda^{-1} r) \le \int_{B_r\cap \Omega} \abs{Df}^2 \le D(\lambda r),\]
and we conclude the proof from \eqref{eq: D comparison}.
When $\phi = \mathbf{1}_{[0,1]}$ we can choose any $\lambda >1$. Note moreover that the constant $\bar C$ in
\eqref{eq:Dirichlet comparison} can be taken to be $e^{Cr_1} \lambda^\tau$
where the exponent $\tau$ depends only on $I_0$ and $m$. The last claim of the corollary is thus obvious.
\end{proof}

\begin{lemma}\label{lem:one sided tangent functions}
	Let $\Omega \subset \R^m$ be an open set of class $C^3$ with $0\in \partial \Omega$. Furthermore assume $f\in W^{1,2}(\Omega\cap B_1, \Iqs)$ satisfies the assumption of Theorem \ref{thm:limit_ff}. Then, for any $r_k \downarrow 0 $, there is a subsequence, not relabeled, such that 
\footnote{Here again we are using the following abuse of notation: if $\lambda$ is a scalar and $P= \sum_i \a{P_i}$ an element in $\Iqs$, then
$\lambda P= \sum_i \a{\lambda P_i}$.}
	\begin{itemize}
		\item[(a)] $\hat{f}_{k}(x):= \left(r_k^{2-m}\int_{B_{r_k}\cap \Omega} \abs{Df}^2 \right)^{-\frac12} f(r_k x)$ converges to a map $g \in W^{1,2}(H, \Iqs)$ such that $g=Q\a{0}$ on $\partial H$, where $H$ is some halfspace containing the origin.
		\item[(b)] $g$ is Dirichlet minimizing, in the sense that 
		\[
		\D(g, B_R \cap H)\le \D(h)
		\] 
		for every $R>0$ and for every $h \in W^{1,2}(H\cap B_R, \Iqs)$ such that $g|_{\partial (H \cap B_R)}= h|_{\partial (H \cap B_R)}$. 
		\item[(c)] $g(x)=\abs{x}^{I_0} g(\frac{x}{\abs{x}})$, where 
		\[
		I_0 = \lim_{r \downarrow 0 }I_{d, \phi}(0)
		\] 
		(which exists thanks to  Theorem \ref{thm:limit_ff}).
	\end{itemize}
\end{lemma}

\begin{proof}
Let $d, \phi$ be a distance function and cut-off function that are admissible in the sense of Theorem \ref{thm:limit_ff}. As before we introduce the usual scaling $f_r(x)=f(rx)$,  $d_r (x) = r^{-1} d (rx)$ and $\Omega_r:=\left\{ x/r:\,x\in\Omega\right\}$. Observe that $\Omega_r$ converges locally in $C^2$ to a halfspace ${\mathbb H}$, which up to a rotation we may assume to be $\{x: x_m >0\}$. 
  Furthermore, by Remark \ref{rmk:zoom} $d_r(x) \to \abs{x}$ in $C_{\rm loc}^2(\R^m\setminus\{0\})$.   Moreover, by direct computation, $H_{\phi, d_r} (f_r, R) = r^{m-1} H_{\phi,d} (f,rR)$ and 
$D_{\phi, d_r} (f_r, R) = r^{m-2} D_{\phi,d} (f, rR)$, for any $R>0$.

Let us pick  $\lambda$ and  $r_1>0$ such that the conclusions of Corollary \ref{cor:Dirichlet comparison} apply. 
  Then, for every $R>1$, the following estimate holds provided $r$ is sufficiently small:
  \[ \int_{B_R\cap {\rm Dom}\, (\hat{f}_r^\pm)} \abs{D\hat{f}_r}^2 \le C(I_0,m)R^{m -2 + 2I_0^\pm} \int_{B_1 \cap
{\rm Dom}\, (\hat{f}_r^\pm)} \abs{D\hat{f}_r}^2\, , 
\]
where ${\rm Dom}\, (\hat f^\pm)$ denote the domains of the rescaled functions $\hat{f}^\pm$.
  Appealing to \cite[Theorem 3.6]{Jonas} we deduce the existence of $g$ satisfying (a) and (b).
  
  It remains to prove (c). Observe that (a), (b) together with $d_r \to \abs{\cdot}$ in $C^2$ imply, for  $R>0$,
  \begin{align*} I_{d,\phi}(0)&=\lim_{k \to \infty} \frac{Rr_k D_{d,\phi} (f, r_kR)}{H_{d,\phi} (f, r_kR)}= \lim_{k \to \infty} \frac{R D_{d_{r_k},\phi} (\hat{f}_{r_k}, R)}{H_{d_{r_k},\phi} (\hat{f}_{r_k}, R)}
= \frac{R D_{\abs{\cdot},\phi} (g, R)}{H_{\abs{\cdot},\phi} (g, R)}.\end{align*}
Now (iii) follows by straightforward adaption of the proof of \cite[Corollary 3.16]{DS1} using \eqref{eq: equality in frequency}. 
\end{proof}

\section{Blowup: proof of Theorem {\ref{thm:collasso}} with $\varphi \equiv 0$}

The proof is based on the monotonicity of the frequency function and the fact that it ensures two things: non-triviality of the blow-ups and radial homogeneity.

More precisely, we have the following:

\begin{lemma}\label{lem:blowup}
Let $(f^+, f^-)$ be a $\qhalf$ $\D$-minimizer which collapses at the  interface $(\gammado, 0)$, where $\gammado$ is $C^3$. Fix $p\in \gammado$
and, unless $(f^+, f^-)$ is identically $(Q\a{0}, (Q-1) \a{0})$ in some ball $B_r (0)$, for every $r$ define
\[
\hat{f}^\pm_{p, r} (x) := \frac{1}{\Delta_{p,r}} f^\pm (p + r x)\,.
\]
The normalizing factor $\Delta_{p,r}$ is chosen to fulfill 
\[
\Delta_{p,r}^2=r^{2-m} \int_{B_r^+(p)} \abs{Df^+}^2 +r^{2-m} \int_{B^-_r(p)} \abs{Df^-}^2, 
\]
so that
\[
\D (\hat{f}^+_{p, r}, B_1) + \D (\hat{f}^-_{p, r}, B_1) = 1.
\]
If we set $\pi = T_p \gammado$, then, up to subsequences, the pair of sequences $(f^+_{p, r}, f^-_{p, r})$ converges to a
$\qhalf$ $\D$-minimizer $(g^+, g^-)$ which collapses at the interface $(\pi, 0)$ satisfying the following properties:
\begin{itemize}
\item[(a)] The convergence is as in Theorem \ref{thm:cpt_dm}.
\item[(b)] $\D (g^+) + \D (g^-) =1$.
\item[(c)] $(g^+, g^-)$ is radially homogeneous, namely $g^\pm (r x) = r^{I_0} g^\pm (x)$, where, if we fix $\phi = \mathbf{1}_{[0,1]}$
in Definition \ref{def:ff}, then
\begin{equation}\label{e:Io}
I_0 = \lim_{r\downarrow 0} \frac{r \left(D (f^+, r) + D (f^-, r)\right)}{H (f^+, r)+ H (f^-,r)}\, 
\end{equation}

\end{itemize}
\end{lemma}
\begin{proof} 
After a translation we may assume that $p=0$. Observe that both $x \mapsto f^+(x)$ and $x \mapsto f^-(x)$ satisfy the assumptions of Theorem \ref{thm:limit_ff}. 
Let us define the single normalization factors 
\[(\Delta^\pm_r)^2:=r^{2-m} \int_{B_r^\pm} \abs{Df^\pm}^2, \]
so that $\Delta^2_r=(\Delta^+_r)^2 + (\Delta_r^-)^2$. Thanks to Lemma \ref{lem:one sided tangent functions}, given any sequence $r_k \to 0$ there is a subsequence (not relabeled) such that $ \tilde{f}^\pm_k(x):= \frac{1}{\Delta_{r_k}^\pm} f^\pm(r_k x )$ converge to some $\tilde{g}^\pm(x)$, which are homogeneous with exponent $I_0^\pm$.  Since 
\[ 
\left(\hat{f}^+_r(x), \hat{f}^-_r(x)\right)= \left(\frac{\Delta_r^+}{\Delta_r} \tilde{f}^+_r(x), \frac{\Delta_r^-}{\Delta_r} \tilde{f}^-_r(x) \right),
\]
it is sufficient to understand the possible limits of $\alpha_k^\pm:=\frac{\Delta_{r_k}^\pm}{\Delta_{r_k}} \in [0,1]$.
Up to subsequences, we may assume that their limits exist and are $\alpha^\pm \ge 0$. Due to the properties of $\Delta^\pm_r$ and $\Delta_r$, we have
\[ 
(\alpha^+)^2+ (\alpha^-)^2 = 1.
\]
Point (a) agrees with the statement of Theorem \ref{thm:cpt_dm} since 
\[ \left(\hat{f}^+_{r_k}(x), \hat{f}^-_{r_k}(x)\right) \to (\alpha^+ \tilde{g}^+, \alpha^- \tilde{g}^-)=(g^+, g^-) .\]
We now distinguish three cases depending on the values of 
\[
I^\pm=\lim_{r \to 0} \frac{r D(f^\pm, r)}{H(r)}
\]
\emph{Case $I_0^+=I_0^-$:} In this case the tangent function $(g^+, g^-)$ is $I_0^+=I_0^-$ homogeneous and satisfies (b). Point (c) follows from the simple observation that
\[\frac{r \left(D (f^+, r) + D (f^-, r)\right)}{H (f^+, r)+ H (f^-,r)}=\frac{ \left(\frac{\Delta^+_r}{\Delta_r}\right)^2 D (\tilde{f}_r^+, 1) + \left(\frac{\Delta^-_r}{\Delta_r}\right)^2D (\tilde{f}_r^-, 1)}{\left(\frac{\Delta^+_r}{\Delta_r}\right)^2 H (\tilde{f}_r^+, 1)+ \left(\frac{\Delta^-_r}{\Delta_r}\right)^2 H(\tilde{f}^-_r,1)}.\]
\emph{Case $I_0^+ > I_0^-$:} We claim that in this case $\alpha^{+} =0 $, so that $(g^{+} , g^-)=(Q \a{0} , \tilde{g}^-)$ 
is $I_0=I_0^-$ - homogeneous. Pick $\lambda>1$ such that $\lambda I_0^- < \lambda^{-1} I_0^+$. For $r_1>0$ sufficiently small, such that Corollary \ref{cor:Dirichlet comparison} applies for $f^+$ and $f^-$, we may choose $r< r_1$. Using \eqref{eq:Dirichlet comparison}, for some fixed $t<r_1$ and for any $s<t$, we have that
\[ 
\frac{ \int_{B_s^+} \abs{Df^+}^2}{\int_{B_s^-} \abs{Df^-}^2} \le \lambda^{2m + 2\lambda I_0^-} \left(\frac{s}{t}\right)^{\lambda^{-1}I_0^+ - \lambda I_0^-} \frac{ \int_{B_t^+} \abs{Df^+}^2}{\int_{B_t^-} \abs{Df^-}^2}. 
\]
By our choice of $\lambda$ this converges to $0$ as $s \to 0$. \\
\emph{Case $I_0^+ < I_0^-$:} We argue as in the previous case swapping $+$ and $-$ and conclude that $\alpha^-=0$.
\end{proof}

\begin{definition}\label{d:tangent_function}
A $(g^+, g^-)$ as above will be called, from now on, a \emph{tangent function}\index{Tangent function} to $(f^+, f^-)$ at $p$. 
\end{definition}

\begin{remark}\label{rmk:inv} Let $(g^+, g^-)$ be a tangent function to some  $(f^+, f^-)$ at some point \(p\). Let  $q\in T_p \gammado\setminus \{0\}$ and let us consider a further tangent function $(g^+_1, g^-_1)$ to $(g^+, g^-)$ at $q$. 
Then, by \cite[Lemma 12.3]{DS1},  $(g^+_1, g^-_1)$ is invariant along the direction $q$, namely $g^\pm_1 (x+\lambda q) = g^\pm (x)$ for every $\lambda \in \R$.
\end{remark}

As a simple corollary we then conclude the following:

\begin{lemma}\label{lem:fed_red}
Let $(f^+, f^-)$ and $p\in \gammado$ be as in Lemma \ref{lem:blowup}. Consider a tangent function $(g^+, g^-)$ to $(f^+, f^-)$ at $p$.
Moreover fix a base $e_1, \ldots , e_{m-1}$ of $\pi = T_p \gammado $, and define inductively $(g_1^+, g_1^-)$ to be a tangent function to $(g^+, g^-)$ at $e_1$
and $(g^+_j, g^-_j)$ to be a tangent function to $(g^+_{j-1}, g^-_{j-1})$ at $e_j$. Then $(h^+, h^-)= (g^+_{m-1}, g^-_{m-1})$ is given by
$(Q \a{L}, (Q-1) \a{L})$, where $L$ is a nonzero linear function which vanishes on $\pi$.
\end{lemma}
\begin{proof} Assume $\pi = \{x: x_m =0\}$.
Applying the remark above \(m\) times we infer the existence of a  map  $(h^+, h^-)$ with the following properties:
\begin{itemize}
\item $(h^+, h^-)$ is a $\qhalf$ $\D$-minimizer which collapses at the interface $(\pi, 0)$;
\item $(h^+, h^-)$ depends only on $x_m$, namely there exist  $Q$-valued function $\alpha^+:\mathbb R_{+}\to \Iqs$ and a 
$(Q-1)$-valued function $\alpha^-: \mathbb R_{-}\to \Iqqs$ such that $h^\pm (x) = \alpha^\pm (x_m)$;
\item $(h^+, h^-)$ is an $I$-homogeneous function for some $I>0$, namely there is a $Q$-point $P$ and a $(Q-1)$-point $P'$ such that $\alpha^+ (x_m)= x_m^I P$ and $\alpha^- (x_m) = (-x_m)^I P'$.
\item $\D (h^+, B_1)+ \D (h^-, B_1)  = 1$.
\end{itemize}
Since $(h^+, h^-)$ is a $\D$-minimizer both $h^+$ and $h^-$ are classical harmonic functions and, since they depend only upon one variable, we necessarily have that $I=1$. So there are coefficients $\beta_1^+, \ldots, \beta_Q^+$ and $\beta_1^-, \ldots , \beta_{Q-1}^-$ such that
\begin{align*}
h^+ (x) = & \sum_{i=1}^Q \a{\beta_i^+ x_m}\\
h^- (x) = & \sum_{i=1}^{Q-1} \a{\beta_i^- x_m}\, .
\end{align*}
If $Q=1$, then there is nothing to prove. If $Q>1$, then necessarily for every choice of $i$ and $j$ the function 
\[
k (x) =
\left\{
\begin{array}{ll}
\beta^+_j x_m \qquad & \mbox{if $x_m\geq 0$}\\ \\
\beta_i^- x_m \qquad & \mbox{if $x_m<0$}
\end{array}
\right.
\]
must be harmonic and hence linear. This implies that all $\beta^-_i$ and $\beta^+_j$ coincide. The claim of the lemma follows.
\end{proof}

\begin{remark}\label{rmk:bcond}
The above result is the key step to establish Theorem \ref{thm:collasso}. Note that in proving that the only \(1\) homogeneous \(1\) dimensional $\qhalf$ $\D$-minimizer which collapses at the interfaces $(\pi, 0)$ we have used in an essential way that only one sheet has to take care of the interface, while the values of the others can be modified even over $\gamma$. In other words the above result is easily seen to be false if we would have required to be minimizers only with respect to variations that keep the pair $f^+$ and $f^-$ completely fixed over $\gamma$.
\end{remark}

As a simple corollary  of the above Lemma we have:

\begin{corollary}\label{cor:av=0}
Assume $(f^+, f^-)$ is a $\qhalf$ $\D$-minimizer with collapsed interface $(\gammado, 0)$, where $\gammado$ is $C^3$. If $\etaa \circ f^- = \etaa\circ f^+ = 0$, then $f^+= Q\a{0}$ and $f^- = (Q-1) \a{0}$. 
\end{corollary}
\begin{proof} If $(f^+, f^-)$ is identically $(Q\a{0}, (Q-1) \a{0})$ in a neighborhood $U$ of a point $p\in \gammado$, then, by the interior regularity theory of $\D$-minimizer, $(f^+, f^-)$ is identically $(Q\a{0}, (Q-1) \a{0})$ in the connected component of the domain of $(f^+, f^-)$ which contains $p$. Thus, if the corollary were false, then there would be a point $p$ such that $\D (f^+, B_r (p)) + \D (f^-, B_r (p)) >0$ for every $r>0$.

If we consider $(h^+, h^-)$ as in Lemma \ref{lem:fed_red}, we conclude that $\etaa \circ h^+ = \etaa \circ h^-=0$, since such property is inherited by each tangent map. But then the nonzero linear function $L$ of the conclusion of Lemma \ref{lem:fed_red} should equal $\etaa \circ h^+$ on $\{x_m>0\}$ and $\etaa\circ h^-$ on $\{x_m \leq 0\}$. Hence $L$ should vanish identically, contradicting Lemma \ref{lem:fed_red}.
\end{proof}

\begin{corollary}\label{cor:interface=0}
Theorem \ref{thm:collasso} holds when $\varphi =0$.
\end{corollary}
\begin{proof}
We start noticing that by classical elliptic regularity, the functions $\etaa \circ f^\pm$  belong to \(C^1(\Omega^\pm\cup \gamma)\).
Let $\nu$ be the unit normal to $\gammado$.  We claim that 
\begin{equation}\label{e:int0_claim}
\partial_\nu (\etaa\circ f^+)(p)= \partial_\nu (\etaa\circ f^-)(p) \qquad \textrm{for all \(p\in \gamma \cap \Omega\)}.
\end{equation} 
The claim will be proved below, whereas we first show that it is enough to conclude.
Indeed it implies that the function
\begin{equation}\label{e:incolla_media}
\zeta =
\left\{\begin{array}{ll}
\etaa \circ f^+\qquad \mbox{on $\Omega^+$}\\ \\
\etaa \circ f^-\qquad \mbox{on $\Omega^-$}
\end{array}\right.
\end{equation}
is a harmonic function. Now let us subtract it from $(f^+, f^-)$, namely let us define the functions
\begin{align}
\tilde{f}^+ = & \sum_i \a{f^+_i - \zeta}\label{e:tilde+}\\
\tilde{f}^- = & \sum_i \a{f^-_i - \zeta}.\label{e:tilde-}
\end{align}
We conclude that $(\tilde{f}^+, \tilde{f}^-)$ is a $\qhalf$ $\D$-minimizer which collapses at the interface $(\gammado , 0)$
and that $\etaa\circ \tilde{f}^+ = \etaa\circ \tilde{f}^- =0$. Thus we apply Corollary \ref{cor:av=0} and conclude that $\tilde{f}^+ = Q \a{0}$ and $\tilde{f}^- = (Q-1) \a{0}$, which complete the proof.

To prove claim \eqref{e:int0_claim} assume by contradiction that, at some point $p\in \gammado\cap \Omega$, we have $\partial_\nu (\etaa\circ f^+) (p) \neq \partial_\nu (\etaa \circ f^-) (p)$ and consider a tangent function $(g^+, g^-)$ to $(f^+, f^-)$ at $p$, which is the limit of some $(f^+_{p, \rho_k}, f^-_{p, \rho_k})$. Observe that, since at least one among $\partial_\nu (\etaa\circ f^+) (p)$ and $\partial_\nu (\etaa \circ f^-) (p)$ differs from $0$, we necessarily have
\[
\D (f^+, B_{\rho_k} (p)) + \D (f^-, B_{\rho_k} (p)) \geq c_0 \rho_k^m
\]
for some constant $c_0$. We then have just two possibilities:
\begin{itemize}
\item[(A)] $\limsup_k (\rho_k)^{-m} (\D (f^+, B_{\rho_k} (p)) + \D (f^-, B_{\rho_k} (p))) = \infty$. In this case the tangent function $(g^+, g^-)$ has zero average, i.e. 
\[
\etaa\circ g^+ = \etaa\circ g^-=0\, .
\] 
By Corollary \ref{cor:interface=0}, $(g^+, g^-)$ should be trivial. But this is not possible because $\D (g^+, B_1) + \D (g^-, B_1)=1$. 
\item[(B)] $\limsup_k (\rho_k)^{-m} (\D (f^+, B_{\rho_k} (p)) + \D (f^-, B_{\rho_k} (p))) < \infty$. In this case we have that $\etaa\circ g^+$ and $\etaa\circ g^-$ are also nontrivial and linear. Moreover they are two distinct linear functions.
\end{itemize}
We can apply this argument to the tangent functions of $(g^+, g^-)$ and since the case (A) is always excluded, after applying it $m-1$ times, we reach  a pair $(h^+, h^-)$ as in Lemma \ref{lem:fed_red}, with the property that $\etaa\circ h^+$ and $\etaa\circ h^-$ are two distinct linear functions. However this contradicts the conclusion of Lemma \ref{lem:fed_red}.
\end{proof}

\section{Proof of Theorem {\ref{thm:collasso}}: general case}

\begin{proof}
Let $\nu$ be the unit normal to $\gammado$. As above, we claim that \[
\partial_\nu (\etaa\circ f^+) = \partial_\nu (\etaa\circ f^-)\,.
\]
With this claim,
proceeding as in the proof of Corollary \ref{cor:interface=0}, we can define $\zeta$ as in \eqref{e:incolla_media} and conclude that it is a harmonic function. We then define $(\tilde{f}^+, \tilde{f}^-)$ as in \eqref{e:tilde+} and \eqref{e:tilde-}. To this pair we can apply Corollary \ref{cor:av=0} and conclude.

To prove the claim, assume by contradiction that, for some $p\in \gammado$, we have that  $\partial_\nu (\etaa\circ f^+)(p)\ne \partial_\nu (\etaa\circ f^-)(p)$. . Without loss of generality we can assume that $p=0$, $\varphi (0) =0$ and $D\varphi (0) =0$. Since at least one among $Df^\pm (0)$ does not vanish, we must have
\begin{equation}\label{e:lineare}
\D (f^+, B_\rho) + \D (f^-, B_\rho) \geq c_0 \rho^{m}
\end{equation}
for some positive constant $c_0$. It also means that there exist a
constant $\eta>0$ and a sequence $\rho_k\downarrow 0$ such that
\[
\D (f^+, B_{\rho_k}) + \D (f^-, B_{\rho_k}) \geq \eta (\D (f^+, B_{2\rho_k}) + \D (f^-, B_{2\rho_k}))\, ,
\]
otherwise we would contradict the lower bound \eqref{e:lineare}. If we now define the blow-up functions 
\[
f^\pm_{ \rho_k}(x):=\frac{f^\pm(\rho_k x)}{\D (f^+, B_{\rho_k}) + \D (f^-, B_{\rho_k})}\, .
\]
 we see  that they have finite energy on $B_2$ and thus there is strong convergence of a subsequence to a $\qhalf$ $\D$-minimizer $(g^+, g^-)$ with interface $(T_p \gammado, 0)$. 
 The latter must then have Dirichlet energy $1$ on $B_1$. We then have two possibilities:
\begin{itemize}
\item[(A)] $\limsup_k (\rho_k)^{-m} (\D (f^+, B_{\rho_k}) + \D (f^-, B_{\rho_k})) = \infty$. Arguing as in the proof of Corollary \ref{cor:av=0}, this gives that $\etaa \circ g^+ = \etaa\circ g^-=0$. Thus, applying Corollary \ref{cor:av=0} we conclude that $(g^+, g^-)$ is trivial, which is a contradiction.
\item[(B)] $\limsup_k (\rho_k)^{-m} (\D (f^+, B_{\rho_k}) + \D (f^-, B_{\rho_k})) < \infty$. Assuming in this case that $T_0 \gammado = \{x_m =0\}$, we conclude that $(g^+, g^-)$ is a $\qhalf$ $\D$-minimizer with flat interface $(T_0 \gammado, 0)$, but also that
$\etaa \circ g^\pm (x) = \bar{c} \partial_\nu (\etaa \circ f^\pm) (0) x_m$ for some positive constant $\bar c$. By Corollary \ref{cor:interface=0}, we then conclude that $\partial_\nu (\etaa\circ f^+) (0) = \partial_\nu (\etaa\circ f^-) (0)$.
\end{itemize}
\end{proof}

\chapter{First Lipschitz approximation and harmonic blow-up}\label{chap:Lip1}

In this chapter we assume that $\pi_0 = \R^m\times \{0\}$\index{aagq\pi_0@$\pi_0$} and we use the notation $\bp$
and $\bp^\perp$
for the orthogonal projections onto $\pi_0$ and $\pi_0^\perp$ respectively., whereas $\bp_\pi$ and $\bp_\pi^\perp$ will denote, respectively, the orthogonal projections onto the plane $\pi$ and its orthogonal complement $\pi^\perp$. We also introduce the notation $B_r (p, \pi)$\index{aalb B_r(p, \pi)@$B_r(p, \pi)$} for the disks $\bB_r (p) \cap (p+\pi)$ and 
$\bC_r (p, \pi)$\index{aalc\bC_r (p, \pi)@$\bC_r (p, \pi)$} for the cylinders $B_r (p, \pi) + \pi^\perp$. If $\pi$ is omitted, then we assume $\pi=\pi_0$. 

\begin{definition}\label{e:excess}
For a current $T$ in a cylinder $\bC_r (p, \pi)$ we define the \emph{cylindrical excess}\index{Cylindrical excess} $\bE$\index{aale\bE (T, \bC_r (p, \pi))@$\bE (T, \bC_r (p, \pi))$} and the \emph{excess measure}\index{Excess measure} $\be_T$\index{aale\be_T@$\be_T$} of a set $F\subset B_{4r}(\bp_\pi(p),\pi)$ as
\begin{align*}
\bE (T, \bC_r (p, \pi)) := \; &\frac{1}{2\omega_m r^m} \int_{\bC_r (p, \pi)} |\vec{T}-\vec{\pi}|^2\, d\|T\|\\
\be_T(F):=\; & \frac{1}{2} \int_{F + \pi^\perp} |\vec{T}-\vec{\pi}|^2\, d\|T\|\,.
\end{align*}
The \emph{height}\index{Height}\index{aalh\bh@$\bh$} in a set $G\subset \mathbb R^{m+n}$ with respect to a plane $\pi$ is defined as
\begin{equation}
\bh (T, G, \pi) := \sup \{|\bp^\perp_\pi (q-p)|: q,p\in \supp (T)\cap G\}\, .
\end{equation}
\end{definition}

The aim of this chapter is to produce a Lipschitz $\qhalf$-valued approximation for area-minimizing currents in a neighborhood of boundary points where the latter are sufficiently flat. For this reason we will introduce a set of assumptions: in this chapter we will work under these assumptions and only later we will show
when we will in fact fall under them. In what follows, in order to simplify our notation, we will assume that $(x,0)\in \pi_0$ and we will abuse the notation by identifying
$\mathbb R^m$ with $\pi_0 = \mathbb R^m\times \{0\}$: in particular we will use $\bC_r (x)$ for the cylinder $\bC_r (x, \pi_0)$ and
we will use the same symbol $F$ for subsets $F\subset \mathbb R^m$ and for the corresponding $F\times \{0\} \subset \pi_0$. Similarly we will write $F\times \mathbb R^n$ for the set $F\times \{0\} + \pi_0^\perp$. 

\begin{ipotesi}\label{Ass:app} $\gammaup \subset \Sigma$ is a $C^2$ submanifold of dimension $m-1$ and $\Sigma \subset \R^{m+n}$ is a $C^2$ submanifold of dimension $m + \bar{n}= m+ n - l$ containing $\gammaup$. We assume moreover that both $\Sigma$ and $\gammaup$ are graphs of entire functions\index{aagx\Psi@$\Psi$}\index{aagx\psi@$\psi$} $\Psi: \R^{m+ \bar{n}} \to \R^l$ and $\psi: \R^{m-1}\to \R^{\bar{n}+1+l}$ satisfying the bounds \index{aala\bA@$\bA$}
\begin{equation}\label{eq:joni-1}
\norm{D\psi}_0 + \norm{D\Psi}_0\le c_0 \text{ and } \bA:= \norm{A_\gammaup}_0 + \norm{A_\Sigma}_0 \le c_0
\end{equation}
where $c_0$ is a positive (small) dimensional constant.

$T$ is an integral current of dim. $m$ with $\partial T \res \bC_{4r} (x)= \a{\gammaup}\res \bC_{4r} (x)$ and $\supp (T)\subset \Sigma$. Moreover we assume that
\begin{itemize}
\item[(i)] $p = (x,0) \in \gammaup$ and $T_p \gammaup = \R^{m-1}\times \{0\}\subset \pi_0$;
\item[(ii)] $\gammado=\bp(\gammaup)$ divides $B_{4r}(x)$ in two disjoint open sets $\Omega^+$ and $\Omega^-$;
\item[(ii)]  for some integer $Q$
\begin{equation}\label{eq:joni-2}
\bp_{\#} T= Q \a{ \Omega^+} + (Q-1) \a{\Omega^-};
\end{equation}
\item[(iv)] $T$ is area minimizing in $\Sigma\cap \bC_{4r}(x)$;
\item[(v)] $Q-\frac{1}{2} \leq \Theta (T,q)$ for every $q\in \gammaup \cap \bC_{4r} (x)$.  
\end{itemize}
\end{ipotesi}

Observe that thanks to  \eqref{eq:joni-2} we have the identities
\begin{align}\label{e:excess_2}
\bE (T, \bC_{4r} (x)) =\; &\frac{1}{\omega_m r^m} \left( \|T\| (\bC_{4r} (x)) - (Q |\Omega^+| + (Q-1) |\Omega^-|)\right)\\
\be_T (F) =\; & \|T\| (F\times\R^n) - (Q |\Omega^+\cap F| + (Q-1) |\Omega^-\cap F|)\,.\label{e:excess_3}
\end{align}

\begin{definition}\label{def:mexc}
Given a current $T$ in a cylinder $\bC_{4r} (p,\pi)$ we introduce the \emph{non-centered maximal function}\index{Non-centered maximal function}\index{aalm\me_T@$\me_T$} of $\be_T$ as
\[
\me_T(y):=\sup_{y\in B_s(z, \pi)\subset B_{4r}(p, \pi)}\frac{\be_T(B_s(y, \pi))}{\omega_m s^m}\,.
\]
\end{definition}

Again abusing the notation, under Assumption \ref{Ass:app} we regard $\me_T$ has a function on $B_{4r} (x) \subset \mathbb R^m$. 

In what follows, given a \(Q\)-valued function \(u\), we denote by $\gr (u)$\index{aalg\gr(u)@$\gr (u)$} and $\bG_u$\index{aalg\bG_u@$\bG_u$} respectively the set theoretic graph of $u$ and the integer rectifiable current
naturally induced by it. For the precise definition we refer to \cite{DS2}.  We next rotate the coordinates keeping $\pi_0$ fixed and
achieving suitable estimates for $D\Psi$: the argument is the same as in \cite[Remark 2.5]{DS3}. 

\begin{remark}[Estimates on $\Psi$ in good Cartesian coordinates]\label{r:Psi}
Assume that $T$ is as in Assumption \ref{Ass:app} in the cylinder $\bC_{4r} (x)$. If 
$E := \bE( T, \bC_{4r} (x))$ is smaller than a geometric constant,
we can assume, without loss of generality, that the function
$\Psi: \R^{m+\bar n} \to \R^l$ parameterizing $\Sigma$ satisfies 
$\Psi(x) = 0$, $\|D \Psi\|_0 \leq C\, E^{\sfrac{1}{2}} + C \bA r$ and $\|D^2 \Psi\|_0 \leq C\bA$.
Indeed observe that
\[
E = \bE(T, \bC_{4r} (x)) =
\frac{1}{2\,\omega_m\,(4r)^m} \int_{\bC_{4r}(x)} |\vec T(y) - \vec \pi_0|^2 \, d\|T\|(y)\, .
\]
Thus, we can fix a point $p \in \supp(T)\cap \bC_{4r} (x)$ such that 
$|\vec{T}(p) - \vec \pi_0|\leq C\,E^{\sfrac{1}{2}}$.
Then, we can find an associated rotation $R \in  {\mathcal O}(m+\overline n,\R)$ such that
$R_\sharp\vec{T}(p) = \vec\pi_0$ and $|R- {\rm Id}|\leq C\, E^{\sfrac{1}{2}}$. 
It follows that $\pi := R (T_p\Sigma)$ is a $(m + \bar n)$-dimensional 
plane such that $\pi_0 \subset\pi$ and 
$\|\pi - T_p \Sigma\| \leq C E^{\sfrac12}$. We choose new coordinates
so that $\pi_0$ remains equal to $\R^m\times \{0\}$ but $\R^{m+\bar{n}}\times \{0\}$
equals $\pi$. 
Since the excess $E$ is assumed to be sufficiently small, 
we can write $\Sigma$ as the graph of a function $\Psi: \pi\to \pi^\perp$. If $(z, \Psi (z))=p$,
then $|D \Psi (z)|\leq C \|T_p \Sigma - \R^{m+\bar n}\times\{0\}\| \leq C E^{\sfrac12}$.
However, $\|D^2 \Psi\|_{0}\leq C\bA$ and so
$\|D \Psi\|_{0} \leq C E^{\sfrac{1}{2}} + C\bA r$. Moreover,
$\Psi (x) =0$ is achieved translating the system of reference by a vector orthogonal to
$\R^{m+\bar n}\times \{0\}$ and, hence, belonging to $\{0\}\times \R^l$.
\end{remark}

We introduce the notation $\Lip (u)$ for the Lipschitz constant of a $Q$-valued map $u= \sum_i u_i$ and ${\rm osc}\, u$ for its oscillation, which is defined
as in \cite{DS3} by \index{aall\Lip@$\Lip (u)$} \index{oscillation} \index{aalo{\rm osc}@${\rm osc}\, (u)$}
\[
{\rm osc}\, (u) = \sup_{z,y,i,j} |u_i (z) - u_j (y)|\, ,
\]
and let $\psi':\gamma \to \mathbb R^n$ be the function\footnote{ If $\psi_1$ is the first of component of the map $\psi$, then 
\[
\gamma = \{(x', \psi_1 (x'), 0): x'\in \mathbb R^{m-1}\}\, .
\] 
In particular $\psi'$ can be regarded as a function of $x'$ and in particular we have $\psi (x')=(\psi_1 (x'), \psi' (x'))$. In the remaining part of the section we will adopt the latter convention.} whose graph coincides with $\Gamma$.

\begin{theorem}\label{t:Lipschitz_1}\label{T:LIPSCHITZ_1}
There are positive geometric constants $C$ and $c_0$ with the following properties. Assume $T$ satisfies Assumption \ref{Ass:app},
$E := \bE (T, \bC_{4r} (x)) \leq c_0$ and $\|D\Psi\|_0 \leq C ( E^{\sfrac{1}{2}} + \bA r)$. Then, for any $\delta_*\in (0,1)$, there are a closed set $K\subset B_{3r} (x)$ and a $\qhalf$-valued function $(u^+, u^-)$
on $B_{3r} (x)$ which collapses at the interface {\color{red} $(\gammado, \psi')$} satisfying the following properties:
\begin{align}
& \Lip (u^\pm)\leq C  (\delta_*^{\sfrac{1}{2}} + r^\frac12 \bA^{\frac12})\label{prop:Lipschitz_11}\\
& {\rm osc} (u^\pm) \leq  C \bh (T, \bC_{4r} (x), \pi_0) + C r E^{\sfrac{1}{2}} + C r^2 \bA\\
& \gr (u^\pm)\subset \Sigma\\
& K \subset B_{3r} (x) \cap \{\me_T\leq \delta_*\}\label{prop:Lipschitz_16}\\
& \bG_{u^\pm} \res[(K \cap \Omega^\pm) \times \R^n] = T \res [(K\cap \Omega^\pm)\times \R^n] \label{e:differenza}\\
& |B_{s} (x) \setminus K| \leq \frac{C}{\delta_*}\:\be_T\left(\{\me_T> \delta_*\}\cap B_{s+ r_1 r}(x)\right) \quad \forall s \le (3-r_1)r \label{e:stimaK}\\
& \frac{\|  T- \bG_{u^+}-\bG_{u^-}\|(\bC_{3r}(x))}{r^m}\le\frac{C(m,n,Q)}{\delta_*}E \quad  \label{e:graphmass}
\end{align}
where $r_1 = c \sqrt[m]{\frac{E}{\delta_*}}$.
\end{theorem}

From now on the approximation of Theorem \ref{t:Lipschitz_1} is called the \emph{$\delta_*^\frac 12$-approximation of $T$ in $\bC_{3r} (x)$}\index{delta@$\delta_*^\frac 12$-approximation of $T$ in $\bC_{3r} (x)$}. 
Actually in the sequel we will choose \(\delta_*^\frac 12\) to be \(E^\beta\) for a suitable chosen small \(\beta\).

In a second step we will prove that, if \(E\)  is chosen sufficiently small and $T$ is area minimizing, then $u$ is close to a $\qhalf$ $\D$-minimizer which which collapses at its interface
and thus, by Theorem \ref{thm:collasso}, consists of a single harmonic sheet. 

\begin{theorem}\label{t:harm_1}\label{T:HARM_1}
For every $\eta_*>0$ and every $\beta\in (0, \frac{1}{4m})$ there exist constants $\varepsilon>0$ and $C>0$ with the following property.
Let $T$ be as in Theorem \ref{t:Lipschitz_1} and mass-minimizing in $\Sigma$, let $(u^+, u^-)$ be the $E^\beta$-approximation of $T$ in $B_{3r} (x)$ and let $K$ be the set satisfying all the properties \eqref{prop:Lipschitz_11}-\eqref{e:graphmass}. If $E\le \varepsilon$ and 
$r\bA\le \varepsilon E^{\frac 12}$, then 
\begin{equation}\label{e:eta-star-condition}
\be_T(B_{5r/2}\setminus K))\le \eta_*E\,,
\end{equation}
and 
\begin{equation}\label{e:small-energy-condition}
\D(u^+,\Omega^+\cap B_{2r}(x)\setminus K)+\D(u^-,\Omega^-\cap B_{2r}(x)\setminus K)\le C\eta_* E\,. 
\end{equation}
Moreover, there exists a (single) harmonic function $h:B_{2r}(x)\to \R^{\overline n}$ such that $h|_{x_m =0} \equiv  0$ and the function $\kappa(y):=(h(y),\Psi(y,h(y)))$ satisfies the following inequalities:
\begin{align}
&r^{-2}\int_{B_{2r}(x)\cap\Omega^+}\!\!\G(u^+,Q\a{\kappa})^2\nonumber\\
+ &\int_{B_{2r}(x)\cap\Omega^+}\!\!\left(|Du^+|-\sqrt Q|D\kappa|\right)^2\le\; \eta_*E r^m\,\label{e:harm-app1}\\
&r^{-2}\int_{B_{2r}(x)\cap\Omega^-}\!\!\G(u^-,(Q-1)\a{\kappa})^2\nonumber\\
+ & \int_{B_{2r}(x)\cap\Omega^-}\!\!\left(|Du^-|-\sqrt{Q-1}|D\kappa|\right)^2 \le\;  \eta_*E r^m\,\label{e:harm-app2}\\
&\int_{B_{2r}(x) \cap \Omega^\pm} \abs{D(\etaa\circ u^\pm) - D\kappa}^2 \le\;  \eta_* E r^m\,.\label{e:harm-app3}
\end{align}
\end{theorem}

\begin{remark}\label{r:odd_harmonic}
Observe that from the Schwarz reflection principle and the unique continuation for harmonic functions, it follows immediately that the $h$ of the previous theorem is in fact odd in the variable $x_m$.
\end{remark}

\section{Proof of Theorem \ref{t:Lipschitz_1}}

\subsection{Artificial sheet and ``bad set''}\label{s:artificial}
Since the statement is invariant under translations and dilations, without loss of generality we assume $x =0$ and $r=1$.  We add to the current $T$ an artificial sheet \index{artificial sheet}, constructed by translating the boundary $\gammaup$ in the ``negative direction'' $-e_m$
over the negative domain $\Omega^-$. Clearly,  if the current $T$ were area minimizing, the addition would (in general) destroy such property. On the other hand  we do not assume that $T$ is area minimizing in Theorem \ref{t:Lipschitz_1} and the ``augmented current''
has
no boundary in the cylinder, while it still has small excess. This will allow us to apply the first part of the approximation theory in the interior developed in \cite[Section 3]{DS3}, where the area minimizing assumption is not relevant. 

Let  therefore $\psi(x')=(\psi_1(x'), \psi'(x'))$ be the map introduced in Assumption \ref{Ass:app}, whose graph gives $\gammaup$, and let $(x',x_m)=x$ be the coordinates of $\R^m$.
We introduce further the map $G_{\psi'} : \pi_0=\R^m  \to \R^{m+ \bar{n} + l}$ given by $G_{\psi'}(x', x_m):=( x',x_m, \psi'(x'))$: the image of $G_{\psi'}$ is just the translation of $\Gamma$ in the direction $e_m = (0, \ldots , 0, 1, 0, \ldots , 0)$. Consider then the current $Z:= {G_{\psi'}}_{\#} \a{\Omega^-} $, cf. Figure \ref{fig:current $Z$}.

\begin{figure}[htbp]
\begin{center}\label{fig:current $Z$}
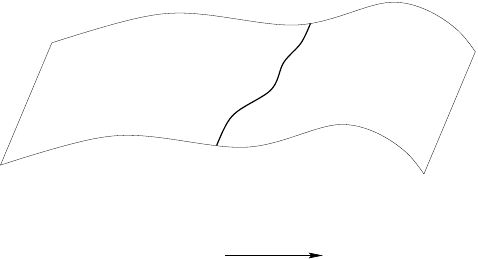
\end{center}
\caption{The current $Z$ is the graph over $\Omega^-$ of a function $\psi'$ which does not depend on $x_m$: $\psi'$ is chosen so that $\partial Z = \a{\gammaup}$.}
\end{figure}

Using the Taylor expansion of the mass, e.g. \cite[Remark 5.4]{DS3}, we can estimate, for any Borel set $F \subset \R^m$.
\[ \bM(Z \res (F\times \mathbb R^n)) = \abs{F \cap \Omega^-} + \int_{F \cap \Omega^-} \frac{ \abs{D\psi'}^2}{2} + \int_{F \cap \Omega^-} R(D\psi') \]
where \(R(D\psi')=O(|D\psi'|^4)\). By assumption 
\[
\abs{D\psi'(x')} \le \abs{x'} \norm{D^2\psi'}_\infty \le c \abs{x'} \bA
\] 
for some dimensional constant $c$. Hence, assuming that the constant $c_0$ in \eqref{eq:joni-1} sufficiently small,
\[ \be_{Z}(F) \le  \int_{ F \cap \Omega^-} \abs{D\psi'}^2 \le c \bA^2 \abs{ F \cap \Omega^-}\,.\]
By construction we have $\partial Z \res \bC_{4}  = {G_{\psi'}}_{\#} \a{ \partial \Omega^- \cap B_4} = - \a{\gammaup}$ and $\bp_{\#} Z = \a{\Omega^-}$. Therefore $ S:= T + Z$ satisfies
\begin{align}\label{e:exess_S} 
 &\bp_{\#} S = Q \a{B_4}\,,\quad \partial S \res \bC_4 = 0\quad  \text{and} \nonumber\\
 &\be_S(F) \le \be_T(F) + \be_{Z}(F)  \le \be_T(F) + c {\bA}^2 \abs{ F \cap {\Omega^-} }\,.
\end{align}
We can thus apply the modified Jerrard-Soner estimate of \cite[Proposition 3.3]{DS3} which gives:
\begin{itemize}
\item[(JS)]
For every $\varphi \in C^\infty(\R^n)$ set $\mathbf{\Phi}_\varphi(x):= S_x(\varphi)$ with 
\[
S_x:= \bp^\perp_{\#} \langle S, \bp, x\rangle \in \mathbf{I}_0 (\R^n)
\] 
(the space of zero-dimensional integral currents in $\R^n$). If $\norm{D\varphi}_\infty \le 1$ then $\mathbf{\Phi}_\varphi(x) \in BV(B_4)$ and satisfies
\begin{equation}\label{e:BVestimate} \left( \abs{D\mathbf{\Phi}_\varphi} ( F) \right)^2 \le 2 m^2 \be_S(F) \norm{S}(F \times \pi_0^\perp) \text{ for every Borel set } F \subset B_4.\end{equation}
\end{itemize}
Following a classical terminology we define noncentered maximal functions for Radon measures $\mu$ and (Lebesgue) integrable functions $f: \R^k \to \R_+$ by setting
\begin{align*}
\bmm(f)(z) & := \sup_{z\in B_s(y)\subset B_{4}}\frac{1}{\omega_m s^m} \int_{B_s(y)} f\\
 \bmm(\mu)(z) &:= \sup_{z\in B_s(y)\subset B_{4}}\frac{\mu(B_s(y))}{\omega_m s^m}. 
\end{align*}
Note that the functions $ z\mapsto \bmm(f)(z), z\mapsto \bmm(\mu)(z)$ and $z \mapsto \me_Z(z)$ are lower semi-continuous. Indeed, since
$\bmm(f)$ is obviously the maximal function of the measure $f \mathscr{L}^m$, it suffices to show the claim for $\bmm (\mu)$. Next observe that for a general Radon measure $\mu$ the map $y \mapsto \mu(B_s(y))$ is lower semicontinuous, and thus the claim follows from the fact that the map
$z\mapsto \bmm (\mu) (z)$ is the supremum of lower semicontinuous functions.

Let us fix a small constant $0<\lambda <1$ and define the following ``bad'' sets, which are, respectively, the upper level set $U$ of $\me_T$
\begin{equation}\label{e:badsets}
U:= \{ x \in B_4 \colon \me_T(x) > \delta_*\} 
\end{equation}
and the upper level set of $\bmm (\mathbf{1}_U)$:
\begin{equation}\label{e:badsets_2}
U^*:= \{ x \in B_4 \colon \bmm(\mathbf{1}_U)(x) > \lambda \}\, .
\end{equation}
As proven in \cite[Proposition 3.2.]{DS3} we have a weak $L^1$ estimate for the Lebesgue measure of $U$. Indeed, fix $r<3$ and
for every
point $x\in U\cap B_r$ consider a ball $B^x$ of radius $r(x)$ which contains $x$ and satisfies $\me_T (B^x) \geq \delta_* \omega_m r(x)^m$. Since
$\me_T (B^x) \leq E$ we obviously have
\[
r (x) \leq r_0 = \sqrt[m]{\frac{E}{\omega_m \delta_*}}
\]
Now, by the definition of the maximal function it follows clearly that $B^x \subset U \cap B_{r+r_0}$. In turn, by the $5r$ covering theorem we
can select countably many pairwise disjoint $B^{x_i}$ such that the corresponding concentric balls $\hat{B}^i$ with radii $5 r (x_i)$ cover $U\cap B_r$  
Then we get
\[ \abs{U \cap B_r} \le 5^m\sum_i \omega_m r(x_i)^m \leq \frac{5^m}{\delta_*} \sum_i \me_T (B^{x_i}) \leq \frac{5^m}{\delta_*} \me_T (U\cap B_{r+r_0})\, . 
\]
Since $U$ is open we have $U \subset U^*$ and by the classical weak $L^1$ estimate  (see e.g. \cite[1.3 Theorem 1]{Stein}), 
we have again
\begin{equation}\label{e:O*}
\abs{ U^*\cap B_r} \le \frac{5^m}{\lambda} \abs{U\cap B_{r+r_1} } \quad \forall r< 3 \text{, where } r_1 = 5 \sqrt[m]{\frac{E}{\omega_m \lambda \delta_*}}.
\end{equation}

\subsection{Lipschitz estimate}
Since $\delta_* + c\bA^2 <1 $, we infer that $\bM (S_x) < Q+1$ for a.e. $x \notin U$. Indeed recall that $\norm{S}(F \times \pi_0^\perp) \ge \int_{F} \bM (S_x)\, dx $ for every open set $F$ (e.g. \cite[Lemma 28.5]{Sim}). Therefore 
 using \eqref{e:exess_S}
\begin{align*}\bM(S_x)& \le \lim_{r \to 0} \frac{ \norm{S}(\bC_r(x))}{\omega_m r^m} \\
 &\le \lim_{r\to 0} \frac{\norm{T}(\bC_r(x))}{\omega_m r^m} +  c\bA^2 \le \me_T(x) + c\bA^2 + Q. 
\end{align*}
There are then $Q$ measurable functions $_i: B_4 \setminus U \to \R^n$ such that $S_x = \sum_{i=1}^Q \a{g_i(x)}$ and we define $g: B_4\setminus U \to \Iqs$ by 
\[
g(x)= \sum_{i=1}^Q \a{g_i(x)}\, .
\] 
Since the slicing is a linear operator and $Z_x =  Z_{(x', x_m)} = \bp^\perp_{\#} \langle Z, \bp, x\rangle = \a{\psi'(x')} $ for all $x \in \Omega^-$,  we have that 
\[
S_x = \sum_{i=1}^{Q-1} \a{g_i(x)} + \a{\psi'(x')} \quad \mbox{for a.e. $x \in \Omega^- \setminus U$.}
\] 
In conclusion  we can define  a $\qhalf$-valued function  $(g^+, g^-)$ as
\begin{align*}
g^+(x) & := \sum_{i=1}^Q \a{g_i(x)} \quad \text{ for a.e. } x \in \Omega^+ \setminus U\\
g^-(x)  &:= \sum_{i=1}^{Q-1} \a{g_i(x)} \quad \text{ for a.e. } x \in \Omega^-\setminus U,
\end{align*}
i.e. $g(x)= g^-(x) + \a{\psi'(x')}$ for all $x \in \Omega^-\setminus U$. 

Combining \eqref{e:BVestimate} and \eqref{e:exess_S} we infer
\begin{align*} 
\bmm{\abs{D\mathbf{\Phi}_\varphi}} (x)^2 &\le 2m^2 ( \me_T(x) + c\bA^2) ( \me_T(x)+ c\bA^2 + Q)\\
& \le 2m(Q+1)(\delta_* + c\bA^2)\, .
\end{align*}
Therefore, the theory of BV functions gives a dimensional constant $C$ such that, for any $\varphi \in C^\infty(\R^n)$ with $\norm{D\varphi}_\infty  \le 1$, 
\begin{align*}
 \abs{ \mathbf{\Phi}_{\varphi}(x) - \mathbf{\Phi}_{\varphi}(y)} & \le C \sqrt{2m(Q+1)(\delta_* + c\bA^2)} \abs{x-y}\\
& \le L_* \abs{x-y} \qquad\qquad \text{ for } x,y \in B_3 \setminus U\, ,
\end{align*}
where $ L_*:=C \sqrt{2m(Q+1)} (\delta_*^\frac12 + c^\frac12 \bA) $.
As pointed out in the proof of \cite[Proposition 3.2]{DS3} one has
\[ 
\sup\{ \abs{\mathbf{\Phi}_{\varphi}(x) - \mathbf{\Phi}_{\varphi}(y)} \colon \abs{D\varphi}_\infty \le 1\} = W_1(g(x), g(y))
\]
where we have set 
\begin{align*}
W_1(S_1,S_2) &:= \sup\{  (S_1- S_2) (\varphi) : \norm{D\varphi}_\infty \le 1 \}\\
& = \min_{\sigma \in \mathcal{P}_Q} \sum_{i} \abs{S_{1i} - S_{2\sigma(i)}} \ge \G(S_1,S_2) 
\end{align*}
for $S_k = \sum_{i=1}^Q \a{S_{ki}} \in \Iqs$.
This implies the Lipschitz continuity of $g$ on $B_3 \setminus U$ and of $g^\pm$ on \(\Omega^\pm\setminus U\). For $g$  it follows directly from the above estimate:
\begin{equation}\label{e:Lipschitz_g} \G(g(x), g(y)) \le W_1(g(x), g(y)) \le L_* \abs{x-y} \text{ for all } x,y \in B_3 \setminus U\end{equation}
and similarly for $g^+$ and $x,y \in \Omega^+ \cap B_3\setminus U$. In the case of $g^-$ we use the triangle inequality  to infer 
\begin{align*}
 & \G(g^-(x),g^-(y)) \le W_1(g^-(x), g^-(y))\\
 \le & W_1(g^-(x) + \a{\psi'(x')}, g^-(y) + \a{\psi'(y')}) + W_1(\a{\psi'(x')},  \a{\psi'(y')})\\
\le & L_* \abs{x-y} + \abs{\psi'(x') - \psi'(y')} \le (L_* + c \bA) \abs{x-y}.
\end{align*}
We now claim that for some dimensional constant $a>c$ we have
\begin{align*} \G(g^+(y), Q \a{\psi'(x')}) &\le 33\sqrt{Q}(L_* + a\bA^\frac12)\abs{ y - x}
\end{align*}
 for all  $y \in \Omega^+ \setminus U^*, x \in \gammado$ and 
\begin{align*}
\G(g^-(y),(Q-1) \a{\psi'(x')}) &\le 33\sqrt{Q}(L_*+ a \bA^\frac12) \abs{ y - x} 
\end{align*}
for all  $y \in \Omega^- \setminus U^*, x \in \gammado$.
The latter estimates are implied by the following claim: 
\begin{itemize}
\item[(Cl)] for  $y \in B_3 \setminus U^*$ with $ \abs{x -y} = \dist(y, \gammado)$ we have
\[ \abs{ g_i(y) - \psi'(x')} \le 33(L_*+a\bA^\frac12) \abs{x-y}\qquad \forall i \]
(where we recall that, given a point $x\in \mathbb R^m$, we write $x'$ for the vector $x'\in \mathbb R^{m-1}$ having the first $m-1$ coordinates
of $x$.)
\end{itemize}
We will argue by contradiction. Assume $y_0 \in B_3\setminus U^*$, $x_0 \in \gammado$ and $i\in \{ 1, \dotsc, Q \}$ satisfy \[\abs{ g_i(y_0) - \psi'(x_0')} \ge 33(L_* + a \bA^\frac12)r, \] where $r=\abs{y_0-x_0}=\dist(y_0,\gammado)<1$. Firstly, we note that
\begin{equation}\label{e:Lipschitz_psi}
\abs{\psi'(x'_1) - \psi'(x'_2)} \le c\bA \abs{x_1-x_2} \text{ for all } x_1, x_2 \in B_4.
\end{equation}
Moreover $g_i(y_0) \in \supp(T)\setminus \supp(Z)$ . Secondly, since $y_0 \notin U^*$ we have $\bmm(\mathbf{1}_U)(y_0)\le\lambda$ and so \begin{equation}\label{e:density}\abs{B_r(x_0)\cap U} \le \lambda \abs{B_r(x_0)}.\end{equation}
Due to \eqref{e:Lipschitz_g} for all $y\in B_r(x_0)\setminus U$ there must be a $j \in \{1, \dotsc, Q\}$ 
with \[\abs{g_j(y)- \psi'(x'_0)} \ge \abs{g_i(y_0) - \psi'(x_0')} -\G(g(y),g(x_0)) \ge 32(L_* +a \bA^\frac12) r \]
and, because of \eqref{e:Lipschitz_psi}, $g_j(y) \in \supp(T)\setminus \supp(Z)$.

Choose $N \in \N$ such that
\begin{equation}\label{e:choiceN} 
\frac{1}{N} \le (4 (L_* + a \bA^{\frac12}))^2 < \frac{1}{(N-1)}
\end{equation} 
and set $r_i: = (1 -  \frac{i}{2N})r$ for $i={\color{red} 0, \ldots, N}$.  This choice ensures that, if $(y,z) \in \bB_{r_i}((x_0, \psi'(x'_0)))$ and $ y $ belongs to the annulus $A_i:= B_{r_i}(x_0) \setminus B_{r_{i+1}}(x_0)$, we must have 
\[
\abs{z- \psi'(x'_0)}^2 \le r_i^2 - r_{i+1}^2 \le \frac{1}{N} r r_i \le (4(L_* + a\bA^\frac12))^2 r^2.
\]
 Therefore, if $ y \in A_i\setminus U$,  the point $(y, g_j(y))$ determined above cannot be contained in $\bB_{r_i}((x_0, \psi'(x'_0)))$. In order to simplify our notation, set $p_0:= (x_0, \psi' (x'_0))$. We then have
 \[
 A_i\setminus U\subset \bp \big(\supp T\cap  \bC_{r_i}(p_0)\setminus \bB_{r_i} (p_0)\big)\, 
 \]
and thus
 \begin{equation}\label{e:misura}
 \norm{T}\big (\bC_{r_i}(p_0)\setminus \bB_{r_i} (p_0)\big) \ge  |A_i\setminus U|.
 \end{equation}

We now claim that there should be \(i\in 1,\dotsc, N\) such that  $\abs{A_i\setminus U} \ge \frac12 \abs{A_i}$, indeed  otherwise
 \begin{align*}
\abs{B_r(x_0)\cap U}&\ge \sum_{ i=0}^N  \abs{A_i \cap U}\ge \frac12 \sum_{ i=0}^N \abs{A_i} \ge \frac12\abs{B_r(x_0)\setminus B_{\frac{r}{2}}(x_0)}\\
&\ge \frac12\left(1- \frac{1}{2^m}\right)\abs{B_r(x_0)}
 \end{align*}
which contradicts \eqref{e:density} because $\lambda \le \frac{1}{4}$. Fix an annulus $A_i$ with $\abs{A_i\setminus U} \ge \frac12 \abs{A_i}$ and define $\rho:=r_i$.
Now we can estimate the mass of $T$ in $\bB_{\rho}(p_0)$ from above using \eqref{e:excess_3}, in fact
\begin{align}
 &\norm{T}(\bB_{\rho}(p_0) = \norm{T}(\bC_\rho(p_0)) - \norm{T}(\bC_\rho(p_0)\setminus \bB_{\rho}(p_0))\nonumber\\
& \overset{\eqref{e:misura}}{\le} \norm{T}(\bC_\rho(p_0))-\frac12 \abs{A_i}\nonumber
\\
& \overset{\eqref{e:excess}} {\le} Q \abs{\Omega^+\cap B_\rho(x_0)} + (Q-1) \abs{\Omega^-\cap B_\rho(x_0)} + \me_T(B_\rho(x_0)) - \frac12 \abs{A_i}\nonumber\\
&\le Q \abs{\Omega^+\cap B_\rho(x_0)} + (Q-1) \abs{\Omega^-\cap B_\rho(x_0)}\nonumber\\
&\qquad + \me_T(B_\rho(x_0)) -  \frac{m}{4N} \abs{B_\rho(x_0)}.\label{e:above}
\end{align}
Notice that
\begin{align} 
 Q \abs{\Omega^+\cap B_\rho(x_0)} &+ (Q-1) \abs{\Omega^-\cap B_\rho(x_0)}\nonumber\\
\le &\Big(Q-\frac12\Big) \abs{B_\rho(x_0)} + \abs{B_\rho(x_0)\cap \{ \psi_1(x')\le x_m < \psi_1(x_0')\}}\nonumber\\
\le &\Big(Q-\frac12\Big) \abs{B_\rho(x_0)} + c\bA \rho \abs{B_\rho(x_0)}.\label{e:above_2}
\end{align}
Moreover $B_\rho(x_0)\setminus U \neq \emptyset$ and  $\me_T(B_\rho(x_0))\le \delta_* \abs{B_\rho(x_0)}$. Combining the latter inequality with \eqref{e:above} and \eqref{e:above_2} we have 
\begin{equation}\label{e:cazzo}
\norm{T}(\bB_{\rho}(p_0)) \le  \abs{B_\rho(x_0)} \left( \Big(Q-\frac12\Big) + c \bA \rho + \delta_* - \frac{1}{4N}\right).
\end{equation}
On the other hand, by   Allard's monotonicity formula and  (v) in Assumption \eqref{Ass:app} we have
\[
e^{C_0 \bA\rho} \omega_m \rho^{-m}  \norm{T}(\bB_{\rho}(p_0))\ge \Theta(T,p_0)\ge Q-\frac{1}{2}
\]
from which we deduce that 
\begin{equation}\label{e:cazzo1}
 \norm{T}(\bB_{\rho}(p_0)) \ge (1 - C_0 \bA \rho) \Big(Q - \frac12\Big) \abs{B_\rho(x_0)}
 \end{equation}

The comparison of \eqref{e:cazzo} and \eqref{e:cazzo1}  gives a contradiction, because, for sufficiently large $a>0$,
\begin{align*}
 \delta_* + (c+ C_0) \bA \rho - \frac{1}{4N} &\le  L_*^2 +4(c+C_0)\bA- \frac{1}{8}\frac{1}{N-1}\\
&\stackrel{\eqref{e:choiceN}}{\le} L_*^2 +(c +C_0) \bA - 4 L_*^2 - 4a^2 \bA < 0.
\end{align*}
This concludes the proof of the claim (Cl). 

\subsection{Conclusion}
Having established the Lipschitz bounds, first we restrict $g^\pm$ to the sets $\Omega^\pm \cap B_3 \setminus U^*$ and then we extend them to $\gammado$ setting:
\begin{align*} g^+(x) &= Q\a{\psi'(x')} \\ g^-(x) &= (Q-1) \a{\psi'(x')}.\end{align*}
We define the ``good'' set to be 
\begin{equation}
K:= (\Omega\cap B_3 \setminus U^* )\cup \gammado
\end{equation}
and \eqref{e:O*} agrees with the claimed estimate on $\abs{B_s \setminus K}$.

Next, write $g^{\pm}(y) = \sum_{i} \a{ ( h_i^\pm(y), \Psi(y, h_i^\pm(y)))}$. Obviously the maps
\[
y \mapsto h^\pm(y):= \sum_{i} \a{h_i^\pm(y)}
\]
are Lipschitz on $K^\pm:=K\cap \Omega^{\pm}$ with Lipschitz constant $33(L_*+ a \bA^\frac12)$.
Recalling \cite[Theorem 1.7]{DS1}, we can extend $h^\pm$ to maps $\bar{u}^\pm \in \Lip (B_3\cap \Omega^\pm, \Iq (\R^{\bar n}))$ satisfying \[ \Lip(\bar u^\pm)\leq C(\delta_*^{\sfrac{1}{2}} + a \bA^{\frac12}) \quad \text{ and } \quad
{\rm osc}\, ( \bar{u}^\pm)\leq C {\rm osc}\, ( h^\pm).\]
Set finally $u^\pm(x) := \sum_i \a{(\bar{u}^\pm_i (x), \Psi(x, \bar{u}^\pm_i(x)))}$. 
We start showing the Lipschitz bound. Fix $x_1, x_2 \in B_3\cap \Omega^\pm$ and assume, without loss of generality, that $\G(\bar u^\pm(x_1), \bar u^\pm(x_2))^2 = \sum_i |\bar u^\pm_i(x_1) - \bar u^\pm_i(x_2)|^2$. Then
\begin{align*}
&\G(u^\pm(x_1), u^\pm(x_2))^2\\ 
& \leq \sum_i \big\vert (\bar{u}^\pm_i (x_1), \Psi(x_1, \bar{u}^\pm_i(x_1))) - (\bar{u}^\pm_i (x_2), \Psi(x_2, \bar{u}^\pm_i(x_2))) \big\vert^2\\
&\leq 2\sum_i \Big((1+\|D_y\Psi\|_0^2)|\bar u^\pm_i(x_1) - \bar u^\pm_i(x_2)|^2 + \|D_x\Psi\|_0^2 |x_1-x_2|^2 \Big)\\
& \leq 2 (1+\|D\Psi\|_0^2) \G(\bar u^\pm(x_1), \bar u^\pm(x_2))^2 + 2 \|D\Psi\|_0^2 |x_1-x_2|^2\\
&\leq C (\delta_* + a^2 \bA+\|D\Psi\|_0^2) |x_1 -x_2|^2\, .
\end{align*}
Recalling that $\norm{D\Psi}_0 \le C(E^{\sfrac{1}{2}} + \bA)$ the Lipschitz bound follows. 
As for the $L^\infty$ bound, recall that
 ${\rm osc}(u^\pm) = \inf_p \sup_{x \in B_3} \G(u^\pm(x), Q\a{p})$. Proceeding as above we then conclude
\begin{align*}
{\rm osc}(u^\pm)^2 &\leq \inf_p \sup_{x \in B_3} \G(u^\pm(x), Q\a{(p, \Psi(0,p))})^2 \\
&\leq {\color{red} 2} \inf_p \sup_{x \in B_3} \Big((1+\|D \Psi\|_0^2)\G(\bar u^\pm(x), Q\a{p^\pm})^2  +  \|D \Psi\|_0^2 |x|^2\Big)\\
& \leq {\color{red} 2} (1+\|D\Psi\|_0^2) {\rm osc}(\bar u^\pm)^2 + {\color{red} 18} \,\|D\Psi\|_0^2.
\end{align*}
The identity $\mathbf{G}_{u^\pm} \res (K^\pm\times \R^{n})= T\res (K^\pm\times\R^{n})$ is a consequence
of $u^\pm(x) = T_x$ for a.e. $x\in K^\pm$. Indeed, recall that both $T$ and $\bG_{u^\pm}$ are rectifiable and observe that\footnote{ Here we use the notation $\langle \vec{v_1}, \vec{v_2}\rangle$ for the standard inner product between $m$-vectors and $S\res \omega$ for the restriction of currents $S$ on forms $\omega$.} $\langle \vec{T}, 
\vec\pi_0\rangle \neq 0$ $\|T\|$-a.e. on $K\times \R^n$, because $\me_T < \infty$ on $K$. Similarly,
$\langle \vec{\bG}_{u^\pm}, \vec{\pi}_0\rangle \neq 0$ $\|\bG_{u^\pm}\|$-a.e. on $K^\pm\times \R^n$, by \cite[Proposition 1.4]{DS2}.
Thus, $(\bG_{u^\pm} - T)\res K^\pm\times \R^n =0$ if and only if $(\bG_{u^\pm} - T) \res dx_1\wedge \ldots \wedge dx_n\,  {\mathbf 1}_{K^\pm\times \R^n}=0$.
The latter identity follows from the slicing formula and the property $\langle T, \bp, x \rangle =
\langle \bG_{u^\pm}, \bp, x\rangle = \sum_i \a{(x, u^\pm_i (x))}$, valid for a.e. $x\in K^\pm$. Finally, to prove  \eqref{e:graphmass} we simply not that by \eqref{e:stimaK}, \eqref{e:differenza} and \eqref{e:excess_3},
\[
\begin{split}
\|T-\bG_{u^+}-\bG_{u^-}\|(\bC_s(x))&=\|T-\bG_{u^+}-\bG_{u^-}\|(\bC_s(x)\setminus (K\times \R^n))
\\
&\le \|T\|(\bC_s(x)\setminus (K\times \R^n))+C|B_{3}\setminus K|\\
&\le E+(C+Q)|B_{3}\setminus K|\le C E.
\end{split}
\]

\section{Lipschitz approximation of Sobolev maps}
Before coming to Theorem \ref{t:harm_1}, we need a preliminary lemma, which is a modification
of a corresponding statements in \cite{DS3}.

\begin{lemma}\label{l:lip_app}
Let $(f^+, f^-)$ be a $\qhalf$-valued function on $B_r$ with interface $(\gammado, 0)$ where $\gammado = \{x_m =0\}$. Then for every $\varepsilon$ there
exists a $\qhalf$-valued function $(f_\varepsilon^+, f_\varepsilon^-)$ with interface $(\gammado, 0)$ such that
\begin{itemize}
\item[(a)] $f_\varepsilon^+$ and $f_\varepsilon^-$ are Lipschitz continuous;
\item[(b)] The following estimate holds:
\begin{align}
\int_{B^\pm_r}\G(f^\pm,f^\pm_\varepsilon)^2 &+\int_{B^\pm_r}\big(|Df^\pm|-|Df^\pm_\varepsilon|\big)^2\nonumber\\
&+ \int_{B^\pm_r} |D (\etaa\circ f^\pm)- D(\etaa\circ f^\pm_\eps)|\big)^2
\leq \eps.\label{e:lip_smoothing}
\end{align}
\end{itemize}
If $f\vert_{\partial B^\pm_r}\in W^{1,2}(\partial B^\pm_r,\Iq)$,
then $f^\pm_\eps$ can be chosen to satisfy also
\begin{equation}\label{e:approx bordo}
\int_{\partial B^\pm_r}\G(f^\pm,f^\pm_\eps)^2+\int_{\partial B^\pm_r}\big(|Df^\pm|-|Df^\pm_\eps|\big)^2 \leq \eps.
\end{equation}
\end{lemma}
\begin{proof}
Firstly we argue that once we have the properties (a) and (b), the additional conclusion \eqref{e:approx bordo}
can be easily inferred using the same trick of \cite[Lemma 4.5]{DS3}. Indeed, without loss of generality, assume $r=1$ and, using the hypothesis $f\vert_{\partial B^\pm_1}\in W^{1,2}(\partial B^\pm_1,\Iq)$, extend
the maps on $B^\pm_{2}\setminus B^\pm_1$ as $0$-homogeneous: the extension $(\hat{f}^+, \hat{f}^-)$ are then still in $W^{1,2}$ and they form a $\qhalf$-valued function with interface $(\gammado, 0)$ (note that \(\gamma\) is flat). Moreover $\hat f^\pm ((1+\delta) x) = f^\pm (x)$ for every $\delta>0$ and every $x\in \partial B^\pm_1$. 

Assuming that we can prove (a) and (b) for a general $r$, we infer the existence of a sequence $(u_k^+, u_k^-)$ of Lipschitz $\qhalf$ approximations such that
\begin{align*}
\int_{B^\pm_2}\G(\hat f^\pm,u^\pm_k)^2 &+\int_{B^\pm_2}\big(|D\hat f^\pm|-|Du^\pm_k|\big)^2\\
&+ \int_{B^\pm_2} |D (\etaa\circ \hat f^\pm)- D(\etaa\circ u^\pm_k)|\big)^2 \to 0\, .
\end{align*}
By Fubini, there is a sequence $\delta_k\downarrow 0$ such that
\[
\int_{\partial B^\pm_{1+\delta_k}}\G(\hat f^\pm,u^\pm_k)^2+\int_{\partial B^\pm_{1+\delta_k}}\big(|D\hat f^\pm|-|Du^\pm_k|\big)^2  \to 0\, .
\]
By a straightforward computation, if we define $f^\pm_k (x) := u^\pm_k (x/(1+\delta_k))$, then we have at the same time
\begin{align*}
&\int_{B^\pm_1}\G(f^\pm,f^\pm_k)^2+\int_{B^\pm_1}\big(|D f^\pm|-|Df^\pm_k|\big)^2\\
&\qquad\qquad\qquad + \int_{B^\pm_1} |D (\etaa\circ f^\pm)- D(\etaa\circ f^\pm_k)|\big)^2 \to 0\\
&\int_{\partial B^\pm_{1}}\G(f^\pm,f^\pm_k)^2+\int_{\partial B^\pm_{1}}\big(|D f^\pm|-|Df^\pm_k|\big)^2  \to 0\, .
\end{align*}

\medskip

We now come to the main part of the lemma, namely the points (a) and (b). First of all, without loss of generality, we can assume that $r=1$. We next define the auxiliary function $h\in W^{1,2} (B_1, \Iqs)$ as
\[
h (x) :=
\left\{\begin{array}{ll}
f^+ (x) \qquad &\mbox{if $x_m>0$}\\
f^- (x) + \a{0} &\mbox{if $x_m <0$.}
\end{array}\right.
\]
Observe that $|Df^+ (x)| = |Dh (x)|$ for every $x\in B_1^+$ and $|Df^- (x)| = |Dh (x)|$ for every $x\in B_1^-$. Consider the maximal function
$\bmm (|Dh|) (x)$ and let
\[
K_\lambda := \{x: \bmm (|Dh|) (x) \leq \lambda\}\, 
\]
which is a closed set, since maximal functions are lower semicontinuous. 
Arguing as in \cite[Proposition 4.4]{DS1} we conclude that $h|_{K_\lambda}$ is Lipschitz with a constant $C \lambda$ (where $C$ depends only upon $m$). Moreover, by the standard maximal function estimates, we have 
\begin{equation}\label{maxest}
\lambda^2 |B_1\setminus K_\lambda| \leq C \int_{B_1\setminus K_{\lambda/2}} |Dh|^2\, .
\end{equation}
We next consider the symmetrized set 
\[
K^s_\lambda := \{(x', x_m)\in K_\lambda : (x', -x_m)\in K_\lambda\}
\] 
and observe that
\[
|B_1\setminus K_\lambda^s|\leq 2 |B_1\setminus K_\lambda|\, .
\]
By an elementary comparison\footnote{ Indeed, fix $x$ and $y$ and assume without loss of generality that $h_Q (x) = h_Q (y) =0$, and that $h_i (x)= f^-_i (x)$ and $h_i (y) = f^-_i (y)$ for every $i\leq Q-1$. Let $\pi$ be a permutation of the set $\{1, \ldots, Q\}$ such that 
\[
\mathcal{G} (h(x), h (y))^2 = \sum_i |h_i (x) - h_{\pi (i)} (y)|^2\, .
\]
We define a permutation $\sigma$ of $\{1, \ldots , Q-1\}$ in the following way. If $\pi (Q) = Q$, then we simply set $\sigma (j) = \pi (j)$ for every $j\leq Q-1$ and we easily that $\mathcal{G} (h(x),h(y))\geq \mathcal{G} (f^- (x), f^- (y))$. Otherwise there is a $j_0 \leq Q-1$ such that $\pi (j_0) = Q$ and an $i_0 \leq Q-1$ such that $\pi (i_0) = Q$. We then set $\sigma (i_0) = j_0$ and
$\sigma (k) = \pi (k)$ for every $k\in \{1, \ldots , Q-1\}\setminus \{i_0\}$. We can therefore compute
\begin{align*}
& \mathcal{G} (f^- (x), f^- (y))^2\\ 
&\leq \sum_{i\leq Q-1} |f_i^- (x) - f_{\sigma (i)}^- (y)|^2
= \sum_{i\leq Q-1, i\neq i_0} |h_i (x) - h_{\pi (i)} (y)|^2 + |h_{i_0} (x) - h_{j_0} (y)|^2\\
&\leq \sum_{i\leq Q-1, i\neq i_0} |h_i (x) - h_{\pi (i)} (y)|^2 + 2 |h_{i_0} (x)|^2 + 2|h_{j_0} (y)|^2\\ 
&=  \sum_{i\leq Q-1, i\neq i_0} |h_i (x) - h_{\pi (i)} (y)|^2 + 2 |h_{i_0} (x) - h_{\pi (i_0)} (y)|^2 + 2 |h_Q (x) - h_{\pi (Q)} (y)|^2\\
&= \mathcal{G} (h(x). h(y)^2 +  |h_{i_0} (x) - h_{\pi (i_0)} (y)|^2 + |h_Q (x) - h_{\pi (Q)} (y)|^2 \leq 2 \mathcal{G} (h(x), h(y))^2\, .
\end{align*}}  
we easily see that  
\[
\G (f^- (x), f^- (y))\leq \sqrt{2}\, \G (h (x), h(y))\, .
\]
 Hence the Lipschitz constant of 
the restriction of $f^-$ to \(K_\lambda^s\cap B_1^-\) is at most $3C\lambda$ and we can extend it to a function $g^-$ on $B_1^-$ with Lipschitz constant at most $C'\lambda$, for some $C'$ depending only upon $m,n$ and $Q$, cf. \cite[Theorem 1.7]{DS1}. Consider now the function $k: B_1^-\cup (B_1^+\cap K^s_\lambda) \to \Iqs$ such that
\[
k (x) := \left\{
\begin{array}{ll}
g^- (x) + \a{0}\qquad &\mbox{for $x\in B_1^-$}\\
f^+ (x) \qquad &\mbox{for $x\in B_1^+ \cap K^s_\lambda$}\, .
\end{array}\right.
\]
We claim that $k$ is in fact Lipschitz with constant at most $C \lambda$. Fix two points $x,y$ in the domain of the function: if they are both in $B^+_1$ or both in $B^-_1$ then our claim is obvious, given the Lipschitz bounds on $g^-$ and $f^+|_{K^s_\lambda}$, respectively. Fix otherwise 
$x = (x', x_m)\in K^s_\lambda \cap B^+_1$ and $y\in B^-_1$. Consider now $x^s := (x', -x_m)$ and observe that $x^s\in K^s_\lambda$. On the other hand
\[
|x^s -x| = 2 x_m \leq 2 |x-y|\, .
\]
We can therefore estimate
\begin{align*}
\G (k(x), k(y)) \leq\; &\G (k(x), k(x^s)) + \G (k (x^s), k (y))\\
 =\; & \G (h(x), h(x^s)) + \G (k (x^s), k (y))\\
\leq\; & \G (h(x), h(x^s)) + 3 \G (g^- (x^s), g^- (y))\\
\leq\;  &C \lambda |x-x^s| + C \lambda |x^s - y| \leq C \lambda |x-y|\, .
\end{align*}
We can now extend $k$ to a Lipschitz map on the whole ball $B_1$ and we define $g^+ (x)$ equal to such extension for every $x\in B_1^+$. Observe therefore that $(g^+, g^-)$ is a $\qhalf$-valued function with interface $(\gammado, 0)$. Moreover the Lipschitz constant is controlled by $C\lambda$.  Note also that  $g^\pm$ and $f^\pm$ coincide on $K_\lambda^s \cap B_1^\pm$. 

Consider next that the functions 
\[
\alpha^\pm := \G (f^\pm, g^\pm)\, ,
\]
vanish on $K^s_\lambda$. Furthermore  by  choosing $\lambda$ sufficiently large we can assume that \(|K^s_\lambda\cap B_1^\pm|\ge 1/2 |B_1^\pm|\). Thus the Poincar\'e inequality gives
\[
\int_{B_1^\pm} \G (f^\pm, g^\pm)^2 = \int_{B_1^\pm} (\alpha^\pm)^2 \leq C \int_{B_1^\pm} |D\alpha^\pm|^2\, .
\]
Moreover, recalling that \(|B_1 \setminus K^s_\lambda|\leq 2 |B_1\setminus K_\lambda|\) and \eqref{maxest}
\begin{align*}
 \int_{B_1^\pm} &\left( |D\alpha^\pm|^2 + (|Df^\pm| - |Dg^\pm|)^2 + |D (\etaa\circ f^\pm) - D (\etaa\circ g^\pm)|^2\right)\\
\leq &\;C  \int_{B_1^\pm\setminus K^s_\lambda} \left( |Df^\pm|^2 + |Dg^\pm|^2\right)
\leq C \int_{B_1^\pm\setminus K^s_\lambda} \left( |Df^\pm|^2 + \lambda^2\right)\\
\leq & \; C \int_{B_1^\pm\setminus K^s_\lambda} |Df^\pm|^2 + C\lambda^2 |B_1\setminus \lambda|\\
\leq& C \int_{B_1^\pm\setminus K^s_\lambda} |Df^\pm|^2+C\int_{B_1\setminus K_{\lambda/2}} |Dh|^2\,\to 0\, .
\end{align*}
Since the latter converges to $0$ as $\lambda\to \infty$, we conclude the proof.
\end{proof}

%

\section{Proof of Theorem \ref{T:HARM_1}}
It is not restrictive to assume that
$x=0$ and $r=1$. Thus $\Psi (0) =0$ and $\psi (0) =0$.

\subsection{Proof of (\ref{e:eta-star-condition}) and (\ref{e:small-energy-condition}).}
Firstly we want to note that \eqref{e:small-energy-condition} is a consequence of \eqref{e:eta-star-condition}. Indeed, use first
\eqref{prop:Lipschitz_16}, \eqref{e:stimaK} and
\eqref{e:eta-star-condition} to estimate
\[
|B_{2}\setminus K| \leq C \eta_* E^{1-2\beta}\, .
\]
Since $\Lip (u^\pm) \leq C E^{\beta}$, \eqref{e:small-energy-condition} follows easily.

We  fix $\beta$ and $\eta_*$. Assuming by contradiction that the statement is false we find a sequence of area-minimizing currents $T_k$ and submanifolds $\Sigma_k$, $\gammaup_k$ satisfying the following properties:
\begin{itemize}
\item[(i)] The cylindrical excesses satisfy the estimate
\begin{equation}
E_k := \bE (T_k, \bC_4 (0), \pi_0) = \frac{1}{2\omega_m} \int_{\bC_4 (0, \pi_0)} |\vec{T_k}-\vec{\pi_0}|^2\, d\|T_k \| \leq \frac{1}{k}\, .
\end{equation}
\item[(ii)]  $\gammaup_k$ are smooth submanifolds of dimension $m-1$ and $\Sigma_k \subset \R^{m+n}$ are smooth submanifolds of dimension $m + \bar{n}= m+ n - l$ containing $\gammaup_k$. After possibly changing coordinates appropriately (cf. Remark \ref{r:Psi}), $\Sigma_k$ and $\gammaup_k$ are graphs of entire functions $\Psi_k: \R^{m+ \bar{n}} \to \R^l$ and $\psi_k: \R^{m-1}\to \R^{\bar{n}+1+l}$ satisfying the bounds 
\begin{align}
\norm{\Psi_k}_{C^2 (B_8)}\le C (E_k^{\sfrac{1}{2}} + \bA_k) \le C E_k^{\sfrac{1}{2}}\\
\norm{\psi_k}_{C^2 (B_8)}\le C \bA_k\leq \frac{C}{k} E_k^{\sfrac{1}{2}}\, \label{e:interface_k}.
\end{align}
\item[(iii)] Assumption \ref{Ass:app} holds for each $T_k$. 
\item[(iv)] The estimate \eqref{e:eta-star-condition} fails, i.e.,
\begin{equation}\label{e:contradiction_1}
\be_{T_k} (B_{5/2}\setminus K_k)>  \eta_*E_k = 5 c_2 E_k\, ,
\end{equation}
for some positive $c_2$. The pair of $\qhalf$-valued maps $(f_k^+, f_k^-)$ denotes the $E_k^\beta$-Lipschitz approximations of the current $T_k$. 
\end{itemize}
For every $s> 5/2$, we have
\begin{equation}\label{e:improv}
\be_{T_k}(K_k\cap B_s)\leq \be_{T_k}(B_s)-5\,c_2\,E_k.
\end{equation}
In order to simplify our notation, we use $B^\pm_{k,r}$ for the domains of the functions $f^\pm_k$ intersected with the ball $B_r (0) \subset \pi_0$. Instead $B_r^\pm$ denotes the corresponding limits,
namely the sets $B_r^\pm := B_r (0) \cap \{\pm x_m \geq 0\}$. Using this notation and
the Taylor expansion of the area functional, since
$E_k\downarrow 0$, we conclude the following inequalities for every $s\in\left[5/2,3\right]$:
\begin{align}\label{e:improv2}
\int_{K_k\cap B^+_{k,s}}\frac{|Df^+_k|^2}{2}+ \int_{K_k\cap B^-_{k,s}}\frac{|Df^-_k|^2}{2} & \leq (1+C\,E_k^{2\beta})\,\be_{T_k}(K_k\cap B_s)\notag\\
&\leq (1+C\,E_k^{2\beta}) \,\Big(\be_{T_k}(B_s)-5\,c_2\,E_k\Big)
\;\\
& \leq \be_{T_k}(B_s)-4c_2\,E_k.
\end{align}
Our aim is to show that \eqref{e:improv2} contradicts the minimizing property of $T_k$. To construct a competitor we write $f^\pm_k(x) =  \sum_{i} \a{ (f^\pm_k)_i(x)}$ and denote by $(f^\pm_k)''_i(x)$ the first $\bar{n}$ components of the point $(f^\pm_k)_i(x)$. This induces a $\qhalf$ valued map $(f^\pm_k)'':=\sum_{i} \a{(f^\pm_k)''_i(x)}$, namely a pair of maps taking values, respectively, in $\Iq (\R^{\bar{n}})$ and $\mathcal{A}_{Q-1} (\R^{\bar n})$. 
Observe that, since $(f^\pm_k)_i(x)$ are indeed point of the manifold $\Sigma_k$, then 
\[ f^\pm_k(x) = \sum_i \a{\left((f^\pm_k)''_i(x), \Psi_k(x,(f^\pm_k)''_i(x)) \right) }\,.\]
Moreover, by \eqref{e:improv2},  the fact that $\Lip(f_k^\pm)\le C E_k^\beta$ and $\abs{B_3\setminus K_k} \le C E_k^{1-2\beta}$ gives 
\begin{equation}\label{e:dirichletbound} 
\D(f_k^+) + \D(f_k^-) \le C E_k\,.
\end{equation}

Let $((\psi_k)^1(x'), (\psi_k)''(x'))$ be the first $\bar{n}+1$ components of the map $\psi$ whose graph gives $\gammaup_k$. We consider the $\qhalf$ valued map $(g^+_k, g^-_k)$ with $g^\pm_k:= E_k^{-\frac 12} (f^\pm_k)''$ with interface $(\gammado_k, \varphi_k)$ where 
\[
\gammado_k=\{x_m =  (\psi_k)^1(x') \}\qquad \textrm{and}\qquad \varphi_k(x') = E_k^{-\frac 12} (\psi_k)''(x'). 
\]
By assumption \eqref{e:interface_k}, denote by $\gamma$ the plane $\{x_m=0\}\subset \pi_0$, we have that $(\gammado_k, \varphi_k) \to (\gammado, 0)$ in $C^1$. 

For each $k$ we let $\Phi_k$ be a diffeomorphism which maps $B_3$ onto itself and $\gammado_k\cap B_3$ onto $\gammado\cap B_3$. Clearly this can be done so that $\|\Phi_k - {\rm Id}\|_{C^1} \to 0$. Moreover, given the convergence of $\gammado_k$ to $\gammado = \{x_m =0\}$, it is not difficult to see that we can require the property $\Phi_k (\partial B_r) = \partial B_r$ for
every $r\in [2,3]$ (provided $k$ is large enough)\footnote{ A simple procedure to define the map on each sphere $\partial B_r$ is the following. Consider the north and south poles $P^\pm_r= (0,\ldots, 0, r)$. On each great circle $C_r$ passing through $P^+_r$ and $P^-_r$ consider the corresponding half circles connecting $P^\pm_r$. Each have exactly one intersection with, respectively, $\{x_m =0\}$ and $\gamma_k$. We then map both half circles onto themselves by keeping the map an identity around the poles and moving the intersections with $\gamma_k$ to the intersections with $\{x_m = 0\}$. If we use polar coordinates on the circle $C_r$ so that the north and south poles are given by $\pm\frac{\pi}{2}$, we then can assume that one half circle is parametrized by $[-\frac{\pi}{2}, \frac{\pi}{2}]$: we seek a map which is the identity around $\pm \frac{\pi}{2}$ and which maps a small given $\alpha$ in $0$. Consider then a bump function $\lambda$ which is supported in $(-1,1)$ and identically $1$ on $(-\frac{1}{2}, \frac{1}{2})$: an explicit formula for such a map is 
\[
\theta \mapsto \theta (1-\lambda (\theta)) + \lambda (\theta) (\theta -\alpha)\,.
\]} 
Furthermore we have that $\norm{\varphi_k \circ \Phi^{-1}_k }_{C^1(B_3)} \to 0$ so we can choose $\varkappa_k \in C^1(B_3)$ with $\varkappa_k = \varphi_k \circ \Phi^{-1}_k $ on $\gammado$ and $\norm{\varkappa_k}_{C^1(B_3)} \to 0 $. Now define the $\qhalf$ valued maps
\[ \hat{g}^\pm_k(x):= \sum_i \a{ (g^\pm_k)_i\circ \Phi_k^{-1}(x) - \varkappa_k(x) }. \]
We observe that $(\hat{g}_k^+, \hat{g}_k^-)$ is a $\qhalf$ valued map with interface $(\gammado, 0)$ and by straightforward computations 
\begin{align}
 &\D(\hat{g}_k^\pm , \Phi^{-1}_k(A)\cap B^\pm)\nonumber\\
=\; & (1+ o(1)) \left(\D(g_k^+, A\cap B^\pm_{k})+ \D(g_k^-)\right) +o(1)\label{e:sconv}
\end{align}
for all measurable \(A\subset B_{3}\)\, 
where \(o(1)\) is independent of the set \(A\).
From \eqref{e:dirichletbound} we conclude that the Dirichlet energy of $(\hat{g}_k^+, \hat{g}_k^-)$ is uniformly bounded. By the Poincar\'e inequality and since the maps collapse at their  interfaces, their $L^2$ norms are uniformly bounded as well. By compactness we can find a subsequence (not relabeled) and a $\qhalf$ valued map $(g^+, g^-)$ with interface $(\gammado, 0)$ such that 
\[
\norm{\G(\hat{g}^\pm_k\circ \Phi_k^{-1}, g^\pm)}_{L^2(B^\pm_3)} \to 0
\] 
and 
\begin{align*}
\D(g^+) + \D(g^-) &\le \liminf_{k \to \infty } (\D(\hat{g}^+_k) + \D(\hat{g}^-_k))\\
 &= \liminf_{k \to \infty } (\D(g^+_k) + \D(g^-_k))\,.
\end{align*}
Up to extracting a subsequence, we can assume that \(|D\hat g_k^\pm|\weak G^\pm\) weakly in \(L^2(B_3)\). One can then easily check, see for instance the proof of \cite[Proposition 4.3]{DS3}, that 
\[
|D g^\pm|\le G^\pm.
\]
In particular, since  \(|B_3\setminus K_k|\to 0\), we deduce that for every \(s \in (0,3)\):
\begin{equation}\label{e:stica}
\begin{split}
 &\D(g^\pm, B^\pm_s)\le \liminf_{k\to \infty} \int_{B_s^\pm \cap\Phi_k(K_k)}  (G^\pm)^2
\\
\le\; & \liminf_{k\to \infty}\D(\hat g_k^\pm, B^\pm_s\cap \Phi_k(K_k) )\le \liminf_{k\to \infty}\D(g_k^\pm, B^\pm_s\cap K_k )
\end{split}
\end{equation}
where in the last inequality we have used \eqref{e:sconv}.

Let \(\varepsilon>0\) be a small parameter to be chosen later, we apply Lemma \ref{l:lip_app} to $(g^+, g^-)|_{B_3}$ with \(\varepsilon \) to  produce  a Lipschitz functions  $(g^+_\varepsilon, g^-_\varepsilon)$ satisfying all the estimates there.

We would like  to use Lemma \eqref{l:interpolation} to interpolate between  $(\hat{g}^+_k, \hat{g}^-_k)$ and $(g^+_\varepsilon, g^-_\varepsilon)$ (note that  both have interface $(\gammado,0)$). However we would like the functions $(\hat{g}^+_k, \hat{g}^-_k)$ not to  concentrate too much  energy in the transition region. To this end let us define the Radon measures 
\[
\mu_k(A) = \int _{A\cap B^+_{3}} |D \hat{g}^+_k|^2+\int _{A\cap B^-_{3}} |D \hat{g}^-_k|^2 \qquad A \subset B_3.
\]
Up to the extraction of a subsequence we can assume that \(\mu_k\weaks \mu\) for some Radon measure \(\mu\). We now choose \(r\in (5/2,3)\) and a subsequence, not relabeled, such that 
\begin{itemize}
\item[(A)] \(\mu(\partial B_r)=0\) 
\item[(B)] \(\bM (\langle T_k - (\bG_{f_k^+}+ \bG_{f_k^-}), |\bp|, r \rangle) \leq C E_k^{1-2\beta}\), 
where the map $|\bp|$ is given by $\pi_0\times \pi_0^\perp \ni (x,y) \to |x|$. 
\end{itemize}
Indeed (A) is true for all but countably many radii while (B) can be obtained from the estimate \eqref{e:graphmass} through the combination of Fatou's Lemma and Fubini's Theorem.  In particular, by (A) and the properties of weak  convergence of measures,  we have
\begin{equation*}
\begin{split}
\limsup_{s\to r} &\limsup_{k\to \infty} \int _{ B^+_{r}\setminus B^+_s} |D \hat{g}^+_k|^2+\int _{A\cap B^-_{r}\setminus B^-_s} |D \hat{g}^-_k|^2
\\
&\le \limsup_{s \to r} \mu (\overline{B}_r\setminus B_s)=0.
\end{split}
\end{equation*}
Hence, given \(r\in (5/2,3)\) satisfying (A) and (B) above, we can now choose \(s\in (5/2,3)\) such that 
\begin{equation}\label{e:aereo}
\limsup_{k\to \infty} \int _{ B^+_{r}\setminus B^+_s} |D \hat{g}^+_k|^2+\int _{B^-_{r}\setminus B^-_s} |D \hat{g}^-_k|^2\le \frac{c_2}{3}.
\end{equation}
We now apply, for  each $k$,  Lemma \eqref{l:interpolation} to connect the functions  $(\hat{g}^+_k, \hat{g}^-_k)$ and $(g^+_\varepsilon, g^-_\varepsilon)$ on the annulus \(B_r\setminus B_s\) . This gives sets $\overline{B}_{s} \subset V^k_{\lambda,\varepsilon} \subset W^k_{\lambda,\varepsilon} \subset B_r$ and a $\qhalf$ valued interpolation map $(\zeta_{k, \varepsilon}^+, \zeta_{k, \varepsilon}^-)$ with
\begin{align*}
& \int_{(W^k_{\lambda,\varepsilon})^\pm \setminus V_{\lambda,\varepsilon}^k} \abs{D\zeta^\pm_{k, \varepsilon}}^2\\
\le\; & C\lambda \int_{(W^k_{\lambda,\varepsilon})^\pm \setminus V_{\lambda,\varepsilon}^k} \abs{D\hat{g}^\pm_k}^2 + \abs{Dg^\pm_\varepsilon}^2 + \frac{C}{\lambda} \int_{(W^k_{\lambda,\varepsilon})^\pm \setminus V_{\lambda,\varepsilon}^k} \G(\hat{g}^\pm_k, g^\pm_\varepsilon)^2
\\
\le\; & C\lambda \int_{(W^k_{\lambda,\varepsilon})^\pm \setminus V_{\lambda,\varepsilon}^k} \abs{D\hat{g}^\pm_k}^2 + \abs{Dg^\pm_\varepsilon}^2\\
&\quad + \frac{C}{\lambda} \int_{(W^k_{\lambda,\varepsilon})^\pm \setminus V_{\lambda,\varepsilon}^k} \big(\G(\hat{g}^\pm_k, g^\pm)^2 +\G(\hat{g}^\pm, g^\pm_\varepsilon)^2\big)
\end{align*}
Hence
 \[
 \limsup_{\lambda \to 0}\limsup_{\varepsilon \to 0} \limsup_{k\to  \infty } \int_{(W^k_\lambda)^\pm \setminus V_\lambda^k} \abs{D\zeta^\pm_{k, \varepsilon}}^2=0.
 \]
Thus we can find $\lambda, \varepsilon>0$  sufficiently small such that 
\begin{equation}\label{e:sonno}
\limsup_{k\to \infty} \int_{(W^k_{\lambda,\varepsilon})^\pm \setminus V_{\lambda,\varepsilon}^k} \abs{D\zeta^\pm_{k, \varepsilon}}^2 < \frac{c_2}{3}.
 \end{equation}
Moreover, up to further reduce \(\varepsilon\), we can also assume that  
 \begin{equation}\label{e:sonno2}
 \int_{B_r^\pm} |D g_\eps^\pm|^2\le \int_{B_r^\pm} |D  g^\pm|^2+\frac{c_2}{6}.
 \end{equation}
Next we define Lipschitz-continuous function  on $B_r$ with interface $(\gammado, 0)$ by (note that since \(\lambda\) and \(\varepsilon\) are fixed we drop the dependence on those parameters for the sake of readability)
\begin{equation}\label{e:h}
\hat{h}^\pm_k := 
\begin{cases} \hat{g}_k^\pm &\text{ on } B_r \setminus (W_{\lambda, \varepsilon}^k)^\pm \\
 \zeta^\pm_{k, \varepsilon} & \text{ on } (W^k_{\lambda, \varepsilon})^\pm \setminus V^k_\lambda\\
 g^\pm_\varepsilon & \text{ on } (V^k_{\lambda, \varepsilon})^\pm.	
 \end{cases}
\end{equation}
Let us  then consider the functions  $h_k^\pm:=\sum_i \a{ (\hat{h}_k^\pm)_i \circ \Phi_k + \varkappa_k \circ \Phi_k }$,  defined on $B^\pm_{k,3}$. The resulting $\qhalf$ valued map $(h_k^+, h_k^-)$ has interface $(\gammado_k, \varphi_k)$ and satisfies 
\begin{align}
\liminf_{k \to \infty} \,&\Big(\D( h_k^+, B_{k,r}^+)+ \D(h_k^-, B_{k,r}r^-) \Big)\nonumber
\\
&= \liminf_{k \to \infty} \Big(\D( \hat h_k^+, B_r^+)+ \D(\hat h_k^-, B_r^-) \Big)\nonumber
\\
 &\le  \D(g_\eps^+, B_r^+) + \D(g_\eps^-, B_r^-)\nonumber
 \\
 &\quad
 +\limsup_{k\to \infty} \Big(\D( \zeta_k^+, (W^k_{\lambda,\varepsilon})^+ \setminus V_{\lambda,\varepsilon}^k)+ \D(\zeta _k^-, (W^k_{\lambda,\varepsilon})^- \setminus V_{\lambda,\varepsilon}^k)\big)\nonumber
 \\
 &\quad+\limsup_{k\to \infty} \Big(\D( \hat g_k^+, B_r^+\setminus B_s)+ \D(\hat g_k^-, B_r^-\setminus B_s)\Big)\nonumber
 \\
 &\le \D(g^+, B_r^+) + \D(g^-, B_r^-)+c_2\label{e:one_c2_lost}
 \\
 &\le \liminf_{k\to \infty} \Big(\D(\hat g_k^+, B_r^+\cap K_k ) + \D(\hat g_k^-, B_r^-\cap K_k)\Big)+c_2\label{e:liminf}
\end{align}
where  in the third inequality  we have used  \eqref{e:sonno}, \eqref{e:sonno2}, \eqref{e:aereo} and  the fourth  one \eqref{e:stica}.

We thus conclude that, for infinitely many  $k$,
\begin{align}
 E_k &\D (h_k^+, B_{k,r}^+) + E_k \D (h_k^-, B_{k,r}^-)\nonumber\\
&\leq\;  \D ((f^+_k)'', B_{k,r}^+\cap K_k) + \D ((f_k^-)'', B_{k,r}^-\cap K_k) + 2 c_2 E_k\, .\label{e:small_loss2}
\end{align}
Let us  consider the functions
\[
v_k^\pm (x) := E_k^{\sfrac{1}{2}} h_k^\pm (x)  
\]
and
\[
w_k^\pm (x) := \sum_i \a{\left(v_k^\pm (x), \Psi_k (x, v_k^\pm (x))\right)}\, .
\] 
Observe that
$w_k^\pm|_{\partial B_r} = f_k^\pm$ and $\Lip(w_k^\pm) \le C E_k^\beta$. 

We are now ready to construct our competitor currents to test the minimality of the sequence $T_k$. First of all, by the isoperimetric inequality, there is a current $S_k$ supported in \(\Sigma_k\) such that 
\begin{align*}
& \partial S_k = \langle T_k - (\bG_{f_k^+}+ \bG_{f_k^-}), |\bp|, r \rangle\\ 
\mbox{and}\quad
&\bM (S_k) \leq C (E_k^{1-2\beta})^{\frac{m}{m-1}} = o (E_k)\, .
\end{align*}
where we have used that \(\beta<\frac{1}{4m}\).
Let $Z_k = \bG_{w_k^+} \res \bC_r + \bG_{w_k^-} \res \bC_r + S_k$. We easily see that the boundary of $Z_k$ matches that of $T_k \res \bC_r$ and that the support of $Z_k$ is contained in $\Sigma_k$. Thus it is an admissible competitor and we must have
\[
\bM (Z_k) \geq \bM (T_k \res \bC_r)\,.
\]
On the other hand, using the Taylor expansion of the mass, the bound on $\Lip (h_k^\pm)$ and the bound on $\bM (S_k)$, we easily conclude that
\begin{equation}\label{e:test_minimality}
\D (w_k^+, B^+_{k,r}) + \D (w_k^-, B^-_{k,r}) \geq 2 \be_{T_k} (B_r) - o (E_k)\, . 
\end{equation}
We next compute
\begin{equation}\label{e:bus}
\begin{split}
& \D (w^+_k, B_{k,r}^+) - \D (f^+_k , B_{k,r}^+ \cap K_k)\nonumber\\
= &
\underbrace{\int_{B_{k,r}^+} |D v^+_k |^2- \int_{B_{k,r}^+\cap K_k} |D(f_k^+)''|^2}_{I_1}\nonumber\\
&\quad 
+ \underbrace{\int_{B_{k,r}^+} |D (\Psi_k (x, v^+_k))|^2 - \int_{B_{k,r}^+} |D (\Psi_k (x, (f^+_k)''))|^2}_{I_2}\,\nonumber\\ 
&\quad +\underbrace{\int_{B_{k,r}^+\setminus K_k} |D(\Psi_k (x, (f_k^+)''))|^2}_{I_3}\notag .
\end{split}
\end{equation}
By \eqref{e:small_loss2} we already know that $I_1\leq 2 c_2 E_k$ for infinitely many \(k\). For what concerns $I_2$, we proceed as follows. First we write 
\begin{align*}
I_2 & = \sum_i\int_{B_{k,r}^+} (D(\Psi_k(x,v^+_k(x))_i-D(\Psi_k(x,(f_k^+)''(x))_i):\\
&\qquad\qquad\qquad\qquad(D(\Psi_k(x,v^+_k(x))_i+D(\Psi_k(x,(f_k^+)'' (x))_i) .\notag
\end{align*}
Next, recalling the chain rule \cite[Proposition~1.12]{DS1}, we get
\begin{align}
&\big|D(\Psi_k (x,v^+_k(x))_i + D(\Psi_k (x,(f_k^+)'' (x))_i\big|\nonumber\\
\leq\; & C \|D_x\Psi_k\|_0 + C \|D_u \Psi_k\|_0 (\Lip (v_k) + \Lip ((f_k^+)''))
 = C E_k^{\sfrac{1}{2}}\, .\notag
\end{align}
Using the latter inequality and the chain rule  again, we obtain
\begin{align}
I_2 \leq & C E_k^{\sfrac{1}{2}} \int_{B^+_{k,r}} \Big(\sum_i | D_x\Psi_k (x, (v_k^+)_i(x))  - D_x\Psi_k (x, ((f_k^+)'')_i (x))|\notag\\
& \qquad \qquad\qquad + \|D_u \Psi_k\|_0 \left(|Dv^+_k| + |D (f^+_k)''|\right)\Big)\notag\\
\leq{}& C\, E_k^{\sfrac{1}{2}} \|D^2\Psi_k\|_0 \int_{B^+_{k,r}} \G(v^+_k,(f^+_k)'') + C\,E_k \int_{B^+_{k,r}} \left(|Dv^+_k| + |D(f_k^+)''|\right)\notag\\
\leq{}& C\, E_k^{\sfrac{3}{2}}\, .
\end{align}
Finally,
 \[
 I_3\leq C \|D\Psi_k\|_\infty^2 |B_3\setminus K_k|+C\|D_u \Psi_k\|_\infty^2 \int_{B_r} |(Df^+_k)''|^2 \leq C E_k^{2-2\beta}+CE_{k}^{2}.
 \]
Hence $I_1+I_2+ I_3 \le 2c_2E_k + o (E_k)$. Since an analogous estimates holds replacing \(+\) with \(-\), we conclude that
\begin{align}
& \D (w_k^+, B_{k,r}^+) + \D (w_k^-, B_{k,r}^-)\nonumber\\
\leq & \D (f_k^+, B_{k,r}^+ \cap K_k) + \D (f_k^-, B_{k,r}^-\cap K_k) + 4 c_2 E_k + o (E_k)\, . \label{e:small_loss3}
\end{align}
However, the latter inequality combined with \eqref{e:improv2} implies
\begin{equation}\label{e:improv3}
\D (w_k^+, B_{k,r}^+) + \D (w_k^-, B_{k,r}^-) \leq 2 \be_{T_k} (B_r ) -  c_2 E_k + o (E_k)\, .
\end{equation}
Clearly \eqref{e:test_minimality} and \eqref{e:improv3} are incompatible for $k$ large enough. 
This completes the proof of the first part of the theorem.

\subsection{Proof of {\eqref{e:harm-app1}}, {\eqref{e:harm-app2}} and {\eqref{e:harm-app3}}.} We again argue by contradiction. Assume the second part of the theorem is false for some $\eta_*$. We then have again a sequence of area-minimizing currents $T_k$ and submanifolds $\Sigma_k$, $\gammaup_k$ satisfying the properties (i), (ii) and (iii) of the previous step, which we recall here for the reader's convenience together with 
the  fourth contradiction assumption. More precisely:
\begin{itemize}
\item[(i)] The cylindrical excesses satisfy the estimate
\begin{equation}
E_k := \bE (T_k, \bC_4 (0), \pi_0) = \frac{1}{2\omega_m} \int_{\bC_r (0, \pi_0)} |\vec{T_k}-\vec{\pi_0}|^2\, d\|T_k \| \leq \frac{1}{k}\, .
\end{equation}
\item[(ii)] $\gammaup_k$ are smooth submanifolds of dimension $m-1$ and $\Sigma_k \subset \R^{m+n}$ are smooth submanifolds of dimension $m + \bar{n}= m+ n - l$ containing $\gammaup_k$. $\Sigma_k$ and $\gammaup_k$ are graphs of entire functions $\Psi_k: \R^{m+ \bar{n}} \to \R^l$ and $\psi_k: \R^{m-1}\to \R^{\bar{n}+1+l}$ satisfying the bounds 
\begin{align}
\norm{\Psi_k}_{C^2 (B_8)}\le C (E_k^{\sfrac{1}{2}} + \bA_k) \le C E_k^{\sfrac{1}{2}}\\
\norm{\psi_k}_{C^2 (B_8)}\le C \bA_k\leq \frac{C}{k} E_k^{\sfrac{1}{2}}\, \label{e:interface_k_2}.
\end{align}
\item[(iii)] Assumption \ref{Ass:app} holds for each $T_k$. 
\item[(iv)] The $E_k^\beta$-Lipschitz approximations $(f_k^+, f_k^-)$ fail to satisfy one among the estimates \eqref{e:harm-app1},
\eqref{e:harm-app2} and \eqref{e:harm-app3} for any choice of the function $\kappa$. 
\end{itemize}
As in the previous step we write $f^\pm_k(x) =  \sum_{i} \a{ (f^\pm_k)_i(x)}$ and denote by $(f^\pm_k)''_i(x)$ the first $\bar{n}$ components of the point $(f^\pm_k)_i(x)$. This induces a $\qhalf$ valued function  $(f^\pm_k)'':=\sum_{i} \a{(f^\pm_k)''_i(x)}$ with values in $\Iqs(\R^{\bar{n}})$ and $\mathcal{A}_{Q-1} (\R^{\bar{n}})$. Observe that, since $(f^\pm_k)_i(x)$ are indeed points of the manifold $\Sigma_k$, then 
\[ f^\pm_k(x) = \sum_i \a{\left((f^\pm_k)''_i(x), \Psi_k(x,(f^\pm_k)''_i(x)) \right) }\,.\]
We keep using the notation of the previous step. In particular we let 
\[
((\psi_k)^1(x'), (\psi_k)''(x'))
\] 
be the first $\bar{n}+1$ components of the graph map of $\gammaup_k$ and $\varphi_k = E_k^{-\frac 12} (\psi_k)''(x')$. We consider the $\qhalf$ valued map $(g^+_k, g^-_k)$ defined by 
\[
g^\pm_k:= E_k^{-\frac 12} (f^\pm_k)''\, ,
\] 
with interface $(\gammado_k, \varphi_k)$. For each $k$ we let $\Phi_k$ be a diffeomorphism which maps $B_3$ onto itself and $\gammado_k\cap B_3$ onto $\gammado\cap B_3$. 
Again this is done in such a way that $\|\Phi_k - \Phi\|_{C^1} \to 0$, where $\Phi$ is the identity map. Furthermore, since $\norm{\varphi_k \circ \Phi^{-1}_k }_{C^1(B_3)} \to 0$, we can choose $\varkappa_k \in C^1(B_3)$ with $\varkappa_k = \varphi_k \circ \Phi^{-1}_k $ on $\gammado$ and $\norm{\varkappa_k}_{W^{1,2}(B_3)} \to 0 $. Now define the $\qhalf$ valued maps
\[ \hat{g}^\pm_k(x):= \sum_i \a{ (g^\pm_k)_i\circ \Phi_k^{-1}(x) - \varkappa_k(x) }. \]
As in the previous step we can find a subsequence (not relabeled) and a $\qhalf$ valued map $(g^+, g^-)$ with interface $(\gammado, 0)$ such that $\norm{\G(\hat{g}^\pm_k, g^\pm)}_{L^2(B^\pm_3)} \to 0$.
We next claim that
\begin{itemize}
\item[(A)] The convergence of $\hat g_k^\pm$ to $g^\pm$ is strong in $W^{1,2} (B_{5/2})$, namely 
\[
\lim_{k\to \infty} (\D (\hat g_k^+, B_{5/2}^+) + \D (\hat g_k^-, B_{5/2}^-)) = \D (g^+, B_{5/2}^+) + \D (g^-, B_{5/2}^-)\, .
\]
\item[(B)] $g^\pm$ is a $\qhalf$-minimizer. 
\end{itemize}
Assuming that (A) and (B) are proved, from  Theorem \ref{thm:collasso} we would then infer the existence of a classical harmonic function $\hat h$ which vanishes identically
on $\{x_m =0\}$ and such that $g^+ = Q \a{h}$ and $g^- = (Q-1) \a{h}$. Setting $h_k := E_k^{\sfrac{1}{2}} \hat h$ and
$\kappa_k (x) := (h_k (x), \Psi_k (x, h_k (x)))$ we would then conclude that
\begin{align*}
&\int_{B_{k,5/2 }^+}\G(f_k^+,Q\a{\kappa_k})^2+\int_{B_{k,5/2}^+}\left(|Df_k^+|-\sqrt Q|D\kappa_k|\right)^2 = \; o (E_k)\, ,\\
&\int_{B_{k,5/2 }^-}\G(f_k^-,(Q-1)\a{\kappa_k})^2+\int_{B_{k,5/2}^-}\left(|Df_k^-|-\sqrt {(Q-1)}|D\kappa_k|\right)^2\\
& = \; o (E_k)\,, \\
&\int_{B_{k,5/2}^\pm} \abs{D(\etaa\circ f_k^\pm) - D\kappa_k}^2 =\; o (E_k)\,.
\end{align*}
But these estimates are incompatible with (iv) above.  Hence,  at least one between (A) and (B) needs to fail. As in the previous section we will use this to contradict the minimality of \(T_k\). Note that in both cases there exists a  $\qhalf$ valued function  $(\bar g^+, \bar g^-)$ with  interface \((\gammado, 0)\), \(\gammado=\{x_m=0\}\), and a positive constant \(c_3>0\), such that 
\begin{equation}\label{e:lossliminf}
\D(\bar g^+,B^+_s)+\D(\bar g^-,B^-_s)\le \liminf_{k\to \infty} \D(\hat g_k^+,B^+_s)+\D(\hat g_k^-,B^-_s)-2c_3
\end{equation}
for all \(s\in (5/2,3)\). Indeed this is  true with $(\bar g^+, \bar g^-)=(g^+,  g^-)$ if (A) fails, while if (B) fails we choose $(\bar g^+, \bar g^-)$ to be a $\qhalf$-minimizer with boundary data \(g^\pm\) on \(\partial B_{5/2}\) extended to be equal to \(g^\pm\) on  \(B_{3}\setminus B_{5/2}\). We can now argue exactly as in the previous step to find a radius \(r\in (5/2,3)\) and functions  \(\hat h_k^\pm\) such that
\[
\bM (\langle T_k - (\bG_{f_k^+}+ \bG_{f_k^-}), |\bp|,r \rangle) \leq C E_k^{1-2\beta}
\]
and, arguing as we have done for \eqref{e:one_c2_lost},
\begin{align}
&\liminf_{k\to \infty} \D(h^+,B_{k,r}^+)+ \D(h^-,B_{k,r}^-)\le \D(\bar g^+,B^+_r)+\D(\bar g^-,B^-_r)+c_3
\\
&\le \liminf_{k\to \infty} \D(g^+,B_{k,r}^+)+ \D(g^-,B_{k,r}^-)-c_3.\label{e:ddai}
\end{align}
As in the previous section we consider $v_k^\pm (x) := E_k^{\sfrac{1}{2}} h_k^\pm (x)$ and 
\[
w_k^\pm (x) := \sum_i \a{\left(v_k^\pm (x), \Psi_k (x, v_k^\pm (x))\right)}
\] 
and observe that
$w_k^\pm|_{\partial B_r} = f_k^\pm$. We then construct the same competitor currents to test the minimality of $T_k$. First we consider a current \(S_k\) supported in \(\Sigma_k\) such that 
\[
\partial S_k = \langle T_k - (\bG_{f_k^+}+ \bG_{f_k^-}), |\bp|, r\rangle \mbox{ and }
\bM (S_k) \leq C (E_k^{1-2\beta})^{\frac{m}{m-1}} = o (E_k)\, ,
\]
Then we define, as before,  $Z_k := \bG_{w_k^+} \res \bC_r + \bG_{w_k^-} \res \bC_r + S_k$, for which we can verify that
\begin{equation}\label{e:contra100}
\bM (Z_k) \geq \bM (T_k \res \bC_r)\, .
\end{equation}
By  the result of the previous section,  we know that 
\begin{equation}\label{e:contra101}
2 \be_{T_k}(B_r)= \D (f_k^+, B_{k,r}^+) + \D (f_k^-, B_{k,r}^-) +  O(\eta_k E_k)\, .
\end{equation}
Observe that now we can choose \(\eta_k\to 0\) as \(k\to \infty\). On the other hand, using the bound on $\bM (S_k)$ and 
Taylor expansion we infer 
\begin{equation}\label{e:contra102}
2 \be_{Z_k}(B_r) =  \D (w_k^+, B_{k,r}^+) + \D (w_k^-, B_{k,r}^-)+  o (E_k)\, .
\end{equation}
Arguing as in the previous section (see \eqref{e:bus}) and relying on \eqref{e:ddai} we also have 
\begin{align}
 &\D (w_k^+, B_{k,r}^+) + \D (w_k^-, B_{k,r}^-)\nonumber\\ 
\leq & \D (f_k^+, B_{k,r}^+) + \D (f_k^-, B_{k,r}^-) - c_3 E_k + o (E_k)\, . \label{e:contra103}
\end{align}
Clearly \eqref{e:contra100}, \eqref{e:contra101}, \eqref{e:contra102} and \eqref{e:contra103} are in contradiction for $k$ large enough, which completes the proof.

\chapter{Decay of the excess and uniqueness of tangent cones}\label{chap:decay}

In this chapter we prove the decay of the excess at totally collapsed points for area minimizing currents. As a consequence
we will conclude that the tangent cone at each such point is in fact unique.

\begin{definition}\label{def_spher} Let $T$ be an integral current of dimension $m$ in $\R^{m+n}$. We define the \emph{excess}\index{Excess} $\bE (T, \bB_r (p), \pi)$\index{aale\bE (T, \bB_r (p), \pi)@$\bE (T, \bB_r (p), \pi)$} of $T$ in the ball $\bB_r (p)$ with respect to the (oriented) plane $\pi$ as
\begin{equation}\label{e:spher_excess}
\bE (T, \bB_r (p), \pi) := \frac{1}{2\omega_m r^m} \int_{\bB_r (p)} |\vec{T} (x) - \vec{\pi}|^2 \, d\|T\| (x)\, .
\end{equation}
If $T$ is area minimizing in a Riemannian manifold $\Sigma\subset \R^{m+n}$, we then define the \emph{spherical excess}\index{Spherical excess}\index{aale\bE (T, \bB_r (p))@$\bE (T, \bB_r (p))$} of $T$ at any ball $\bB_r (p)$ centered at some point $p\in \supp (T)\subset \Sigma$ as
\begin{equation}\label{e:spher_excess_2}
\bE (T, \bB_r (p)) := \min \{\bE (T, \bB_r (p), \pi): \pi\subset T_p \Sigma\}\, .
\end{equation}
We underline that $\pi$ is constrained to be a subset of $T_p \Sigma$, so probably a more appropriate, yet cumbersome, notation would be $\bE^\Sigma (T, \bB_r (p))$. 
Moreover we let $\bh (T, \bB_r (p))$\index{aalh\bh (T, \bB_r (p))@$\bh (T, \bB_r (p))$} be the minimum of $\bh (T, \bB_r (p), \pi)$ while $\pi\subset T_p \Sigma$ runs among those planes which optimize the right hand side of \eqref{e:spher_excess_2}. 
\end{definition}

Before stating the main theorem of this chapter we need to introduce a modified excess function for boundary points, where we constrain the ``minimal'' reference planes to contain $T_{p}\gammaup$.

\begin{definition}\label{def:exc_flat}
Let $T$, $\Sigma$ and $\gammaup$ be as in Assumption \ref{ass:main} and assume that $p\in\gammaup$. We define the \emph{modified excess}\index{Modified excess for boundary points}\index{aale\bE^\flat(T,\bB_r(p))@$\bE^\flat(T,\bB_r(p))$} in $\bB_r(p)$ as
\begin{equation}\label{e:exc_flat}
\bE^\flat(T,\bB_r(p)):=\min\left\{\bE(T,\bB_r(p),\pi):\,T_{p}\gammaup\subset\pi\subset T_{p}\Sigma\right\}\,.
\end{equation}
\end{definition}

With this notation, the main result of this chapter is the following

\begin{theorem}\label{thm:decay_and_uniq}\label{THM:DECAY_AND_UNIQ}
Let $\gammaup$ be a $C^2$ $(m-1)$-dimensional submanifold of a  $C^2$ $(m+\bar n)$-dimensional submanifold $\Sigma \subset \R^{m+n}$ and consider an area minimizing current $T$ in $\Sigma$ with the property that 
$\partial T \res U = \a{\gammaup}$ for some open set $U$. If $p\in \gammaup \cap U$ is a collapsed point with density $\Theta (T, p) = Q - \frac{1}{2}$, then there exists  \(r>0\) such that:
\begin{itemize}
\item[(a)] Each $q\in \gammaup \cap \bB_r (p)$ is a collapsed point for $T$ with density $Q-\frac{1}{2}$;
\item[(b)] At each $q\in \gammaup \cap \bB_r (p)$ there is a unique flat tangent cone $Q \a{\pi (q)^+} + (Q-1) \a{\pi (q)^-}$, where 
$\pi (q)\subset T_q \Sigma$ is an oriented $m$-dimensional plane containing $T_q \gammaup$;
\item[(c)] For each $\varepsilon >0$ there is a constant $C = C (\varepsilon)$ with the property that
\begin{align}
\bE^\flat (T, \bB_\rho (q)) & \leq \bE (T, \bB_\rho (q), \pi (q))\nonumber\\
& \leq C \left(\frac{\rho}{r}\right)^{2-2\varepsilon}  \bE^\flat (T, \bB_{2r} (p)) + C \rho^{2-2\varepsilon}r^{2\varepsilon} \bA^2 \label{e:decay-1}
\end{align}
for all $q\in \gammaup \cap \bB_r (p)$ and for all $\rho\in ]0, r[$;
\item[(d)] For each $\varepsilon > 0$ there is a constant $C = C (\varepsilon)$ such that
\begin{equation}\label{e:hoelder-1}
|\pi (q) - \pi (q')| \leq C (r^{\eps-1} \bE^\flat (T, \bB_{2r} (p))^{\sfrac{1}{2}} + \bA r^\varepsilon) |q'-q|^{1-\varepsilon}
\end{equation}
$\forall q, q'\in \gammaup \cap \bB_r (p)$;
\item[(e)] There is a constant $C $ such that
\begin{equation}\label{e:decay-height}
\bh (T, \bB_\rho (q), \pi (q)) \leq C (r^{-1} \bE^\flat (T, \bB_{2r} (p)) + \bA)^{\sfrac{1}{2}} \rho^{\sfrac{3}{2}} 
\end{equation}
for all $q\in \gammaup \cap \bB_r (p)$ and  for all $\rho\in ]0, {\textstyle{\frac{r}{2}}}[$.
\end{itemize}
\end{theorem}

Before coming to the proof we state an important corollary of the theorem which will be used often in the remaining chapters (for a geometric illustration of the conclusions we refer to Figure \ref{fig:cones}). 

\begin{corollary}\label{c:cone_cut}\label{C:CONE_CUT}
Let $\gammaup, \Sigma, T$ and $p$ be as in Theorem \ref{thm:decay_and_uniq},  assume $r=2\sigma$ is a radius for which all the conclusions of Theorem \ref{thm:decay_and_uniq} hold, set
$E= \bE^\flat (T, \bB_{r} (p))$. Furthermore let $\pi$ be an optimal plane for the right hand side of \eqref{e:exc_flat} and $\pi(q)$ be the tangent plane to \(T\) in \(q\) as in conclusion (b) of  Theorem \ref{thm:decay_and_uniq}. If we denote by $\bp, \bp^\perp, \bp_q$ and $\bp_q^\perp$ respectively the
orthogonal projections onto $\pi, \pi^\perp, \pi (q)$ and $\pi (q)^\perp$, then
\begin{equation}\label{e:tilt_pi(q)}
|\pi (q) - \pi| \leq C (E + \bA r)\, ,
\end{equation}
\begin{align}
	\supp(T) \cap \bB_\sigma (q) & \subset  \{x  \colon \abs{\bp^\perp(x-q)} \le C (E + \bA r)^{\sfrac{1}{2}} \abs{x-q} \}\label{eq:height-envelope2}\\
\end{align}
for all $q \in \gammaup\cap \bB_\sigma (p)$ and
\begin{align}
	\supp(T) \cap \bB_\sigma (q) & \subset \{ x : \abs{\bp^\perp_q(x-q)} \le C (r^{-1} E + \bA)^{\sfrac{1}{2}} \abs{x-q}^{\frac32} \} \label{eq:height-envelope1} 
\end{align}
for all $q\in \gammaup\cap \bB_\sigma (p)$.
\end{corollary}

\begin{figure}[htbp]
\begin{center}\label{fig:cones}
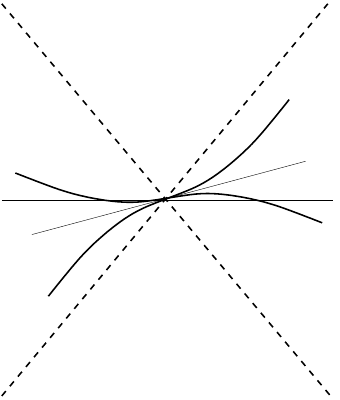
\end{center}
\caption{The region delimited by the thick curved lines is the right hand side of \eqref{eq:height-envelope1}, whereas the cone delimited by the thick dashed straight lines is the right hand side of \eqref{eq:height-envelope2}}. 
\end{figure}

\section{Hardt--Simon height bound}

In this section we show the validity, at the boundary, of the classical interior height bound, under Assumption \ref{Ass:app}.
The argument follows an important idea of Hardt and Simon in \cite{HS} and takes advantage of an appropriate 
variant of Moser's iteration on varifolds, due to Allard, combined with a crucial use of the remainder in the monotonicity formula. 

\begin{theorem}\label{t:height_bound}
There are positive constants $\eps=\eps (Q, m, \bar n, n)$ and $C_0=C_0 (Q, m, \bar n, n)$ with the following property.
Let $T$, $\bC_{4r} (x)$, $\Sigma$, $\gammaup$ and $\pi_0:= \R^m \times \{0\}$ be as in Assumption \ref{Ass:app} and set
\[
E := \bE (T, \bC_{4r} (x))\, , \quad \ba:= \|A_\gammaup\|_0\quad \mbox{and} \quad \bar\ba:= \|A_\Sigma\|_0\, .
\]
If $E + \ba + \bar \ba \leq \varepsilon$, then 
\[
\bh (T, \bC_{2r} (x), \pi_0) \leq C_0 (E^{\sfrac{1}{2}} + \ba^{\sfrac{1}{2}} r^{\sfrac{1}{2}} + \bar\ba r) r\, .
\]
\end{theorem}

We will split the proof of the theorem in the following two lemmas, where again the corresponding geometric constants $C_0$ depend only upon $m,\bar n, n$ and $Q$. 

\begin{lemma}\label{l:Moser}
Under the assumptions of Theorem \ref{t:height_bound} there is a constant $C_0$ such that 
\begin{align}
\sup_{z\in \supp (T) \cap \bC_{2r} (x)} |\bp_{\pi_0}^\perp (z-x)|^2 & \leq C_0 r^{-m} \int_{\bC_{3r} (x)}  |\bp_{\pi_0}^\perp (z-x)|^2\, d\|T\| (z)\nonumber\\
&\qquad + C_0 (\ba^2 + \bar \ba^2) r^4\, .\label{e:Moser}
\end{align}
\end{lemma}

\begin{lemma}\label{l:L^2_est}
Under the assumptions of Theorem \ref{t:height_bound} there is a constant $C_0$ such that
\begin{equation}\label{e:Hardt-Simon}
r^{-m} \int_{\bC_{3r} (x)}  |\bp_{\pi_0}^\perp (z-x)|^2\, d\|T\| (z) \leq C_0 E r^2 +  C_0 \bar\ba^2 r^4 + C_0 \ba r^3\, .
\end{equation}
\end{lemma}

After rescaling and translating we can assume in all our statements that $r=1$ and $x=0$ . Moreover, we use $\bp$ and $\bp^\perp$ in place
of $\bp_{\pi_0}$ and $\bp_{\pi_0}^\perp$. 

\subsection{Proof of Lemma \ref{l:Moser}}
The estimate is a classical one in Allard's interior regularity theory. The proof in our setting follows from a minor modification of the arguments, which we however report for the reader's convenience. 

We fix a system of coordinates so that $\pi_0 = \{y: y_{m+1} = \ldots = y_{m+n}=0\}$ and fix $i\in \{m+1, \ldots, m+n\}$. We fix a constant $C_0$, to be chosen in a moment, and consider the function
\[
f (x) :=  \max \{x_i - C_0 \ba+ C_0 \bar\ba |x|^2, 0 \}\, .
\]
We wish to show the estimate
\begin{equation}\label{e:Moser_2}
\sup_{z\in \supp (T)\cap \bC_2} f^2 (z) \leq C_1 \int_{\bC_3} f^2 (z)\, d\|T\| (z)\, ,
\end{equation}
from which we will get \eqref{e:Moser} simply summing up all the corresponding inequalities when taking $i\in \{m+1, \ldots, m+n\}$ and $-y_i$ in place of $y_i$. 

In fact we let $r_{+,\delta}$ be a suitable convex smoothing of the function $\mathbb R\ni t\mapsto r_+(t) := \max \{t, 0\}$, with the additional properties that $r_{+,\delta}$ vanishes on the negative half line and equals the identity for $t> \delta$: then we will show the inequality \eqref{e:Moser_2} for the function $f (x) := r_{+,\delta} (x_i - C_0 \ba + C_0 \bar \ba |x|^2)$. Since the constant $C_1$ will not depend on $\delta$, we will achieve the correct inequality by simply letting $\delta \downarrow 0$.  For the rest of this proof $f$ denotes such a fixed smoothing of  \(\max \{x_i - C_0 \ba+ C_0 \bar\ba |x|^2, 0 \}.\) 

Observe that, by choosing $C_0$ sufficiently large we achieve that $f$ vanishes on $\gammaup$ and, according to \cite[Section 7.5]{All}, that $f$ is subharmonic\footnote{We recall that a function \(h\) is said to be subharmonic on the varifold induced by \(T\) if 
\[
\int \nabla _T h\ \cdot \nabla_T \varphi \,d\|T\| \le 0 \qquad \mbox{$\forall \varphi \in C_c^1$ with $\varphi \ge 0$,} 
\]
where \(\nabla_T  h$ is the orthogonal projection of $\nabla h$ on the tangent space to $T$ (i.e., if $v_1, \ldots, v_m$ is an orthonormal frame such that $\vec{T} (x) = v_1 \wedge \ldots \wedge v_m$, then $\nabla_T h =\sum_i \frac{\partial h}{\partial v_i} v_i$).} on the varifold  induced by $T$. 

 We next show that \eqref{e:Moser_2} holds under these two assumptions. Note that Allard in \cite[Section 7.5]{All} proves precisely this statement, but we cannot use \cite[Theorem 7.5(6)]{All} directly because the constant in the inequality depends upon the distance of the support of  $f$ and the boundary $\gammaup$: the purpose of the following argument is to show that in fact such dependence is absent in our case.

We denote by $\bC^k$ the decreasing sequence of cylinders $\bC_{2+ 2^{-k}}$. We then observe that the (short) paragraph proving \cite[Lemma 7.5(5)]{All} applies to our situation and implies the inequality 
\begin{equation}\label{e:Caccioppoli}
 \int_{\bC^{k+1}} |\nabla_T h|^2 d\|T\| \leq 2^{2k+2} \int_{\bC^k} h^2 d\|T\|
\end{equation}
for any subharmonic function $h$ which vanishes on a neighborhood of $\gammaup$.  
We next use the Sobolev inequality on stationary varifolds, namely from \cite[Theorem 7.3]{All} we know that, for $\bar\ba$ smaller than a positive geometric constant,
\begin{equation}\label{e:Sobolev}
\left(\int_{\bC^k} (h\varphi)^{\frac{m}{m-1}} d\|T\|\right)^{\frac{m-1}{m}} \leq C_0 \int_{\bC^k} |\nabla_T (h\varphi)|\, 
\end{equation}
whenever $\varphi$ is a smooth function compactly supported in $\bC^k$ (remember that \(h\) vanishes in a  neighborhood of $\gammaup$). 

Following the classical scheme of Moser's iteration, cf.  \cite[Theorem 7.5(6)]{All}, we introduce $\beta:= \frac{m}{m-1}$ and 
\[
I (k) := \left(\int_{\bC^{2k}} f^{2\beta^k}\right)^{\sfrac{1}{\beta^k}}\, .
\]
Next we fix a cutoff $\varphi_k$ identically equal to $1$ on $\bC^{2k+2}$, compactly supported in $\bC^{2k+1}$ and with $|\nabla \varphi_k|\leq C_0 2^{2k}$. Substituting $h = f^{2\beta^k}$ and $\varphi = \varphi_k$ inside \eqref{e:Sobolev} we then conclude
\begin{equation}\label{e:Moser_piece1}
I (k+1)^{\beta^k} \leq C_0 \int_{\bC^{2k+1}} |\nabla_T (f^{2\beta^k})| d\|T\| + C_0 2^{2k} \int_{\bC^{2k+1}} f^{2\beta^k} d\|T\|\, .
\end{equation}
Next we compute
\begin{align*}
& \int_{\bC^{2k+1}} |\nabla_T (f^{2\beta^k})| d\|T\| \leq 2 \int_{\bC^{2k+1}} |\nabla_T (f^{\beta^k})| f^{\beta^k}|d\|T\|\\
&\leq 2 \left(\int_{\bC^{2k+1}} |\nabla_T (f^{\beta^k})|^2 d\|T\|\right)^{\sfrac{1}{2}} \left(\int_{\bC^{2k+1}}f^{2\beta^k}d\|T\|\right)^{\sfrac{1}{2}}\, .
\end{align*}
Now, since $\mathbb R_{+}\ni t \mapsto t^{\beta^k}$ is $C^2$, convex, and increasing, the function $h := f^{\beta^k}$ is subharmonic (cf. \cite[Lemma 7.5(4)]{All}). Moreover it vanishes in a neighborhood of \(\gammaup\).  From \eqref{e:Caccioppoli}, we then conclude
\begin{equation}\label{e:Moser_piece2}
 \int_{\bC^{2k+1}} |\nabla_T (f^{2\beta^k})| d\|T\| \leq 2^{2k+2} \int_{\bC^{2k}}f^{2\beta^k}d\|T\|\, .
\end{equation}
Putting together \eqref{e:Moser_piece1} and \eqref{e:Moser_piece2}, we then easily conclude
\[
I (k+1) \leq C^{k/ \beta^k} I (k)\, .
\]
The estimate \eqref{e:Moser_2} follows from
\[
\sup_{z\in \supp (T)\cap \bC_2} f^2 (z) \leq \limsup_{k\to\infty} I (k) \leq C I (0)\, .
\]

\subsection{Proof of Lemma \ref{l:L^2_est}}
We follow here the proof of \cite[Lemma 1.8]{Spolaor} (note that essentially the same idea was used in \cite{HS}). First of all,
we let $r=4$ and $s$ go to $0$ in \eqref{e:monot_identity} to achieve
\begin{align}\label{e:HS1}
\int_{\bB_4} \frac{|x^\perp|^2}{|x|^{m+2}} d\|T\| (x) \leq 4^{-m} \|T\| (\bB_4) - \omega_m \Theta (T, 0) + {\rm Err_1} + {\rm Err_2}\, ,
\end{align}
where
\begin{align*}
{\rm Err_1} := & \int_0^4 \rho^{-m-1} \int_{\bB_\rho} |x^\perp \cdot \vec{H}_T (x)| d\|T\| (x)\, d\rho\\
{\rm Err_2} := & \int_0^4 \rho^{-m-1} \int_{\bB_\rho\cap \gammaup} |x\cdot \vec{n} (x)|\, d\cH^{m-1} (x)\, d\rho\, .
\end{align*}
Straightforward computations\footnote{Observe that $|x^\perp\cdot \vec{H}_T (x)|\leq \frac{1}{8\rho } |x^\perp|^2 + 2\rho |\vec{H}_T (x)|^2$, while $|\vec{H}_T (x)|\leq m \|A_\Sigma\|_0 \leq m \ba$ by \eqref{e:mean_curv}.} show that $|x\cdot \vec{n} (x)| \leq C_0 \ba |x|^2$ for \(x\in \gammaup\) and  $|x^\perp\cdot \vec{H}_T (x)|\leq \frac{1}{8\rho} |x^\perp|^2 + 2m^2 \rho\bar \ba^2$. Thus we can bound
\begin{align*}
{\rm Err_2} \leq & C_0 \ba \int_0^4 \rho^{1-m} \cH^{m-1} (\bB_\rho \cap \gammaup)\, d\rho \leq C_0 \ba\, 
\end{align*}
and
\begin{align*} 
{\rm Err_1} \leq & \frac{1}{8} \int_0^4 \frac{1}{\rho^{m+2}} \int_{\bB_\rho} |x^\perp|^2 \, d\|T\| (x)\, d\rho + 2  m^2 \bar \ba^2 \int_0^4 \frac{\|T\| (\bB_\rho)}{\rho^m} \, d\rho\\
\leq & \frac{1}{2} \int_{\bB_4} \frac{|x^\perp|^2}{|x|^{m+2}} d\|T\| (x) + 2 C_0 \bar \ba^2 \|T\| (\bB_4)
\end{align*}
where in the last inequality we have used the monotonicity of \(\rho\mapsto e^{C\rho}\rho^{-m}\|T\|(B_\rho)\). Plugging these two estimates in \eqref{e:HS1} and recalling that $\Theta (T, 0) \geq Q-\frac{1}{2}$ we then conclude
\begin{equation}\label{e:HS2}
\int_{\bB_4} \frac{|x^\perp|^2}{|x|^{m+2}} d\|T\| (x) \leq 4^{-m} \|T\| (\bB_4) - (Q-\textstyle{\frac{1}{2}}) \omega_m + C_0 \ba + C_0 \bar \ba^2 \|T\| (\bB_4)\, .
\end{equation}
Next,  by \eqref{e:excess_2} and computations as in \eqref{e:above_2}, we infer
\begin{align}
& 4^{-m} \|T\| (\bB_4) - (Q-\textstyle{\frac{1}{2}}) \omega_m \nonumber\\
= & \omega_m \left(\frac{\|T\| (\bC_4)}{\omega_m 4^m}  - (Q-\textstyle{\frac{1}{2}})\right)\leq  \omega_m \bE (T, \bC_4) +  C_0 \ba\, .
\end{align}
Hence we easily conclude from \eqref{e:HS2} that
\begin{equation}\label{e:HS3}
\int_{\bB_4} |x^\perp|^2 d\|T\| (x) \leq C_0 (E + \ba + \bar\ba^2)\, .
\end{equation}
Next, a straightforward computation gives
\[
|z^\perp|^2 \geq \frac{1}{2}|\bp^\perp (z)|^2 - |z|^2 |\vec{T} (z) - \pi_0|^2
\]
for every $z\in \supp (T)$. Integrating the latter inequality and inserting in \eqref{e:HS3} we then conclude
\begin{equation}\label{e:HS4}
\int_{\bB_4} |\bp^\perp (z)|^2 d\|T\|(z) \leq C_0 (E + \ba + \bar\ba^2)\, .
\end{equation}

In order to complete the proof we need to show that $\supp (T) \cap \bC_{3} \subset \bB_4$, if  the parameter $\varepsilon$ in Theorem \ref{t:height_bound} is chosen sufficiently small. Arguing by contradiction, if this were not the case there would be a sequence of currents $T_k$ in $\bC_4$ and submanifolds $\gammaup_k$, $\Sigma_k$ satisfying all the requirements of Assumption \ref{Ass:app} with $\bE (T_k, \bC_4) + \|A_{\gammaup_k}\|_0 + \|A_{\Sigma_k}\|_0 \to 0$ but with the additional property that there is a point $p_k\in \supp (T_k) \cap \bC_3$ with $|p_k |\geq 4$. Note however that, under these assumptions, the mass of $T_k$ in $\bC_4$ converges to $(Q-\frac{1}{2}) 4^m \omega^m$ and $T_k$ converges, up to subsequences, to a current $T_\infty$ of the form $Q \a{\bC_4 \cap \pi_0^+} + 
(Q-1) \a{\bC_4 \cap \pi_0^-}$. On the other hand this means that, for some geometric constant $r>0$, $\bB_r (p_k)$ has positive distance from the plane $\pi_0$ and is contained in $\bC_4$. Let $U$ be an open set which contains the closure of $\bC_4 \cap \pi_0$ and has empty intersection with $\bB_r (p_k)$. 
Then
\[
\bM (T_k) \geq \|T_k\| (U) + \|T_k\| (\bB_r (p_k))\, .
\]
Letting $k\to\infty$ and using the semicontinuity of the mass we conclude
\[
\Big(Q-\frac{1}{2}\Big) 4^m \omega^m \geq \|T_\infty\| (U) + \limsup_{k\to\infty} \|T_k\| (\bB_r (p_k))\, .
\] 
On the other hand $\|T_\infty\| (U) =  (Q-\frac{1}{2}) 4^m \omega^m$ and so 
\[
\lim_{k\to\infty}  \|T_k\| (\bB_r (p_k)) = 0\, .
\] 
Since $p_k \in \supp (T_k)$ and $\bB_r (p_k)\subset \bC_4 \setminus \gammaup$, for $k$ large enough we contradict the interior monotonicity formula. 

\section{Excess decay}\label{sec:exc_dec}

The core of Theorem \ref{thm:decay_and_uniq} is in fact the decay estimate \eqref{e:decay-1}, which we prove in this section for the modified excess function introduced in Definition \ref{def:exc_flat}, under a suitable smallness assumption.

\begin{theorem}\label{t:exc_dec}
For any $\varepsilon >0$ there is an $\eps_0=\varepsilon_0 (\varepsilon, Q, m,n)>0$ and a $M_0=M_0 (\varepsilon, Q, m,n)$ with the following property. Let $T$, $\Sigma$ and $\gammaup$ be as in Assumption \ref{ass:main} and assume that 
\begin{itemize}
\item[(i)] $\bA^2 \sigma^2 + E = (\|A_\Sigma\| + \|A_\gammaup\|)^2 \sigma^2 + \bE^\flat (T, \bB_{4 \sigma} (q)) < \varepsilon_0$;
\item[(ii)]  $\Theta (T,x) \geq Q-\frac{1}{2}$ for all $x\in \gammaup\cap \bB_{4\sigma} (q)$;
\item[(iii)] $q\in \gammaup$ and $\|T\| (\bB_{4\sigma} (q)) \leq (Q-\frac{1}{4}) \omega_m (4\sigma)^m$.
\end{itemize}
Then, if we set $e (t) := \max\{\bE^\flat (T, \bB_t (q)) ,M_0 \bA^2 t^2\}$ we have 
\begin{equation}\label{e:decay_alternativo}
e (\sigma) \leq \max \{ 2^{-4+4\varepsilon} e (4\sigma),
 2^{-2+2\varepsilon}  e (2\sigma)\}\, .
\end{equation}
\end{theorem}

The rest of this section is devoted to the proof of Theorem \ref{t:exc_dec}.

\subsection{Preliminary considerations}\label{subsec:exc_dec_prel}
Without loss of generality by scaling, translating and rotating, we can assume $\sigma=1$, $q=0$, $\bE^\flat(T,\bB_2)=\bE (T,\bB_2,\pi_0)$, where $\pi_0=\R^m\times\{0\}\subset T_0\Sigma=\R^m\times\R^{\overline n}\times\{0\}$, and $T_0 \gammaup = \R^{m-1}\times \{0\}$. We also recall that, if we do not specify the center of a ball or a cylinder, we implicitly assume that such center is the origin.  

We start by observing that, without loss of generality, we can assume
\begin{equation}\label{e:Apiccolo}
\bE^\flat(T,\bB_2)\ge 2^{-m} M_0 \bA^2,
\end{equation}
and 
\begin{equation}\label{e:2/4}
\bE^\flat(T,\bB_2)\ge 2^{-4-m}  \bE^\flat(T,\bB_4).
\end{equation}
Indeed, note that 
\[
e(1) = \max\{ M_0\bA^2, \bE^\flat (T, \bB_1)\} \leq \max \{M_0 \bA^2, 2^{m}\bE^\flat(T,\bB_2)\}\, .
\] 
So, if \eqref{e:Apiccolo} fails, then
\begin{align*}
e(1) &\leq M_0 \bA^2= 2^{-2} (2^2 M_0 \bA^2) \leq 2^{-2} e (2)\, ,
\end{align*}
whereas, if \eqref{e:2/4} fails, then
\begin{align*}
e(1)&\leq \max \{M_0 \bA^2, 2^{-4} \bE^\flat (T, \bB_4)\} = 2^{-4} e(4)\, .
\end{align*}
Hence in both cases the conclusion would hold trivially.

Summarizing, under  assumptions \eqref{e:Apiccolo} and \eqref{e:2/4}, we need to show the decay estimate:
\begin{equation}\label{e:gollonzo}
\bE^\flat(T,\bB_1)\le 2^{2\varepsilon-2}\bE^\flat(T,\bB_2)\,.
\end{equation}

Let us now fix a positive $\eta<1$, to be chosen sufficiently small later, and consider the cylinder $U:=B_{4-\eta} (0, \pi_0) + B^n_{\sqrt\eta} (0, \pi_0^\perp)$, which by abuse of notation we denote by $B_{4-\eta}\times B^n_{\sqrt{\eta}}$. If $\varepsilon_0$ is sufficiently small, we claim that 
\begin{align}
\supp (T)\cap\partial U & \subset \partial B_{4-\eta}\times B^n_{\sqrt\eta} \label{e:supp1}\\
\bB_{4-\eta}\cap\supp(T)& \subset  U\,.\label{e:supp2}
\end{align} 
Otherwise, arguing by contradiction, we would have a sequence of currents $T_k$ satisfying the assumptions of the theorem with $\varepsilon_0=\frac 1k$, but violating either \eqref{e:supp1} or \eqref{e:supp2}.  Then $T_k$ would converge, in the sense of currents, to 
\[
T_\infty:=Q' \a{B^+_4}+(Q'-1)\a{B^-_4}\, ,
\]
where $B_4^\pm=B_4 (0, \pi_0) \cap\{\pm x_m>0\}$ and $Q'$ is a positive integer. By the area-minimizing property, this implies that the supports of $T_k$ converge to either $\overline B_4$ (if $Q'>1$) or $\overline B_4^+$ (if $Q'=1$) in the Hausdorff sense in every compact subset of $\bB_4$. This would be a contradiction because both  $\overline{\bB_{4-\eta}\setminus U}$ and $\partial U\setminus(\partial B_{4-\eta}\times B^n_{\sqrt \eta})$ are compact subsets of $\bB_4$ with positive distance from $\overline B_{4}$. We have therefore proved \eqref{e:supp1} and \eqref{e:supp2}. 

\medskip

We remark further that we must necessarily have
$\|T_\infty\| (\bB_4) \leq (Q-\frac{1}{4}) \omega_m 4^m$ by assumption (iii). Hence, by the monotonicity formula $Q' - \frac{1}{2} = 
\Theta (T_\infty, 0) \leq Q-\frac{1}{4}$. On the other hand, by assumption (ii) and the upper semicontinuity of the density of area-minimizing currents under convergence of the latter, we must have $\Theta (T_\infty, 0) \geq Q-\frac{1}{2}$. Since $Q'$ is an integer we conclude $Q'=Q$.  Observe also that, by the area-minimizing property, $\|T_k\|(A)\to\|T_\infty\|(A)$ for every compact subset $A$ of $\bB_4$. Thus, for $\varepsilon_0$ is sufficiently small, we
have that:
\begin{itemize}
\item[(A)] the mass of $T$ in the ball $\bB_r$ is, for any radius $1\le r\le 4-\frac\eta 2$ and up to a small error, $\qhalf \omega_mr^m$ . 
\end{itemize}

\medskip

Next, let us define $T_0:=T\res U$. Observe that \eqref{e:supp1} and \eqref{e:supp2} imply:
\begin{itemize}
\item[(B)] $\partial T_0\res\bC_{4-\eta}=\a{\gammaup\cap\bC_{4-\eta}}$;
\item[(C)] $T\res\bB_{4-\eta}=T_0\res\bB_{4-\eta}$.
\end{itemize}
Choose a plane $\overline\pi\subset T_0\Sigma$ which contains $T_0\gammaup$ and such that 
\[
\bE(T,\bB_4,\overline\pi)=\bE^\flat(T,\bB_4)\, .
\] 
Let us observe that (since \(\pi_0\) is the optimal plane for \(\bE^\flat(T,\bB_2)\)):
\begin{align*}
|\overline\pi-\pi_0|^2\|T\|(\bB_2) & =\int_{\bB_2}|\overline\pi-\pi_0|^2\,d\|T\|\\
& \le 2\int_{\bB_2}|\vec{T}-\pi_0|^2\,d\|T\|+2\int_{\bB_2}|\vec T-\overline\pi|^2\,d\|T\|\\
& \le 2\cdot 2^m \omega_m\bE^\flat(T,\bB_2)+2\cdot 4^m\omega_m\bE^\flat(T,\bB_4)\\
&\le C\bE^\flat(T,\bB_4)\,.
\end{align*}
Moreover
\begin{align}
\bE (T_0,\bC_{4-\eta}) & \le\bE(T,\bB_{4-\frac\eta 2},\pi_0)\nonumber\\
&\le 2\bE^\flat(T,\bB_{4-\frac{\eta}{2}})+\textstyle\frac{2}{\omega_m 4^m}|\overline\pi-\pi_0|^2\|T\|(\bB_{4-\frac\eta 2})\nonumber\\
& \le 2\bE^\flat(T,\bB_{4-\frac\eta 2})+C|\overline\pi-\pi_0|^2\|T\|(\bB_2)\le C\bE^\flat(T,\bB_4)\,,\label{e:nonsisa}
\end{align}
where in the third inequality we have used (A), namely that the mass of $T$ in a ball of radius $r\le 4-\frac\eta 2$ is comparable to $\qhalf\omega_m r^m$. 
Thus 
\begin{itemize}
\item[(D)] $\bE (T_0,\bC_{4-\eta})\le C\bE^\flat(T,\bB_4)$.
\end{itemize}
Moreover, recalling that ${\bf p}:\R^{m+n}\to\pi_0$ is the orthogonal projection, by the Constancy Theorem
\begin{itemize}
\item[(E)] ${\bf p}_\sharp T_0=Q^*\a{\Omega^+}+(Q^*-1)\a{\Omega^-}$, where $Q^*$ is a suitable positive natural number
and $\Omega^\pm$ are the regions in which $B_4$ is divided by ${\bf p} (\gammaup)$; in particular  
\[
\partial\a{\Omega^+}\res\bC_{4-\eta}=-\partial\a{\Omega^-}\res\bC_{4-\eta}={\bf p}_\sharp\a{\gammaup}\res\bC_{4-\eta}\, .
\]
\end{itemize}
Since $T_0 = T \res U$ and $U\subset \bB_{4-\eta/2}$, clearly $\|T_0\| (\bC_{4-\eta}) \leq \|T\| (\bB_{4-\eta/2})$. On the other hand, by (D) and (E), 
\[
\|T_0\| (\bC_{4-\eta}) \geq Q^* |\Omega^+| + (Q^*-1) |\Omega^-|\, .
\]
Assuming that the constant $\varepsilon_0$ in the assumption (i) of the theorem is sufficiently small, we conclude that ${\bf p}_\sharp\a{\gammaup}\res\bC_{4-\eta}$ is close to an $m-1$-dimensional plane passing through the origin. In particular 
$Q^* |\Omega^+| + (Q^*-1) |\Omega^-|$ is close to $(Q^*-\frac{1}{2}) \omega_m (4-\eta)^m$. Thus, if $\varepsilon_0$ is smaller than a geometric constant, we infer
\[
\|T_0\|  (\bC_{4-\eta}) \geq (Q^* - \frac{3}{4}) \omega_m (4-\eta)^m\, .
\]
However, by (A), a sufficiently small $\varepsilon_0$ would imply $\|T\| (\bB_{4-\eta/2}) \leq (Q-\frac{1}{4}) \omega_m (4-\frac{\eta}{2})^m$ and hence we achieve
$Q^* \leq Q$ provided $\eta$ is chosen smaller than a geometric constant.

On the other hand, 
\[
\|T_0\| (\bC_{4-\eta}) \leq Q^* |\Omega^+| + (Q^*-1) |\Omega^-| + \bE (T_0,\bC_{4-\eta})\, .
\] 
Using (D) and the argument above, if $\varepsilon_0$ is sufficiently small we get $\|T_0\| (\bC_{4-\eta}) \leq (Q^*-\frac{1}{4}) \omega_m (4-\eta)^m$. 
Recall that we have shown that $T\res \bB_{4-\eta} = T_0 \res \bB_{4-\eta}$. Thus $\|T\|(\bB_{4-\eta})\le\|T_0\|(\bC_{4-\eta})$ and, using (A), we also have $\|T\| (\bB_{4-\eta}) \geq (Q-\frac{3}{4}) (4-\eta)^m$. Thus necessarily $Q^*\geq Q$.

Next, since $T\res\bB_2=T_0\res\bB_2$, then
\[
\begin{split}
\bA^2&\stackrel{\eqref{e:Apiccolo}}{\le}  2^{m+2}M_0^{-1}\bE^\flat(T,\bB_2)\le 2^{m+2}\left(\frac{2}{4-\eta}\right)^mM_0^{-1}\bE(T_0,\bC_{4-\eta})
\\
&\stackrel{\eqref{e:nonsisa}}{\le} C M_0^{-1}\bE^\flat(T,\bB_4)\,.
\end{split}
\]
Thus we can apply Theorem \ref{t:harm_1} with $\beta=\frac{1}{5m}$ and a sufficiently small parameter $\eta_*$ to be chosen later, provided $\varepsilon_0$ is sufficiently small and $M_0$ is sufficiently large. 

\subsection{Reduction to excess decay for graphs}

From now on we let $(u^+,u^-), h$ and $\kappa$ be as in Theorem \ref{t:harm_1}. In particular, recall that $(u^+,u^-)$ is the $E^\beta$-approximation of Theorem \ref{t:Lipschitz_1} (and therefore it satisfies the estimate \eqref{prop:Lipschitz_11}-\eqref{prop:Lipschitz_16}) and $h$ is the single harmonic function which ``supports'' the collapsed $\qhalf$ $\D$-minimizer $\kappa$. Moreover, denote by $E$ the excess $\bE(T_0,\bC_{4-\eta})$ and record the estimates:
\begin{align}
\bA^2 & \le C_0M_0^{-1}E\label{e:bibi}\\
E & \le C_0\bE^\flat(T,\bB_2)\,,\label{e:bibo}
\end{align}
where $C_0$ is a geometric constant and the second inequality follows by combining \eqref{e:nonsisa} and \eqref{e:2/4}. Next, define $\pi$ to be the plane given by the graph of the linear function $x\mapsto (Dh(0)x,0)$. Since, by Remark \ref{r:odd_harmonic}, \(h(x',0)=0\) we have  that 
\[
\pi\supset T_{0}\gammaup=\R^{m-1}\times\{0\}.
\]
Moreover, by elliptic estimates, 
\begin{equation}\label{e:grst}
 |\pi|\le |Dh(0)|\le ( C\D(h,B_{\frac 52(4-\eta)}))^\frac 12\le CE^\frac 12.
\end{equation}

Fix $\overline\eta$ to be chosen later; in the next steps we show that
\begin{equation}\label{e:gollonzo2}
\bE(\bG_{u^+}+\bG_{u^-},\bC_1,\pi)\le(2-\overline\eta)^{ - (2-\varepsilon)}\bE(\bG_{u^+}+\bG_{u^-},\bC_{2-\overline\eta})+\overline\eta E\,.
\end{equation}
From this we easily conclude \eqref{e:gollonzo} as follows. First of all, by the Taylor expansion of the mass of a Lipschitz graph and the Lipschitz bounds on $u^\pm$, we conclude
\begin{align*}
&\bE(\bG_{u^+}+\bG_{u^-},\bC_{2-\overline\eta})\\
\le & \bE(T_0,\bC_{2-\overline\eta})+ C\int_{\Omega^+\setminus K}|Du^+|^2+C\int_{\Omega^-\setminus K}|Du^-|^2\,.
\end{align*}
Secondly, 
\begin{align*}
\bE(T,\bB_1,\pi)&\le\bE(T_0,\bC_1,\pi)\\
& \le\bE(\bG_{u^+}+\bG_{u^-},\bC_{1},\pi)+2\be_T(B_1\setminus K)+2|\pi|^2|B_1\setminus K|\,.
\end{align*}
From \eqref{e:eta-star-condition}, \eqref{e:small-energy-condition} and \eqref{e:grst} we infer
\begin{align*}
\bE(\bG_{u^+}+\bG_{u^-},\bC_{2-\overline\eta}) & \le\bE(T_0,\bC_{2-\overline\eta})+C\eta_*E\\
\bE(T,\bB_1,\pi) & \le \bE(\bG_{u^+}+\bG_{u^-},\bC_1,\pi)+C\eta_* E\,.
\end{align*}
Combining these two last inequalities with \eqref{e:gollonzo2}, we conclude
\begin{equation}\label{e:quasi_gollonzo1}
\bE(T,\bB_1,\pi)\le(2-\overline\eta)^{2-\varepsilon}\bE(T_0,\bC_{2-\overline\eta})+C\eta_*E+\overline\eta E\,.
\end{equation}
Using the height bound in Theorem \ref{t:height_bound}, we infer
\[
\supp(T)\cap\bC_{2-\overline\eta}\subset\bB_2\,.
\] 
Since $T_0\res\bB_2=T\res\bB_2$, \eqref{e:quasi_gollonzo1} gives us that
\begin{align*}
\bE^\flat(T,\bB_1) & \le \bE(T,\bB_1,\pi)\\
&\le(2-\overline\eta)^{- (2-\varepsilon)}\left(\frac{2}{2-\overline\eta}\right)^m\bE(T,\bB_2,\pi_0)+C\eta_*E+\overline\eta E\\
& = (2-\overline\eta)^{- (2-\varepsilon)}\left(\frac{2}{2-\overline\eta}\right)^m\bE^\flat(T,\bB_2)+C\eta_*E+\overline\eta E\,.
\end{align*}
Hence, since the constant $C$ in the last inequality is independent of the parameters $\eta_*,\overline\eta$, choosing the latter sufficiently small and recalling \eqref{e:bibo}, we conclude \eqref{e:gollonzo}.

\subsection{Reduction to $L^2$-decay}

In this section we want to replace the excesses in \eqref{e:gollonzo2} with suitable $L^2$ quantities. In particular the Taylor expansion of the area functional and the estimate ${\rm Lip}(u^\pm)\le E^\beta$ give
\begin{align}
& \Bigg|2\omega_m(2-\overline\eta)^m\bE(\bG_{u^+}+\bG_{u^-},\bC_{2-\overline\eta})\nonumber\\
&\qquad\quad -\int_{B_{2-\overline\eta}\cap\Omega^+}|Du^+|^2+\int_{B_{2-\overline\eta}\cap\Omega^-}|Du^-|^2\Bigg|\nonumber\\
\le & CE^{2\beta}\left(\int_{B_{2-\overline\eta}\cap\Omega^+}|Du^+|^2+\int_{B_{2-\overline\eta}\cap\Omega^-}|Du^-|^2\right)\le \frac{\overline\eta}{3}E\,,\label{e:exc_vs_dir1}
\end{align}
provided $\varepsilon_0$ is sufficiently small. Let us define the linear map $x\mapsto Ax:=(Dh(0)x,0)$. We now claim that
\begin{align}
2\omega_m\bE(\bG_{u^+}+\bG_{u^-},\bC_1,\pi)&\le\int_{B_1\cap\Omega^+}\!\!\G(Du^+\!,Q\a{A})^2\nonumber\\
&\qquad +\!\int_{B_1\cap\Omega^-}\!\!\G(Du^-\!,(Q-1)\a{A})^2+\frac{\overline\eta}{3}E\,.\label{e:exc_wwass_claim}
\end{align}
If we introduce the notation $\vec\tau$ for the unit simple $m$-vector orienting $\pi$, then the latter inequality is implied by 
\begin{equation}\label{e:tilt_wass}
 \int_{\Omega^+\cap B_1\times\R^n}\left|\vec\bG_{u^+}-\vec\tau\right|^2\,d\|\bG_{u^+}\|\le\int\G(Du^+,Q\a{A})^2+\frac{\overline\eta}{3}E
\end{equation}
and the analogous inequality for $u^-$. In fact, since the argument is entirely similar, we only show \eqref{e:tilt_wass}. The argument follows the one of \cite[Theorem 3.5]{DS2}. Arguing as in \cite{DS2}, thanks to \cite[Lemma 1.1]{DS2},
we can write $u^+= \sum_i \a{u^+_i}$ and process local computations (when needed) as if each $u^+_i$ were Lipschitz. 
Moreover, we have that
\[
\vec\tau = \textstyle{\frac{\xi}{|\xi|}} \quad \text{with }\;
\xi = (e_1+ A \,e_1)\wedge \ldots \wedge (e_m+A\,e_m).
\]
Here and for the rest of this proof, we identify $\R^m$ and $\R^n$ with
the subspaces $\R^m \times \{0\}$ and $\{0\} \times \R^n$ of $\R^{m+n}$, respectively:
this justifies the notation $e_j+A\,e_j$ for $e_j \in \R^m$ and $A\, e_j \in \R^n$.
Next, we recall that
\[ 
|\xi| = \sqrt{\langle \xi, \xi\rangle} = \sqrt{\det (\delta_{ij} + \langle A \, e_i , A\, e_j\rangle)} = 1 + \textstyle{\frac{1}{2}} |A|^2
 +O (|A|^4).
\]
By \cite[Corollary 1.11]{DS2}
\begin{align}
E^{\rm tilt}:= & \hphantom{-}\int_{(\Omega^+\cap B_1)\times\R^n} \left| \vec{\bG}_{u^+} - \vec\tau\right|^2\, d \|\bG_{u^+}\|\\
 = & 2\,\bM (\bG_{u^+}) - 2 \int_{(\Omega^+\cap B_1)\times\R^n} \langle \vec\bG_{u^+}, \vec\tau \,\rangle\, d\|\bG_{u^+}\|\nonumber\\
= & \hphantom{-}2 \, Q \, |\Omega^+\cap B_1| + \int_{\Omega^+\cap B_1} (|Du^+|^2 + O (|Du^+|^4))\nonumber\\ 
&- 2\int_{\Omega^+\cap B_1} \sum_i \langle (e_1+Du^+_i\, e_1)\wedge \ldots \wedge (e_m+Du^+_i \, e_m), \vec\tau \,\rangle\nonumber.
\end{align}
On the other hand $\langle A\,e_j , e_k\rangle = 0 = \langle Du^+_i\, e_j , e_k\rangle$. Therefore,
\begin{align*}
 & \langle (e_1+Du^+_i \, e_1)\wedge\ldots \wedge (e_m+Du^+_i\,e_m), \vec \tau \, \rangle\\
= & |\xi|^{-1} \det (\delta_{jk} +
\langle Du^+_i \, e_j , A \,e_k\rangle)\nonumber\\
= &
\left(1 + \frac{|A|^2}{2} + O (|A|^4)\right)^{-1}
 \left(1 + Du^+_i : A + O (|Du^+|^2|A|^2)\right) \, .
\end{align*}
By the mean value property of harmonic functions 
\begin{equation}\label{e:grst2}
|A|=\left|\mint_{B_1}Dh\right|\le CE^\frac 12
\end{equation}
 and the Lipschitz bound ${\rm Lip}(u^+)\le E^\beta$, we conclude
\begin{align*}
E^{\rm tilt} & = \int_{B_1\cap\Omega^+} |Du^+|^2 +  Q\,|\Omega^+\cap B_1|\,|A|^2\\
&\qquad - 2 \int_{B_1\cap\Omega^+} \sum_i Du^+_i :A  + O \left(E^{1+2\beta}\right)\\ 
&=\int_{\Omega^+\cap B_1} \sum_i |Du^+_i - A|^2 + O\left(E^{1+2\beta}\right)\\
& = \int_{\Omega^+\cap B_1} \G (Du^+, Q\a{A})^2 + O(E^{1+2\beta})\, .
\end{align*}
The claim \eqref{e:exc_wwass_claim} follows from the latter identity for $\varepsilon_0$ small enough.

Combining \eqref{e:exc_vs_dir1} and \eqref{e:exc_wwass_claim}, \eqref{e:gollonzo2} is reduced to
\begin{align}
& \int_{\Omega^+\cap B_1}\G(Du^+,Q\a{A})^2+\int_{\Omega^+\cap B_1}\G(Du^-,(Q-1)\a{A})^2\nonumber\\
< &\,(2-\overline\eta)^{-m-2+\varepsilon}\left(\int_{\Omega^+\cap B_{2-\overline\eta}}|Du^+|^2+\int_{\Omega^-\cap B_{2-\overline\eta}}|Du^-|^2\right)+\frac{\overline\eta}{3}E\,.\label{e:gollonzo3}
\end{align}

\subsection{Reduction to $L^2$-decay for harmonic functions} 

As a first step, we substitute $u^+$ and $u^-$ in the inequality \eqref{e:gollonzo3} with $Q\a{\kappa}$ and $(Q-1)\a{\kappa}$, where $\kappa$ is as in Theorem \ref{t:harm_1}. In fact, from \eqref{e:harm-app1} and \eqref{e:harm-app2}
\begin{align*}
 &\int_{\Omega^+\cap B_{2-\overline\eta}}\!\!|Du^+|^2+\int_{\Omega^-\cap B_{2-\overline\eta}}\!\!|Du^-|^2\\
\ge & Q\int_{\Omega^+\cap B_{2-\overline\eta}}\!\!|D\kappa|^2+(Q-1)\int_{\Omega^-\cap B_{2-\overline\eta}}|\!\!|D\kappa|^2-4\sqrt{\eta_*}E\,.
\end{align*}
Moreover, using again \eqref{e:harm-app1}, \eqref{e:harm-app2} and \eqref{e:harm-app3}, the identity 
\begin{align*}
&\int_{\Omega^+\cap B_1}\G(Du^+,\a{A})^2\\ =&\int_{\Omega^+\cap B_1}\left(|Du^+|^2-2 Q (D(\etaa\circ u^+):A) + Q|A|^2\right)\,,
\end{align*}
and \eqref{e:grst2}, we also conclude
\begin{align*}
  & \int_{\Omega^+\cap B_1}\G(Du^+,\a{A})^2+\int_{\Omega^-\cap B_1}\G(Du^-,(Q-1)\a{A})^2\\
&\le Q\int_{\Omega^+\cap B_1}|D\kappa-A|^2+(Q-1)\int_{\Omega^-\cap B_1}|D\kappa-A|^2+C\eta_*^{\sfrac{1}{2}}E\,.
\end{align*}
Next, notice that
\[
\left|\Omega^+\setminus B^+_{2-\overline\eta}\right|+\left|B^+_{2-\overline\eta}\setminus\Omega^+\right|\le C\|\bA_\gammaup\|\le C\bA\le CM_0^{-\sfrac 12}E^{\sfrac 12}
\]
and compute
\begin{align*}
|D\kappa| \le &|Dh|+|D_x\Psi(x,h)|+|D_u\Psi(x,h)||Dh|\nonumber\\
\le & \frac{C}{\bar \eta^m} E^\frac 12 \qquad \mbox{for $x\in B_{2-\eta}$.}
\end{align*}
In the latter estimate we are using that the harmonic function $h$ is defined on $B_{2-\frac{\eta}{2}}$ and that $\int |Dh|^2 \leq C E$, together with the usual interior estimates for harmonic functions. Note that, in particular, we have the better bound $|D\kappa|\leq C E^\frac 12$ on the smaller ball $B_1$. 

Thus
\begin{align*}
& \,Q\int_{\Omega^+\cap B_1}|D\kappa-A|^2+(Q-1)\int_{\Omega^-\cap B_1}|D\kappa-A|^2 \\ 
\le & \,Q\int_{B_1^+}|D\kappa-A|^2+(Q-1)\int_{B_1^-}|D\kappa-A|^2+CE^\frac 32\\
\end{align*}
and
\begin{align*}
&\, Q\int_{\Omega^+\cap B_{2-\overline\eta}}|D u^+|^2+(Q-1)\int_{\Omega^+\cap B_{2-\overline\eta}}|Du^-|^2\\
\ge &\, Q\int_{B^+_{2-\overline\eta}}|D\kappa|^2+(Q-1)\int_{B_{2-\overline\eta}^-}|D\kappa|^2-\frac{C}{\bar \eta^m} E^\frac 32\,.
\end{align*}
In conclusion, if $\varepsilon_0$ is sufficiently small (depending on $\bar \eta$) \eqref{e:gollonzo3} is reduced to
\begin{align}\label{e:gollonzo4}
&\, Q\int_{B_1^+}|D\kappa-A|^2+(Q-1)\int_{B_1^-}|D\kappa-A|^2\nonumber\\
\le & \,(2-\overline\eta)^{-m-2+\varepsilon}\left(Q\int_{B^+_{2-\overline\eta}}|D\kappa|^2+(Q-1)\int_{B^-_{2-\overline\eta}}|D\kappa|^2\right)+\frac{\overline\eta}{8}E\,.
\end{align}

Now we will substitute $\kappa$ with the harmonic function $h$ in \eqref{e:gollonzo4}. To this regard, recall that $A=(Dh(0),0)$ and  \[
D\kappa=(Dh,D_x\Psi+D_u\Psi(x,h)Dh)\,,
\]
where
\[
|D_x\Psi|+|D_u\Psi|\le C\bA\le \frac{C}{M_0^\frac 12}E^\frac 12\,.
\]
Therefore 
\begin{align*}
|D\kappa-A|^2 & \le|Dh-Dh(0)|^2+\frac{C}{M_0}E\\
|D\kappa|^2 & \ge |Dh|^2\,.
\end{align*}
Hence, assuming $M_0$ sufficiently large, the proof of \eqref{e:gollonzo4} will be completed in the next paragraph, where we show that
\begin{align}
& \, Q\int_{B_1^+}|Dh-Dh(0)|^2+(Q-1)\int_{B_1^-}|Dh-Dh(0)|^2\nonumber\\
\le & \,(2-\overline\eta)^{-m-2}\left(Q\int_{B^+_{2-\overline\eta}}|Dh|^2+(Q-1)\int_{B^-_{2-\overline\eta}}|Dh|^2\right)\,.\label{e:gollonzo5}
\end{align}

Recall that $h$ vanishes on $\{x_m=0\}$, hence by the Schwarz reflection principle and unique continuation for harmonic functions, $h(x',x_m)=-h(x',-x_m)$
(see Remark \ref{r:odd_harmonic}). This implies that the left hand side of \eqref{e:gollonzo5} equals $\qhalf\int_{B_1}|Dh-Dh(0)|^2$, whereas the right hand side equals $(2-\overline\eta)^{-m-2}\qhalf\int_{B_{2-\overline\eta}}|Dh|^2$. Thus \eqref{e:gollonzo5} is equivalent to
\begin{equation}\label{e:gollonzo6}
\int_{B_1}|Dh-Dh(0)|^2\le(2-\overline\eta)^{-m-2}\int_{B_{2-\overline\eta}}|Dh|^2\,,
\end{equation}
which is a classical inequality for harmonic functions. In order to show \eqref{e:gollonzo6} it suffices to decompose
$Dh$ in series of homogeneous harmonic polynomials $Dh (x) = \sum_{i =0}^\infty P_i (x)$, where $i$ is the degree. 
In particular the restriction of this decomposition
on any sphere $S := \partial B_\rho$
gives the decomposition of $Dh|_S$ in spherical harmonics, see \cite[Chapter 5, Section 2]{SW}. It turns out, therefore, that the $P_i$ are $L^2 (B_\rho)$-orthogonal.  Since the constant polynomial $P_0$ is $Dh (0)$ and $\int_{B_{1}} |P_i|^2 = (2-\overline\eta)^{-m-2i}\int_{B_{2-\overline\eta}} |P_i|^2$, \eqref{e:gollonzo6} follows at once.

\section{Proof of Theorem \ref{thm:decay_and_uniq}}\label{sec:proof_exc_dec}

We  first notice that, by definition of collapsed point, for every \(\delta>0\) there exists \(\bar\rho =\bar\rho (\delta)\) small such that  
\begin{itemize}
\item[(i)]  \(\bE^\flat(T,\bB_{2\sigma}(p))+ 4 \bA \sigma^2\le \delta\) for every $\sigma\leq \bar\rho$;
\item[(ii)] \(\Theta(T,q)\ge \Theta(T,p)=Q-\frac{1}{2}\) for all \(q\in \gammaup \cap \bB_{2\bar\rho} (p)\).
\end{itemize}
Next, since $\Theta (T, p)= Q-\frac{1}{2}$, if the radius $\bar\rho$ is chosen small enough we can assume that
\[
\|T\| (\bB_{4\bar\rho} (p)) \leq \omega_m \left(Q-\frac{3}{8}\right) (4\bar\rho)^m\, .
\]
By a simple comparison, for $\eta$ sufficiently small, if $q\in \bB_\eta (p)\cap \Gamma$ and $\bar\rho' =\bar \rho-\eta$, then
\begin{align*}
\|T\| (\bB_{4\bar\rho '} (q)) \leq \|T\| (\bB_{4\bar\rho} (p)) &\leq \omega_m \left(Q-\frac{3}{8}\right) (4\bar\rho)^m \\
&\leq \omega_m \left(Q-\frac{5}{16}\right) (4\bar\rho')^m\, .
\end{align*}
Next, by the monotonicity formula
\begin{align*}
\sigma^{-m} \|T\| (\bB_\sigma (q)) &\leq e^{\bA (4\bar\rho'-\sigma)} (4\bar \rho')^{-m} \|T\| (\bB_{4\bar \rho'} (q))\\
&\leq e^{\bA (4\bar\rho'-\sigma)}\omega_m \left(Q-\frac{5}{16}\right)\\ 
&\leq e^{4 \bA\bar\rho} \omega_m \left(Q-\frac{5}{16}\right)
\end{align*}
for all $\sigma \leq 4\bar\rho' $. In particular, if $\bar\rho$ is chosen sufficiently small, we then conclude
\begin{equation}\label{e:condizione_iii_Teorema6.8}
\|T\| (\bB_\sigma (q)) \leq \omega_m \left(Q-\frac{1}{4}\right) \sigma^m \qquad \forall q\in \bB_\eta (p)\cap \Gamma 
\;\mbox{and}\; \forall \sigma \leq 4\bar \rho' \, .
\end{equation}
Set now $r:= \min \{\eta, \bar\rho'\}$. For all  points \(q\) in \(\bB_r\cap \gammaup\) we claim that
\begin{equation}\label{e:start}
\bE^\flat(q,\bB_r) \le 2^m \bE^\flat (p, \bB_{2r}) + C \bA^2 r^2 \leq C\delta.
\end{equation}
Indeed let $\pi$ be a plane for which $\bE^\flat (p, \bB_{2r} (p)) = \bE (p, \bB_{2r} (p), \pi)$. By the regularity of
$\Gamma$ and $\Sigma$ we find a plane $\pi (q)$ such that $|\pi - \pi (q)|\leq C r \bA$ and 
$T_q \Gamma \subset \pi (q) \subset T_q \Sigma$. Then we can estimate
\begin{align*}
\bE^\flat(T,\bB_{r}(q))& \le\bE(T,\bB_{r}(q),\tilde\pi(q))\le 2^m \bE(T,\bB_{2r}(p),\pi (q))\nonumber\\
& \le 2^m \bE^\flat (T,\bB_{2r}(p)) + C r^2 \bA^2 \le C\delta\, .
\end{align*}

We will now show that the conclusions of the theorem hold for this particular radius $r$. 
First, without loss of generality we translate $p$ in $0$ and rescale $r$ to $1$. Summarizing our discussion above,
for every $q\in \bB_1 \cap \Gamma$ we have the following three properties
\begin{itemize}
\item[(A)] $\bE^\flat (T, \bB_1 (q)) + \bA^2 \leq 2^m \bE^\flat (T, \bB_2) + C \bA^2 \leq C\delta$;
\item[(B)] $\Theta (T, x) \geq Q-\frac{1}{2}$ for every $x\in \bB_1 (q) \cap \Gamma$;
\item[(C)] $\|T\| (\bB_s (q)) \leq (Q-\frac{1}{4}) \omega_m s^m$ for every $s\leq 1$.  
\end{itemize}
We now fix any point \(q\in \Gamma\cap \bB_1\) and define $e (s) := \bE^\flat (T, \bB_s (q))$.  We claim that
\begin{equation}\label{e:dec_et}
e(2^{-k-1})\le \max\{2^{-2(1-\varepsilon)k}e({\textstyle{\frac{1}{4}}}),2^{-2(1-\varepsilon)k+2}e({\textstyle{\frac{1}{2}}})\} \qquad
\mbox{for all $k\in \N$.}
\end{equation}
We prove it by induction on $k$: notice that the inequality is trivially true for $k=0,1$. If the inequality is true for $k=k_0\geq 1$, we want to show it for $k=k_0+1$. We set $\sigma=2^{-k-2}$ and notice that, by inductive assumption  
\[
e(4\sigma)\le\max\{e({\textstyle{\frac{1}{4}}}),e({\textstyle{\frac{1}{2}}})\} \le C e (1) \stackrel{(A)}{\leq} C\delta.
\]
Hence, provided we choose \(\delta=\delta(m,Q)\) (and thus \(r\)) sufficiently small,  we are in the position of applying Theorem \ref{t:exc_dec}: note that the induction assumption covers hypothesis (i) of Theorem \ref{t:exc_dec}, whereas (B) and (C)
imply the hypotheses (ii) and (iii). We thus deduce that  
\begin{align*}
e(2^{-k-2}) &= e(\sigma)\le\max\{2^{-2+2\varepsilon}e(2\sigma),2^{-4+4\varepsilon}e(4\sigma)\}\\
\le &\max\{2^{-2(1-\varepsilon)k}e({\textstyle{\frac{1}{4}}}),2^{-2(1-\varepsilon)k+2}e({\textstyle{\frac{1}{2}}})\}\,.
\end{align*}
From \eqref{e:dec_et} we easily conclude that for all such points \(q\) and for \(\rho\in ]0,\frac{1}{2}[\)
\begin{align}
\bE(T,\bB_\rho (q))\le\bE^\flat(T,\bB_\rho(q)) & \le C\rho^{2-2\varepsilon}e({\textstyle{\frac{1}{2}}})\\
&\le C\rho^{2-2\varepsilon}\bE^\flat(T,\bB_{1}(q))+C\rho^{2-2\varepsilon}\bA^2\nonumber\\
&\le C\rho^{2-2\varepsilon}\bE^\flat(T,\bB_{1} (q))+C\rho^{2-2\varepsilon}\bA^2\nonumber\\
&\stackrel{(A)}{\le} C \rho^{2-2\varepsilon} \bE^\flat (T, \bB_2) + C \rho^{2-2\varepsilon}\bA^2\,. \label{e:dec_rrho}
\end{align}
In addition, the estimate is trivial for $\frac{1}{2} \leq \rho \leq 1$.  Next, given \(0<t<s<1\),  if \(\pi(q,s)\) and \(\pi(q,t)\) are the optimal planes for \(\bE(q,t)\) and \(\bE^\flat(q,s)\), \eqref{e:dec_rrho} implies 
\begin{align*}
|\pi(q,s)-\pi(q,t)|^2 & \le\frac{1}{\|T\|(\bB_s(q))}\int_{\bB_s(q)}|\pi(q,t)-\pi(q,s)|^2\\
& \le C\bE(T,\bB_s(q),\pi(q,s))+C\bE(T,\bB_t(q)),\pi(t))\\
& \le Cs^{2-2\varepsilon}\bE^\flat(T,\bB_1)+Cs^{2-2\varepsilon}\bA^2\,.
\end{align*}
We thus conclude the existence of a unique limit $\pi(q)$ such that
\begin{equation}\label{e:franco}
|\pi(q)-\pi(q,s)|^2\le Cs^{2-2\varepsilon}\bE^\flat(T,\bB_1)+Cs^{2-2\varepsilon}\bA^2\quad\forall\,s\le 1\,.
\end{equation}
From the latter inequality and \eqref{e:dec_rrho}, we conclude \eqref{e:decay-1}, namely statement (c) of the theorem, for all \(q\in \bB_1\cap \gammaup\).

Next, notice that, at every such $q\in \bB_1\cap \gammaup $, $T_q\gammaup\subset \pi(q)\subset T_q\Sigma$ and that, from \eqref{e:decay-1}, the tangent cone is unique and takes the form
\[
Q^*\a{\pi(q)^+}+(Q^*-1)\a{\pi(q)^-}\,.
\]
for some \(Q^*\in \mathbb N\) (since the tangent cone is an integral current). By (ii)  $Q^*-\frac{1}{2} =\Theta(T,q)\ge Q-\frac{1}{2}$. Furthermore, by (C) $Q< Q+1$ and thus \(Q^*=Q\). Therefore \(\Theta(T,q)=Q-\frac{1}{2}\) and  this proves  statements (a) and (b) of the theorem.

We next turn to (e): arguing as in Section \ref{subsec:exc_dec_prel}, we let 
\[
T_0=T\res\left(B_\rho(q,\pi(q))\times B^n_\rho(0,\pi(q)^\perp)\right)
\]
and we note that it  satisfies \eqref{Ass:app} in the cylinder $\bC_\rho(q,\pi(q))$. In addition we have  
\[
\bE(T_0,\bC_\rho(q,\pi(q)))\le C\bE(T,\bB_\rho(q),\pi(q))
\] 
and $T\res\bB_\rho(q)=T_0\res\bB_\rho(q)$. Thus, we can apply Theorem \ref{t:height_bound} to get
\begin{align*}
\bh(T,\bB_\rho(q),\pi(q)) &\le\bh(T_0,\bC_\rho(q,\pi(q)),\pi(q))\\
& \le C(\bE(T,\bB_\rho(q),\pi(q))^\frac 12+\bA^\frac 12\rho^\frac 12)\rho\,.
\end{align*}
The estimate \eqref{e:decay-height} follows at once from the latter inequality and \eqref{e:decay-1}.

We conclude by proving (d) of Theorem \ref{thm:decay_and_uniq}. First of all, observe that it suffices to show \eqref{e:hoelder-1} when $\rho:=|q-q'|\le 1/2$. Recall the estimate \eqref{e:franco}: 
\[
\max\{|\pi(q)-\pi(q,\rho)|,|\pi(q')-\pi(q',\rho)|\}\le C(\bE^\flat(T,\bB_{1})^\frac 12+\bA) \rho^{1-\varepsilon}\,.
\]
Hence to complete the proof of \eqref{e:hoelder-1}, we notice that
\begin{align*}
& |\pi(q,\rho)-\pi(q',\rho)|^2  \le\mint_{B_\rho(q)\cap B_\rho(q')}|\pi(q,\rho)-\pi(q',\rho)|^2\\
& \le \frac{C}{\omega_m\rho^m}\int_{B_\rho(q)}|\vec T-\pi(q,\rho)|^2+\frac{C}{\omega_m\rho^m}\int_{B_\rho(q')}|\vec T-\pi(q',\rho)|^2\\
& =C(\bE^\flat(T,\bB_\rho(q))+\bE^\flat(T,\bB_\rho(q')))\\
& \le C(\bE^\flat(T,\bB_1)+\bA^2)\rho^{2-2\varepsilon}\, ,
\end{align*}
where we have also used that \(\|T\|(B_{\rho}(p)\ge c \rho^m\), a simple consequence of the monotonicity formula in Theorem \ref{thm:allard}.

\section{Proof of Corollary \ref{c:cone_cut}}

The inclusion \eqref{eq:height-envelope1} follows immediately from \eqref{e:decay-height} applied to some $\rho$ with $2|x-q|> \rho > |x-q|$, where $x\in \supp (T) \cap \bB_\sigma (q)$. 
Next we observe that \eqref{eq:height-envelope1} is in fact stronger than \eqref{eq:height-envelope2}, because, by \eqref{e:tilt_pi(q)}, we can control the tilt $|\pi (q)-\pi(p)|$. 
Indeed,
\begin{align*}
\abs{\bp^\perp -\bp_q^\perp}^2&= \abs{\bp -\bp_q}^2\le m \abs{\pi- \pi(q)}^2 \stackrel{\eqref{e:franco}}{\le} C E.
\end{align*}
Using Theorem \ref{thm:decay_and_uniq}(d) with $q'=p$ and $\varepsilon = \frac{1}{2}$ we conclude the crude estimate $|\pi (q)-\pi (p)|\leq C (E^{\sfrac{1}{2}} + \bA r)$. In particular
\[
\abs{\bp_q^\perp -\bp^\perp}^2= \abs{\bp_q -\bp}^2\le m \abs{\pi(q)- \pi}^2 \le C (E + \bA^2 r^2)\, .
\]
Fix therefore a point $x\in \bB_\sigma (q)\cap \supp (T)$. Then
\begin{align*}
|\bp^\perp (x-q)|\leq & |x-q| |\bp^\perp - \bp_q^\perp|+ |\bp_q^\perp (x-q)|\nonumber\\
\leq &C (E^{\sfrac{1}{2}} + \bA r) |x-q| + C (r^{-1} E + \bA)^{\sfrac{1}{2}} \abs{x-q}^{\frac32}\nonumber\\
\leq & C (E + \bA r)^{\sfrac{1}{2}} |x-q|\, ,
\end{align*}
which proves \eqref{eq:height-envelope2}.

\chapter{Second Lipschitz approximation}\label{chap:Lip2}

Recalling Theorem \ref{thm:step0}, our main task is to show that, under Assumption \ref{ass:main}, any  collapsed point $q\in \gammaup$ is regular. 
By  the usual scaling and translation argument,  we can moreover assume that:
\begin{itemize}
\item[(i)] $0\in \gammaup$ is a collapsed point with multiplicity $\Theta (T, 0) = Q - \frac{1}{2}$;
\item[(ii)] at any point $q\in \gammaup\cap \bB_1$ the conclusions of Theorem \ref{thm:decay_and_uniq} apply for every radius $r\leq 1$;
\item[(iii)] $\bA$ and $\bE^\flat (T, \bB_2)$ are small, namely
\begin{equation}\label{eq:parametro_eps_0}
\bA^2 + \bE^\flat (T, \bB_2) < \varepsilon_0\, ,
\end{equation}
where $\varepsilon_0$ is a sufficiently small constant whose choice will be specified in the remaining proofs. 
\end{itemize}

Let $\pi_0$ be a plane which minimizes the expression defining $\bE^\flat (T, \bB_1)$. By Corollary \ref{c:cone_cut}, we know that
\begin{equation}
\supp (T)\cap \bB_1 \subset \{x: |\bp_0^\perp (x)|\leq C \varepsilon_0^{\sfrac{1}{2}}\abs{x}\}\, ,
\end{equation}
where $\bp_0^\perp$ is the orthogonal projection on $\pi_0^\perp$.
Since we can restrict the current $T$ to $\bB_1$ and further scale by a factor $2$, we can assume, without loss of generality, that
\begin{itemize}
\item[(iv)] There is a plane $\pi_0$ such that $\bE^\flat (T, \bB_2) =\bE (T, \bB_2, \pi_0)$, $T_0 \gammaup \subset \pi_0\subset T_0 \Sigma$ and 
\begin{equation}
\supp (T)\cap \bB_2 \subset \{x: |\bp_0^\perp (x)|\leq C \varepsilon_0^{\sfrac{1}{2}}\abs{x}\} \, .
\end{equation}
\end{itemize}

From now on we will work under the above assumptions, which we summarize together in the following

\begin{ipotesi}\label{ass:decay+cone}
 $T$, $\Sigma$ and $\gammaup$ are as in Assumption \ref{ass:main} and they satisfy additionally the conditions (i), (ii), (iii) and (iv) above. 
\end{ipotesi}

In particular, Theorem \ref{thm:step0} is implied by the following milder version:

\begin{theorem}\label{thm:step4}
If $T, \Sigma$ and $\gammaup$ are as in Assumption \ref{ass:decay+cone}, then $0$ is a regular boundary point of $T$. 
\end{theorem}

In this framework we can then refine our Lipschitz approximation in cylinders with small excess. We first note the following corollary
of Theorem \ref{thm:decay_and_uniq} and of the cone condition in Assumption \ref{ass:decay+cone}(iv). 

\begin{proposition}\label{p:good_cylinders}
Let $T, \Sigma$  and $\gammaup$ be as in Assumption \ref{ass:decay+cone} with $\varepsilon_0$ sufficiently small (depending only upon $m,n, \bar n$ and $Q$). Then there are positive constants $C = C(m,n,\bar n, Q)$ and $\bar \varepsilon = \bar \varepsilon (m,n,\bar n,Q)$ with the following properties. Assume that $q\in \gammaup\cap \bB_1$, $r<\frac{1}{8}$ and $\pi$ is an $m$-dimensional plane such that $T_q \gammaup \subset \pi \subset T_q \Sigma$ and
\begin{equation}\label{e:exc_small_100}
E = \bE (T, \bC_{4r} (q,\pi)) < \bar \varepsilon\, .
\end{equation}
Then 
\[
\supp (\partial (T\res \bC_{4r} (q, \pi))) \subset \partial \bC_{4r} (q, \pi)  \cup \gammaup
\]
and
\begin{equation}\label{e:height_100}
\bh (T, \bC_{2r} (q, \pi), \pi) \leq C r (E+ \bA r)^{\sfrac{1}{2}}\, .
\end{equation}
\end{proposition}

We are then ready to state  our improved approximation theorem:

\begin{theorem}\label{thm:second_lip} \label{THM:SECOND_LIP}
Let $T$, $\Sigma$, $\gammaup$, $q$, $r$ and $\pi$ be as in Proposition \ref{p:good_cylinders}. Consider the orthogonal projection $\gammado$ of $\gammaup\cap \bC_{4r} (q, \pi)$ onto the plane $q+\pi$ and observe that, since $\varepsilon_0$ is sufficiently small, $\gammaup \cap \bC_{4r} (q, \pi)$ is the graph over $\gammado$ of a $C^{3,a_0}$ function $\psi$. 
Then there are a closed set $K\subset B_{r} (q) = B_r (q, \pi)$ and a $\qhalf$-valued map $(u^+, u^-)$ on $B_r (p)$ which collapses at the interface $(\gammado, \psi)$ satisfying the following estimates:
\begin{align}
\Lip (u^\pm)\leq\; & C  (E+ \bA^2 r^2)^{\sigmaexp}\label{e:slip1}\\
{\rm osc} (u^\pm) \leq\; & C (E +  \bA r)^{\sfrac{1}{2}} r\label{e:slip2}\\
\bG_{u^\pm} \res[(K \cap \Omega^\pm) \times \pi^\perp] =\; & T \res [(K\cap \Omega^\pm)\times \R^n]\label{e:slip3}\\
\gr (u^\pm)\subset\; &\Sigma\label{e:slip4}\\
|B_r (q)\setminus K|\leq & C (E + \bA^2 r^2)^{1+\sigmaexp}r^m\label{e:slip4.25}\\
\be_T(B_{r} (q)\setminus K) \le & C (E + \bA^2 r^2)^{1+\sigmaexp}r^m\label{e:slip4.5}\\
\int_{B_r (q)\setminus K} \abs{Du}^2 \le & C (E +\bA^2r^2)^{1+\sigmaexp}r^m \label{e:slip5}\\
\left| \be_T(F) - \frac12 \int_F \abs{Du^\pm}^2\right| \le & C (E + \bA^2 r^2)^{1+\sigmaexp} r^m \; \quad \forall F \subset \Omega^\pm\text{ measurable,}\label{e:slip6}
\end{align}
where $\Omega^\pm$ are the two regions in which $B_r (q)$ is divided by $\gamma$, whereas
$C\ge 1$ and $\sigmaexp\in ]0,\frac{1}{4}[$ are two positive constants which depend on $m,n, \bar n$ and $Q$. 
\end{theorem}

\section{Preliminary observations}

We start recalling \cite[Theorem 2.4]{DS3} in our context. 

\begin{theorem}[Almgren's strong approximation]\label{t:old_approx}
There exist constants $C, \sigmaexp,\bar \eps>0$ (depending on $m,n,\bar n,Q$)
with the following property. Let $T$, $\Sigma$ and $\gammaup$ be as in Assumption \ref{ass:decay+cone}, $\pi$, $q$ and $r$ as in Proposition \ref{p:good_cylinders} and let \(x\in \bB_1\) such that 
\begin{itemize}
\item[(i)] the cylinder $\bC := \bC_{4\rho} (x, \pi)$ does not intersect $\gammaup$ and is contained in $\bC_{4r} (q, \pi)$; 
\item[(ii)] $\bA^2 \rho^2 + \bar E = \bA^2 + \bE(T,\bC_{4\,\rho} (x, \pi)) < \bar \varepsilon$. 
\end{itemize}
Then, there is a map $f: B_\rho (x, \pi) \to \mathcal{A}_Q (\pi^\perp)$, or a map $f:B_\rho (x, \pi) \to \mathcal{A}_{Q-1} (\pi^\perp)$, with $\supp (f(z))\subset \Sigma$ for every  $z\in B_\rho(x, \pi)$,
and a closed set $\bar K\subset B_\rho (x, \pi) $ such that
\begin{equation}\label{e:main(i)}
\Lip (f) \leq C (\bar E + \bA^2\rho^2)^{\sigmaexp}
\end{equation}
\begin{align}
&\bG_f\res (\bar K\times \R^n)=T\res (\bar K\times\R^{n})\nonumber\\
\quad\mbox{and}\quad
&|B_\rho (x, \pi)\setminus \bar K| \leq
 C  \left(\bar E+ \bA^2\rho^2\right)^{1+\sigmaexp}\, \rho^m, \label{e:main(ii)}
\end{align}
\begin{align}
 & \left| \|T\| (\bC_{s\rho} (x)) - Q \,\omega_m\,(s\rho)^m -
{\textstyle{\frac{1}{2}}} \int_{B_{s\rho} (x, \pi)} |Df|^2\right|\\ 
\leq
& C   \left(\bar E+ \bA^2\rho^2\right)^{1+\sigmaexp}\, \rho^m \quad \forall\,0<s \leq 1\, \label{e:main(iii)}
\end{align}
and
\begin{equation}\label{e:L-infty_est}
{\rm osc}\, (f) \leq C \bh (T, \bC,  \pi) + C (\bar E^{\sfrac{1}{2}} + \bA\rho)\, \rho\, .
\end{equation}
\end{theorem}

From now on, in order to simplify our notation, we assume that $\pi=\pi_0 = \mathbb R^m \times \{0\}$ and use the shorthand notation $B_t (x)$ for $B_t (x, \pi)$.

In addition to the conclusions of the theorem above, we observe that they imply the following further estimates 
\begin{align}
\be_T(B_\rho (x)\setminus \bar K) \le & C (\bar E+\rho^2\bA^2)^{1+\sigmaexp}\rho^m\label{eq:additional_estimates_strongapprox1}\\
\int_{B_\rho(x)\setminus \bar K} \abs{Df}^2 \le & C (\bar E+\rho^2\bA^2)^{1+\sigmaexp}\rho^m\label{eq:additional_estimates_approx3} \\
\label{eq:additional_estimates_strongapprox2}
\left| \be_T(F) - \frac12 \int_F \abs{Df}^2\right| \le & C \bar (\bar E+\rho^2\bA^2)^{1+\sigmaexp}\rho^m \;\; \forall F \subset B_\rho (x)\; \text{ measurable.}	
\end{align}
This can be seen as follows. First of all \eqref{e:main(i)} and \eqref{e:main(ii)} give
\[ \int_{F\setminus \bar K} \abs{Df}^2 \le C (\bar E+\bA^2\rho^2 )^{2\sigmaexp} \abs{B_\rho (x)\setminus \bar K} \le C (\bar E+\bA^2\rho^2 )^{1+\sigmaexp}\rho^m \]
for every $F \subset B_\rho (x)$ measurable.
In particular we achieve \eqref{eq:additional_estimates_approx3} setting $F= B_\rho (x)$. 

Next recall that $\|T\| (B_\rho (x)) - Q \omega_m \rho^m = \be_T (B_\rho (x))$ and hence \eqref{e:main(iii)} can be reformulated, for $ s=1$, as
\[ \left|\be_T (B_\rho (x)) - {\textstyle{\frac12}}\int_{B_\rho (x)} \abs{Df}^2\right| \le  C (\bar E+\bA^2\rho^2)^{1+\sigmaexp}\rho^m\, . \]
In particular
\[
\frac{1}{2} \int_{B_\rho (x)} \abs{Df}^2 \leq\left(\bar E +  C(\bar E+\bA^2\rho^2)^{1+\sigmaexp}\right)\rho^m \leq C\left(\bar E +\bA^2\rho^2\right)\rho^m\, .
\]
Secondly, the Taylor expansion of the area functional and \eqref{e:main(i)} give
\[ 
\left|\be_{\bG_f}(F)-{\textstyle{\frac12}} \int_F \abs{Df}^2\right| \le C\Lip(f)^2 \int_{F}\abs{Df}^2 \le C (\bar E+\bA^2\rho^2)^{1+2\sigmaexp}\rho^m
\]
for every $F \subset B_\rho (x)$ measurable.

Combining the inequalities just obtained we achieve 
\begin{align*}
& \be_T(B_\rho (x)\setminus \bar K) = \be_T(B_\rho (x)) - \be_{\bG_f}(B_\rho (x)\cap \bar K) \\
\le &\left|\be_T(B_\rho (x))-{\frac12} \int_{B_\rho (x)} \abs{Df}^2 \right|\\
&\quad + \left|{\frac12} \int_{B_\rho (x)\cap \bar K} |Df|^2- \be_{\bG_f}(B_\rho (x)\cap \bar K)\right| + \int_{B_\rho (x)\setminus \bar K} \abs{Df}^2\\
\le &C(\bar E+\bA^2\rho^2 )^{1+\sigmaexp}\rho^m\, ,
\end{align*}
which implies \eqref{eq:additional_estimates_strongapprox1}.

Finally, for every $F\subset B_\rho (x)$ measurable we have
\begin{align*}
\left| \be_T(F) - \frac12 \int_F \abs{Df}^2\right| &\le \left|\be_{\bG_{\color{blue}f}}(F\cap K) - \frac12 \int_{F\cap K} \abs{Df}^2\right|\\
&\qquad + \be_T(F\setminus K) + \frac12 \int_{F\setminus K} \abs{Df}^2\\
&\le C (\bar E+\bA^2\rho^2 )^{1+\sigmaexp}\rho^m\, .
\end{align*}

\section{Proof of Theorem \ref{THM:SECOND_LIP}}

Without loss of generality we assume that $T_q \gammaup = \R^{m-1}\times \{0\}$, $\pi = \R^m \times \{0\}$ and $T_q \Sigma = \R^{m+\bar n}\times \{0\}$. We then use $\bC_s (q)$ in place of $\bC_s (q, \pi)$,  and $B_s (q)$ in place of
$B_s (q, \pi)$. 
Note that
\begin{align}
\partial T \res \bC_{4r} (q) &= \a{\gammaup\cap \bC_{4r} (q)}\nonumber\\
\qquad\mbox{and}\qquad  \bp_\sharp (\partial T \res \bC_{4r} (q)) &=\a{\gammado\cap B_{4r}(\bp(q))}\, .\label{eq:assumptions1}
\end{align}
As in the previous sections, denote by $\Omega^+$ and $\Omega^-$ the two connected components of $B_{4r} (q)\setminus \gammado$, chosen so that
\[
\bp_\sharp T\res \bC_{4r} (q) = Q \a{\Omega^+} + (Q-1) \a{\Omega^-}\, .
\] 
Let $L_0$ be the cube $q+ [-r,r]^m$ and, for any natural number $k$, let $\mathcal{Q}_k$ be the collection of cubes $L$ of the form
\[
L = q + r 2^{-k} x + [-2^{-k}r, 2^{-k} r]^m\, 
\] 
for $x\in \mathbb Z^m$, which are contained in $L_0$ and intersect $B_r (q)$. 
We fix a number $N\in \N$ such that the $16\sqrt{m}2^{-N}r$-neighborhood of $\cup_{L\in \mathcal{Q}_N} L$ is
contained in $\bC_{4r}(q)$ and construct a Whitney decomposition of 
\[
\tilde\Omega = \bigcup_{L\in \mathcal{Q}_N} L \setminus \gammado
\]
in the following way. We set $\mathcal{R}_N = \mathcal{Q}_N$. If $L\in \mathcal{R}_N$ has ${\rm diam}\, (L) \leq \frac{1}{16} {\rm sep}\, (L, \gammado)$, then
we assign $L$ to the class $\mathcal{W}_N$. Here and in what follows we set
\[
{\rm sep}\, (L, \gamma) = \min \{|x-y|:x\in \gamma, y\in L \}\, .
\]
Otherwise we subdivide it in $2^m$ subcubes of side $2^{-N} r$ and assign them to $\mathcal{R}_{N+1}$. We then inductively define $\mathcal{W}_k$ and $\mathcal{R}_{k+1}$ for every $k\geq N$. The Whitney decomposition $\mathcal{W} = \cup_{k\geq N} \mathcal{W}_k$ is then a collection of closed dyadic cubes whose interiors are pairwise disjoint, which cover $\Omega^+ \cup \Omega^-$ and such that 
\begin{equation}\label{eq:Whitney_second_lip}
\min\left\{\frac{1}{32} {\rm sep}\, (L, \gammado), \sqrt{m} 2^{-N+1}\right\} \leq\; \diam(L) \le \frac{1}{16}{\rm sep}\, (L, \gammado).
\end{equation}
We denote with $c_L$ the center of the cube $L\in \mathcal{W}$ and set $r_L:= 3\, \diam(L)$ so that 
$L \subset B_{\frac14 r_L}(c_L)$.

\begin{figure}[htbp]
\begin{center}\label{fig:whitney}
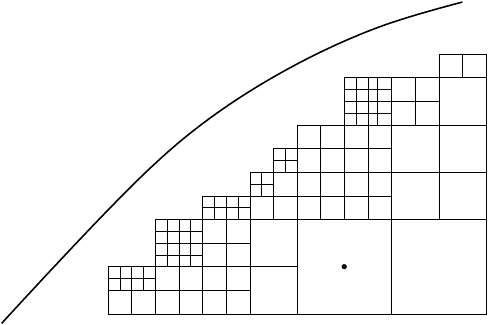
\end{center}
\caption{The Whitney decomposition $\mathcal{W}$ in $\Omega^-$.}
\end{figure}

We claim that for each cube $L$ the current $T$ restricted to the cylinder $\bC_{4r_L}(c_L)$ satisfies the assumptions of Theorem \ref{t:old_approx}.

 First note that, by the  construction of the Whitney decomposition, we have $\bC_{4r_L}(c_L) \cap \gammaup = \emptyset$ and $\bB_{6r_L}(c_L)\subset \bB_{4r}(q)$ and thus $\partial T \res \bC_{4r_L} (c_L) =0$. Moreover, either $B_{4r_L} (c_L) \subset \Omega^+$ or $B_{4r_L} (c_L) \subset \Omega^-$ and thus
$\bp_\sharp T \res \bC_{4r_L} (c_L)$ equals either $Q \a{B_{4r_L} (c_L)}$ or $(Q-1) \a{B_{4r_L} (c_L)}$.

To check the second assumption of Theorem \ref{t:old_approx} we distinguish the two cases $r_L = 2^{-N}r$ and $r_L< 2^{-N}r$.
If $r_L = 2^{-N}r$ we simply have
\[ 
\bE(T,\bC_{4r_L}(c_L)) \le 2^{Nm} \bE(T,\bC_{4r}(q)) = 2^{Nm}E. 
\]
For each $L \in \mathcal{W}$ with $r_L < 2^{-N}r$ let $x_L$ be the point of $\gammado$ closest to $c_L$ and let $q_L\in \gammaup$ be the point $(x_L, \psi(x_L))$. From the first inequality of \eqref{eq:Whitney_second_lip} we deduce that $\bC_{4r_L}( c_L ) \subset \bC_{13r_L}(q_L)$. In particular notice that by the cone condition \eqref{e:height_100}, $\supp (T) \cap \bC_{14r_L} (q_L)\subset \bB_{16r_L} (q_L)$ and by our choice of $N$ we have $\bC_{14r_L} (q_L)\subset \bB_{16r_L} (q_L) \subset \bC_{4r} (q)$. 

Next, observe that
\begin{align*}
\bE(T, \bC_{4r_L}(c_L))& \le 4^m \bE(T, \bB_{16r_L} (q_L), \pi)\\
& \leq C \bE (T, \bB_{16 r_L}, \pi (q_L)) + C |\pi- \pi (q_L)|^2
\end{align*}
According to Theorem \eqref{thm:decay_and_uniq} we then conclude
\begin{equation}\label{e:bruno}
\bE(T, \bC_{4r_L}(c_L)) \leq C (E + \bA^2 r^2)\, .
\end{equation}

So, provided $\varepsilon_0$ is chosen sufficiently small, we can apply Theorem \ref{t:old_approx} in every cylinder $\bC_{4r_L}(c_L)$ and obtain:
\begin{itemize}
\item[-] a $Q$-valued (or $(Q-1)$-valued) map $f_L$ on each ball $B_{r_L}(c_L)$ with $\supp(f_L(x)) \in \Sigma$ for every $x\in B_{r_L}(c_L)$
\item[-] a closed sets $K_L \subset B_{r_L}(c_L)$
\end{itemize}
 such that 
 \begin{align}
\Lip(f_L) \le & C (E + \bA^2 r_L^2)^{\sigmaexp} \label{eq:strong-approx-in_L_1}\\
\bG_{f_L}\res(K_L \times \R^n)= & T\res(K_L \times \R^n)\\
\abs{ B_{r_L}(c_L)\setminus K_L} \le & C (E + \bA^2 r_L^2)^{1+\sigmaexp} r_L^m \label{eq:strong-approx-in_L_2}\\
\be_T(B_{r_L}(c_L)\setminus K_L) \le & C (E + \bA^2 r_L^2)^{1+\sigmaexp}r_L^m\\
\int_{B_{r_L}(c_L)\setminus K_L} \abs{Df_L}^2 \le & C (E+ \bA^2 r_L^2)^{1+\sigmaexp}r_L^m \label{eq:strong-approx-in_L_3}\\
\left| \be_T(F) - \frac12 \int_F \abs{Df_L}^2\right| \le & C (E + \bA^2r_L^2 )^{1+\sigmaexp} r_L^m\nonumber\\
&\qquad\qquad\qquad \;\; \forall F \subset B_{r_L}(c_L) \text{ measurable}	 \label{eq:strong-approx-in_L_4}
\end{align}
whereupon \eqref{eq:strong-approx-in_L_3}, \eqref{eq:strong-approx-in_L_4} follow as explained in \eqref{eq:additional_estimates_strongapprox1}, \eqref{eq:additional_estimates_strongapprox2}.

Next, for each $L$ we let $\mathcal{N}^+ (L)$ be the neighboring cubes in $\mathcal{W}$ with larger or equal radius, i.e. 
\[
\mathcal{N}^+(L)=\{ H \in \mathcal{W} \colon H \cap L \neq \emptyset, r_H \ge r_L \}.
\]
Note that by the construction of the Whitney decomposition we ensured that if $H\in \mathcal{N}^+ (L)$, then $L \subset B_{r_{H}} (c_{H})$.  
We define 
\begin{align*}
K'_L &= K_L \cap \bigcap_{H\in \mathcal{N}^+ (L)} K_{H}\\
K^+ &= \bigcup_{L \in \mathcal{W}, L\subset \Omega^+} K'_L\cap L\\
K^- &=\bigcup_{L \in \mathcal{W}, L\subset \Omega^-} K'_L\cap L
\end{align*}
and further 
\[ 
\tilde{u}^+(x) := f_L(x) \text{ if } x \in L \cap K^+ \text{ and } \tilde{u}^-(x) := f_L(x) \text{ if } x\in L\cap K^-. 
\]
Since the cardinality of $\mathcal{N}^+ (L)$ is bounded by a geometric constant $C(m)$,  we conclude from 
from \eqref{eq:strong-approx-in_L_2} that
\begin{equation}\label{eq:strong_Jonas_additional}
|L\setminus K'_L|\leq C (E + \bA^2 r^2)^{1+\sigmaexp}r_L^m.
\end{equation}
In particular, if $\varepsilon_0$ is sufficiently small, we conclude that $L\cap K'_L \neq \emptyset$. 
We next claim that
\begin{align}
\Lip (\tilde{u}^\pm) \leq & C (E + \bA^2 r^2)^{\sigmaexp}\label{e:muta1}\\
\bG_{\tilde{u}^\pm}\res (K^\pm\times \R^n) = & T \res (K^\pm\times \R^n)\label{e:muta2}\\
\be_T (L\setminus K'_L) \leq &  C (E + \bA^2 r^2)^{1+\sigmaexp}r_L^m\label{e:mutanda4.5}\\
\int_{L\setminus K'_L} \abs{D\tilde{u}^\pm}^2 \le & C (E+ \bA^2 r^2)^{1+\sigmaexp}r_L^m\label{e:mutanda5}\, .
\end{align}
Inequalities \eqref{e:muta2}, \eqref{e:mutanda4.5} and \eqref{e:mutanda5} follows easily by the fact that $L\setminus K'_L \subset B_{r_L} (c_L) \setminus K_L$ and $\tilde{u}^\pm$ coincides with $f_L$ on $K'_L$. To show the  the Lipschitz \eqref{e:muta1} we let  $H, L \in \mathcal{W}$ be any two cubes and we  assume that $\diam(H) \ge \diam(L)$ and $x \in H, y \in L$.

If $H \cap L \neq \emptyset$ (and in particular if $H=L$) by construction $\tilde{u}^\pm=f_H$ on $K^\pm\cap B_{r_H}(c_H) \subset K_H$, hence the inequality $\G (\tilde{u}^\pm (x), \tilde{u}^\pm (y))\leq C  (E + \bA^2 r^2)^{\sigmaexp} |x-y|$ follows from the Lipschitz bound for $f_H$.

If $H \cap L = \emptyset$ we have
\[ \frac{1}{2\sqrt{m}} r_H \le \abs{x-y}. \]
In case $r_H=2^{-N}r$ then the Lipschitz estimate follows from the hight bound \eqref{e:height_100}: $\G (\tilde{u}^+ (x), \tilde{u}^+ (x'))\leq 2 Cr(E+\bA r)^{\sfrac12} \le C (E+\bA r)^{\sfrac12} \abs{x-x'}.$

If $r_H < 2^{-N}r$ consider for the points $x,y\in \gammado$  which are the closest to  $x',y'$ respectively
We claim  that
\begin{align}
\G (\tilde{u}^\pm (x), Q \a{\psi (x')}) \leq & C |x-x'| (E + \bA r)^{\sfrac{1}{2}}\\
\G (\tilde{u}^\pm (y), Q \a{\psi (y')}) \leq & C |y-y'| (E + \bA r)^{\sfrac{1}{2}}\, .
\end{align}
Indeed, both inequalities are due to the fact that $\dist\, (x, \gammado)$ is comparable to $r_L$ and that, in the cylinder $\bC_{C 16r_L} (x')$, we have the height bound \eqref{e:height_100} (recall that the points $(x', \psi (x'))$ and $(x, \tilde{u}_i (x))$ are all in the support of the current $T$). Note also that, by the regularity of \(\Gamma\), 
\[
|\psi (x')-\psi (y')|\leq C (E + \bA r)^{\sfrac{1}{2}} |x'-y'|\, .
\]
In particular we can estimate
\begin{align*}
 & \G (\tilde{u}^\pm (x), \tilde{u}^\pm (y))\\ 
\leq & \G (\tilde{u}^\pm (x), Q \a{\psi (x')}) + Q^{\sfrac{1}{2}} |\psi (x') - \psi (y')|
+ \G (\tilde{u}^\pm (y), Q \a{\psi (y')}) \\
\leq & C (E + \bA r)^{\sfrac{1}{2}} (|x-x'| + |x'-y'| + |y'-y|)\\
\leq & C (E + \bA r)^{\sfrac{1}{2}}  (2|x-x'| + |x-y| +2 |y'-y|)\\ 
\leq &C (E + \bA^2 r^2)^{\sigmaexp} |x-y|
\end{align*}
where we have used that $\sigmaexp\leq \frac{1}{4}$ and that
\[
|x-x'|+|y'-y|=\dist(x,\gammado)+\dist(y,\gammado)\le C(r_L+r_H)\le Cr_H\le C|x-y|.
\] 
Note in particular that we have also  proved  that $\tilde{u}^+$ (resp. $\tilde{u}^-$) has a unique Lipschitz extension to $(K^+\cup \gammado)\cap B_r (q)$ (resp. $(K^-\cup \gammado) \cap B_r (q)$) which on $\gammado \cap B_r (q)$ coincides with $Q \a{\psi}$ (resp. $(Q-1) \a{\psi}$). 

We next wish to extend $\tilde{u}^\pm$ to the whole $\Omega^\pm$ keeping the Lipschitz estimate (up to a multiplicative geometric constant) and the property that $\supp (x, \tilde{u}^\pm (x))\subset \Sigma$. This can be easily done observing that $\Sigma\cap \bC_r (q)$ is the graph of a function 
$\Psi: T_0 \Sigma \cap \bB_r (q) \to T_0 \Sigma^\perp = \{0\} \times \R^{n-\bar n}$ with Lipschitz constant controlled by $C \bA r$. Therefore we can write 
\[
\tilde{u}^\pm (x) =\sum_i \a{v^\pm_i (x), \Psi (x, v^\pm_i (x))}
\]
for an appropriate Lipschitz $Q$-valued map $v^+ : K^+ \to \mathcal{A}_Q (\R^{\bar n})$ and an appropriate Lipschitz $(Q-1)$-valued map $v^-: K^- \to \mathcal{A}_{Q-1} (\R^{\bar n})$ with $\Lip (v^\pm) \leq C (E + \bA^2 r^2)^{\sigmaexp}$. Extending first $v^\pm$ to $\Omega^\pm$ and then composing with $\Psi$, we achieve the desired extension $u^\pm$ of $\tilde{u}^\pm$ to $\Omega^\pm$. Note moreover that, by the observation above, the pair $(u^+, u^-)$ collapses at the  interface $(\gammado\cap B_r (q), \psi)$. Recalling the height estimate \eqref{e:height_100}, we also have that ${\rm osc}\, (\tilde{u}^\pm) \leq C (E + \bA r)^{\sfrac{1}{2}} r$ and the Lipschitz extension can be constructed so to preserve the oscillation bound as well (up to a geometric factor, cf. \cite[Theorem 1.7]{DS1}).

Setting $K = K^+\cup K^-$, we have so far proved the conclusions \eqref{e:slip1}, \eqref{e:slip2}, \eqref{e:slip3} and \eqref{e:slip4}. For the remaining estimates, observe first that
\[
\sum_{L\in \mathcal{W}} r_L^m \leq C(m) r^m\, .
\]
Hence, \eqref{e:slip4.25}, \eqref{e:slip4.5} and \eqref{e:slip5} follow from summing, respectively, \eqref{eq:strong_Jonas_additional}, \eqref{e:mutanda4.5} and \eqref{e:mutanda5}. 

Finally, fix a measurable set $F\subset \Omega^+$ and observe that, for any cube $L$ in the Whitney decomposition of $\Omega^+$
\begin{align*}
& \left| \be_T (F\cap L) - \frac{1}{2} \int_{F\cap L} |Du^+|^2\right|\\
\leq\  & 
\left| \be_T (F\cap L\cap K^+) - \frac{1}{2} \int_{F\cap L\cap K^+} |Du^+|^2\right|\\
& + \be_T (L\setminus K^+) + \Lip (u^+)^2 |L\setminus K^+|\\
\leq\  & \left| \be_T(F\cap L\cap K^+) - \frac12 \int_{F \cap L\cap K^+} \abs{Df_L}^2\right| + C (E + \bA^2r^2 )^{1+\sigmaexp} r_L^m\\
 \leq \ &
 C (E + \bA^2r^2 )^{1+\sigmaexp} r_L^m\, .
\end{align*}
Summing over $L$ we obtain \eqref{e:slip6}. The same arguments work for $u^-$ and conclude the proof.


\chapter{Center manifolds}\label{chap:center_manifolds}

As already pointed out in the previous chapter, our task is to prove  Theorem \ref{thm:step4}, which for the reader's convenience we recall here:

\begin{theorem}\label{thm:step4_bis}
If $T, \Sigma$ and $\gammaup$ are as in Assumption \ref{ass:decay+cone}, then $0$ is a regular boundary point of $T$. 
\end{theorem}

We thus work from now on under the assumption that $0$, the origin of our system of coordinates, is a collapsed point and that
\begin{align*}
T_0\gammaup & =\R^{m-1}\times\{0\}\\
T_0\Sigma & =\R^{m+\overline n}\times\{0\}\ \text{and}\\
\R^n & =\R^{m+\overline n+l}\,.
\end{align*}
Therefore, the tangent cone of $T$ at $p=0$ is $Q\a{\pi_0^+}+(Q-1)\a{\pi_0^-}$, where 
\[
\pi_0^\pm = \{x\in \R^n :\pm x_m>0, x_{m+1} = \ldots = x_{n+m}=0\}\, .
\]

As in the previous chapters, we denote by $\gammado$ the projection on $\pi_0$ of $\gammaup$ and, given any sufficiently small open set $\Omega\subset \pi_0$ which is contractible and contains $0$,
we denote by $\Omega^\pm$ those portions of $\Omega$ lying on the right and left of $\gammado$. 
We are going to build two separate  $m$-dimensional surfaces $\mathcal{M}^\pm$\index{aalm\mathcal{M}^\pm@$\mathcal{M}^\pm$} of class $C^3$ which will be called (respectively)  \emph{left and right center manifolds}\index{Left/right center manifold}. Both surfaces lie in the manifold $\Sigma$. $\mathcal{M}^+$ will be a graph over $B_{3/2}^+ (0, \pi_0)$ (which from now on we denote by $B^+_{3/2}$) of some function ${\bm \varphi}^+$\index{aagv{\bm\varphi}^\pm@${\bm\varphi}^+$ (resp. ${\bm\varphi}^-$)} and $\mathcal{M}^-$ a graph over $B_{3/2}^- (0, \pi_0)$ of some function ${\bm \varphi}^-$ . Both center manifolds will have $\gammaup\cap \bC_{3/2} (0, \pi_0)$ as a boundary, when considered as surfaces in the cylinder $\bC_{3/2} (0, \pi_0)$ and will be $C^3$ (in fact $C^{3, \kappa}$ for a suitable positive $\kappa$) {\em up to the boundary}. In addition, at each point $p\in \gammaup \cap \bC_{3/2} (0, \pi_0)$ the tangent space to both manifolds
will be the same and will coincide with the plane $\pi (q)$ of Theorem \ref{thm:decay_and_uniq}. In particular $\mathcal{M} = \mathcal{M}^+\cup \mathcal{M}^-$ will be
a $C^{1,1}$ submanifold of $\Sigma \cap \bC_{3/2} (0, \pi_0)$ {\em without boundary}. 

Finally we remark that at this stage we do not have  any information about higher regularity of $\mathcal{M}$: in particular we do not yet know that the second derivatives of the two functions ${\bm \varphi}^\pm$ coincide at $\gammado$. At the very end of the proof of Theorem \ref{thm:step4_bis}, which will be accomplished in the final chapter, it will however turn out that $\mathcal{M}$ is indeed $C^3$ and that $T\res \bC_{3/2} (0, \pi_0) = Q \a{\mathcal{M}^+} + (Q-1) \a{\mathcal{M}^-}$. 

\section{Construction of the center manifolds}

\subsection{Boundary dyadic cubes and non-boundary dyadic cubes}
We focus on the construction of $\mathcal{M}^+$ (the one of $\mathcal{M}^-$ follows a ``specular'' algorithm). We start by describing a procedure which reaches a suitable Whitney-type decomposition of $B_{3/2}^+$ with cubes whose sides are parallel to the coordinate axes and have sidelength $2 \ell(L)$. The center of any such cube $L$ considered in the procedure will be denoted by $c(L)$ and its sidelength will be denoted by $2 \ell (L)$. 
We start by introducing a family of dyadic cubes $L\subset \pi_0$ in the following way: for $j\geq N_0$ (an integer whose choice will be specified below),
we introduce the families 
\[
\sC_j:=\{L:\,L\text{ is a dyadic cube of side }\ell(L)=2^{-j}\text{ and } B^+_{3/2}\cap L\neq\emptyset\}\,,
\]
For each $L$ define a radius
\[
r_L:=M_0\sqrt m\ell(L)\,,
\]
with $M_0\geq 1$ to be chosen later.  
We then subdivide $\sC := \cup_j \mathscr{C}_j$ into, respectively, \emph{boundary cubes}\index{Boundary cube}\index{aalc\sC^\flat@$\sC^\flat$} and \emph{non-boundary cubes}\footnote{Observe that some boundary cubes can be completely contained in $B_{3/2}^+$. For this reason we prefer to use the term ``non-boundary'' rather than ``interior'' for the cubes in $\sC^\natural$.} \index{Non-boundary cube}\index{aalc\sC^\natural@$\sC^\natural$}
\begin{align*}
\sC^\flat & :=\{L\in \sC:\,\dist(c(L),\gammado)<64 r_L\}\\
\sC^\natural & :=\{L\in \sC:\,\dist(c(L),\gammado)\ge 64 r_L\}\,.
\end{align*}
Likewise we also use the notation $\mathscr{C}^\flat_j$ and $\mathscr{C}^\natural_j$ for $\mathscr{C}^\flat \cap \mathscr{C}_j$
and $\mathscr{C}^\natural_j = \mathscr{C}^\natural \cap \mathscr{C}_j$. Indeed in what follows, without mentioning it any further, we will often use the same convention for
several other subfamilies of $\mathscr{C}$.

\begin{definition}
If $H, L \in \mathscr{C}$ we say that:
\begin{itemize}
\item $H$ is a \emph{descendant}\index{Descendant (in a Whitney decomposition)} of $L$ (and $L$ is an \emph{ancestor}\index{Ancestor (in a Whitney decomposition)} of $H$) if $H\subset L$;
\item $H$ is a \emph{son}\index{Son (in a Whitney decomposition)} of $L$ (and $L$ is the \emph{father}\index{Father (in a Whitney decomposition)} of $H$) if $H\subset L$ and $\ell (H) = \frac{1}{2} \ell (L)$;
\item $H$ and $L$ are \emph{neighbors}\index{Neighbors (in a Whitney decomposition)} if $\frac{1}{2} \ell (L) \leq \ell (H) \leq \ell (L)$ and $H\cap L \neq \emptyset$. 
\end{itemize}
\end{definition}

Note, in particular, the following elementary consequence of the subdivision of $\mathscr{C}$:

\begin{lemma}\label{l:son_father}
Let $H$ be a boundary cube. Then any ancestor $L$ and any neighbor $L$ with $\ell (L) = 2 \ell (H)$ is necessarily a boundary cube. In particular: the descendant of a non-boundary cube is a non-boundary cube.
\end{lemma}
\begin{proof} For the case of ancestors it suffices to prove that if $L$ is a father of a boundary cube $H$, then $L$ as well is a boundary cube, and
since the father of $H$ is a neighbor of $H$ with $\ell (L) = 2\ell (H)$, we only need to show the second part of the statement of the lemma. The latter is a simple consequence of the following chain of inequalities:
\begin{align}
\dist (c(L), \gammado) &\leq \dist (c(H), \gammado) + |c(H)-c(L)|\nonumber\\ 
&= \dist (c(H), \gammado) + 3 \sqrt{m} \ell (H)\nonumber\\ 
& < 64 r_H + 3 \frac{r_H}{M_0} \leq \left(64 + 3 M_0^{-1}\right) \frac{r_L}{2} \leq \frac{67}{2} r_L < 64 r_L\, . \nonumber\qedhere
\end{align}
\end{proof}

Moreover, we set the following:
\begin{itemize}
\item If $L\in\sC_j^\natural$, then $\bB_L$\index{aalb_l@$\bB_L$} is a  ball in $\R^{m+\overline n+l}$ with radius $64 r_L$ and center some chosen point $p_L\in\supp(T)$\index{aalp_l@$p_L$}  such that $\bp_{\pi_0}(p_L)=c(L)$ (note that such \(p_L\) is a priori not unique: we just make an arbitrary choice) and $\pi_L$ is a plane which minimizes the excess in $\bB_L$, namely $\bE (T, \bB_L)= \bE (T, \bB_L, \pi_L)$ and $\pi_L \subset T_{p_L} \Sigma$.
\item If $L\in \mathscr{C}^\flat$, then  $\bB^\flat_L$\index{aalb_l^\flat@ $\bB^\flat_L$} is the ball in $\R^{m+\overline n+l}$ with radius $2^7 64 r_L$ and center $p^\flat_L\in\gammaup$\index{aalp_l@$p_L^\flat$} such that $|\bp_{\pi_0}(p_L^\flat)-c(L)|=\dist(c(L),\gammado)$. Note that in this case the point $p^\flat_L$ is uniquely determined because $\gammaup$ is regular and $\bA$ is assumed to be sufficiently small. Likewise $\pi_L$\index{aagp_L@$\pi_L$} is a plane which minimizes the excess $\bE^\flat$, namely such that $\bE^\flat (T,\bB^\flat_L) = \bE (T, \bB^\flat_L, \pi_L)$ and $T_{p^\flat_L} \gammaup \subset \pi_L \subset T_{p^\flat_L} \Sigma$.
\end{itemize}

A simple corollary of Theorem \ref{thm:decay_and_uniq} and Corollary \ref{c:cone_cut} is the following lemma.

\begin{lemma}\label{lem:raffinamento}
Let $T, \Sigma$ and $\gammaup$ be as in Assumption \ref{ass:decay+cone}. Then there
is a positive dimensional constant $C (m,n)$ such that, if  the starting size of the Whitney decomposition is fine enough, namely if  $2^{N_0} \geq C (m,n) M_0$, then the balls $\bB^\flat_L$ and $\bB_L$ are all contained in $\bB_2$.

Moreover,  there exists $\varepsilon_1$ such that, for any choice of $M_0 , \alpha_\be>0$ and $\alpha_\bh<\frac 12$, if
\begin{equation}\label{e:eps_condition}
\bE^\flat(T,\bB_2)+ \|\Psi\|_{C^{3,a_0}}^2 + \|\psi\|_{C^{3,a_0}}^2<\varepsilon_1\,,
\end{equation}
then for every cube $L\in\sC^\flat$ we have 
\begin{align}
\bE^\flat(T, \bB^\flat_L) &\le C_0 \varepsilon_1 r_L^{2-2\alpha_\be}\, , \label{e:fa} \\
\bh(T, \bB_L^\flat,\pi_L)&\le C_0 \varepsilon_1^{\sfrac{1}{4}} r_L^{1+\alpha_\bh}\, , \label{e:fb}\\
|\pi_L - \pi_0| &\le C_0 \varepsilon_1^{\sfrac{1}{2}}\, ,\label{e:fc}\\
\abs{\pi_L - \pi(p_L^\flat)} &\le C_0  \varepsilon_1^{\sfrac{1}{2}} r_L^{1-a_\be} \label{e:fd}
\end{align} 
where, \(\pi(p_L^\flat)\) has been defined in (b) of Theorem  \ref{thm:decay_and_uniq} and $C_0$ depends only upon $\alpha_\be$, $\alpha_\bh$, $m$ and $n$. 
\end{lemma}
\begin{proof} The first part of the statement is just a direct inspection. Estimate \eqref{e:fa}  is a direct consequence of \eqref{e:decay-1}. Consider now $\pi (p^\flat_L)$ as in Theorem \ref{thm:decay_and_uniq}.
By the monotonicity formula we know that
\[
\|T\| (\bB^\flat_L) \geq \omega_m (2^7 64 r_L)^m\, 
\]
because we know that $\Theta (T, p^\flat_L) = Q-\frac{1}{2} \geq \frac{3}{2}$.
Moreover \eqref{e:decay-1} implies
\[
\bE (T, \bB^\flat_L, \pi_L) \leq \bE (T, \bB^\flat_L, \pi (p^\flat_L)) \leq C_0 {\color{red} \varepsilon_1} r_L^{2-2\alpha_\be}\, .
\]
Thus
\[
|\pi (p^\flat_L) - \pi_L|^2 \leq C_0 \big (\bE (T, \bB^\flat_L, \pi_L) + \bE (T, \bB^\flat_L, \pi (p^\flat_L))\big) \leq C_0 {\color{red} \varepsilon_1^{\sfrac{1}{2}}} r_L^{2-2\alpha_\be}\, .
\]
which proves \eqref{e:fd}. \eqref{e:fc} is now a direct consequence of \eqref{e:tilt_pi(q)} and  \eqref{e:fd} while \eqref{e:fb} is  direct consequence of
 \eqref{e:decay-height}. 
\end{proof}

\subsection{Decomposition and stopping conditions}
We will now defined a suitable refining procedure of our initial Whitney decomposition. To this end let \(C_\be, C_\bh\) be two positive constants that will be fixed later, see Assumption \ref{ass:cm} below.  We take a cube $L\in\sC_{N_0}$ and we {\em do not} subdivide it if it belongs to one of the following sets:
\begin{enumerate}
\item[(1)] $\sW_{N_0}^{\be}:=\{L\in\sC_{N_0}^\natural:\, \bE(T, \bB_L)> C_\be \varepsilon_1 \ell(L)^{2-\alpha_\be}\}$;
\item[(2)] $\sW_{N_0}^{\bh}:=\{L\in\sC_{N_0}^\natural:\, \bh(T, \bB_L,\pi_L)> C_\bh\varepsilon_1^{\sfrac{1}{2m}} \ell(L)^{1+\alpha_\bh}\}$.
\end{enumerate}
We then define 
\[
\sS_{N_0}:=\sC_{N_0}\setminus \left(\sW_{N_0}^{\be}\cup\sW_{N_0}^{\bh}\right)\,.
\]
The cubes in $\sS_{N_0}$ will be subdivided in their sons.
In fact we will ensure that $\sW_{N_0} := \sW_{N_0}^{\be}\cup\sW_{N_0}^{\bh} = \emptyset$ (and therefore $\mathscr{C}_{N_0} = \sS_{N_0}$) by choosing $C_\be$ and $C_\bh$ large enough, depending only upon $\alpha_\bh, \alpha_\be, M_0$ and $N_0$, see Proposition \ref{pr:tilting_cm} below. 

We next describe the refining procedure assuming inductively that for a certain step $j\geq N_0+1$ we have defined
the families $\sW_{j-1}$ and $\sS_{j-1}$. 
In particular we consider all the cubes $L$ in $\mathscr{C}_j$ which are contained in some element of $\sS_{j-1}$. Among them we select and set aside in the classes $\sW_j:=\sW_j^{\be}\cup\sW_j^{\bh}\cup\sW_j^{\bf n}$ those cubes where the following stopping criteria are met:
\begin{enumerate}
\item[(1)] $\sW_j^{\be}:=\{L \mbox{ son of } K \in\sS_{j-1}^\natural:\, \bE(T, \bB_L)> C_\be\varepsilon_1 \ell(L)^{2-{\color{red} 2} \alpha_\be}\}$\index{aalw\sW_j^{\be}@$\sW_j^{\be}$};
\item[(2)] $\sW_j^{\bh}:=\{L \mbox{ son of } K\in \sS_{j-1}^\natural:\, L\not\in \sW_j^\be\mbox{ and }$\\
 $\phantom{.}\qquad\qquad\bh(T, \bB_L,\pi_L)> C_\bh\varepsilon_1^{\sfrac{1}{2m}} \ell(L)^{1+\alpha_\bh}\}$\index{aalw\sW_j^{\bh}@$\sW_j^{\bh}$};
\item[(3)] $\sW_j^{\bf n}:=\{L\mbox{ son of } K\in \sS_{j-1}:\, L\not\in \sW_j^\be\cup \sW_j^\bh \mbox{ but }$\\
$\phantom{.}\qquad\qquad\exists L'\in \sW_{j-1} \text{ with }L\cap L' \neq \emptyset\}$\index{aalw\sW_j^{\bf n}@$\sW_j^{\bf n}$}.
\end{enumerate}
Note, in particular, that the refinement of boundary cubes can {\em never} be stopped because of the conditions (1) and (2). Indeed we could have included analogous stopping conditions for boundary cubes as well, but Lemma \ref{lem:raffinamento} would have implied in any case that these conditions would never stop the refining of boundary cubes. In principle a boundary cube might still be stopped because of the third condition, but we will see in Lemma \ref{cor:no_stop_b} that this possibility can be excluded as well. Thus boundary cubes always belong to $\sS$. Clearly, descendants of boundary cubes might become non-boundary cubes and so their refining can be stopped. 

We finally set $\sW_j := \sW_j^{\be}\cup\sW_j^{\bh}\cup\sW_j^{\bf n}$\index{aalw\sW_j@$\sW_j$} and
we keep refining the decomposition in the set \index{aals\sS_j@$\sS_j$} 
\[
\sS_j:=\left\{ L\in \mathscr{C}_j \mbox{ son of } K\in \sS_{j-1}\right\} \setminus \sW_j \, .
\] 
Observe that it might happen that the son of a cube in $\sS_{j-1}$ does not intersect $B_{3/2}^+$: in that case, according to our definition, the cube does not belong to $\sS_j$ neither to $\sW_j$: it is simply discarded. 

As already mentioned, we use the notation $\sS_j^\flat$\index{aals\sS_j^\flat@$\sS_j^\flat$}  and $\sS_j^\natural$\index{aals\sS_j^\natural@$\sS_j^\natural$}  respectively for $\sS_j\cap \sC^\flat$ and $\sS_j\cap \sC^\natural$.
Furthermore we set\index{aals\mathbf S^+@$\mathbf S^+$}
\begin{align*}
\sW & :=\bigcup_{j\ge N_0} \sW_j\\
\sS & :=\bigcup_{j\ge N_0} \sS_j\\
\mathbf S^+ & :=\bigcap_{j\ge N_0}\Big( \bigcup_{L\in \sS_j} L\Big)=B_{3/2}^+\setminus \bigcup_{H\in \sW} H\, .
\end{align*}
 We emphasize that $B_{3/2}^+$ includes $\gamma\cap B_{3/2}$.

\begin{lemma}\label{cor:no_stop_b}
$\mathscr{C}^\flat_j \cap \sW = \emptyset$ for every $j\geq N_0$ and in particular $\gammado \cap B_{3/2}^+ \subset \mathbf S^+$. 
\end{lemma}

\begin{proof}
Assume there is a boundary cube in $\sW$ and let $L$ be a boundary cube in $\sW$ with largest side length. The latter must then belong to  $\sW^{\bf n}_j$ for some $j$. However this would imply the existence of a neighbor $L' \in \sW$ with $\ell (L') = 2\ell (L)$: by Lemma \ref{l:son_father}
$L'$ would be a boundary cube in $\sW$, contradicting the maximality of $L$. 
\end{proof}

\subsection{Hierarchy of parameters}
From now on we specify a set of assumptions on the various choices of the constants involved in the construction.
\begin{ipotesi}\label{ass:cm}
$T, \Sigma$ and $\gammaup$ are as in Assumptions \ref{ass:decay+cone} and we also assume that
\begin{itemize}
\item[(a)] $\alpha_\bh$\index{aaga_h@$\alpha_\bh$} is smaller than $\frac{1}{2m}$ and $\alpha_\be$\index{aaga_e@$\alpha_\be$} is positive but small, depending only on $\alpha_{\bh}$,
\item[(b)] $M_0$\index{aaM_0@$M_0$} is larger than a suitable constant, depending only upon $\alpha_\be$,
\item[(c)] $2^{N_0} \geq C (m,n, M_0)$\index{aalN_0@$N_0$}, in particular it satisfies the condition of Lemma \ref{lem:raffinamento},
\item[(d)] $C_\be$\index{aalC_e@$C_\be$} is sufficiently large depending upon $\alpha_\be$, $\alpha_\bh$, $M_0$ and $N_0$,
\item[(e)] $C_\bh$\index{aalC_h@$C_\bh$} is sufficiently large depending upon $\alpha_\be, \alpha_\bh, M_0, N_0$ and $C_\be$,
\item[(f)] \eqref{e:eps_condition} holds with an $\varepsilon_1$\index{aage_1@$\varepsilon_1$} sufficiently small depending upon all
the other parameters.
\end{itemize}
Finally, there is an exponent $\sigmaexpcm$\index{aaga_L@$\sigmaexpcm$}, which depends only on $m,n,\bar n$ and $Q$ and which is independent of all the other parameters, in terms of which several important estimates in Theorem \ref{thm:cm_app} will be stated.
\end{ipotesi}

Note that the parameters are chosen following a precise hierarchy, in particular ensuring that there is a nonempty set of
parameters satisfying all the requirements. The hierarchy is consistent with that of \cite{DS4}, in particular the reader can
compare Assumption \ref{ass:cm} with \cite[Assumption 1.9]{DS4}.

\subsection{Interpolating functions}\label{s:interpolating} In this section we define the ``interpolating functions'' $g_L$ for each cube $L$. In particular, over the set $B_{3/2}^+\setminus \mathbf{S}^+$, the function $\boldsymbol{\varphi}^+$ is defined by patching together the $g_L$'s with a partition of unity subordinate to the cover $\sW$. Since however we need to define ${\bf \varphi}^+$ over $\mathbf{S}^+$ as well, we introduce all the necessary objects for {\em any} cube in $\sS\cup \sW$. 

\begin{proposition}\label{prop:yes_we_can}
If $T, \Sigma$ and $\gammaup$ are as in Assumptions \ref{ass:decay+cone} and if the various parameters $\alpha_\be,\alpha_\bh,M_0,N_0,C_\be,C_\bh,\varepsilon_1$ fulfill the Assumptions \ref{ass:cm} we have
\begin{align*}
\supp(T)\cap\bC_{36r_L}(p_L,\pi_L)& \subset\bB_L
\qquad
\mbox{when $L\in\sS^\natural_j\cup\sW_j$,}
\end{align*}
\begin{align*}
\supp(T)\cap\bC_{2^7 36r_L}(p^\flat_L,\pi_L)& \subset\bB_L^\flat \qquad \mbox{when $L\in\sS^\flat_j$,} 
\end{align*}
and the current $T$ satisfies the assumptions of Theorem  \ref{t:old_approx} in the cylinder $\bC_{36 r_L} (p_L, \pi_L)$, resp. the assumptions of Theorem \ref{thm:second_lip} in the cylinder $\bC_{2^736 r_L}(p^\flat_L,\pi_L)$. 
\end{proposition}

 We omit the proof here and in fact a strengthened version of the proposition is included in Proposition \ref{prop:yes_we_can_more}.
In each cube $L\in\sS_j^\flat$ (resp. $L\in\sS_j^\natural\cup\sW_j$) we define $(f_L^-, f_L^+)$ (resp. $f_L$) to be the Lipschitz approximation of $T$ in the cylinder $\bC_{2^7 9 r_L}(p^\flat_L,\pi_L)$ (resp.  $\bC_{9r_L}(p_L,\pi_L)$). Moreover we define the multifunctions $\bar f^\pm_L$ (respectively $\bar f_L$) by projecting the values of $f^\pm_L$ (resp. $f_L$) on the plane
$T_{p^\flat_L} \Sigma$ (resp. $T_{p_L} \Sigma$). More precisely, if we introduce the plane $\varkappa_L := \pi_L^\perp \cap T_{p^\flat_L} \Sigma$ (resp. $\varkappa_L := \pi_L^\perp \cap T_{p_L} \Sigma$), which is the orthogonal complement of $\pi_L$ in $T_{p^\flat_L} \Sigma$ (resp. in $T_{p_L} \Sigma$), the functions $f^\pm_L$ and $f_L$ are defined by
\[
\bar f^+_L = \sum_{i=1}^Q \a{\bp_{\varkappa_L} ((f^+_L)_i)}
\quad \bar f^-_L = \sum_{i=1}^{Q-1} \a{\bp_{\varkappa_L} ((f^-_L)_i)}
\]
\[
\mbox{and} \quad \bar f_L = \sum_{i=1}^Q \a{\bp_{\varkappa_L} ((f_L)_i)}\, .
\]
We can therefore regard each value $(f^\pm_L)_i (x)$ (resp. $(f_L)_i (x)$) as an element of the product space $\varkappa_L \times T^\perp_{p^\flat_L} \Sigma$ (resp.
$\varkappa_L \times T_{p_L}^\perp \Sigma$). Hence, if we let $\Psi_L:T_{p^\flat_L}\Sigma\to T^\perp_{p^\flat_L}\Sigma$  (resp. $\Psi_L:T_{p_L}\Sigma\to T^\perp_{p_L}\Sigma$) be the parametrization of the ambient manifold $\Sigma$ (in such a way that locally $\Sigma={\rm Graph}(\Psi_L)$), we have the identities
\begin{align*}
(f^\pm_L)_i (x) &= ((\bar{f}^\pm_L)_i (x), \Psi_L (x, (\bar f^\pm_L)_i (x)))\\
(f_L)_i (x) &= ((\bar f_L)_i (x), \Psi_L (x, (\bar f_L)_i (x)))\, .
\end{align*}
Although abusive, in order to make our notation less cumbersome 
we will then write 
$f_L^\pm=(\overline f_L^\pm,\Psi_L\circ\overline f_L^\pm)$ (resp. $f_L=(\overline f_L,\Psi_L\circ\overline f_L)$ and we will adopt the same convention for other maps with the same structure.

\begin{definition}\label{def:pi-approximation}
The maps $f_L^\pm$\index{aalf_L^\pm@$f_L^\pm$} and $f_L$\index{aalf_L@$f_L$} defined above will be called \emph{$\pi_L$-approximations}\index{p@$\pi_L$-approximation of $T$} of $T$ in the respective cylinders (indeed  $f_L^\pm$ approximates the current on
the ``half cylinder''  $\bp_{\pi_L}^{-1} (B_{2^7 9r_L}^\pm)$). 
\end{definition}

We next let $\overline h_L$\index{aalh_L@$\overline h_L$} be the solution of a suitable elliptic system (coming from the linearization of the mean curvature condition for minimal surfaces in $\Sigma$), subject to appropriate boundary conditions, which differ depending on whether $L$ is a non-boundary or a boundary cube. More precisely, for each cube, we introduce the constant matrix $\mathbf{L}$ as 
\begin{align}
\mathbf{L}^{ik} &= - \sum_j \Delta_x\Psi_L^j(p_L)\partial^2_{y_ix_k}\Psi_L^{j}(p_L)\qquad \mbox{if $L\in \mathscr{C}^\natural$}\label{e:matrix_L}\\
\mathbf{L}^{ik} &= - \sum_j \Delta_x\Psi_L^j(p^\flat_L)\partial^2_{y_ix_k}\Psi_L^{j}(p^\flat_L)\, \qquad
\mbox{if $L\in \mathscr{C}^\flat$.}
\end{align}
and we impose that
\begin{equation}\label{e:defining_PDE_i}
\left\{
\begin{array}{ll}
\Delta\overline h_L=\mathbf{L} \cdot (x-\bp_{\pi_L} (p_L)) & \\ \\
\overline h_L=\etaa\circ\overline f_L\qquad & \text{on }\partial B_{5r_L}(p_L,\pi_L)\, ,\
\end{array}
\right.
\end{equation}
when $L$ is a non-boundary cube and that
\begin{equation}\label{e:defining_PDE_b}
\left\{
\begin{array}{ll}
\Delta\overline h_L=\mathbf{L} \cdot (x-\bp_{\pi_L} (p^\flat_L)) & \\ \\
\overline h_L=\etaa\circ\overline f_L^+\qquad & \text{on }\partial  \big(B^+_{2^7 5r_L}(p^\flat_L,\pi_L) \big)\, ,
\end{array}
\right.
\end{equation}
when $L$ is a boundary cube.

\begin{definition}\label{def:tilted_interpolating}
The function\index{aalh_l@$h_L$}
\[
h_L:=(\overline h_L,\Psi_L\circ\overline h_L)\,
\]
will be called the \emph{tilted $L$-interpolating function}\index{tilted $L$-interpolating function}.
\end{definition}

We now are ready to define the final function, $g_L$, on our ``reference coordinate system''
(i.e. the domain of $g_L$ is contained in $\pi_0$ and its values are contained in $\pi_0^\perp$) with the property that its graph
coincides with (a suitable portion of) the graph of $h_L$. For this reason we need the following proposition ((cf. \cite[Appendix B]{DS4}).

\begin{proposition}\label{prop:yes_we_can_2}
Under the assumptions of Proposition \ref{prop:yes_we_can}, for every $L$ as above the function $h_L$ is Lipschitz on $B^+_{ 2^7\cdot 9r_L/2} (p^\flat_L, \pi_L)$ (resp. $B_{9r_L/2} (p_L, \pi_L)$) and we can define a function  $g_L:B^+_{2^7 4r_L}(p^\flat_L,\pi_0)\to \pi_0^\perp$ (resp.$g_L:B_{4r_L}(p_L,\pi_0)\to\pi_0^\perp $)  such that
\begin{align*}
&\bG_{g_L}=\bG_{h_L}\res B^+_{2^7 4r_L} (p^\flat_L, \pi_0) \times \R^{\bar n + l}\\
 \big(\mbox{resp. }\; &\bG_{g_L}=\bG_{h_L}\res \bC_{4r_L} (p_L, \pi_0)\big)\, .
\end{align*}
\end{proposition}

\begin{definition}\label{def:interpolating}
The function $g_L$ is called \emph{$L$-interpolating function}\index{L@$L$-interpolating function}.
\end{definition}

\subsection{Glued interpolations and center manifolds}\label{ss:glued_interpolations}
Let us define the Whitney cubes at the step $j$ as
\[
\mathscr P_j:=\sS_j\cup\bigcup_{i=N_0+1}^j \sW_i\,.
\]
Note that $\mathscr P_j$ is a ``Whitney family of dyadic cubes'' in the sense that if $K, L \in \mathscr{P}_j$ have non empty intersection, then $\frac{1}{2} \ell (L) \leq \ell (K) \leq 2 \ell (L)$. Consistently with the notation introduced in the previous section we let $\varkappa_0 := \pi_0^\perp \cap T_0 \Sigma$ be the orthogonal complement of $\pi_0$ in $T_0 \Sigma$. 
Recall then the map $\Psi:\pi_0\times \varkappa_0 = T_0 \Sigma\to T_0\Sigma^\perp$, which is the graphical parametrization of $\Sigma$ with respect to $T_0 \Sigma$. We fix a function $\vartheta\in C^\infty_c ([-\frac{17}{16}, \frac{17}{16}]^m, [0,1])$ which is identically $1$ on $[-1,1]^m$. For each cube $L$ we define further\index{aaghtilde_L@$\tilde{\vartheta}_L$}
\[  \tilde{\vartheta}_L(y)
 := \vartheta \left(\frac{y-c (L)}{\ell (L)}\right)\, .
\]
We obtain a partition of unity of $B_{3/2}^+$ by setting\index{aagh_L@$\vartheta_L$}
\[\vartheta_L(y) := \frac{\tilde{\vartheta}_L(y)}{\sum_{H \in \mathscr P_j} \tilde{\vartheta}_H(y)}\,.\] 
\begin{definition}\label{def:glued_interpolations}
We set \index{aagvbar_j@$\overline\varphi_j$}
\[
\overline\varphi_j:=\sum_{L\in\mathscr P_j}\vartheta_L  \overline g_L\, ,
\] 
and \index{aagv_j@$\varphi_j$}
\[
\varphi_j:=(\overline\varphi_j,\Psi\circ\overline\varphi_j)\,.
\]
The latter map is called the {\em glued interpolation at the step $j$.}\index{Glued interpolation} 
\end{definition}
We are now ready to state the main theorem regarding the construction of the {\em right center manifold}. 

\begin{theorem}\label{thm:center_manifold}\label{THM:CENTER_MANIFOLD}
If $T, \Sigma$ and $\gammaup$ are as in Assumptions \ref{ass:decay+cone} and $\alpha_\be,\alpha_\bh,M_0,N_0,C_\be,C_\bh,\varepsilon_1$ fulfill the Assumptions \ref{ass:cm}, then there is a $\kappa>0$, depending only upon $\alpha_\be$ and $\alpha_\bh$, such that
\begin{itemize}
\item[(a)] $\|\varphi_j\|_{3, \kappa, B_{3/2}^+} \leq C \varepsilon_1^{\sfrac{1}{2}}$, for some $C=C(\alpha_\be, \alpha_\bh, M_0, C_{\be}, C_{\bh})$;
\item[(b)] If \(i\le j\), $L\in \sW_{i-1}$ and $H$ is a cube concentric to $L$ with $\ell (H) = \frac{9}{8} \ell (L)$, then $\varphi_j = \varphi_i$ on $H$;
\item[(c)] $\varphi_j$ converges in $C^3$ to a map ${\bm \varphi}^+: B^+_{3/2}\to \R^n$\index{aagv{\bm\varphi}^\pm@${\bm\varphi}^+$ (resp. ${\bm\varphi}^-$)}, 
whose graph is a $C^{3,\kappa}$ submanifold $\mathcal M^+$ of $\Sigma$, which will be called {\em right center manifold}\index{Right (resp. left) center manifold};
\item[(d)] ${\bm \varphi}^+ = \psi$ on $\gammado\cap B_{3/2}$, namely $\partial \mathcal M^+ \cap \bC_{3/2} = \gammaup \cap \bC_{3/2}$;
\item[(e)] For any $q\in \partial \mathcal M^+ \cap \bC_{3/2}$, the tangent plane $T_q \mathcal M^+$ coincides with the plane $\pi (q)$ in Theorem \ref{thm:decay_and_uniq}.
\end{itemize}

\end{theorem}

The construction of $\mathcal M^+$ made in Theorem \ref{thm:center_manifold} is based on the decomposition of $B_{3/2}^+$. Under Assumption \ref{ass:cm}, the same construction can be made for $B_{3/2}^-$ and gives a $C^{3,\kappa}$ map ${\bm\varphi}^-: B^-_{3/2}\to\R^n$ which agrees with $\psi$ on $\gammado\cap B_{3/2}$. The graph of ${\bm\varphi}^-$ is a $C^{3,\kappa}$ submanifold $\mathcal M^-\subset\Sigma$, which will be called {\em left center manifold}. Clearly its boundary in the cylinder $\bC_{3/2}$, namely $\partial \mathcal M^- \cap \bC_{3/2}$, coincides, in a set-theoretical sense, with $\partial \mathcal M^+ \cap \bC_1$, but it has opposite orientation, and moreover its tangent plane $T_q\mathcal M^-$ coincides with $\pi(q)$ for every point $q\in\partial\mathcal{M}^-\cap\bC_{3/2}$. In particular, the union $\mathcal M:= \mathcal M^+\cup\mathcal M^-$ of the two submanifolds  is a $C^{1,1}$ submanifold of $\Sigma\cap \bC_{3/2}$ without boundary (in $\bC_{3/2}$), which will be called {\em center manifold}\index{Center manifold}. Moreover, we will often state properties of the center manifold related to cubes $L$ in one of the collections $\sW_j$ described above. Therefore, we will denote by $\sW^+$ the union of all $\sW_j$ and by $\sW^-$ the union of the corresponding classes of cubes which lead to the left center manifold $\mathcal{M}^-$. 

\begin{remark}
We emphasize again that so far we can only conclude the $C^{1,1}$ regularity of $\mathcal M$, because we do not know that the traces of the second derivatives of ${\bm\varphi}^+$ and ${\bm\varphi}^-$ coincide on $\gammado$.
\end{remark}

\begin{definition}
Let us define the graph parametrization map of $\mathcal M^+$\index{aagv{\bm\Phi}^+@${\bm\Phi}^+$} as ${\bm\Phi}^+(x):=(x,{\bm\varphi}^+(x))$. We will call {\em right  contact set}\index{Right contact set} the subset $\mathbf{K}^+ := {\bm\Phi}^+(\mathbf S^+)$\index{aalk\mathbf{K}^+@$\mathbf{K}^+$ (resp. $\mathbf{K}^-$)}. For every cube $L\in \sW^+$ we associate a {\em Whitney region} $\mathcal L$\index{aalL@$\mathcal L$} on $\mathcal M^+$ as follows: 
\begin{itemize}
\item $L:= {\bm \Phi}^+(H\cap B_1)$ where $H$ is the cube concentric to $L$ such that $\ell(H)=\frac{17}{16}\ell(L)$.
\end{itemize}
Analogously we define the map $\mathbf \Phi^-$, the contact set $\mathbf{K}^-$ and the Whitney regions on the left center manifold $\mathcal{M}^-$. 
\end{definition}

\section{The approximation on the normal bundle of $\mathcal{M}$}

In what follows we assume that Theorem \ref{thm:center_manifold} may be applied and we fix a corresponding center manifold $\mathcal M$, subdivided into its left and right portions. For any Borel set $\mathcal{V} \subset \mathcal{M}$ we denote by $|\mathcal{V}|$ its Hausdorff $m$-dimensional measure and we write $\int_{\mathcal{V}} f$ for $\int_{\mathcal{V}} f\, d\cH^m$. 

Since the two portions $\mathcal{M}^-$ and $\mathcal{M}^+$ are $C^{3, \kappa}$ and they join with $C^1$ regularity along $\gammaup$, in a sufficiently small normal neighborhood of $\mathcal{M}$ there is a well defined orthogonal projection $\mathbf{p}$ onto $\mathcal{M}$. The thickness of the neighborhood is inversely proportional to the size of the second derivatives of ${\bm \varphi}^\pm$ and hence, for $\varepsilon_1$ sufficiently small, we can assume it is $2$. Summarizing, in the rest of the section we make the following assumptions:

\begin{ipotesi}\label{ass:cm_app}
$T, \Sigma$ and $\gammaup$ are as in Assumption \ref{ass:decay+cone} and the various parameters $\alpha_\be,\alpha_\bh,M_0,N_0,C_\be,C_\bh,\varepsilon_1$ satisfy Assumption \ref{ass:cm}. In particular Theorem \ref{thm:center_manifold} applies and we let $\mathcal{M}$ be the union of the left and right center manifolds. $\varepsilon_1$ is sufficiently small so that, if
\begin{align}
\mathbf{U} &:= \{q\in \R^{m+n} :\,\exists! q' = \mathbf{p} (q) \in \mathcal{M} \mbox{ s.t. $|q-q'|<1$}\nonumber\\
&\qquad\qquad \qquad\qquad\qquad\qquad\qquad\qquad \mbox{and $q-q' \perp \mathcal{M}$}\}\, ,\label{e:tub_neigh}
\end{align}
then the map $\mathbf{p}$ extends to a  Lipschitz map to  the closure $\overline{\mathbf{U}}$  which is  $C^{2, \kappa}$ on \(\mathbf{U}\setminus \mathbf{p}^{-1} (\Gamma)\) and
\[
\mathbf{p}^{-1} (q') = q' + \overline{B_1 (0, (T_{q'} \mathcal{M})^\perp)} \qquad \mbox{for all $q'\in \overline{\mathcal{M}}$.}
\]
\end{ipotesi}

We then have the following as a consequence of the construction algorithm:

\begin{corollary}\label{c:cm}\label{C:CM}
Under Assumption \ref{ass:cm_app} the following holds: 
\begin{itemize}
\item[(a)] $\supp (\partial (T\res\mathbf{U}))\cap \bC_1 \subset \Gamma \cup \mathbf{p}^{-1} (\partial \mathcal{M})$, $\supp (T) \cap \bC_1 \subset \mathbf{U}$ and
\[
\mathbf{p}_\sharp (T\res \mathbf{U}) = (Q-1) \a{\mathcal{M}^-} + Q \a{\mathcal{M}^+}\, ;
\]
\item[(b)] $\supp (\langle T, \mathbf{p},  x\rangle) \subset \{y: | x-y|\leq C \varepsilon_1^{\sfrac{1}{2m}} \ell (L)^{1+\alpha_{\bh}}\}$ for a $C = C(\allin)$ and every $x\in \mathcal{L}$ Whitney region corresponding to $L\in \sW^+ \cup \sW^-$;
\item[(c)] $\langle T, \mathbf{p},q \rangle = Q \a{q}$ $\forall q\in \mathbf{K}^+\setminus \gammaup$ and 
$\langle T, \mathbf{p}, q\rangle = (Q-1) \a{q}$ $\forall q\in \mathbf{K}^-\setminus \gammaup$;
\item[(d)] $\mathbf{K}^+ \cap \mathbf{K}^- = \gammaup \cap \bC_{3/2}$ and $\supp (T \cap \mathbf{p}^{-1} (q)) = \{q\}$ for every $q\in \gammaup \cap \bC_{3/2}$. 
\end{itemize}
\end{corollary}

\subsection{Local estimates}
The center manifold is coupled with a map on $\mathcal{M}$ taking values in the normal bundle which approximates the current $T$ with very high accuracy.

\begin{definition}\label{d:app_pair}
Given a center manifold $\mathcal{M}$ as in Assumption \ref{ass:cm_app}, an $\mathcal{M}$-normal approximation of $T$ is given by a triple $(\mathcal{K}, F^+, F^-)$ such that
\begin{itemize}
\item[(A1)] $F^+: \mathcal{M}^+\cap \bC_1 \to \mathcal{A}_Q (\mathbf{U})$ and $F^-: \mathcal{M}^-\cap \bC_1 \to \mathcal{A}_{Q-1} (\mathbf{U})$\index{aalF^\pm@$F^\pm$} are Lipschitz and take the form $F^\pm (x) = \sum_i \a{x+N^\pm_i (x)}$ with $N^\pm_i (x) \perp T_x \mathcal{M^\pm}$ and $x+ N^\pm_i (x) \in \Sigma$ for every $i$ and every $x\in \mathcal{M^\pm}$;
\item[(A2)] $\mathcal{K} \subset \mathcal{M}$\index{aalK@$\mathcal{K}$} is closed and $\mathbf{T}_{F^\pm} \res \mathbf{p}^{-1} (\mathcal{K}\cap \mathcal{M}^\pm) = T \res \mathbf{p}^{-1} (\mathcal{K}\cap \mathcal{M}^\pm)$, where \(\mathbf T_{F^\pm}:=F^\pm_\sharp \a{\mathcal M}\), see \cite{DS2}\index{aalt\mathbf{T}_{F^+}@$\mathbf{T}_{F^+}$}\index{aalt\mathbf{T}_{F^-}@$\mathbf{T}_{F^-}$} ;
\item[(A3)] $\mathbf{K}^+ \cup \mathbf{K}^- \subset \mathcal{K}$\index{aalK^\pm@$\mathbf{K}^\pm$} and moreover $F^+ (x) =  Q \a{x}$ (resp. $F^- (x) = (Q-1) \a{x}$) on $\mathbf{K}^+$ (resp. $\mathbf{K}^-$).
\end{itemize}
\end{definition}

 Observe that the pairs $(F^+, F^-)$ and $(N^+, N^-)$\index{aalN^\pm@$N^\pm$} can be regarded as $\qhalf$-valued maps.
The following theorem, which is a consequence of the construction and of the estimates leading to Theorem \ref{thm:center_manifold}, ensures the existence of an $\mathcal{M}$-normal approximation which describes the current $T$ with a high degree of accuracy:

\begin{theorem}[Local estimates for the $\mathcal{M}$-normal approximation]\label{thm:cm_app}\label{THM:CM_APP}
Under Assumption \ref{ass:cm_app} there is a constant $\sigmaexpcm>0$ (depending only on $m, n, \overline{n}, Q$) such that
there is an $\mathcal{M}$-normal approximation $(\mathcal{K}, (F+, F^-))$ satisfying the following estimates on any Whitney region $\mathcal{L}\subset \mathcal{M}$ associated to a cube $L \in \sW^+\cup \sW^-$ (where to simplify the notation we use $N$ in place of $N^+$ and $N^-$):
\begin{align}
\Lip (N|_{\mathcal{L}}) & \leq C\varepsilon_1^{\sigmaexpcm} \ell (L)^\sigmaexpcm \label{e:cm_app1}\\
\|N_{\mathcal{L}}\|_0 &\leq C \varepsilon_1^{\sfrac{1}{2m}} \ell (L)^{1+\alpha_\bh}\label{e:cm_app2}\\
|\mathcal{L}\setminus \mathcal{K}| + \|\mathbf{T}_F -T\| (\mathbf{p}^{-1} (\mathcal{L})) & \leq C \varepsilon_1^{1+\sigmaexpcm} \ell(L)^{m+2+\sigmaexpcm}\label{e:cm_app3}\\
\int_{\mathcal{L}} |DN|^2 & \leq C \varepsilon_1 \ell (L)^{m+2-2\alpha_\be}\label{e:cm_app4}
\end{align}
for a constant $C = C (\allin)$.

Moreover, for any $a>0$ and any Borel $\mathcal{V} \subset \mathcal{L}$,
\begin{align}
\int_{\mathcal{V}} |\etaa\circ N| &\leq C \varepsilon_1 \left(\ell(L)^{m+3+\alpha_\bh/3} + a \ell (L)^{2+\sigmaexpcm/2} 
|\mathcal{V}|\right) \nonumber\\
&\quad+ \frac{C}{a} \int_{\mathcal{V}} \mathcal{G} (N, Q \a{\etaa\circ N})^{2+\sigmaexpcm}\, .\label{e:cm_app5}
\end{align}
\end{theorem}

\subsection{Separation and domains of influence} We next analyze suitable ``bounds from below'' induced by the stopping conditions in the center manifold construction. The next proposition shows that the current ``separates'' suitably on top of Whitney regions corresponding to cubes in $\sW^\bh$.

\begin{proposition}[Separation]\label{prop:sep}\label{PROP:SEP}
Under the assumptions of Theorem \ref{thm:cm_app} (recall, in particular, that $C_{\bh}\gg C_{\be}$), the following conclusions hold for every Whitney region $\mathcal{L}$ corresponding to a cube $L\in \sW^\bh \subset \sW^+$:
\begin{itemize}
\item[(S1)] $\Theta (T, p) \leq Q - \frac{1}{2}$ for every $p\in \bB_{16 r_L} (p_L)$;
\item[(S2)] $L\cap H =\emptyset$ for every $H\in \sW^{\mathbf{n}}$ with $\ell (H) \leq \frac{1}{2} \ell (L)$;
\item[(S3)] $\mathcal{G} (N^+ (x), Q \a{\etaa\circ  N^+ (x)}) \geq \frac{1}{4} C_{\bh} \varepsilon_1^{\sfrac{1}{2m}} \ell(L)^{1+\alpha_{\bh}}\ \forall x \in  \mathcal{M}^+ \cap \bC_{2\sqrt{m} \ell (L)} (p_L)$.  
\end{itemize}
For $L\in \sW^\bh \subset \sW^-$ the same conclusions, where in (S1) we replace $Q - \frac{1}{2}$ with $Q-\frac{3}{2}$.\footnote{Observe that, when $Q=2$, we actually draw the conclusion that no cube $L\subset \sW^-$ can belong to $\sW^\bh$: in fact when $Q=2$, we could use directly Allard's regularity theorem to prove that the ``left'' side of the current coincides with a single smooth classical graph over $B^-_{3/2}$. In order to make our work shorter we prefer however to treat the case $Q=2$ together with the general one $Q>2$.}
\end{proposition}

A simple corollary of the previous proposition is then the following

\begin{corollary}\label{c:domains}\label{C:DOMAINS}
Given any $H\in \sW^{\bf n}\subset \sW^+$ (resp. $\subset \sW^-$) there is a chain $L =L_0, L_1, \ldots, L_j = H$ such that:
\begin{itemize}
\item[(a)] $L_0\in \sW^\be\subset \sW^+$ (resp. $\subset \sW^-$) and $L_i\in \sW^{\bf n}\subset \sW^+$ (resp. $\sW^-$) for all $i>0$; 
\item[(b)] $L_i\cap L_{i-1}\neq\emptyset$ and $\ell (L_i) = \frac{1}{2} \ell (L_{i-1})$ for all $i>0$.
\end{itemize}
In particular,  $H\subset B_{3\sqrt{m}\ell (L_{0})} (x_{L_{0}}, \pi_0)$.
\end{corollary}

We use this last corollary to partition $\sW^{\bf n}$.

\begin{definition}[Domains of influence]\label{d:domains}
We first fix an ordering of the cubes in $\sW^\be\subset \sW^+$ (resp. $\subset \sW^-$) as $\{J_i\}_{i\in \N}$ so that their side lengths do not increase. Then $H\in \sW^{\bf n}$
belongs to $\sW^{\bf n} (J_0)$ (the domain of influence of $J_0$) if there is a chain as in Corollary \ref{c:domains} with $L_0 = J_0$.
Inductively, $\sW^{\bf n} (J_r)$ is the set of cubes $H\in \sW^{\bf n} \setminus \cup_{i<r} \sW^{\bf n} (J_i)$ for which there is
a chain as in Corollary \ref{c:domains} with $L_0 = J_r$.
\end{definition}

\subsection{Splitting before tilting} Next we show that even around cubes $L\in \sW^\be$ the sheets of the current ``open up'' in a suitable quantitative way. Again we bundle the estimates for the two maps $N^\pm$ in single statements using the letter $N$ to denote both of them.

\begin{proposition}[Splitting]\label{p:splitting}\label{P:SPLITTING}
Under the Assumptions of Theorem \ref{thm:cm_app}
the following holds. If $L\in \sW^\be\subset \sW^+$ (resp. $\subset\sW^-$), $q\in \pi_0$ with $\dist (L, q) \leq 4\sqrt{m} \,\ell (L)$ and $\Omega = \bC_{\ell (L)/4} (q)\cap \mathcal{M}$, then (with $C, C^* = C (\allin)$):
\begin{align}
&C_\be \varepsilon_1 \ell(L)^{m+2-2\alpha_\be} \leq \ell (L)^m \mathbf{E} (T, \bB_L) \leq C \int_\Omega |DN|^2\, ,\label{e:split_1}\\
&\int_{\mathcal{L}} |DN|^2 \leq C \ell (L)^m \mathbf{E} (T, \bB_L) \leq C^* \ell (L)^{-2} \int_\Omega |N|^2\, . \label{e:split_2}
\end{align}
\end{proposition}

\section{Estimates on tilting and optimal planes}

\begin{proposition}[Tilting and optimal planes]\label{pr:tilting_cm}
Under the Assumptions \ref{ass:decay+cone} and \ref{ass:cm} we have $\sW_{N_0} = \emptyset$.  Then the following estimates hold for any couple of neighbors $H,L\in \sS \cup \sW$ and for every $H, L \in \sS\cup \sW$ with $H$ descendant of $L$:
\begin{itemize}
\item[(a)]  denoting by $\pi_H, \pi_L$ the excess-minimizing planes in $\bB_H$ and $\bB_L$, respectively, 
\[
|\pi_H - \pi_L|\leq \bar C \eps_1^{\sfrac{1}{2}} \ell (L)^{1-\alpha_\be}\qquad |\pi_H-\pi_0| \leq \bar C\varepsilon_1^{\sfrac{1}{2}};
\]
\item[(b)$^\natural$] $\bh (T, \bC_{48 r_H} (p_H, \pi_0)) \leq C \varepsilon_1^{\sfrac{1}{2m}} \ell (H)$ and\\ 
$\supp (T)\cap \bC_{48r_H} (p_H, \pi_0) \subset \bB_H$ if $H\in \mathscr{C}^\natural$; 
\item[(b)$^\flat$] $\bh (T, \bC_{2^7 48 r_H} (p^\flat_H, \pi_0)) \leq C \varepsilon_1^{\sfrac{1}{4}} \ell (H)$ and\\ 
$\supp (T)\cap \bC_{2^7 48 r_H} (p^\flat_H, \pi_0) \subset \bB^\flat_H$ if $H\in \mathscr{C}^\flat$;
\item[(c)$^{\natural}$] $\bh (T, \bC_{36 r_L} (p_L, \pi_H))\leq C \varepsilon_1^{\sfrac{1}{2m}} \ell (L)^{1+\alpha_\bh}$ and\\ 
$\supp (T) \cap \bC_{36 r_L} (p_L, \pi_H) \subset \bB_L$ if $H, L \in \mathscr{C}^\natural$;
\item[(c)$^{\flat}$] $\bh (T, \bC_{2^7 36r_L} (p^\flat_L, \pi_H))\leq C \varepsilon_1^{\sfrac{1}{4}} \ell (L)^{1+\alpha_\bh}$\\ 
and $\supp (T) \cap \bC_{2^7 36 r_L} (p^\flat_L, \pi_H)) \subset \bB^\flat_L$ if $L \in \mathscr{C}^\flat$;
\end{itemize}
where $\bar C = \bar C (\alpha_{\be}, \alpha_{\bh}, M_0, N_0, C_\be)$ and $C = C (\alpha_{\be}, \alpha_{\bh}, M_0, N_0, C_\be, C_\bh)$.
\end{proposition}

\begin{proof} In this proof, constants denoted by $C$ will be assumed to depend on $m,n,Q$ and all the parameters
$\alpha_{\be}, \alpha_{\bh}, M_0, N_0, C_\be, C_\bh$, constants denoted by $\bar C$ will be assumed to depend on $m,n,Q, \alpha_{\be}, \alpha_{\bh}, M_0,$ 
$N_0, C_\be$ and constants denoted by $C_0$ will be assumed to depend only upon $m,n$ and $Q$. Constants depending on other subsets of the parameters above will be explicitly mentioned. We first show that $\sW_{N_0} = \emptyset$. We have already proved that $\sW$ does not contain boundary cubes in Lemma \ref{cor:no_stop_b}. Next, if $H\in \mathscr{C}_{N_0}^\natural$, $\bB_H \subset \bB_2$ by Lemma \ref{lem:raffinamento} and thus we can estimate
\begin{align}
\bE(T, \bB_H, \pi_0) \leq C (M_0, N_0) \bE(T, \bB_2, \pi_0) \leq C (M_0, N_0) \varepsilon_1\, .
\end{align}
Next, let $\pi$ be the projection of the plane $\pi_0$ in $T_{p_H} \Sigma$. Since $\pi_0 \subset T_0 \Sigma$, by the regularity assumption \eqref{e:eps_condition} on $\Sigma$,
\[
|\pi_0 - \pi|\leq C_0 \varepsilon_1^{\sfrac{1}{2}}\, .
\]
In particular, since by the monotonicity formula we can assume 
\[
\|T\| (\bB_H) \leq C_0(64 r_H)^m\, ,
\] 
we conclude
\[
\bE (T, \bB_H) \leq \bE (T, \bB_H, \pi) \leq C (M_0, N_0) \varepsilon_1 \leq C (M_0, N_0) \varepsilon_1 \ell (H)^{2-2\alpha_\be}\, .
\]
By our assumptions on the parameters, since $C_\be \geq C (M_0, N_0)$, we conclude that $L\not\in \sW^\be$. 

Next, notice that, since $p_H \in \supp (T)$, by the monotonicity formula we know
\begin{equation}\label{e:lower_bound_BH}
\|T\| (\bB_H) \geq \frac{1}{2} \omega_m (64 r_H)^m\, .
\end{equation}
Thus we can estimate
\begin{align*}
|\pi_H - \pi_0|^2 & \leq C_0 \bE (T, \bB_H) + C_0 \bE (T, \bB_H, \pi_0)\\
& \leq C_0 \varepsilon_1 + C (M_0, N_0) \bE (T, \bB_2, \pi_0)
\nonumber\\
& \leq C (M_0, N_0) \varepsilon_1\, .
\end{align*}
Hence,
\begin{align*}
\bh (T, \bB_H) & =  \bh (T, \bB_H, \pi_H)\\
& \leq C_0 |\pi_H - \pi_0| (r_H + \bh (T, \bB_H, \pi_0)) +  \bh (T, \bB_H, \pi_0)\\
& \leq  C (M_0, N_0) \varepsilon_1^{\sfrac{1}{2m}}\, .
\end{align*}
Since $C_{\bh}$ is assumed to be large enough compared to $M_0$ and $N_0$, we conclude that $H\not\in \sW^\bh$. 

\medskip

We next prove  (b)$^\flat$, (c)$^\flat$ and (a) when $H\in \mathscr{C}^\flat$. Since  the conclusions (b)$^\flat$ and (c)$^\flat$ are direct consequences of Corollary \ref{c:cone_cut} and (a), it will be enough to prove \((a)\) for $H\in \mathscr{C}^\flat$. To this end, note that  the second part of the statement is in Lemma \ref{lem:raffinamento}. 
We start with the first part of (a) in the case of $L$ is a boundary cube. 
In this is case the we can use  Lemma \ref{lem:raffinamento} and Theorem \ref{thm:decay_and_uniq} part (c) to conclude that 
\begin{align}\label{eq:(b) for boundary cubes}
	\abs{\pi_H - \pi_L}^2 &\le 3 \left( \abs{\pi_H - \pi(p^\flat_H) }^2+ \abs{\pi_L - \pi(p^\flat_L) }^2   + \abs{\pi(p^\flat_H) - \pi(p^\flat_L) }^2\right)\\
	&\le 3 C_0 \varepsilon_1 \ell(H)^{2-2\alpha_\be} + 3C_0 \varepsilon_1 \ell(L)^{2-2\alpha_\be} + 3 C_0\varepsilon_1 \ell(L)^{2-2\alpha_\be} \nonumber.
\end{align}
where we have also used that, by regularity of \(\gammaup\),    $\abs{p^\flat_H - p^\flat_L} \le C_0 \abs{c (H) - c (L)} \le C_0 \ell(L)$. Since   \(\ell(H) \le 2\ell(L)\) this proves (a) when \(L\in \mathscr{C}^\flat\).

It remains the case that $L$ is not a boundary cube. Since \(H\) is a boundary cube, Lemma \ref{l:son_father} implies that  $\frac{1}{2} \ell(H)\le \ell(L) \le \ell(H)$. In this case from Corollary \ref{c:cone_cut}, equation \eqref{eq:height-envelope2}, and the very definition of \(p^\flat_{H}\)  we deduce that
\begin{align}
& (1-C_0\varepsilon_1^{\frac12})\abs{p_{L}- p^\flat_{H}} \le \abs{\bp_{\pi_0}(p_{L}- p^\flat_{H})}\nonumber\\
 &\le \abs{c (L) - c (H)} + \abs{c (H)- \bp_{\pi_0}(p^\flat_{H})} \le 65 r_H.\label{eq:distance of centers-0}
\end{align}
Hence we conclude that $\bB_L\subset \bB_{H}^\flat$ and so arguing as above  
\[ 
\abs{\pi_{L} - \pi_{H}}^2 \le C_0\bE(T,\bB_{L}) + C_0 \bE^\flat(T,\bB^\flat_{H}).
\]
If $L \notin \sW^\be$ we conclude that $\abs{\pi_{L} - \pi_{H}} \le \overline{C}\varepsilon_1^{\frac12} \ell(H)^{1-\alpha_\be}$. Otherwise let $\pi$ be the projection of $\pi_{H}$ onto $T_{p_L}\Sigma$. By the regularity assumptions on $\Sigma$ and the estimate \eqref{eq:distance of centers-0} we have $\abs{\pi- \pi_{H}} \le C_0 \varepsilon_1^{\frac12}\ell(H)$ and so 
\begin{align*}
\bE(T,\bB_{L}) \le & \bE(T, \bB_{L}, \pi) \le C_0 \bE^\flat(T, \bB^\flat_{H}) + C_0 \abs{\pi- \pi_{H}}^2\\ 
\le & C_0 \varepsilon_1^{\frac12} \ell(H)^{2-2\alpha_\be}.
\end{align*}
Hence we conclude as well if $L \in \sW$ $\abs{\pi_{L} - \pi_{H}} \le \overline{C}\varepsilon_1^{\frac12} \ell(H)^{1-\alpha_\be}$, since \(\ell(H) \le  2\ell(L)\), this concludes the proof of (a) if \(H\) is a boundary cube.

\bigskip

Now we now turn to the proof of  (a),  (b)$^\natural$ and (c)$^\natural$.  To do so we first pick \(H\in \sC^\natural\) and we start by considering  a chain of ancestor-cubes $H=H_{j_0+1} \subset H_{j_0} \subset \dotsm \subset H_{\bar j}$ such that  $H_j$ is the father of $H_{j+1}$ and $H_{\bar j}$ is the first ancestor that is a boundary cube or $\bar j=N_0$. We want to show by  induction that
\begin{itemize}
\item[(i)\textsuperscript{j}] $\abs{\pi_{H_j}- \pi_{H_{j-1}}} \le \overline{C}_1 \varepsilon_1^{\frac12} \ell(H_j)^{1-\alpha_\be}$ and $\abs{\pi_{H_j}- \pi_0} \le \overline{C}_1 \varepsilon_1^{\frac12}$;
\item[(ii)\textsuperscript{j}] $\supp(T) \cap \bC_j \subset \bB_{H_j}$ and  $\bh(T, \bC_j, \pi_0)\le C_1 \varepsilon_1^{\frac{1}{2m}} \ell(H_j)$ with $\bC_j:=\bC_{48 r_j}(p_{H_j}, \pi_0)$; 
\end{itemize}
for suitable constants  $\overline C_1 = \overline C_1 (\alpha_{\be}, \alpha_{\bh}, M_0, N_0, C_\be)$\\ 
and $C _1= C_1 (\alpha_{\be}, \alpha_{\bh}, M_0, N_0, C_\be, C_\bh)$.

\medskip
\noindent
{\em Base Step, \(j=\bar j\)}:  If $H_{\bar j}=H_{N_0}$ we have shown already  that 
\[
\abs{\pi_{H_{N_0}}- \pi_0} \le C(M_0,N_0) \varepsilon_1^{\frac12}\ell(H_{N_0})^{1-\alpha_\be}
\] 
and $\supp(T)\cap \bC_{N_0} \subset \bB_{H_{N_0}}$. Hence we need to consider only the case in which $H_{\bar j}$ is a boundary cube. In this case  we argue as in \eqref{eq:distance of centers-0} to deduce
\begin{align}
& (1-C_0\varepsilon_1^{\frac12})\abs{p_{H_{\bar j+1}}- p^\flat_{H_{\bar j}}} \le \abs{\bp_{\pi_0}(p_{H_{\bar j+1}}- p^\flat_{H_{\bar j}})}\nonumber\\ 
\le & \abs{c_{H_{\bar j+1}} - c_{H_{\bar j}}} + \abs{c_{H_{\bar j}}- \bp_{\pi_0}(p^\flat_{H_{\bar j}})} \le 65 r_{H_{\bar j}}.\label{eq:distance of centers-1}
\end{align}
In particular this implies that $\bB_{H_{\bar j+1}} \subset \bB^\flat_{H_{\bar j}}$. Hence we have 
\[ 
\abs{\pi_{H_{\bar j+1}} - \pi_{H_{\bar j}}}^2 \le C_0\bE(T,\bB_{H_{\bar j+1}}) + C_0 \bE^\flat(T,\bB^\flat_{H_{\bar j}}).
\]
As before if $H_{\bar j+1} \in \sS_{\bar j+1}$ we directly conclude that 
\[
\abs{\pi_{H_{\bar j+1}} - \pi_{H_{\bar j }}} \le \overline{C} \varepsilon_1^{\frac12} \ell(H_{\bar j+1})^{1-\alpha_\be}.
\] 
Otherwise let $\pi$ be the projection of $\pi_{H_{\bar j}}$ onto the tangent space of \(\Sigma\) at \(p_{H_{\bar j+1}}\). By the regularity of $\Sigma$ and the estimate \eqref{eq:distance of centers-1} we have $\abs{\pi- \pi_{H_{\bar j+1}}} \le C(M_0) \varepsilon_1^{\frac12}\ell(H_{\bar j+1})$. Since 
\(\|T\|(\bB^\flat_{H_{\bar j}})\ge \omega_m r_{H_{\bar j}}^m/2\), 
\begin{align}
\bE(T,\bB_{H_{\bar j+1}}) &\le \bE(T, \bB_{H_{\bar j+1}}, \pi) \le C_0 \bE^\flat (T, \bB^\flat _{H_{\bar j}}) + C_0 \abs{\pi- \pi_{H_{\bar j+1}}}^2\nonumber\\ 
&\le C\varepsilon_1^{\frac12} \ell(H_{\bar j+1})^{2-2\alpha_\be}.\label{e:odio}
\end{align}
We conclude the first part of (i)\textsuperscript{l} for \(j=\bar j\),  while the second one follows from \eqref{e:tilt_pi(q)} and the estimate:
\[
|\pi (p_{H_{\bar j}}^\flat)-\pi_{H_{\bar j}}|\le C_0 \varepsilon_1 r_{H_{\bar j}}^{1-\alpha_\be}.
\]

\bigskip
\noindent
{\em Induction Step}: Let us assume the validity of  (i)\textsuperscript {j'}, (ii)\textsuperscript {j'}  for all  $\bar j\le j' \le j$, we want to show that $(i)\textsuperscript{j+1}, (ii)\textsuperscript{j+1}$ hold true. First note that   $p_{H_{j+1}} \in \bC_j$, and thus, by  (ii)\textsuperscript{j},
\begin{align} 
\abs{p_{H_{j+1}} - p_{H_j}}^2 &\le \abs{c (H_{j+1}) - c (H_j)}^2 + \abs{ \bp^\perp_{\pi_0}(p_{H_{j+1}} - p^\square_{H_j})}^2\nonumber\\
&\le  \left(\frac{9}{M_0^2} + 4C_1 \varepsilon_1 \right) r_{H_{j+1}}^2\, , \label{eq:distance of centers-2}
\end{align}
where \(\square=\flat \) or \(\square=\ \) depending on whether \(H_{l}\) is a boundary or a non-boundary cube.
Hence, provided \(M_0^{-1}\) and \(\varepsilon_1\) are sufficiently small,  $\bB_{H_{j+1}} \subset \bB^\square_{H_j}$. Thus
\[ 
\abs{\pi_{H_{j+1}} - \pi_{H_{j}}}^2 \le C_0\bE^\square(T,\bB^\square_{H_{j}}) + C_0 \bE(T,\bB_{H_{j+1}}).
 \]
Note now that $H_{j}\in \sS_{j}$ (since otherwise it would have not been subdivided to produce \(H_{j+1}\)), hence
\[
 \bE(T,\bB_{H_{j+1}})\le C_0 \bE^\square(T,\bB^\square_{H_{j}})\le C_0 C_{\be} \varepsilon_1 \ell(H_{j})^{2-2\alpha_\be}\le \overline{C} \varepsilon_1 \ell(H_{j})^{2-2\alpha_\be}
\] 
for a constant \(\overline C\) which depends only on  \(m,n,Q,\) and \(C_\be\). This proves   the first part of (i)\textsuperscript{j+1} if  we choose \(\overline C_1\ge  \overline C\). The second part follows from the first one and the inductive assumption via the estimate
\[ 
\abs{\pi_{H_{j+1}} - \pi_0} \le \sum_{j' = \bar j}^{j+1} \abs{\pi_{H_{j'}}- \pi_{H_{j'-1}}} \le \overline C_1 \varepsilon_1^{\frac12}\sum_{j'=\bar j+1}^{j+1} 2^{-(1-\alpha_\be)j'} \le \overline C_1 \varepsilon_1^{\frac12}.
\]
since we can choose \(N_0\)  big enough to ensure \[\sum^\infty_{j'=N_0}2^{-(1-\alpha_\be)j'}\le 1\, .\]

We now prove (ii)\textsuperscript{j+1}. The idea is to first use the inductive assumption (namely the height bound in $\bC_j$) in order to prove that $\supp (T)\cap \bC_{j+1} \subset \bB_{H_{j+1}}$ and hence to use the height bound in $\bB_{H_{j+1}}$ in order to conclude an height bound in $\bC_{j+1}$: in the second step it is crucial that the tilt $|\pi_{H_{j+1}} - \pi_0|$ has already been proved to be under control, cf. Figure \ref{f:induttiva}. Indeed, by (ii)\textsuperscript{j}   for all $x \in \supp(T) \cap \bC_{j+1}\subset \supp(T) \cap \bC_{j}$ we have 
\begin{align}
\abs{x-p_{H_{j+1}}}^2 &\le \left(48r_{H_{j+1}}\right)^2 + \bh (T, \bC_j, \pi_0)\nonumber\\ 
&\leq
 \left(48r_{H_{j+1}}\right)^2 + C_1 4 \varepsilon_1 \ell(H_{j+1})^2 \le (64 r_{H_{j+1}})^2. \label{eq:hight bad induction step} 
\end{align}
provided \(\varepsilon_1\) is small enough. This implies that $\supp(T)\cap \bC_{j+1} \subset \bB_{H_{j+1}}$ and thus the first part of  (ii)\textsuperscript{j+1}. We now  note that, if $H_{j+1} \in \sS_{j+1}$, then
\begin{align*}
\bh (T, \bC_{j+1}, \pi_0) &\leq C_0 r_{H_{j+1}} \abs{\pi_{H_{j+1}} - \pi_0} + \bh (T, \bB_{H_{j+1}}, \pi_{H_{j+1}})\\ 
&\leq C_1 \varepsilon_1^{\sfrac{1}{2m}} \ell (H_{j+1})\, .
\end{align*}
provided \(C_1\) is chosen big enough. If instead  $H_{j+1} \notin \sS_{j+1}$ (which can just happen for \(j=j_0\)) we   just observe that $\bC_{j+1} \subset \bC_j$  and that \(H_{j}\in \sS_j\) (otherwise it would have not been subdivided) and thus, by choosing   \(C_1\) possibly bigger, 
\begin{align*}
\bh(T, \bC_{j+1}, \pi_0) &\le \bh(T, \bC_j, \pi_0)\le C_0 r_{H_{j}} \abs{\pi_{H_{j}} - \pi_0} + \bh (T, \bB^\square_{H_{j}}, \pi_{H_{j}})
\\
&\le  C_0 r_{H_{j+1}} \abs{\pi_{H_{j+1}} - \pi_0} + C_\bh  \varepsilon_1^{\sfrac{1}{2m}} \ell(H_j)^{1+\alpha_\bh}\\
&\le C_1 \varepsilon_1^{\sfrac{1}{2m}} \ell(H_{j+1})
\end{align*}
This complete the proof of (ii)\textsuperscript{j+1} and of the claim.  Note in particular that (ii)\textsuperscript{j+1} implies (b)$^\natural$. 

\begin{figure}[htbp]
\begin{center}\label{f:induttiva}
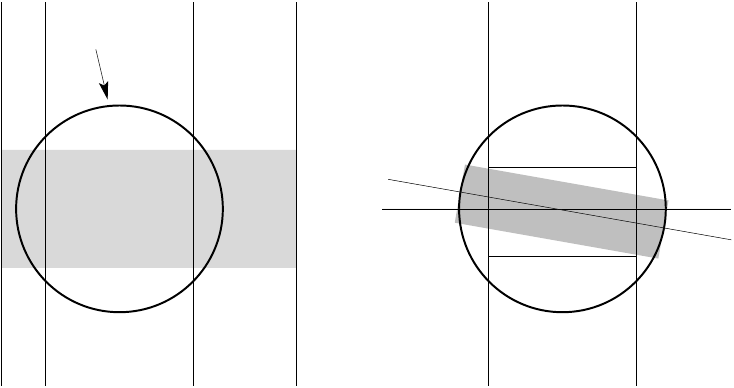
\end{center}
\caption{The inductive proof of (ii)\textsuperscript{j+1} consists of two steps: first the height bound in the cylinder $\bC_{j}$ is used to prove that $\supp (T) \cap \bC_{j+1} \subset \bB_{H_{j+1}}$; then the height bound in $\bB_{H_{j+1}}$ is used to prove the height bound in the cylinder $\bC_{j+1}$.}
\end{figure}

Let us now prove (a), and  (c)\(^\natural\). For (a),  let  $L$  be an ancestor of \(H\), then  either $L=H_{i}$ for some $i\le \bar j$ or $L$ is a boundary cube with $H_{\bar j}\subset L$. In the first case the we use (i)\textsuperscript{j} to deduce that 
\[ 
\begin{split}
\abs{\pi_{H} - \pi_{L}}&=\abs{\pi_{H_{j_0+1}} - \pi_{H_i}}  \le \sum_{j= i+1}^{j_0+1} \abs{\pi_{H_{j}}- \pi_{H_{j-1}}}
\\
 &\le C \varepsilon_1^{\frac12}\ell(H_{i})^{1-\alpha_\be} \sum_{j=1}^{j_0-i} 2^{-(1-\alpha_\be)l'} \le C\varepsilon_1^{\frac12} \ell(H_{j})^{1-\alpha_\be}.
\end{split}
\]
In the second case we use the triangle inequality and (a) for boundary cubes (which has already been shown) to deduce 
\begin{align*}
\abs{\pi_{H}- \pi_{L}} &\le \abs{\pi_{H}- \pi_{H_{\bar j}}}+ \abs{\pi_{H_{\bar j}}- \pi_{H_{L}}}\nonumber\\ 
&\le C\varepsilon_1^{\frac12}\ell(H_{\bar j})^{1-\alpha_\be} + C\varepsilon_1^{\frac12} \ell(L)^{1-\alpha_\be}\le C\varepsilon_1^{\frac12} \ell(L)^{1-\alpha_\be} 
\end{align*}
It remains to prove the second part of (a) in the case that $L,H$ are neighbors and both are non-boundary cubes. Let $M$ be the father of $L$ and we may assume that $\ell(H) \le \ell(L)=\frac12 \ell(M)$. Since $\abs{c (H)-c (M)} \le 3\sqrt{m} \ell(L)$ we have that $p_H \in \bC_{32r_M}(p_M, \pi_0)\cap \supp(T)$ or $p_H \in \bC_{2^7 32 r_M}(p_M^\flat, \pi_0) \cap \supp(T)$ if $M$ is a boundary cube. In both cases, by (b),  $\bB_H \subset \bB_M$ (or $\bB_H \subset \bB^\flat_M$), hence 
\[ 
\abs{\pi_H - \pi_M} \le C \varepsilon_1^{\frac12} \ell(M)^{1-\alpha_\be}.
 \]
Since a symmetric argument holds for \(L\) we  obtain 
 \[
 \abs{\pi_H - \pi_L} \le \abs{\pi_H - \pi_M} + \abs{\pi_L - \pi_M} \le 4C \varepsilon_1^{\frac12} \ell(L)^{1-\alpha_\be}.
 \]
 and this concludes the proof of (a). To prove  (c)$^\natural$ we consider again the chain of ancestors $H=H_{j_0+1} \subset H_{j_0} \subset \dotsm \subset H_{\bar j}$  where \(  H_{\bar j}\) is either the first boundary cube in this chain or \(H_{\bar j}\in \sC_{N_0}\). Let us set  $\bC_j:=\bC_{48 r_{H_j}}(p^\square_{H_j}, \pi_0)$,  (c)$^\natural$ will  follow if we  show  that for all $j \ge \bar j$ 
\begin{equation}\label{eq:inclusion of cylinders} 
\supp(T) \cap  \bC_{36 r_{H_j}}(p^\square_{H_j}, \pi_{H_j}) \subset \supp(T) \cap \bC_j
\end{equation}
(note that the possibility $\square = \flat$ can only occur for $j = \bar j$). 
Indeed the inclusion \(\supp(T) \cap  \bC_{36 r_{H_j}}(p^\square_{H_j}, \pi_H)\subset \bB^\square_L\) will then follow from (b), the arguments in the last step and simple geometric considerations. Moreover, assuming \eqref{eq:inclusion of cylinders} and using (a) we will have 
\[ 
\begin{split}
\bh(T,\bC_{36 r_{H_j}}(p^\square_{H_j}, \pi_{H})) &\le \bh(T, \bC_j, \pi_H) \le \bh(T, \bB^\square_{H_j}, \pi_H) 
\\
&\le \bh(T, \bB^\square_{H_j}, \pi_{H_j}) + C \abs{\pi_H- \pi_{H_j}}  r_{H_j} 
\\
&\le C_\bh \varepsilon_1^{\frac{1}{2m}} \ell(H_j)^{1+\alpha_\bh }+C \varepsilon_1 \ell(H_j)^{2-\alpha_{\be}},
\end{split}
\]
from which we easily conclude. 

We are thus left to show \eqref{eq:inclusion of cylinders}.  First,  note  that from \eqref{eq:distance of centers-2} and (a)  for $j \ge  \bar j$
\[ 
\begin{split}
\abs{\bp_{\pi_H}( p_{H_{j+1}}- p^\square_{H_j})} &\le \abs{\bp_{\pi_0}( p_{H_{j+1}}- p^\square_{H_j})} + C\abs{\pi_0-\pi_H}\abs{ p_{H_{j+1}}- p^\square_{H_j}} 
\\
&\le (3\sqrt{m} + C\varepsilon_1^{\frac12}) \ell(H_j)
\end{split}
 \]
 (recall that $H_{j+1}$ is a non-boundary cube by assumption).
Hence, by choosing first \(M_0\) large and then \(\varepsilon_1\) small,  we always have 
\begin{equation}\label{e:mausiamolo}
\bC_{36 r_{H_{j+1}}}(p_{H_{j+1}}, \pi_H)\subset \bC_{36 r_{H_j}}(p^\square_{H_j}, \pi_H).
\end{equation}
Now, if $H_{\bar j} = H_{N_0}$ we deduce from $\abs{\pi_H - \pi_{H_{N_0}}} \le C \varepsilon_1^{\frac12}$ that 
\[
\bC_{36 r_{H_{N_0}}}(p_{H_{N_0}}, \pi_H)\subset \bC_{N_0}
\] 
if $\varepsilon_1$ is sufficient small.  If $H_{\bar j}$ is a boundary cube,  Corollary \ref{c:cone_cut} implies that $\bC_{2^7 36 r_{H_{\bar j}}}(p^\flat_{H_{\bar j}}, \pi_H) \subset \bC_{2^7 48 r_{H_{\bar j}}}(p^\flat_{H_{\bar j}}, \pi_0)$. Hence, in both cases, \eqref{eq:inclusion of cylinders} holds for \(j=\bar j\). Let us assume now that there exists a first index \(j'\ge \bar j +1\) such that \eqref{eq:inclusion of cylinders} fails. Then there is a point \(p\in \supp(T)\) such that 
\[
 p \in \supp(T) \cap  \bC_{36 r_{H_{j'}}}(p_{H_{j'}}, \pi_H) \setminus \bC_{j'}. 
 \]
 By a simple  geometric  argument and (a),  this implies that
  \[ 
  \abs{\bp_{\pi_0}^\perp( p - p_{H_{j'}} )} \ge \frac{36r_{H_{j'}}}{C \abs{\pi_0 - \pi_H}}\ge \frac{Cr_{H_{j'}}}{\varepsilon_1}.
  \]
  On the other hand, by the inclusion \eqref{e:mausiamolo},  the validity of   \eqref{eq:inclusion of cylinders}  at the step  \(j'-1\) and (b), we have
  \[
  \begin{split}
    \abs{\bp_{\pi_0}^\perp( p - p_{H_{j'}} )}&\le  \abs{\bp_{\pi_0}^\perp( p - p_{H_{j'-1}} )}+\abs{\bp_{\pi_0}^\perp( p_{H_{j'}} - p_{H_{j'-1}} )}
 \\
    &\le 2 \bh(T, \bC_{j'-1}, \pi_0)\le Cr_{H_{j'}}. 
  \end{split}
  \]
  Taking \(\varepsilon_1\) small enough the last two inequality are in contradiction, from which we deduce the validity of \eqref{eq:inclusion of cylinders} for \(j'\).
\end{proof}

In particular, a simple additional argument implies Proposition \ref{prop:yes_we_can}, in the following strengthened version:

\begin{proposition}\label{prop:yes_we_can_more}
Under the Assumptions \ref{ass:decay+cone} and \ref{ass:cm} the following holds for every
couple of neighbors $H, L \in \sS \cup \sW$ and any $H, L\in \sS \cup \sW$ with $H$ descendant of $L$:
\begin{align*}
&\supp(T)\cap\bC_{36r_L}(p_L,\pi_H) \subset\bB_L
&\mbox{when $L\in\mathscr{C}^\natural$,}\\
&\supp(T)\cap\bC_{2^7 36r_L}(p^\flat_L,\pi_H) \subset\bB_L^\flat  &\mbox{when $L\in\sC^\flat$,} 
\end{align*}
and the current $T$ satisfies the assumptions of Theorem \ref{t:old_approx} in the cylinder $\bC_{36 r_L} (p_L, \pi_H)$ (resp. of Theorem \ref{thm:second_lip} in $\bC_{2^7 36r_L}(p^\flat_L,\pi_H)$). 
\end{proposition}

\begin{proof} The first two claims have already been proved in the previous proposition.
We now wish to prove the applicability of Theorem \ref{t:old_approx} in $\bC_{36 r_L} (p_L, \pi_H)$, resp. of Theorem \ref{thm:second_lip} in $\bC_{2^7 36r_L}(p^\flat_L,\pi_H)$. In both cases let $\bC$  be the corresponding cylinder and $B$ their bases, namely $B_{36 r_L} (\bp_{\pi_H} (p_L), \pi_H)$ and $B_{2^7 36 r_L} (\bp_{\pi_H} (p^\flat_L), \pi_H)$.  We only have to show the following properties:
\begin{align}
& \bp_{\pi_H} (T\res \bC) = Q \a{B}\qquad &\mbox{if $L\in \mathscr{C}^\natural$}\\ 
& \bp_{\pi_H} (T \res \bC) = Q \a{B^+} + (Q-1) \a{B^-} \qquad &\mbox{if $L\in \mathscr{C}^\flat$}
\end{align}
where, in the second identity, we consider $B^+$ and $B^-$ as the regions of $B$ which are separated by $\bp_{\pi_H} (\gammaup)$.

We just show the argument for the second case, since the first one is entirely analogous and already contained in \cite{DS4}
(in fact  also the argument for the second case is just a modification of the one contained in \cite{DS4}).

Assume first that $L\not\in \mathscr{C}_{N_0}$, let $M$ be the father of $L$ and let $\bC' = \bC_{2^7 36 r_M} (p^\flat_M, \pi_0)$. Consider that, by case (c)$^\flat$ of the previous proposition, we clearly have $\supp (T)\cap \bC \subset \bC'$. Consider thus a continuous path of planes $[0,1] \ni t\mapsto \pi (t)$ such that $\pi (0) = \pi_0$, $\pi (1) = \pi_H$ and $|\pi (t) - \pi_0| \leq C \varepsilon_1^{\sfrac{1}{2}}$ and let $S:= T\res \bC'$, $\bC (t) := \bC_{2^7 36 r_L} (p^\flat_L, \pi (t))$ and $T (t) := \bp_{\pi (t)} (S\res \bC (t))$. Observe that, by the height bound on $\bC'$, if $\varepsilon_1$ is sufficiently small, then $\supp (\partial S) \cap \bC(t) \subset \gammaup$. In particular, if $B(t) = B_{2^7 36 r_L} (\bp_{\pi (t)} (p^\flat_L), \pi (t))$ and $B(t)^\pm$ are the corresponding regions in which $\bp_{\pi (t)}$ subdivides it, we must have
\[
T (t) = k(t) \a{B(t)^+} + (k (t)-1) \a{B(t)^-}\, 
\] 
for a suitable integer $k(t)$. However, by a simple continuity argument on $t\mapsto T(t)$, the map $t\mapsto k(t)$ must be as well continuous, that is constant. Since $k(0) = Q$, we thus must have $k(1)=Q$ as well. On the other hand $ T(1) = \bp_{\pi_H} (T \res \bC)$, thus implying the desired claim.

In case $L\in \mathscr{C}_{N_0}$ we use the same argument where we define $\bC'$ to be the cylinder $\bC_{2^7 72 r_L} (p^\flat_L, \pi_0)$.
\end{proof}

\section{Interpolating functions and linearized system} 

Consider now a pair $H, L \in \sS \cup \sW$ which are either neighbors or such that $H$ is a descendant of $L$. By Proposition \ref{prop:yes_we_can_more} we can consider corresponding maps $f_{HL}^+$ and $f_{HL}$ as in Section \ref{s:interpolating}, by applying Theorem \ref{thm:second_lip} and Theorem \ref{t:old_approx} in the cylinders
$\bC_{2^7 36r_L}(p^\flat_L,\pi_H)$ and $\bC_{36 r_L} (p_L, \pi_H)$, respectively. Hence we introduce the
corresponding maps $h_{HL} (x)= (\bar h_{HL}(x), \Psi_H(x,\bar h_{HL}(x)))$ where $\bar h_{HL}$ solves
\begin{equation}\label{e:defining_PDE_i_2}
\left\{
\begin{array}{ll}
\Delta\overline h_{HL}=\mathbf{L} \cdot (x-\bp_{\pi_H} (p_H)) & \\ \\
\overline h_{HL}=\etaa\circ\overline f_{HL}\qquad & \text{on }\partial B_{8r_L}(p_L,\pi_H)\, ,\
\end{array}
\right.
\end{equation}
if $H$ and $L$ are both nonboundary cubes,
\begin{equation}\label{e:defining_PDE_b_2}
\left\{
\begin{array}{ll}
\Delta\overline h_{HL}=\mathbf{L} \cdot (x-\bp_{\pi_H} (p_H)) & \\ \\
\overline h_{HL}=\etaa\circ\overline f_{HL}^+\qquad & \text{on }\partial B^+_{2^78}r_L(p^\flat_L,\pi_H)\, ,
\end{array}
\right.
\end{equation}
if $L$ is a boundary cube and $H$ is a non-boundary cube,

\begin{equation}\label{e:defining_PDE_i_3}
\left\{
\begin{array}{ll}
\Delta\overline h_{HL}=\mathbf{L} \cdot (x-\bp_{\pi_H} (p^\flat_H)) & \\ \\
\overline h_{HL}=\etaa\circ\overline f_{HL}\qquad & \text{on }\partial B_{8r_L}(p_L,\pi_H)\, ,\
\end{array}
\right.
\end{equation}
if $L$ is a nonboundary cube and $H$ is a boundary cube and finally
\begin{equation}\label{e:defining_PDE_b_3}
\left\{
\begin{array}{ll}
\Delta\overline h_{HL}=\mathbf{L} \cdot (x-\bp_{\pi_H} (p^\flat_H)) & \\ \\
\overline h_{HL}=\etaa\circ\overline f_{HL}^+\qquad & \text{on }\partial B^+_{2^7 8r_L}(p^\flat_L,\pi_H)\, ,
\end{array}
\right.
\end{equation}
if both $H$ and $L$ are boundary cubes.
The constant coefficient matrix $\mathbf{L}$ is given by
\begin{align}
\mathbf{L}^{ik} &= -  \sum_j \Delta_x\Psi_H^j(p_H)\partial^2_{y_ix_k}\Psi_H^{j}(p_H)\qquad \mbox{if $H\in \mathscr{C}^\natural$}\label{e:matrix_L_2}\\
\mathbf{L}^{ik} &= -  \sum_j \Delta_x\Psi_H^j(p^\flat H)\partial^2_{y_ix_k}\Psi_H^{j}(p^\flat_H)\, \qquad
\mbox{if $H\in \mathscr{C}^\flat$.}
\end{align}
Observe that the third case cannot happen when $H$ is a descendant of $L$ and thus it can only happen
when $H$ and $L$ are neighbors.

In order to simplify our discussion, in what follows we always use the convention that $\varkappa_H$ is the orthogonal complement in $T_{p_H} \Sigma$ (resp. $T_{p^\flat_H} \Sigma)$ of $\pi_H$. Moreover, for every map $u$ defined on a domain $\Omega\subset \pi_H$ and taking values in $\pi_H^\perp$, we denote by $\bar u$ its projection on $\varkappa_H$. In particular, if the graph of $u$ is contained in $\Sigma$, then
we have $ u = (\bar u, \Psi_H \circ \bar u)$. The same convention, given the obvious adjustments, is adopted for multivalued maps.  

The key estimate leading to the proof of Theorem \ref{thm:center_manifold} is contained in the following proposition.

\begin{proposition}\label{pr:elliptic_regularization}
Under the Assumptions \ref{ass:decay+cone} and \ref{ass:cm} the following estimates hold for every pair of cubes $H$ and $L$ which are either neighbors or such that $H$ is a descendant of $L$:

\begin{align}
&\int \left( D (\etaa\circ \bar f_{HL}): D\zeta + \zeta^t \cdot \mathbf{L} \cdot (\bp_{\pi_H} (x- p_H))\right)\nonumber\\ 
\leq & C \varepsilon_1 r_L^{m+1+\alpha_{\bh}} (r_L \|D \zeta\|_0 + \|\zeta\|_0)\label{e:weak_ell_1}\\
&\qquad\qquad \forall \zeta \in C_c^\infty (B_{8r_L} (p_L, \pi_H), \varkappa_H)\qquad \mbox{if $L, H\in \mathscr{C}^\natural$;}
\nonumber\\
&\int \left( D (\etaa\circ \bar f_{HL}): D\zeta + \zeta^t \cdot \mathbf{L} \cdot (\bp_{\pi_H} (x- p^\flat_H))\right)\nonumber\\ 
\leq& C \varepsilon_1 r_L^{m+1+\alpha_{\bh}} (r_L \|D \zeta\|_0 + \|\zeta\|_0)\label{e:weak_ell_4}\\
&\qquad\qquad \forall \zeta \in C_c^\infty (B_{8r_L} (p_L, \pi_H), \varkappa_H)\qquad \mbox{if $L\in \mathscr{C}^\natural$ and $H\in \mathscr{C}^\flat$;}
\nonumber\\
&\int \left( D (\etaa\circ \bar f_{HL}): D\zeta + \zeta^t \cdot \mathbf{L} \cdot (\bp_{\pi_H} (x- p_H))\right)\nonumber\\ 
\leq & C \varepsilon_1 r_L^{m+1+\alpha_{\bh}} (r_L \|D \zeta\|_0 + \|\zeta\|_0)\label{e:weak_ell_2}\\
&\qquad\qquad \forall \zeta \in C_c^\infty (B^+_{2^78 r_L} (p^\flat_L, \pi_H), \varkappa_H)\qquad \mbox{if $L\in \mathscr{C}^\flat$ and $H\in \mathscr{C}^\natural$;}\nonumber\\
&\int \left( D (\etaa\circ \bar f_{HL}): D\zeta + \zeta^t \cdot \mathbf{L} \cdot (\bp_{\pi_H} (x- p^\flat_H))\right)\nonumber\\ 
\leq & C \varepsilon_1 r_L^{m+1+\alpha_{\bh}} (r_L \|D \zeta\|_0 + \|\zeta\|_0)\label{e:weak_ell_3}\\
&\qquad\qquad \forall \zeta \in C_c^\infty (B^+_{2^78 r_L} (p^\flat_L, \pi_H), \varkappa_H)\qquad \mbox{if $L, H\in \mathscr{C}^\flat$.}\nonumber
\end{align}
Moreover,
\begin{align}
\|\bar h_{HL} - \etaa\circ \bar f_{HL}\|_{L^1 (B_{8r_L} (p_L, \pi_H))} & \leq C \varepsilon_1 r_L^{m+3+\alpha_{\bh}}
\;\;\mbox{if $L\in\mathscr{C}^\natural$;}\label{e:L1_interior}\\
\|\bar h_{HL} - \etaa\circ \bar f_{HL}\|_{L^1 (B^+_{2^7 8 r_L} (p^\flat_L, \pi_H))} & \leq C \varepsilon_1 r_L^{m+3+\alpha_{\bh}}
\;\;\mbox{if $L\in\mathscr{C}^\flat$;}\label{e:L1_boundary}\\
\|D \bar h_{HL}\|_{L^\infty (B_{7r_L} (p_L, \pi_H))} & \leq C \varepsilon_1^{\frac12} r_L^{1-\alpha_\be}
\;\;\mbox{if $L\in\mathscr{C}^\natural$;}\label{e:lip_interior}\\
\|D \bar h_{HL}\|_{L^\infty (B^+_{2^7 7 r_L} (p^\flat_L, \pi_H))} &\leq C \varepsilon_1^{\frac12} r_L^{1-\alpha_\be}
\;\;\mbox{if $L\in\mathscr{C}^\flat$.}
\label{e:lip_boundary}
\end{align}
\end{proposition}

\begin{proof} {\bf Proof of \eqref{e:weak_ell_1}, \eqref{e:weak_ell_2} and \eqref{e:weak_ell_3}.}
The argument follows that of \cite[Proposition 5.2]{DS4} with essentially no variations and we report it here for the
reader's convenience. 

In order to simplify our notation we let $p=p_H$ in the first and third cases and $p=p_H^\flat$ in the second and fourth ones and we write $\pi, \varkappa$ and $\varpi$ for the planes $\pi_H, \varkappa_H$ and $T_p \Sigma^\perp$. With a slight abuse of notation we denote by $\Psi$ the map $\Psi_H$, so that the graph of $\Psi: T_p \Sigma \to T_p \Sigma^\perp$ is $\Sigma$. Finally we use the coordinates $(x,y,z)\in \pi\times \varkappa\times \varpi$ to identify points in $\R^{m+\bar n + l} = \R^{m+n}$ and we set $f = f_{HL}$, $f^+= f^+_{HL}$, $r=r_L$.  To avoid cumbersome notation we use $\norm{ \cdot }_0$ for $\norm{ \cdot }_{C^0}$ and $\norm{ \cdot }_1$ for $\norm{ \cdot }_{C^1}$.

In all the cases  the identities are derived by testing the first variation condition $\delta T (\chi) =0$ for the vector field $\chi (x,y,z)= (0, \zeta (x), D_y \Psi (x,y) \cdot \zeta (x))$: in the first case the condition will be tested in the cylinder $\bC := \bC_{8r_L} (p_L, \pi_H)$, whereas in the second and third cases it will be tested in the domain $\bC^+ := B^+_{2^78 r_L} (p^\flat_L, \pi_H) \times \pi_H^\perp$. Note that in both cases the vector field $\chi$ vanishes at the boundaries of the respective domains, whereas the current $T$  has zero boundary in both $\bC$ and $\bC^+$. Finally, although $\chi$ does not have compact support, the currents $T\res \bC$ and $T\res \bC^+$ have both bounded support and thus we have $\delta (T\res \bC) (\chi) =0$, $\delta (T \res\bC^+) (\chi) =0$. Using the formula for the first variation and the estimates in the Theorem \ref{t:old_approx}, in the first case we conclude
\begin{align}
|\delta \bG_f (\chi)| & = |\delta (\bG_f - T \res \bC) (\chi) \leq
\|D\chi\|_0 \bM (T\res\bC - \bG_{f})\nonumber\\
&\leq C_0 \|D\chi\|_0 r^m (\bE (T, \bC, \pi_H)+r^2 \bA^2)^{1+\sigma}\\
&\leq C_0 \|D\chi\|_0 r^m (\varepsilon_1 r^{2-2\alpha_{\be}})^{1+\sigma}\, .
\end{align}
On the other hand $\|\chi\|_0\leq 2\|\zeta\|_0$ and $\|D\chi\|_0 \leq 2\|\zeta\|_0 + 2 \|D\zeta\|_0$, provided $\varepsilon_1$ is sufficiently small. Choosing $\alpha_\bh \leq \frac{\sigma}{2}$ and $\alpha_\be$ small enough so that $(2-2\alpha_\be) (1+\sigma) \geq 2+\frac{\sigma}{2}$, we conclude that
\begin{equation}\label{e:first_var_1}
|\delta \bG_f (\chi)| \leq C \varepsilon_1 r^{m+1+\alpha_{\bh}} (r \|D \zeta\|_0 + \|\zeta\|_0)\, .
\end{equation}
Using the same argument and the estimates in Theorem \ref{thm:second_lip}, we gain the same estimate for the second and third case. 

The remaining computations are the same for all the cases and we give them for case two and three. First we write  $f^+= \sum_i \a{f^+_i}$ and $\bar{f}^+ = \sum_i \a{\bar{f}^+_i}$.
$\gr (f^+) \subset \Sigma$ implies
$f^+ = \sum_i \a{(\bar{f}^+_i, \Psi (x, \bar{f}^+_i))}$.
From \cite[Theorem~4.1]{DS2} we can infer that
\begin{align}
&\delta \bG_{f^+} (\chi) = \nonumber\\
&\int_B \sum_i \Big( \underbrace{D_{xy}\Psi (x, \bar{f}^+_i)\cdot \zeta}_{(A)} + 
\underbrace{(D_{yy} \Psi (x, \bar{f}^+_i)\cdot D_x \bar{f}^+_i) \cdot \zeta}_{(B)}\nonumber\\ 
&\qquad\quad+ \underbrace{D_y \Psi (x, \bar{f}^+_i) \cdot D_x \zeta}_{(C)}\Big)
: \Big(\underbrace{D_x \Psi (x, \bar{f}^+_i)}_{(D)} + \underbrace{D_y \Psi (x, \bar{f}^+_i)\cdot D_x \bar{f}^+_i}_{(E)}\Big)\nonumber\\
&+ \int_B \sum_i D_x \zeta : D_x \bar f^+_i + {\rm Err}\, .\label{e:tayloraccio}
\end{align}
Recalling \cite[Theorem 4.1]{DS2},
the error term ${\rm Err}$ in 
\eqref{e:tayloraccio} satisfies the inequality
\begin{equation}\label{e:tayloraccio2}
|{\rm Err}|\leq C \int |D\chi| |Df^+|^3 \leq \|\chi\|_1 \int |Df|^3 \leq C \|\chi\|_1 \Lip (f^+) \int |Df^+|^2\, .
\end{equation}
Using now the estimates of Theorem \ref{thm:second_lip} and arguing as above we achieve
\begin{equation}\label{e:errore_buono}
|{\rm Err}|\leq \varepsilon_1 r^{m+1+\alpha_{\bh}} (r \|D \zeta\|_0 + \|\zeta\|_0)\, .
\end{equation}
The second integral in \eqref{e:tayloraccio} is obviously 
$Q \int_B D\zeta : D (\etaa \circ \bar f^+)$.
We therefore expand the product in the first integral and estimate
all terms separately. In order to simplify our computations we shift coordinates so that $p = (0,0,0)$.  Recall that this implies that $\abs{\bp_{\pi}(p_L)} \le C_0 \ell(L)$, or $\abs{\bp_{\pi}(p^\flat_L)} \le C_0 64r$ if $L$ is a boundary cube. 

In particular we have $\Psi (0,0) =0$ and $D\Psi (0,0) = 0$. Taking into account the bounds on $\bA$, we then can write the
Taylor expansion 
\[
D \Psi (x,y) = D_x D \Psi (0,0) \cdot x + D_y D\Psi (0,0) \cdot y + O \big(\eps_1^{\sfrac{1}{2}} (|x|^2 + |y|^2)\big)\, .
\] 
In particular we gather the following estimates:
\begin{align}
|D \Psi (x, \bar{f}^+_i)| &\leq C \eps_1^{\sfrac{1}{2}} r\notag\\
D\Psi (x, \bar{f}^+_i) &= D_x D\Psi (0,0)\cdot x + O \big(\varepsilon_1^{\sfrac{1}{2}} r^{1+\ba_h}\big),\notag\\
|D^2 \Psi (x, \bar{f}^+_i)| &\leq \eps_1^{\sfrac{1}{2}}\notag\\
D^2 \Psi (x, \bar{f}^+_i) &= D^2 \Psi (0,0) + O \big(\eps_1^{\sfrac{1}{2}} r\big)\, .\nonumber
\end{align}
We are now ready to compute the behavior of the summands in \eqref{e:tayloraccio}. First
\begin{align}
& \int \sum_i (A):(D) = \int \sum_i (D_{xy} \Psi (0,0) \cdot \zeta) : D_x \Psi (x, \bar{f}^+_i)\nonumber\\
&\qquad\qquad + O\Big( \eps_1r^2 \int |\zeta|\Big) \nonumber\\
= & Q \int \sum_i (D_{xy} \Psi (0,0)\cdot \zeta : D_{xx} \Psi (0,0)\cdot x + O \Big( \eps_1\, r^{1+\alpha_\bh} \int |\zeta|\Big)\, .\label{e:(AD)}
\end{align}
Next, we estimate
\begin{align}
\int \sum_i (A):(E) &  = O \Big(\eps_1 r^{1+\alpha_\bh} \int |\zeta|\Big),\label{e:(AE)}\\
\int \sum_i (B):\left((D)+(E)\right) &  = O\Big(\eps_1
 r^{1+\alpha_\bh} \int |\zeta|\Big),\label{e:(B(D+E))}\\
\int \sum_i (C):(E) &  = O \Big(\eps_1 r^{2+\alpha_\bh} \int |D\zeta|\Big).\label{e:(CE)}
\end{align}
Finally we compute 
\begin{align*}
& \int \sum_i (C):(D) = \int \sum_i ((D_{xy}\Psi (0,0) \cdot x) \cdot D_x \zeta) : D_x \Psi (x, \bar{f}^+_i)\\
&\qquad\qquad\qquad\qquad+ O \Big(\eps_1\,r^{2+\alpha_\bh} \int |D\zeta|\Big)\\
=  & \, Q \int \sum_i (D_{xy}\Psi (0,0) \cdot x) \cdot D_x \zeta) : (D_{xx} \Psi (0,0) \cdot x)\\ 
&\quad+ O \Big(\eps_1 \, r^{2+\alpha_\bh} \int |D\zeta|\Big)\, .
\end{align*}
Summarizing, the first integral in \eqref{e:tayloraccio} takes the following form:
\begin{align*}
& Q \int \sum_{i,j,k,s} \partial^2_{x_i y_j} \Psi^k (0,0) \zeta^j (x) \partial^2_{x_ix_s} \Psi^k (0,0) x_s\, dx\\
+ \, & Q \int \sum_{i,j,k,s,r} \partial^2_{x_i y_j} \Psi^k (0,0) x_i \partial_s \zeta^j (x) \partial^2_{x_rx_s} \Psi^k (0,0) x_r\, dx
+{\rm Err}\, ,
\end{align*}
where ${\rm Err}$ satisfies the estimate \eqref{e:errore_buono}. Integrating by parts the second term we achieve
\[
- Q \int \sum_{i,j} x_i \left(\sum_j \Delta_x \Psi^k (0,0) \partial^2_{x_i y_j} \Psi^k (0,0)\right) \zeta^j (x)\, dx + {\rm Err}\, , 
\]
which completes the proof of the claim.

\medskip

{\bf Proof of \eqref{e:L1_interior} and \eqref{e:L1_boundary}.} The estimate is the same in all cases: we denote by $\Omega$ the domain of the function $\bar h := \bar h_{HL}$ and observe that for the difference $u:=  \bar h - \etaa\circ \bar f$, resp. $u:= \bar h - \etaa\circ \bar f^+$, the function $u$ satisfies $u|_{\partial \Omega}=0$ and 
\[
\left| \int_\Omega Du : D\zeta\right|\leq C r^{m+1+\alpha_\bh} (\|\zeta\|_0 + r \|D\zeta\|_0) \qquad \forall \zeta\in W^{1,2}_0 (\Omega)
\]
(although the estimates in \eqref{e:weak_ell_1}, \eqref{e:weak_ell_2} and \eqref{e:weak_ell_3} were proved for $\zeta\in C^\infty_c (\Omega)$, a simple density argument extends it to the case above). Now, for every $v\in L^2$ consider the unique solution $\zeta :=P(v)\in W^{1,2}_0 (\Omega)$ of $\Delta \zeta = v$. We then have the estimates
\[
r^{-1} \|P (v)\|_0 + \|D ( P(v))\|_0 \leq r \|v\|_0\, .
\]
Therefore we can write
\begin{align*}
\|u\|_{L^1 (\Omega)} &= \sup_{v: \|v\|_0 \leq 1} \int_\Omega u \cdot v = \sup_{v: \|v\|_0 \leq 1} \int_\Omega u \cdot \Delta ( P(v))\\
&= \sup_{v: \|v\|_0 \leq 1} \left(-\int_\Omega Du: D ( P(v))\right)\\
& \leq C \eps_1 r^{m+1+\alpha_{\bh}} \sup_{v:\|v\|_0\leq 1} \left(\| P(v)\|_0 + r\|D(P(v))\|_0\right)\\
& \leq C \eps_1 r^{m+3+\alpha_{\bh}}\, .
\end{align*}

\medskip

{\bf Proof of \eqref{e:lip_interior}}.
We split $h$ as $v+w$, where
\begin{equation}
\left\{
\begin{array}{ll}
\Delta v = 0 &\text{in }B_{8r_L}(p_L,\pi_H) \\ \\
v= \etaa\circ \bar{f} &\text{on } \partial B_{8r_L}(p_L,\pi_H) \, 
\end{array}\right. 
\end{equation}
and
\begin{equation}
\left\{
\begin{array}{ll}
\Delta w =  \mathbf{L} \cdot x &\text{in }B_{8r_L}(p_L,\pi_H) \\ \\
w= 0 &\text{on } \partial B_{8r_L}(p_L,\pi_H)\,
\end{array}\right.
\end{equation}
The estimate \eqref{e:lip_interior} follows from the interior regularity for the Laplace equation. More precisely, for the harmonic part we have 
\begin{align*}
\|Dv\|^2_{L^\infty (B_{7r_L}(p_L))} &\leq C r_L^{-m}\int_{B_{8r_L}(p_L)} \abs{Dv}^2\\
& \le C r_L^{-m}\int_{B_{8r_L}(p_L)} \abs{D\, (\etaa \circ \bar{f})}^2 \le C \varepsilon_1 r_L^{2-2\alpha_\be}\, ,
\end{align*}
whereas for $w$ the estimate holds up to the boundary
\[
\|Dw\|_{L^\infty (B_{8r_L}(p_L))} \leq Cr_L \|\Delta w\|_\infty \leq C \varepsilon_1 r_L^2\, . 
\]
For later use let us note that in particular if $L \in \mathscr{C}_{N_0}^\natural$ we have (for some constant $C$ depending on $N_0$)
\begin{align*} 
\sum_{k=0}^4 \norm{D^kv}_{B_{7r_L}(p_L)} &\le C \norm{Dh}_{L^2(B_{8r_L}(p_L))} \le C \varepsilon^{\frac12}_1 \\ 
\sum_{k=0}^4 \norm{D^kw}_{B_{7r_L}(p_L)}
 & \le C \norm{\Delta w}_{C^2(B_{8r_L}(p_L))} \le C \varepsilon_1\,. 
\end{align*}
Therefore we conclude that, for any $L \in \mathscr{C}_{N_0}^\natural$, 
\begin{equation}\label{e:C^3,kappa estimate for N_0 interior} 
\norm{h_{HL}}_{C^{3,\kappa}(B_{7r_{L}}(p_L))} \le C \varepsilon_1^{\frac12}\,.
 \end{equation}

\medskip
\noindent
{\bf Proof of \eqref{e:lip_boundary}}.  Let  $L$ be  a boundary cube, we want to apply Schauder estimates to prove \eqref{e:lip_boundary}. To this aim we first observe that $\etaa\circ f$ coincides with the $C^{3,a_0}$ function whose graph describes $\Gamma$ on $\gamma = \bp_{\pi} (\gamma)$. For this reason we fix a $C^{3,a_0}$ extension of it to the whole domain $\Omega$. We will show below that, by our assumption on $\Gamma$, we can impose
$\|\phi\|_{3,a_0} \leq C \varepsilon_1^{\sfrac{1}{2}}$. As customary we write $\phi = (\bar \phi, \Psi (x, \bar \phi))$.

We then split $h$ as $v+w+\bar \phi$, where
\begin{equation}
\left\{
\begin{array}{ll}
\Delta v = 0 &\text{in }B^+_{2^7 8r_L}(p^\flat_L,\pi_H) \\ \\
v= \etaa\circ \bar{f} - \bar \phi &\text{on } \partial B^+_{2^7 8r_L}(p^\flat_L,\pi_H) \, 
\end{array}\right. 
\end{equation}
and
\begin{equation}
\left\{
\begin{array}{ll}
\Delta w =  \mathbf{L} \cdot x - \Delta \bar \phi &\text{in }B_{2^7\, 8r_L}(p^\flat_L,\pi_H) \\ \\
w= 0 &\text{on } \partial B^+_{2^7 8r_L}(p^\flat_L,\pi_H)\,.
\end{array}\right.
\end{equation}

\medskip

\emph{Step 1: Definition of $\phi$.} Recall that $\Gamma$ is a $C^{3, a_0}$ graph of a function $\psi_L$ over $\tau_1 := T_{p_L^\flat} \Gamma$ with $\|\psi_L\|_{3, a_0} \leq C \varepsilon_1^{\sfrac{1}{2}}$. Consider now that $|\pi - \pi_L^\flat| \leq C \varepsilon_1^{\sfrac{1}{2}} \ell (L)^{1-\alpha_{\be}} \leq C \varepsilon_1^{\sfrac{1}{2}}$ and hence, if we define $\tau := \bp_{\pi} (\tau_1)$, under the assumption that $\varepsilon_1$ is smaller than a geometric constant we conclude as well that$|\tau - \tau_1|\leq C \varepsilon_1^{\sfrac{1}{2}} \ell (L)^{1-\alpha_\be} \leq C \varepsilon_1^{\sfrac{1}{2}}$. We can now invoke Lemma \ref{l:rotazioni_semplici} below (namely \cite[Lemma B.1]{DS4}) to conclude that $\Gamma$ is the graph of a function $\psi$ over $\tau$ with $\|\psi\|_{3, a_0} \leq C \varepsilon_1^{\sfrac{1}{2}}$. Fix next a unit vector $e$ orthogonal to $\tau$. We can then write $\psi = \tilde \psi e + \tilde{\phi}$, where $\tilde{\phi} = \bp_{\pi^\perp} (\psi)$.
Since 
$\partial B^+_{2^7 8r_L}(p^\flat_L,\pi_H)\cap B_{2^7 8r_L}(p^\flat_L,\pi_H) \subset \bp_{\pi} (\Gamma)$, we infer that 
the graph of $\tilde{\psi}$ over a suitable subdomain of $\tau$ describes $\partial B^+_{2^7 8r_L}(p^\flat_L,\pi_H)\cap B_{2^7 8r_L}(p^\flat_L,\pi_H)$.

Next, for every $x\in \pi$ we let $x = v + t e$ with $v\in \tau$ and define
$\phi (x) = \tilde{\phi} (v)$. Clearly $\|\phi\|_{3, a_0} \leq C \varepsilon^{\sfrac{1}{2}}$. Moreover, when restricted to
$\partial B^+_{2^7 8r_L}(p^\flat_L,\pi_H)\cap B_{2^7 8r_L}(p^\flat_L,\pi_H)$ the graph of the function $\phi$ gives the portion of $\Gamma$ lying over it. Hence $\phi = \etaa \circ f$ over $\partial B^+_{2^7 8r_L}(p^\flat_L,\pi_H)\cap B_{2^7 8r_L}(p^\flat_L,\pi_H)$.
Note in addition that for every $q\in \bB^\flat_L$,
\begin{align*}
|T_q\Gamma - \tau| &\leq |T_q \Gamma - \tau_1| + |\tau_1 - \tau| = |T_q \Gamma - T_{p_L^\flat} \Gamma|\\
&\leq C \varepsilon+1^{\sfrac{1}{2}} |q-p_L^\flat| + C \varepsilon_1^{\sfrac{1}{2}} \ell (L)^{1-\alpha_\be}
\leq C \varepsilon_1^{\sfrac{1}{2}} \ell (L)^{1-\alpha_\be}\, .
\end{align*}
This estimate implies 
\[
\|D\phi\|_\infty \leq C \varepsilon^{\sfrac{1}{2}} \ell (L)^{1-\alpha_\be}\, .
\] 

\medskip
\noindent
\emph{Step 2: Schauder estimates.}
By interpolation
\[
[D\phi]_\alpha\leq C \norm{D\phi}_\infty^{1-\alpha} \norm{D^2\phi}_\infty^\alpha \le C \varepsilon^{\frac12} \ell(L)^{(1-\alpha_\be)(1-\alpha)}.
\]
Since $\frac{1}{m+1}\operatorname{div}( x \otimes x) =x$, we have
\[
\mathbf{L}x - \Delta \phi = \operatorname{div} \Bigg( \frac{1}{m+1} \mathbf{L} x \otimes x  - \nabla \phi\Bigg) = \operatorname{div}(F).
\]
 By classical Schauder theory for operators in divergence form and $0$-boundary conditions, we have
\[ [Dw]_{\alpha} \le C [F]_{\alpha} \le \left[ \frac{1}{m+1} \mathbf{L} x \otimes x  - \nabla \phi \right]_\alpha \le C \varepsilon_1^{\frac12} r_L^{(1-\alpha_\be)(1-\alpha)}. \]
 We moreover have the elementary estimate
\[
\|D w\|_{L^2} \leq C \|F\|_{L^2}\, ,
\]
which follows from multiplying the equation by $w$ and integrating by parts.
Hence we conclude 
\[
\norm{Dw}_{\infty} \le C r^\alpha [Dw]_\alpha \le C \varepsilon^{\frac12} r_L^{1-\alpha_\be}.
\]
It remains to estimate the harmonic part $\norm{Dv}_\infty$. Since  \(v=0\) on \( \partial B^+_{2^7 8r_L}(p^\flat_L,\pi_H) \cap  B_{2^7 8r_L}(p^\flat_L,\pi_H)$ we can use a classical estimate on harmonic functions vanishing on a smooth boundary to deduce that
\begin{align*}
&\norm{Dv}^2_{C^0 (B^+_{2^7 7 r_L} (p^\flat_L, \pi_H))} \le Cr^{-m} \int_{B^+_{2^7 8 r_L}(p^\flat_L, \pi_H)}\abs{Dv}^2\\
\le & Cr^{-m} \int_{B^+_{2^7 8 r_L}(p^\flat_L, \pi_H)}\abs{D(\etaa\circ \bar{f} - \phi)}^2 \le C \varepsilon_1 r_L^{2-2\alpha_\be}.
\end{align*}
Combining all estimates give \eqref{e:lip_boundary}. As in the interior situation let us remark that for $L \in \mathscr{C}_{N_0}^\flat$ there is a constant depending on $N_0$ such that for \(\kappa \le a_0\)
\begin{equation*} 
[D^3v]_{\kappa, B'} +\sum_{k=0}^3 \norm{D^kv}_{C^0 (B')} \le C \|\etaa\circ \bar{f}\|_{C^0} +\norm{\phi}_{C^0} \le C \varepsilon^{\frac12}_1
\end{equation*}
and
\begin{equation*}
 \quad [D^3w]_{\kappa, B'} + \sum_{k=0}^3 \norm{D^kw}_{C^0 (B')} \le C \norm{\Delta w}_{C^{1,\kappa}} \le C \varepsilon^{\frac12}_1\, , \end{equation*}
 where $B' = C^{3,\kappa}(B^+_{2^7 7r_{L}})$. 
Therefore
\begin{equation}\label{e:C^3,kappa estimate for N_0 boundary} 
\norm{h_{HL}}_{C^{3,\kappa}(B^+_{2^7 7r_{L}} (p^\flat_L, \pi_H))} \le C \varepsilon_1^{\frac12}\,. 
\end{equation}

\end{proof}

We end this section by recalling the following simple consequence of the regularity theory for harmonic functions vanishing at a sufficiently smooth portion of the boundary.
\begin{lemma}\label{l:boundary_reg_harm}
Let $r<1$ and consider any $m-1$ dimensional $C^{3,a_0}$ hypersurface $\gammado\subset \mathbb R^m$ which passes through the origin and is the graph of a $C^{3,a_0}$ function $\varphi$ with $\|\varphi\|_{C^{3,a_0}}\leq 1$. Let $B^+$ the subset of $B_1$ lying over $\gammado$. Then there is a constant $C(r, a_0, m)$ such that the following estimate holds for every harmonic function $h$ in $B^+$ which vanishes along $\gammado$:
\begin{equation}\label{e:boundary_reg_harm}
\|h\|_{C^{3, a_0}(B_r\cap B^+)} \leq C (r, a_0, m) \|h\|_{L^1 (B^+)}\,  .
\end{equation}
\end{lemma}

\section{Tilted $L^1$ estimate}

\begin{definition}
Four cubes $H,J,L,M\in \mathscr{C}$ make a \emph{distant relation} between $H$ and $L$ if $J,M$ are neighbors (possibly the same cube) with same side length and $H$ and $L$ are descendants respectively of $J$ and $M$. 
\end{definition}

\begin{lemma}[Tilted $L^1$ estimate]\label{lem:tilted_L1}
Under the Assumptions \ref{ass:decay+cone} and \ref{ass:cm} the following holds for every quadruple $H,J,L$ and $M$ in $\sS\cup \sW$ which makes a distant relation between $H$ and $L$.
\begin{itemize}
\item If $J\in \mathscr{C}^\natural$, then there is a map $\hat{h}_{LM} : B_{4r_J} (p_J, \pi_H) \to \pi_H^\perp$ such that
\[
\bG_{\hat h_{LM}} = \bG_{h_{LM}} \res \bC_{4r_J} (p_J, \pi_H)
\]
and
\begin{equation}\label{e:tilted_L1_int}
\|h_{HJ}- \hat h^\square_{LM}\|_{L^1 (B_{2r_J} (p_J, \pi_H))} \leq C \varepsilon_1 \ell (J)^{m+3+\alpha_{\bh}/2}\, ,
\end{equation}
where $\square = +$ or $\square =\ $ depending on whether $M$ is a boundary or a non-boundary cube.
\item If both $J$ and $M$ belong to $\mathscr{C}^\flat$, then there is a map $\hat{h}_{LM} : B^+_{2^7 4r_J} (p^\flat_J, \pi_H) \to \pi_H^\perp$ such that
\[
\bG_{\hat h_{LM}} = \bG_{h_{LM}} \res \bC_{2^7 4 r_J} (p^\flat_J, \pi_H)
\]
and
\begin{equation}\label{e:tilted_L1_b}
\|h^+_{HJ}- \hat h_{LM}\|_{L^1 (B^+_{2^7 2 r_J} (p^\flat_J, \pi_H))} \leq C \varepsilon_1 \ell (J)^{m+3+\alpha_{\bh}/2}\, .
\end{equation}
\end{itemize}
\end{lemma}

Before coming to the proof we recall the following two lemmas from \cite{DS4}.

\begin{lemma}[Lemma B.1 in \cite{DS4}]\label{l:rotazioni_semplici}
For any $m, n \in \mathbb N \setminus \{0\}$ there are constants $c_0,C_0>0$ with the following properties.
Assume that
\begin{itemize}
\item[(i)] $\varkappa, \varkappa_0\subset \mathbb R^{m+n}$ are $m$-dimensional planes with $|\varkappa-\varkappa_0|\leq c_0$ and $0<r\leq 1$;
\item[(ii)] $p=(q,u)\in\varkappa\times \varkappa^\perp$ and $f,g: B^m_{7r} (q, \varkappa)\to \varkappa^\perp$ are Lipschitz functions such that
\begin{equation*}
\Lip (f), \Lip (g) \leq c_0\quad\text{and}\quad |f(q)-u|+|g(q)-u|\leq c_0\, r.
\end{equation*}
\end{itemize}
Then there are two maps $f', g': B_{5r} (p, \varkappa_0)\to \varkappa_0^\perp$ such that
\begin{itemize}
\item[(a)] $\bG_{f'} = \bG_f \res \bC_{5r} (p, \varkappa_0)$ and $\bG_{g'} = \bG_g \res \bC_{5r} (p, \varkappa_0)$;
\item[(b)] $\|f'-g'\|_{L^1 (B_{5r} (p, \varkappa_0))}\leq C_0\,\|f-g\|_{L^1 (B_{7r} (p, \varkappa))}$;
\item[(c)] { if $f\in C^{3, \kappa} (B_{7r} (p, \varkappa))$
then $f'\in C^{3, \kappa} (B_{5r} (p, \varkappa_0))$ with the estimates
\begin{align}
&\|f'- u'\|_{C^0}\leq C \|f-u\|_{C^0} + C |\varkappa-\varkappa_0| r\label{e:ruota_uno}\\
&\|Df'\|_{C^0} \leq C \|Df\|_{C^0} + C |\varkappa-\varkappa_0|\label{e:uota_due}\\
&\|D^2 f'\|_{C^{1, \kappa}} \leq \Phi (|\varkappa-\varkappa_0|, \|D^2 f\|_{C^{1, \kappa}})\label{e:ruota_tre}
\end{align}
where $(q',u')\in \varkappa_0 \times \varkappa_0^\perp$ coincides with the point $(q,u)\in \varkappa\times \varkappa^\perp$ and $\Phi$ is a smooth function with $\Phi (\cdot, 0) \equiv 0$.} 
\end{itemize}
All the conclusions of the Lemma still hold if we replace the exterior radius $7r$ and interior radius $5r$ with
$\rho$ and $s$: the corresponding constants $c_0$ and $C_0$ (and the function $\Phi$) will then depend also on the ratio
$\frac{\rho}{s}$.
\end{lemma}

\begin{lemma}[Lemma 5.6 of \cite{DS4}]\label{l:cambio_tre_piani}
Fix $m,n,l$ and $Q$. There are geometric constants $c_0, C_0$ with the following property.
Consider two triples of planes $(\pi, \varkappa, \varpi)$ and $(\bar\pi, \bar\varkappa, \bar\varpi)$, where
\begin{itemize}
\item $\pi$ and $\bar\pi$ are $m$-dimensional;
\item $\varkappa$ and $\bar\varkappa$ are
$\bar{n}$-dimensional and orthogonal, respectively, 
to $\pi$ and $\bar\pi$;
\item $\varpi$ and $\bar\varpi$ $l$-dimensional and orthogonal, respectively, to $\pi\times \varkappa$ and $\bar\pi\times
\bar\varkappa$.
\end{itemize}
Assume ${\rm An} := |\pi-\bar\pi| + |\varkappa-\bar\varkappa|\leq c_0$ and
let $\Psi: \pi\times \varkappa \to \varpi$,
$\bar\Psi: \bar\pi\times \bar\varkappa \to \bar\varpi$ be two maps
whose graphs coincide and such that $|\bar \Psi (0)| \leq c_0 r$ and $\|D\bar \Psi\|_{C^{0}} \leq c_0$.
Let $u: B_{8r} (0, \bar{\pi}) \to \Iq (\bar\varkappa)$ be a map with
$\Lip (u) \leq c_0$ and $\|u\|_{C^0} \leq c_0 r$
and set $f (x)= \sum_i \llbracket(u_i (x), \bar\Psi (x, u_i (x)))\rrbracket$ and $\bef (x) = 
(\etaa \circ u (x), \bar\Psi (x, \etaa \circ u (x)))$. Then there are
\begin{itemize}
\item a map $\hat{u}: B_{4r} (0, \pi) \to \Iq (\varkappa)$ such
that the map 
\[
\hat{f} (x) := \sum_i \a{(\hat{u}_i (x), \Psi (x, \hat{u}_i (x)))}
\] 
satisfies $\bG_{\hat{f}}
= \bG_{f} \res \bC_{4r} (0, \pi)$
\item and a map $\hat\bef: B_{4r} (0, \pi) \to \varkappa \times \varpi$ such that
\[
\bG_{\hat\bef} = \bG_{\bef} \res \bC_{4r} (0, \pi)\, .
\] 
\end{itemize} 
Finally, if $\beg (x) := (\etaa \circ \hat{u} (x),
\Psi (x, \etaa \circ \hat{u} (x)))$, then
\begin{align}
\|\hat\bef-\beg\|_{L^1} &\leq C_0 \left(\| f \|_{C^0}+ r {\rm An}\right) \big(\D (f) + r^m \big(\|D\bar \Psi\|^2_{C^0} + {\rm An}^2\big)\big)\, .\label{e:che_fatica}
\end{align}
\end{lemma}

\begin{proof}[Proof of Lemma \ref{lem:tilted_L1}] We start by examining the first case. Using Proposition \ref{pr:elliptic_regularization} we know that
$\|\bar h_{HJ} - \etaa \circ \bar f_{HJ}\|_{L^1 (B_{8r_J} (p_J, \pi_H))} \leq C\eps_1 r_J^{m+3+\alpha_\bh}$. Now, since $\Psi_H$ is Lipschitz and
$h_{HJ} = (\bar h_{HJ}, \Psi (x, \bar h_{HJ}))$, $\bef_{HJ} = (\etaa\circ \bar f_{HJ}, \Psi_H (\etaa\circ \bar f_{HJ}))$, we easily conclude that
\begin{equation}\label{e:triangolo_1}
\|h_{HJ} - \bef_{HJ}\|_{L^1 (B_{8r_J} (p_J, \pi_H))} \leq C \eps_1 r_J^{m+3+\alpha_\bh}\, .
\end{equation}
Similarly, 
\[
\|h_{LM} - \bef_{LM}\|_{L^1 (B_{8r_M} (p_M, \pi_L))}\leq C \eps_1 r_M^{m+3+\alpha_\bh}
\leq C \eps_1 r_J^{m+3+\alpha_\bh}
\]
in case $M$ is a non-boundary cube or
\[
\|h^+_{LM} - \bef^+_{LM}\|_{L^1 (B_{2^7 8r_M} (p^\flat_M, \pi_L))}
\leq C \eps_1 r_J^{m+3+\alpha_\bh}
\]
if it is a boundary cube. 
Since the two situations are entirely analogous, we just focus on the case where $M$ is a non-boundary cube.

Now both $h_{LM}$ and $\bef_{LM}$ are Lipschitz (and well defined!) over $B_{6 r_J} (p_J, \pi_L)$ and recall that, due to Proposition \ref{pr:tilting_cm}, $|\bp_{\pi_L} (p_M - p_J)| \le 3\sqrt{m} \ell(M)$. Moreover they satisfy the assumption (ii) of Lemma \ref{l:rotazioni_semplici} by a simple Chebyshev argument on the $L^1$ estimate above. So we can apply Lemma \ref{l:rotazioni_semplici} to get a function $\hat\bef_{LM}$ the function such that
\[
\bG_{\hat \bef_{LM}} \res \bC_{4r_J} (p_J, \pi_H) = \bG_{\bef_{LM}} \res \bC_{4r_J} (p_J, \pi_H)\, ,
\]  
similarly for $h_{LM}$ and to conclude that
\begin{equation}\label{e:triangolo_2}
\|\hat h_{LM} - \hat\bef_{LM}\|_{L^1 (B_{4r_J} (p_J, \pi_H))}\leq C  \eps_1 r_J^{m+3+\alpha_\bh}\, .
\end{equation}
In order to simplify the notation, shift the center $p_J$ to the origin and 
consider next $\hat f_{LM}$, $\hat u$ and $\beg$ as in Lemma \ref{l:cambio_tre_piani} once we define $f=f_{LM}$, $\pi=\pi_H$ and $\bar \pi = \pi_L$. Now, the graphs of $\hat u$ and $\bar f_{HJ}$ coincides except for a set of  Lebesgue measure bounded by $C r_J^m (\varepsilon_1 r_J^{2-2\alpha_\be})^{1+\sigma}$ because of the Lipschitz approximation theorems. On the other hand the  oscillations  of both functions are bounded
by $C \varepsilon_1^{\sfrac{1}{2m}} r_J^{1+\alpha_\bh}$. It is thus easy to verify that 
\begin{equation}\label{e:triangolo_3}
\|\bef_{HJ} - \beg\|_{L^1 (B_{4r_J} (p_J, \pi_H))}\leq C  \eps_1 r_J^{m+3+\alpha_\bh}\, .
\end{equation}
We now claim that 
\begin{equation}\label{e:triangolo_4}
\|\hat \bef_{LM} - \beg\|_{L^1 (B_{4r_J} (p_J, \pi_H))}\leq C  \eps_1 r_J^{m+3+\alpha_\bh/2}\, ,
\end{equation}
which combined with \eqref{e:triangolo_1}, \eqref{e:triangolo_2} and \eqref{e:triangolo_3} would give the desired estimate. 

In order to reach \eqref{e:triangolo_4} we wish to apply the estimate \eqref{e:che_fatica} in Lemma \ref{l:cambio_tre_piani}. Recall that in our context we have the following estimates:
\begin{align*}
\|f\|_0 &\leq C \varepsilon_1^{\sfrac{1}{2m}} r_J^{1+\alpha_\bh}\\
r &=r_J\\
{\rm An}&\leq C \varepsilon_1^{\sfrac{1}{2}} r_J^{1-\alpha_\be}\\
{\rm Dir}\, (f) & \leq C \varepsilon_1 r_J^{m+2 -2 \alpha_\be}\\
\|D\bar\Psi\|_{C^0} &\leq C \varepsilon_1^{\sfrac{1}{2}} r_J\, .
\end{align*}
Hence the estimate \eqref{e:triangolo_4} follows easily from \eqref{e:che_fatica} once we impose $\alpha_\bh > 4 \alpha_\be$.

\medskip

In the case where both $M$ and $J$ are boundary cubes, the argument is entirely analogous. The only subtlety is that we cannot apply directly
the lemmas \ref{l:rotazioni_semplici} and \ref{l:cambio_tre_piani} since the functions we are dealing with are only defined on a portion of the respective ball, namely on $B^+_{2^7 6 r_J} (p^\flat_J, \pi_L)$. Note however that all functions can be easily extended to the whole 
ball $B_{2^7 6r_J} (p^\flat_J, \pi_L)$ with the following simple trick: on the boundary $\gammado = B_{2^7 6r_J} (p_J^\flat,\pi_L) \cap \partial B^+_{2^7 6 r_J} (p_J^\flat,\pi_L)$ the graph of $h_{LM}$ coincides with the boundary $\gammaup$, hence with a $C^3$ function $\psi$, and the graph of $f_{LM}$ coincides with $Q\a{\psi}$. Note moreover that $\psi$ satisfies the estimates $r_J^{-2} \|\psi\|_0 + r_J^{-1} \|D\psi\|_0 + \|D^2\psi\|_0 \leq C \varepsilon_1^{\sfrac{1}{2}}$. Hence it suffices to extend $\psi$ to $B^-_{2^76r_J} (p_J^\flat,\pi_L)$ to a function $\varphi$ with the same estimates and hence extend $h_{LM}$ and $f_{LM}$ to $B^-_{2^76r_J}  (p_J^\flat,\pi_L)$ by setting them respectively equal to $\psi$ and $Q\a{\psi}$. In this way we keep all the estimates which were essential for the argument above. 
\end{proof}

\section{Construction estimates and proof of Theorem \ref{thm:center_manifold}}

In what follows we use the shorthand notations $x_H$ (resp. $x^\flat_H$) for  the center $c(H) = \bp_{\pi_0} (p_H)$ (resp. $\bp_{\pi_0} (p^\flat_H)$) and
we write $B_r (x)$ for $B_r (x, \pi_0)$. 

\begin{proposition}\label{p:block_estimates}
Let $\kappa := \min \{\alpha_{\bh}/4, a_0/2\}$. 
Under the Assumptions \ref{ass:decay+cone} and \ref{ass:cm} the following holds for every pair of cubes $H, L \in \mathscr{P}_j$ \footnote{Recall the definition of $\mathscr{P}_j$ given in Section \ref{ss:glued_interpolations}}.
\begin{itemize}
\item[(a)] $\|g_H\|_{C^{3,\kappa} (B)}\leq C \varepsilon_1^{\sfrac{1}{2}}$, where $B = B_{4r_H} (x_H)$ when $H\in \mathscr{C}^\natural$
and $B = B^+_{2^7 4r_H} (x^\flat_H)$ when $H\in \mathscr{C}^\flat$;
\item[(b)] If $H$ and $L$ are neighbors then
\begin{align}
&\|g_H- g_L\|_{C^i (B_{r_H} (x_H))} \leq C \varepsilon_1^{\sfrac{1}{2}} \ell (H)^{3+\kappa-i} \quad \forall i\in \{0, 1,2,3\}\nonumber\\
&\qquad\qquad\qquad\qquad \mbox{when $H\in \mathscr{C}^\natural$,}\\
&\|g_H- g_L\|_{C^i (B^+_{2^7 r_H} (x^\flat_H))} \leq C \varepsilon_1^{\sfrac{1}{2}} \ell (H)^{3+\kappa-i} \quad \forall i\in \{0, 1,2,3\}\nonumber\\
&\qquad\qquad\qquad\qquad \mbox{when $H, L\in \mathscr{C}^\flat$;}
\end{align}
\item[(c)] $|D^3 g_H (x_H^\square) - D^3 g_L (x^\square_L)|\leq C \varepsilon_1^{\sfrac{1}{2}}|x^\square_H - x^\square_L|^\kappa$, where $\square = \phantom{\flat}$ if the corresponding cube is a non-boundary cube and $\square = \flat$ if it is a boundary cube;
\item[(d)] $\|g_H - \bp_{\pi_0}^\perp (p_H)\|_{C^0 (B)} \leq C \varepsilon_1^{\sfrac{1}{2m}} \ell (H)$ if $H\in \mathscr{C}^\natural$ and $g_H |_{\gammado \cap \overline{B}} = \psi$ if $H\in \mathscr{C}^\flat$,  where $B$ is as in (a);
\item[(e)] $|\pi_H - T_{(x, g_H (x))} \bG_{g_H} |\leq C \varepsilon_1^{\sfrac{1}{2}} \ell (H)^{1-\alpha_\be}$ for every $x\in B$, where
$B$ is as in (a);
\item[(f)] If $H'$ is the cube concentric to $H\in \sW_j$ with $\ell (H') = \frac{9}{8} \ell (H)$, then
\begin{equation}
\|\varphi_i - g_H\|_{L^1 (H')}\leq C \varepsilon_1 \ell (H)^{m+3+\alpha_\bh/2}\qquad \forall i\geq j+1\, . 
\end{equation}
\end{itemize}
\end{proposition}

%

\begin{proof} {\bf Proof of  (a).} Consider the chain of ancestors $H = H_i \subset H_{i-1} \subset \ldots \subset H_{N_0}$.
Fix any $j$ and consider the two cases where $H_j$ is a boundary cube or where $H_j$ is a non-boundary cube. In the first case observe that $H_{j-1}$ must also be a boundary cube. It follows then that $\bar h_{H H_j} - \bar h_{H H_{j-1}}$ is an harmonic function on $\Omega_j := B_{2^7 7 r_{H_j}} (p^\flat_{H_j}, \pi_H)$ in the first case and in $\Omega_j := B_{7r_{H_j}} (p_{H_j}, \pi_H)$ in the second case. Notice next that, by Proposition \ref{pr:elliptic_regularization}, we have
\[
\|\bar h_{H H_j} - \bar h_{H H_{j-1}}\|_{L^1 (\Omega_j)} \leq \|\etaa\circ \bar f_{HH_j} - \etaa\circ \bar f_{HH_{j-1}}\|_{L^1 (\Omega_j)}\, +  C \varepsilon_1 r_{H_{j-1}}^{m+3+\alpha_{\bh}}.
\]
On the other hand $\etaa\circ \bar f_{HH_j} - \etaa\circ \bar f_{HH_{j-1}}$ vanishes except for a set of Lebesgue measure at most $C \ell (H_{j-1})^m (\varepsilon_1 \ell (H_{j-1})^{2-2\alpha_\be})^{1+\sigma}$.  Taking into account that the oscillation of both functions are bounded by $C \varepsilon_1^{\frac{1}{2m}} r_{H_{j-1}}^{1+\alpha_{\bh}}$ we also know that 
\[
\|\etaa\circ \bar{f}_{HH_j} - \etaa \circ \bar{f}_{HH_{j-1}}\|_{L^1(\Omega_j)} \leq C \varepsilon_1 \ell (H_{j-1})^{m+3+2\alpha_\bh}.
\]
We thus conclude
\[
\|\bar h_{H H_j} - \bar h_{H H_{j-1}}\|_{L^1 (\Omega_j)} \leq C \varepsilon_1 \ell (H_{j-1})^{m+3+\alpha_\bh}\, .
\]
Now, if $H_j$ is a non-boundary cube we immediately conclude from the mean-value inequality for harmonic functions that
\begin{equation}
\sum_{k=0}^4 \ell (H_{j-1})^{k} \|D^k (\bar h_{H H_j} - \bar h_{H H_{j-1}})\|_{C^0 (B_{4r_{H_j}} (p_{H_j}, \pi_H))} \leq C \varepsilon_1 \ell (H_{j-1})^{3+\alpha_\bh}\, .
\end{equation}
In particular we conclude the estimates
\begin{equation}
\|\bar h_{H H_j} - \bar h_{H H_{j-1}}\|_{C^{3,\kappa} (B_{4r_{H_j}} (p_{H_j}, \pi_H))} \leq C \varepsilon_1 2^{-j\kappa}\, .
\end{equation}
Similarly, using an obvious scaling argument together with Lemma \ref{l:boundary_reg_harm}, when $H_j$ is a boundary cube we conclude
\begin{align}
&\sum_{k=0}^3 \ell (H_{j-1})^{k} \|D^k (\bar h_{H H_j} - \bar h_{H H_{j-1}})\|_{C^0 (B_{2^{7} 4r_{H_j}} (p^\flat_{H_j}, \pi_H))} \nonumber\\
\leq & C \varepsilon_1 \ell (H_{j-1})^{3+\alpha_\bh}\\
& [D^3(\bar h_{H H_j} - \bar h_{H H_{j-1}})]_{0, a_0, B_{2^{7} 4r_{H_j}} (p^\flat_{H_j}, \pi_H)}\nonumber\\
\leq & C \varepsilon_1 \ell (H_{j-1})^{\alpha_\bh - a_0}\, .
\end{align}
In particular, 
\begin{equation}
\|\bar h_{H H_j} - \bar h_{H H_{j-1}}\|_{C^{3,\kappa} (B_{2^7 4r_{H_j}} (p^\flat_{H_j}, \pi_H))} \leq C \varepsilon_1 2^{-j\kappa}\, .
\end{equation}
Summing all the estimates we conclude that if $H$ is not a boundary cube then
\begin{equation}
\|\bar h_H\|_{C^{3, \kappa} (B_{4r_h}(p_H,\phi_H))} \leq \|\bar h_{HH_{N_0}}\|_{C^{3, \kappa} (\Omega_{N_0})} + C \varepsilon_1\, .
\end{equation}
If $H$ is a boundary cube we have 
\[\|\bar h_H\|_{C^{3, \kappa} (B^+_{2^7 4r_h}(p^\flat_H,\phi_H))} \leq \|\bar h_{HH_{N_0}}\|_{C^{3, \kappa} (\Omega_{N_0})} + C \varepsilon_1\, .\]
Recall that in previously in \eqref{e:C^3,kappa estimate for N_0 interior}, \eqref{e:C^3,kappa estimate for N_0 boundary} we already showed that 
\[ 
 \|\bar h_{HH_{N_0}}\|_{C^{3, \kappa} (\Omega_{N_0})} \le C \varepsilon_1^{\frac12}\,,
\]
composing with $\Psi_H$ we find the desired regularity for $h_H$.
The regularity for $g_H$ follows then from Lemma \ref{l:rotazioni_semplici}.

\medskip

{\bf Proof of (b).} Consider the function $\hat h_L$ defined by Lemma \ref{lem:tilted_L1} when we take $H=J$ and $L=M$. We then have the two estimates 
\begin{align}
\|h_H- \hat h_L\|_{L^1 (B_{2r_J} (p_J, \pi_H))}\hphantom{t} & \leq C \varepsilon_1 \ell (J)^{m+3+\alpha_{\bh}/2}\, .\\
\|h_H- \hat h_L\|_{L^1 (B^+_{2^7 2 r_J} (p^\flat_J, \pi_H))} & \leq C \varepsilon_1 \ell (J)^{m+3+\alpha_{\bh}/2}\, ,
\end{align}
depending on the two cases under examination ($H$ non-boundary cube or both $H$ and $L$ boundary cube). 

Observe that the graph of $g_L$ coincides with (a portion) of the graph $\hat h_L$. We can thus use Lemma \ref{l:rotazioni_semplici} to prove
\[
\|g_H - g_L\|_{L^1 (\Omega)} \leq C \varepsilon_1 \ell (J)^{m+3+\alpha_{\bh}/2}\, 
\]
 where $\Omega_i$ is either $B_{r_{J}} (x_{J}, \pi_0)$ or $B^+_{2^7 r_{J}} (x^\flat_{J}, \pi_0)$ depending on whether $J$ is a non-boundary cube or a boundary cube
(in the second case we argue as in the proof of Proposition \ref{lem:tilted_L1}: in order to apply Lemma \ref{l:rotazioni_semplici} we extend both maps $h_H$ and $\hat h_L$ so that they are equal on $B^-_{2^7 2 r_J} (p_J, \pi_H)$ and the Lipschitz constant of both remains bounded by $C \varepsilon_1^{\sfrac{1}{2}}$). In order to conclude the estimates we then apply \cite[Lemma C.2]{DS4}. In the case of boundary cubes it is easy to see that the proof given in \cite{DS4} of Lemma \cite[Lemma C.2]{DS4} extends to $B^+_{2^7 2 r_J} (p_J, \pi_H)$ with trivial modifications.

\medskip

{\bf Proof of (c).} If the distance between $H$ and $L$ is larger than $2^{-N_0}$ then there is nothing to prove. Otherwise we can find an ancestor $J$ of $H$ and an ancestor $M$ of $L$ which make a distant relation and such that $\ell (J) = \ell (M)$ is comparable to $|x^\square_H - x^\square_L|$ up to a geometric constant. Consider then the chain of ancestors $H\subset H_{j-1} \subset \ldots \subset J$. Observe that, by the same arguments given in the previous step we can find maps $g_{HH_i}$ whose graphs coincide with (subsets of ) the graphs $h_{HH_i}$ and satisfy the estimates
\[
\|g_{HH_i} - g_{HH_{i-1}}\|_{C^3 (\Omega_i)} \leq C \varepsilon_1^{\sfrac{1}{2}} \ell (H_{i-1})^\kappa\, 
\]
where the domains $\Omega_i$ are either $B_{r_{H_i}} (x_{H_i}, \pi_0)$ or $B_{2^7 r_{H_i}} (x^\flat_{H_i}, \pi_0)$ depending on whether $H_i$ is a non-boundary cube or a boundary cube. Moreover, all the maps $g_{HH_i}$ enjoy uniform $C^{3,\kappa}$ bounds by the same arguments of point (a).
We thus conclude that
\[
|D^3 g_{HH_i} (x^\square_{H_i}) - D^3 g^\square_{HH_{i-1}}(x^\square_{H_{i-1}})|\leq C \varepsilon_1^{\sfrac{1}{2}} 2^{-i\kappa}\, .
\]
Summing all the estimates we then reach
\[
|D^3 g_H (x^\square_H) - D^3 g_{HJ} (x^\square_J)| \leq C \eps_1^{\sfrac{1}{2}} \ell (J)^\kappa \leq C \eps_1^{\sfrac{1}{2}} |x^\square_H - x^\square_L|^\kappa\, .
\]
Arguing similarly we conclude the corresponding estimate
\[
|D^3 g_L (x^\square_L) - D^3 g_{LM} (x^\square_M)| \leq C \eps_1^{\sfrac{1}{2}} |x^\square_H - x^\square_L|^\kappa\, .
\]
Finally, the obvious adaptation of the argument for (b) gives 
\[
|D^3 g_{HJ}(x^\square_J) - D^3 g_{LM} (x^\square_M)| \leq C \eps_1^{\sfrac{1}{2}} |x^\square_H - x^\square_L|^\kappa\, .
\]

\medskip

{\bf Proof of (d).} The claim is obvious by construction for boundary cubes. For non-boundary cubes, consider that the height bound for $T$ and the Lipschitz regularity for $f_H$ give that 
$\|\bp_{\pi^\perp_H} (p_H) - \etaa \circ f_H\|_\infty \leq C \varepsilon_1^{\sfrac{1}{2m}} \ell (H)$. If we set $\bef_H := (\etaa \circ \bar f_H, \Psi_H (x, \etaa\circ \bar f_H))$ we also get $\|\bp_{\pi^\perp_H} (p_H) - {\bf f}_H\|_\infty \leq C \varepsilon_1^{\sfrac{1}{2m}} \ell (H)$. On the other hand the Lipschitz regularity of the  tilted $H$-interpolating function  $h_H$ and the $L^1$ estimate on $h_H - \bef_H$ easily gives $\|\bp_{\pi^\perp_H} (p_H) - h_H\|_\infty \leq C \varepsilon_1^{\sfrac{1}{2m}} \ell (H)$. The estimate claimed in (d) follows then from Lemma \ref{l:rotazioni_semplici}.

\medskip

{\bf Proof of (e).}  The estimates \eqref{e:lip_interior} and \eqref{e:lip_boundary} show that the distance between any tangent to the graph of $h_H$ and $\pi_H$ is at most $C \varepsilon_1^{\sfrac{1}{2}} \ell (H)^{1-\alpha_\be}$ in the corresponding regions, which is just a reformulation of (e).

\medskip

{\bf Proof of (f).} For nearby neighbors $H$ and $L$ we can conclude the estimate $\|g_H-g_L\|_{L^1 (H\cup L)} \leq C \varepsilon_1 \ell (H)^{m+3+ \alpha_{\bh}/2}$ from the corresponding estimate for $h_H-h_L$ and Lemma \ref{l:rotazioni_semplici}. The conclusion is then an obvious consequence of the definition of  the glued interpolation maps $\varphi_i$.
\end{proof}

\begin{proof}[Proof of Theorem \ref{thm:center_manifold}] The estimate in  (a) is a consequence of Proposition \ref{p:block_estimates}: the argument is entirely analogous to that of \cite[Theorem 1.17(i)]{DS4}. Point (b) is a direct consequence of the definition of $\varphi_i$. Points (c) and (d) are a consequence of (a) and of the obvious facts that by construction the graphs of $\varphi_j$ are contained in $\Sigma$ and coincide with $\gammaup \cap \bC_{3/2}$ over  $\gammado \cap B_{3/2}$. Next, take any point $q\in \gammado$ and consider $\varphi_i$. Let $H\in \mathscr{C}_i$ be any cube which contains $q$ and observe that, since $H$ is a boundary cube, it must necessarily be that $H\in \mathscr{S}_i$. In particular we have 
$|\pi_H - T_q \bG_{\varphi_i}|\leq C \varepsilon_1^{\sfrac{1}{2}} 2^{-i(1-\alpha_\be)}$ by Proposition \ref{p:block_estimates} (b)\&(e). Note moreover that by Theorem \ref{thm:decay_and_uniq} we have $|\pi_H - \pi (q)|\leq C \varepsilon_1^{\sfrac{1}{2}} 2^{-i (1-\alpha_\be)}$. On the other hand, as $i\to \infty$ the planes $T_q \bG_{\varphi_i}$ converge to $T_q \mathcal{M}^+$, thus completing the proof of the theorem.
\end{proof}

\section{Proof of Cor. \ref{c:cm} and \ref{c:domains}, Prop. \ref{prop:sep} and Theo. \ref{thm:cm_app}}

Since all of the cubes in $\sW$ are non-boundary cubes, the proofs follow literally the ones of the corresponding corollaries, proposition and theorem  in \cite{DS4}, where Corollary \ref{c:cm} corresponds to \cite[Corollary 2.2]{DS4}, Corollary \ref{c:domains} corresponds to
\cite[Corollary 3.2]{DS4}, Proposition \ref{prop:sep} corresponds to \cite[Proposition 3.1]{DS4} and Theorem \ref{thm:cm_app} corresponds to \cite[Theorem 2.4]{DS4}. Note in particular that the estimates claimed in our statements match the ones of the statements in \cite{DS4} once we identify our parameters $a_0, \alpha_\be, \alpha_\bh, M_0, N_0, C_{\be}, C_{\bh}, \varepsilon_1$ with the parameters $\varepsilon_0, \delta_2, \beta_2, M_0, N_0, C_e, C_h, \mathbf{m}_0$ in \cite{DS4}. Moreover, although the excess $\bE (T, \bB_L)$ used in \cite{DS4} differs slightly from ours (since
it corresponds to minimizing $\bE (T, \bB_L, \pi)$ over all planes $\pi$, whereas in this note we minimize over all planes $\pi \subset T_{p_L} \Sigma$), it is obvious that it is smaller than the one used in this note, which suffices to prove all the estimates claimed. For the reader's convenience we briefly outline the arguments:

\begin{proof}[Proof of Corollary \ref{c:cm}]
First of all, while in \cite[Corollary 2.2]{DS4} it is claimed that the boundary of $T\res \mathbf{U}$ is supported in $\partial_l \mathbf{U}$, in our case we claim that it is supported in $\partial_l \mathbf{U} \cup \Gamma$. This is a consequence of the height bounds in (b)$^\flat$ and (b)$^\natural$ of Proposition \ref{pr:tilting_cm}. In order to prove the second claim of (a) we proceed similarly to the proof of the corresponding statement of \cite[Corollary 2.2]{DS4}. First of all consider that from the first part of the claim we conclude that the current $S := \mathbf{p}_\sharp T \res \bC_1 (0, \pi_0)$ is integer rectifiable and $\partial S \res \bC_1 (0, \pi_0) \subset \Gamma$. In particular we must have $S= k_+ \a{\mathcal{M}^+ \cap \bC_1 (0, \pi_0)}
+ k_- \a{\mathcal{M}^- \cap \bC_1 (0, \pi_0)}$ for some integers $k_0$ and $k_1$. Next fix any cylinder $\bC = \bC (x, r, \pi_0)$ for some point $x\in B_1 (0, \pi_0)\setminus \gamma$ and some $2 r< \dist (x, \gamma)$. We can then repeat literally the argument of \cite[Section 6.1]{DS4} to show that $\mathbf{p}_\sharp T \res \bC  (x, r, \pi_0)$ is either $Q \a{\mathcal{M}^+ \cap \bC}$ or $(Q-1) \a{\mathcal{M}^-\cap \bC}$, depending on whether $x$ belongs to $B_1^+$ or $B_1^-$. We then must have $k_+ = Q$ and $k_- = Q_1$

For the proof of (b) and (c) we can apply the same argument of \cite[Section 6.1]{DS4} used to prove (ii) and (iii) of \cite[Corollary 2.2]{DS4}, since the cylinders and balls considered in the corresponding argument do not touch $\Gamma$. 
The final conclusion (d) of the corollary follows from the fact that boundary cubes are always refined, that the corresponding balls $\bB^\flat_H$ are always centered on points of $\Gamma$ and from (b)$^\flat$ of Proposition \ref{pr:tilting_cm}.
\end{proof}

\begin{proof}[Proof of Theorem \ref{thm:cm_app}] The construction of $(F^+, F^-)$ is done separately on the two manifolds $\mathcal{M}^+$ and $\mathcal{M}^-$ following the exact same procedure of \cite[Section 6.2]{DS4}. Note that for all $L\in \sW^+$ and for all $L\in \sW^-$ the cylinders $\bC_{8r_L} (p_L, \pi_L)$ which are involved in the corresponding argument have empty intersection with $\Gamma$ and enjoy the relevant estimates once we identify our parameters $a_0, \alpha_\be, \alpha_\bh, M_0, N_0, C_{\be},$ $C_{\bh}, \varepsilon_1$ with the parameters $\varepsilon_0, \delta_2, \beta_2, M_0, N_0, C_e, C_h, \mathbf{m}_0$ in \cite{DS4}. This procedure defines $F^+$ on $\mathcal{M}^+\setminus \Gamma$ and $F^-$ on $\mathcal{M}^-\setminus \Gamma^-$. However, using the height bound in the boundary cylinders $\bC_{2^7 36 r_L} (p^\flat_L, , \pi_L)$ of (c)$^\flat$ in Proposition \ref{pr:tilting_cm} it is easy to see that $F^+$ (resp. $F^-$) on $\mathcal{M}^+\setminus \Gamma$ (resp. $\mathcal{M}^-\setminus \Gamma)$ can be extended to a unique Lipschitz map on the whole $\mathcal{M}^+$ (resp. $\mathcal{M}^-$) by setting $F(x)= Q \a{x}$ (resp. $(Q-1) \a{x}$) for every $x\in \Gamma \cap \mathcal{M}^+$ (resp. $\Gamma \cap \mathcal{M}^-$).
\end{proof}

\begin{proof}[Proof of Proposition \ref{prop:sep}] We follow literally the argument given in \cite[Proof of Proposition 3.1]{DS4} given in \cite[Section 7.1]{DS4}. Note in particular that all the cylinders involved in the argument of that proof do not intersect $\Gamma$, because the cubes $H$ and $L$ involved in the statement of Proposition \ref{prop:sep} are all non-boundary cubes. 
\end{proof}

\begin{proof}[Proof of Corollary \ref{c:domains}]
Again we can repeat word by word the proof of \cite[Corollary 3.2]{DS4} given at the end of \cite[Section 7.1]{DS4}: note indeed that all the cubes involved in the argument are necessarily non-boundary cubes.
\end{proof}

\section{Proof of Proposition \ref{p:splitting}}

The proof follows the one of the corresponding statement in \cite{DS4}, namely \cite[Proposition 3.4]{DS4}, with one minor adjustment, which is needed because our excess is not exactly the excess of \cite{DS4} (namely here we minimize only among planes contained in \(T_p\Sigma\)). The adjustment goes as follows. Note first that we know that a cube $H\in \sW^\be$ must be a non-boundary cube. In fact the very same argument given in Proposition \ref{pr:tilting_cm} shows the following simple fact:

\begin{lemma}\label{l:super_technical}
For any fixed $i\in \mathbb N$, if $\varepsilon_1$ is chosen sufficiently small, then for every $H\in \sW^\be$ the chain of ancestors $H = H_j \subset H_{j-1} \subset \ldots \subset H_{j-i}$ consists all of non-boundary cubes (and in particular $j-i \leq N_0$). 
\end{lemma}

The proof given in \cite[Section 7.3]{DS4} of \cite[Proposition 3.4]{DS4} is then based on the following two facts:
\begin{itemize}
\item[(a)] If $H\in \mathscr{sW}^\be$, then the chain of ancestors $H = H_j \subset L = H_{j-1} \subset \ldots \subset H_{j-6}$ consists all of non-boundary cubes;
\item[(b)] The following inequality holds:
\begin{equation}\label{e:super_technical}
\min_\pi \bE (T, \bB_H, \pi) \geq 2^{-2+2\delta_2} \min_{\bar \pi} \bE (T, \bB_L, {\bar \pi})\, ,
\end{equation}
for some positive $\delta_2$: correspondingly $M_0$ will have to be chosen large depending on such $\delta_2$.
\end{itemize}
The first condition is covered by Lemma \ref{l:super_technical}. As for the second condition, observe that we actually have 
\begin{align}
\min_{\pi\subset T_{p_H} \Sigma} \bE (T, \bB_H, \pi) &= \bE (T, \bB_H) \geq 2^{-2+2\alpha_\be} \bE (T, \bB_L)\nonumber\\
& = 2^{-2+2\alpha_\be} \min_{\bar \pi\subset T_{p_L} \Sigma} \bE (T, \bB_L, {\bar \pi})\, .\label{e:super_technical_2}
\end{align}
We now want to show that \eqref{e:super_technical} will indeed follow from \eqref{e:super_technical_2}, provided $\delta_2 = \alpha_\be/2$. In order to apply the argument of \cite[Section 7.3]{DS4} we then  just need $M_0$ to be sufficiently large with respect to $\alpha_\be$, which is indeed one of the requirements of Assumption \ref{ass:cm}.

\medskip

{\bf Proof of \eqref{e:super_technical}}
First of all, in order to simplify our notation, for every $q\in \Sigma$ we denote by $\bp_q$ the orthogonal projection onto $T_q \Sigma$. Moreover, if $\pi$ is an $m$-dimensional (oriented) plane, we let $\vec{\pi}$ be the unit $m$-vector orienting it. Consistently, we denote by $\vec{T} (p)$ the unit $m$-vector orienting the approximate tangent plane of $T$ at $p$ (which exists for $\|T\|$-a.e. $p$). 

Next, clearly 
\begin{equation}\label{e:super_technical_3}
\bE (T, \bB_L) \geq \min _{\bar \pi} \bE (T, \bB_L, {\bar \pi})\, .
\end{equation}
So, we need a reverse inequality between the quantities $\bE (T, \bB_H)$ and $\min_\pi \bE (T, \bB_H, \pi)$. We select thus a $\pi$ which attains the latter minimum. Notice that we have the following inequality
\begin{align*}
& \, \frac{1}{\|T\| (\bB_H)} \int_{\bB_H} |\bp_{p_H} (\vec \pi) - \vec T (q)|^2\, d\|T\| (q)\nonumber\\
 \leq &
\, \frac{2}{\|T\| (\bB_H)} \int_{\bB_H} |\bp_{p_H} (\vec \pi) - \bp_{p_H} (\vec T (q))|^2\, d\|T\| (q)\nonumber\\
& + \frac{2}{\|T\| (\bB_H)} \int_{\bB_H} |\bp_{p_H} (\vec{T} (q)) - \vec T (q)|^2\, d\|T\| (q)\\
\leq & \,C_0 \bE (T, \bB_H) + C_0 \sup_{q\in \Sigma \cap \bB_H} |\bp_{p_H} - \bp_{q}|^2\\
\leq & \,C_0 C_\be \varepsilon_1 \ell (H)^{2-2\alpha_\be} + \bar C \varepsilon_1 \ell (H)^2\, ,
\end{align*}
where $C_0$ is a geometric constant and the constant $\bar C$ depends only upon $M_0$. In particular, since $C_\be$ is assumed to be sufficiently large compared to $M_0$ and $N_0$, we conclude
\[
\frac{1}{\|T\| (\bB_H)} \int_{\bB_H} |\bp_{p_H} (\vec \pi) - \vec T (q)|^2\, d\|T\| (q) \leq C_0 C_\be \varepsilon_1 \ell (H)^{2-2\alpha_\be}\, .
\]
We next use the obvious inequality $|1-|\bp_{p_H} (\vec\pi)|| = ||\vec T (q)| - |\bp_{p_H} (\vec\pi)||\leq |\vec T (q) -  \bp_{p_H} (\vec\pi)|$ to infer
\[
|1-|\bp_{p_H} (\vec\pi)||^2 \leq C_0 C_\be \varepsilon_1 \ell (H)^{2-2\alpha_\be}\, .
\]
Observe also that $|\bp_{p_H} (\vec\pi)|$ is necessarily smaller than $1$, because $\bp_{p_H}$ is a projection. We thus reach
\begin{equation}\label{e:super_technical_4}
1-C_0 C_\be \varepsilon_1 \ell (H)^{2-2\alpha_\be} \leq |\bp_{p_H} (\vec \pi)| \leq 1\, .
\end{equation}
In particular, since $\eps_1$ is assumed to be  small with respect to $\bC_e$, we have $|\bp_{p_H} (\vec \pi)|\geq \frac{1}{2}$. 
Consider now the $m$-dimensional plane $\pi'$ which is oriented by $\bp_{p_H} (\vec\pi)/|\bp_{p_H} (\vec\pi)|$. Clearly $\pi' \subset T_{p_H} \Sigma$. Moreover, since $\vec{T} (q)$ has norm $1$ whereas $\bp_{p_H} (\vec \pi)$ has norm at most $1$, we have the pointwise inequality
\[
|\vec{T} (q) - \pi'|^2 = \left|\vec{T} (q) - \frac{\bp_{p_H} (\vec\pi)}{|\bp_{p_H} (\vec\pi)|}\right|^{2}\leq \frac{1}{|\bp_{p_H} (\vec\pi)|}
|\vec{T} (q)- \bp_{p_H} (\vec\pi)|^{2}\, .
\]
We can thus repeat the computations above to conclude
\begin{align}
|\bp_{p_H} (\vec\pi)| \bE (T, \bB_H) & \leq |\bp_{p_H} (\vec \pi)| \bE (T, \bB_H, \pi')\nonumber\\
& = \frac{|\bp_{p_H}|}{2\omega_m (64 r_H)^m} \int_{\bB_H} \left|\vec{T} (q) - \frac{\bp_{p_H} (\vec\pi)}{|\bp_{p_H} (\vec\pi)|}\right|^2\,
d\|T\| (q)\nonumber\\
& \leq \frac{1}{2\omega_m (64 r_H)^m} \int_{\bB_H} | \vec{T} (q) - \bp_{p_H} (\vec\pi)|^2\, d\|T\| (q)\, .\label{e:super_technical_5}
\end{align}
Next, arguing as few lines above
\begin{align}
&\left(\int_{\bB_H} | \vec{T} (q) - \bp_{p_H} (\vec\pi)|^2\, d\|T\| (q)\right)^{\sfrac{1}{2}}\nonumber\\
\leq & \left(\int_{\bB_H} | \bp_{p_H} (\vec{T} (q)) - \bp_{p_H} (\vec\pi)|^2\, d\|T\| (q)\right)^{\sfrac{1}{2}}\nonumber\\
&+ \left(\int_{\bB_H} | \bp_{p_H} (\vec{T} (q)) - \vec{T} (q)|^2\, d\|T\| (q)\right)^{\sfrac{1}{2}}\nonumber\\
\leq & \left(\int_{\bB_H} | \bp_{H} (\vec{T} (q)) - \bp_{p_H} (\vec\pi)|^2\, d\|T\| (q)\right)^{\sfrac{1}{2}} \\
&\qquad + \bar C (\omega_m (64 r_H)^m)^{\sfrac{1}{2}} \varepsilon_1^{\sfrac{1}{2}} \ell (H)\, .
\end{align}
Combining the latter inequality with \eqref{e:super_technical_5} and with
\begin{align}
 &\frac{1}{2\omega_m (64 r_H)^m} \int_{\bB_H} |\bp_{p_H} (\vec{T} (q)) - \bp_{p_H} (\vec\pi)|^2\, d\|T\| (q)\nonumber\\
&\leq \frac{1}{2\omega_m (64 r_H)^m} \int_{\bB_H} |\vec{T} (q) - \vec\pi|^2\, d\|T\| (q)\nonumber\\
&= \bE (T, \bB_H, \pi) = \min_{\bar\pi} \bE (T, \bB_H, \bar\pi)\, ,
\end{align}
we reach the inequality
\begin{align}
|\bp_{p_H} (\vec\pi)| \bE (T, \bB_H) & \leq \min_{\bar\pi} \bE (T, \bB_H, \bar\pi) + \bar C \left(\min_{\bar\pi} \bE (T, \bB_H, \bar\pi)\right)^{\frac{1}{2}} \eps_1^{\frac{1}{2}} \ell (H)\nonumber\\ 
&+ \bar C \eps_1 \ell (H)^2\label{e:super_technical_7}\, ,
\end{align}
where $\bar C$ depends only upon $M_0$. 
By Young inequality we thus deduce that
\[
|\bp_{p_H} (\vec\pi)| \bE (T, \bB_H)\le 2^{\frac{\alpha_{\be}}{2}}\min_{\bar\pi} \bE (T, \bB_H, \bar\pi)+\hat{C} \eps_1 \ell(H)^{2}
\]
where \(\hat{C}\) depends on \(M_0\) and \(\alpha_\be\). Since  \(H\in \sW^{\be}\),
\[
\bE (T, \bB_H) \geq C_\be \varepsilon_1 \ell (H)^{2-2\alpha_\be}\, ,
\]
hence,  by  also using  \eqref{e:super_technical_4} and that  \(\ell(H) \le 1\), 
\[
(1-C_0 C_{\be} \varepsilon_1) \bE (T, \bB_H)\le 2^{\frac{\alpha_{\be}}{2}}\min_{\bar\pi} \bE (T, \bB_H, \bar\pi) +\frac{\hat C}{C_\be} \ell (H)^{2\alpha_\be} \bE (T, \bB_H),
\]
i.e.
\[
\left(1- C_0 C_{\be} - \frac{\hat C}{C_\be} \ell (H)^{2\alpha_\be}\right) \bE(T, \bB_H) \le  2^{\frac{\alpha_{\be}}{2}}\min_{\bar\pi} \bE (T, \bB_H, \bar\pi)\, .
\]
Since the constant $\hat{C}$ depends only on $M_0$, choosing $N_0$ sufficiently large (which implies that
$\ell (H)^{2\alpha_\be} \leq 2^{-2\alpha_\be N_0}$ is sufficiently small) and
then \(\varepsilon_1\) small we deduce that
\begin{align}\label{e:super_technical_9}
2^{-\alpha_\be} \bE (T, \bB_H) \leq 
\min_{\bar\pi} \bE (T, \bB_H, \bar\pi) \, .
\end{align}
Combining \eqref{e:super_technical_2}, \eqref{e:super_technical_3} and the latter inequality we conclude
\begin{align}
\min_\pi \bE (T, \bB_H, \pi) &\geq 2^{-\alpha_\be} \bE (T, \bB_H)
\geq 2^{-2+\alpha_\be} \bE (T, \bB_L)\nonumber\\
&\geq 2^{-2+ \alpha_\be}  \min_{\bar \pi} \bE (T, \bB_L, {\bar \pi})\, ,\label{e:super_technical_10}
\end{align}
thus \eqref{e:super_technical} holds with $\delta_2 = \alpha_\be/2$ as promised.

\chapter{Monotonicity of the frequency function}\label{chap:frequency}

In this chapter we establish the monotonicity of a suitable frequency function at a  collapsed point. We
assume therefore that $0\in \gammaup$ is a  collapsed point and that Assumption~\ref{ass:cm_app} holds.
In particular we fix a center manifold $\mathcal{M} = \mathcal{M}^+\cup \mathcal{M}^-$ as in Theorem \ref{thm:center_manifold} and
an $\mathcal{M}$-normal approximation as in Theorem \ref{thm:cm_app}. We will indeed consider two different frequency functions:
one related to the ``left side'' of the approximation and the other one related to the ``right side''. Without loss of generality we will
carry on our discussion on $\mathcal{M}^+$. 

\begin{remark}\label{r:tildaM}
By our construction $\mathcal{M}^+$ is the graph of a map $\boldsymbol{\varphi}^+: \pi_0^+ \supset B_1^+ \to \pi_0^\perp$, where we assume that $\pi_0$ is the tangent plane to $T$ in $0 \in \gammaup$. 
For convenience we can extend $\boldsymbol{\varphi}^+$ to a $C^3$ map $\tilde{\boldsymbol{\varphi}}$ 
on the whole ball $B_1\cap \pi_0$. When referring to $\boldsymbol{\varphi}^+$ we will then drop the superscript $+$, but we will keep the notation $\mathcal{M}^+$ for that portion of the extended graph $\{ (x, \tilde{\boldsymbol{\varphi}} (x)) \colon x \in B_1 (0, \pi_0)\}$ which lies over $B_1^+$. The graph of the function $\tilde{\boldsymbol{\varphi}}$ on the whole $B_1 (0, \pi_0)$ will instead be denoted by $\widetilde{\mathcal{M}}$. Note that in this setting the projection \(\bp: \bp^{-1} (\mathcal M^+)\to \mathcal M^+\) is of class \(C^{2,\kappa}\), cf. with Assumption \ref{ass:cm_app}.  
\end{remark}

\section{Frequency function and main monotonicity formula}\label{sec:ff}

In order to define our main quantities, we start with the following simple lemma which is the curvilinear version of Lemma \ref{l:good_vector_field}.

\begin{lemma}\label{l:good_vector_field on M}
There exists a continuous function $d^+:\mathcal M^+ \to \R^+$ which belongs to $C^2(\mathcal M^+\setminus\{0\})$ and satisfies the following properties:
\begin{itemize}
\item[(a)] $d^+(x)=\dist_{\mathcal{M^+}}(x,0) + O(\dist_{\mathcal{M^+}}(x,0)^2)= \abs{x} + O(\abs{x}^2)$;
\item[(b)] $\abs{\nabla d^+(x)} = 1 + O(d^+)$, where $\nabla$ is the gradient on the manifold $\mathcal{M}$;
\item[(c)] $\frac{1}{2}\nabla^2d^2(x)= g + O (d^+)$, where \(\nabla^2\) denotes the covariant Hessian on $\mathcal{M}$ (which we regard as a $(0,2)$ tensor) and $g$ is the induced metric on $\mathcal{M}$ as a submanifold of $\mathbb R^{m+n}$;
\item[(d)] $\nabla d^+(x) \in T_x\gammaup$ for all $x \in \gammaup$, i.e.
\begin{equation}\label{e:sign_condition}
\nabla d^+ \cdot \vec{n}^+ = 0 \qquad \mbox{on $\gammaup$,}
\end{equation}
where $\vec{n}^+$ denotes the outer unit normal to $\mathcal{M}^+$ inside $\widetilde{\mathcal{M}}$.
\end{itemize}
In particular this implies 
\begin{equation}\label{e:hessian d}
\nabla^2d^+ (x)=\frac{1}{d} \Big(g-\nabla d^+(x) \otimes \nabla d^+(x)\Big)+O(1)
\end{equation}
and 
\begin{equation}\label{e:laplacian of the good vectorfield} 
	\Delta\, d^+ = \frac{m-1}{d^+} + O(1)\, 
\end{equation}
where $\Delta$ denotes the Laplace-Beltrami operator on $\mathcal{M}$, namely the trace of the Hessian $\nabla^2$. Moreover:
\begin{itemize}
\item[(S)] All the constants estimating the $O (\cdot)$ error terms in the above estimates can be made smaller than any given $\eta>0$, provided the parameter $\varepsilon_1$ in Assumption \ref{ass:cm} is chosen appropriately small (depending on $\eta$). 
\end{itemize}

On the ``left side'' there exists an analogous function $d^-: \mathcal{M}^- \to \mathbb R^+$ satisfying the properties corresponding to (a), (b), (c), (d) and (S).
\end{lemma}

\begin{proof} For the sake of simplicity we focus on the ``right side'' and we drop the subscript $+$ from the function $d$. 
As noted in Remark \ref{r:tildaM} we can extend  $\mathcal{M}^+$ to a   $C^3$ manifold \(\widetilde{\mathcal{M}}\)  such that   $\gammaup\subset \widetilde{\mathcal{M}}$ is a $C^3$ submanifold of $\widetilde{\mathcal{M}}$ passing through the origin. Hence there exists a $C^2$ regular map  $\Xi: U\times (-\delta, \delta) \to \widetilde{\mathcal{M}}$, $U\subset \R^{m-1}$, with the properties that 
\begin{enumerate}
\item $\Xi(0)=0$ and $D\Xi(0)=0$;
\item $\Xi$ is a local parametrization of $\widetilde{\mathcal{M}}$  and $ y'\ni U \mapsto \Xi(y',0)$  is a local parametrization of $\gammaup$;
\item $\partial_m\Xi(y',0) \perp T_{\Xi(y',0)}\gammaup$ for all $y' \in U$.  	
\end{enumerate}
Hence, if $g:=\Xi^\#\delta$ is the pullback metric of \(\widetilde{\mathcal{M}}$  on  $U\times (-\delta, \delta)$, we have 
\[
g_{ij}(y)=\delta_{ij} + O(\abs{y}^2), \qquad \partial_kg_{ij} = O(\abs{y}),
\]
and similarly  for $g^{ij}$. In particular this implies that $\dist_\mathcal{M}(\Xi(y),0)= \abs{y}+O(\abs{y}^2)$ on \(\mathcal M^+\).
We claim that $d(x):=\abs{\Xi^{-1}(x)}$ has the desired properties. We will check (a) - (c) using the coordinates associated to the map $\Xi$. 
Since 
\[ \abs{\nabla d}^2(\Xi(y)) = g^{ij}\partial_id \partial_jd = g^{ij}(y) \frac{y^iy^j}{\abs{y}^2} = 1 + O(\abs{y}^2) \]
we have that (b) is satisfied. For the Christoffel symbols we have $\Gamma_{ij}^k(y)= O(\abs{y})$ since   $\partial_ig_{ij}= O(\abs{y})$. Hence (c) follows, because 
\[ 
\frac{1}{2}\nabla^2d(\Xi(y))_{ij}= \frac12\partial_{ij}d^2 - \frac12 \Gamma_{ij}^k \partial_kd^2= \delta_{ij} + O(\abs{y}^2) = g_{ij} (y) + O (|y|^2)\, .
 \]
Concerning (d) we just note that, by (3), we have $g^{im}(y',0)=0$ for all $y' \in U$, hence $g^{ij}\partial_jd \in \R^{m-1}\times \{0\}$ for all $y' \in U$ and 
$\nabla d(\Xi(y))= \Xi_{\#} (g^{ij}\partial_jd e_i)$.
Equations  \eqref{e:laplacian of the good vectorfield}  and \eqref{e:hessian d} are now simple consequences of (c) and (b).

Claim (S) follows easily from a closer inspection of the above argument.
\end{proof}

%
 We now fix a cutoff function
\begin{equation}
\phi (t) :=
\left\{
\begin{array}{ll}
1\qquad &\mbox{for $0\leq t \leq \frac{1}{2}$}\\
2 (1-t) &\mbox{for $\frac{1}{2}\leq t \leq 1$}\\
0\qquad &\mbox{for $t\geq 1$.}
\end{array}\right.
\end{equation}
and define
\begin{align}
D_{\phi,d^+} (N^+,r) & := \hphantom{-}\int_{\mathcal{M}^+} \phi\left(\frac{d^+(x)}{r}\right)|DN^+|^2(x)\,\label{e:richiamo_pm}
\\
H_{\phi,d^+} (N^+,r) &:=-\int_{\mathcal{M}^+} \phi'\left(\frac{d^+(x)}{r}\right) |\nabla d^+ (x)|^2 \frac{|N^+(x)|^2}{d^+(x)}\, ,\label{e:richiamo_pm_2}
\end{align}
where all integrals are taken with respect to the standard volume form on $\mathcal{M}^+$.\footnote{The convention of omitting the volume form in the integrals taken over $\mathcal{M}^+$ and $\mathcal{M}^-$ will be used systematically in the rest of the paper.}
The \emph{frequency function} is then defined as  the ratio
\[
I_{\phi,d^+} (N^+,r) :=\frac{rD_{\phi,d^+} (N^+,r)}{H_{\phi,d^+} (N^+,r)}\,.
\]
Analogously we define $D_{\phi, d^-} (N^-, r)$, $H_{\phi, d^-} (N^-, r)$ and $I_{\phi, d^-} (N^-, r)$.


The main theorem of this chapter is then the following counterpart to Theorem \ref{thm:limit_ff}, where we use the notation\index{aalC^\pm@${\mathcal C}^\pm$}
\[
\mathcal C^{\pm}=\big\{y\in \bB_{1}:  \bp(y)\in \mathcal M^{\pm}\textrm{ and } |y-\bp(y)|\le \dist(y,\gammaup)^{3/2}  \big\}
\]
for the horned neighborhoods of $\mathcal{M}^\pm$ in which $T$ is supported (compare with  Corollary \ref{c:cone_cut} and Theorem \ref{thm:center_manifold} (e)).

\begin{theorem}\label{thm:ff_estimate_current}\label{THM:FF_ESTIMATE_CURRENT}
Let \(T\), \(\Sigma\)  and \(\gammaup\) be as in Assumption~\ref{ass:cm_app} and consider $\phi$ and $d$ as above.
Then:
\begin{itemize}
\item[(a)] either $T\res \mathcal{C}^+$ equals $Q\a{\mathcal{M}^+}$ in a neighborhood of $0$, in which case we set $I_0^+=+\infty$;
\item[(b)] or there is a positive number $I_0^+$ such that
\begin{equation}
I^+_0 = \lim_{r\downarrow 0} I_{\phi,d^+} (N^+, r)\, .
\end{equation}
\end{itemize} 
The corresponding statements hold on the left side for the current $T\res \mathcal{C}^-$ and the frequency function $I_{\phi, d^-} (N^-, r)$.
\end{theorem}

\section{Poincar\'e inequality}

From now on, in order to simplify our notation, we drop the supscripts $+$ from $N$ and $d$ and the subscripts $d$ and $\phi$ from $H$, $D$ and $I$.

We notice here the following simple consequence of the fact that $N|_{\gammaup}$ vanishes identically.

\begin{proposition}\label{p:poincare_traccia}
There is a geometric constant $C$ such that
\begin{equation}\label{e:H<rD}
H (r) \leq C r D(r)\, \qquad \mbox{for all sufficiently small $r$.}
\end{equation}
In particular 
\begin{equation}\label{e:lower_bound_I}
I(r) \geq C^{-1}  \qquad \mbox{for all sufficiently small $r$.}
\end{equation}
Moreover,
\begin{equation}\label{e:poincare_1000}
\int_{\{d<r\}\cap \mathcal{M}^+} |N|^2 \leq C r^2 D(r)\, \mbox{for all sufficiently small $r$.}
\end{equation}
\end{proposition}
\begin{proof}
We start noticing that, for $r$ sufficiently small, we can assume
\begin{equation}\label{e:nabla_d_2sided}
\frac{1}{2} \leq |\nabla d|\leq 2\, 
\end{equation}
and that the domains $\{d=r\}\cap \mathcal M^+$ and $\{d<r\}\cap \mathcal M^+$ are diffeomorphic to the corresponding half-sphere and half-ball in $\mathbb R^m_+ =\{x_1 \geq 0\}$, with uniform controls on the first derivative of the diffeomorphism and its inverse. In particular we have the trace Poincar\'e inequality
\[
\int_{\{d=s\}\cap \mathcal{M}^+} |N|^2 \leq C s \int_{\{d<s\}\cap \mathcal{M}^+} |D|N||^2 \leq C s \int_{\{d<s\}\cap \mathcal{M}^+} |DN|^2\, ,
\]
because $|N|$ vanishes identically on $\gammaup$. 

Integrating the latter inequality, using \eqref{e:nabla_d_2sided} and the coarea formula, we achieve
\begin{align*}
H(r) & =-\int_{\frac{r}{2}}^r \frac{1}{s} \phi' \left(\frac{s}{r}\right) \left(\int_{\{d=s\}\cap \mathcal{M}^+} |\nabla d| |N|^2\right)\, ds\\ 
& \leq - C \int_{\frac{r}{2}}^r \phi' \left(\frac{s}{r}\right) \left(
\int_{\{d<s\}\cap \mathcal{M}^+} |DN|^2\right)\, ds\\
&= C r \int_{\frac{r}{2}}^r \left(\int_{\{d=s\}\cap \mathcal{M}^+} |DN|^2 |\nabla d|^{-1}\right)
\phi \left(\frac{s}{r}\right)\\ 
&\quad + C r \phi \left(\frac{r}{2}\right) \int_{\{d<r/2\}\cap \mathcal{M}^+} |DN|^2\\
&\leq C r D(r)\, .
\end{align*}

Next, the inequality \eqref{e:lower_bound_I} is a trivial consequence of \eqref{e:H<rD}. Moreover, \eqref{e:H<rD} and 
\eqref{e:nabla_d_2sided} give 
\[
\int_{\{r/2<d<r\}\cap \mathcal{M}^+} |N|^2 \leq C r^2 D(r)\, .
\]
On the other hand 
\[
\int_{\{d<r/2\}\cap \mathcal{M}^+} |N|^2 \leq C r^2 \int_{\{d<r/2\}\cap \mathcal{M}^+} |DN|^2 \leq C r^2 D(r)
\]
follows from the usual Poincar\'e inequality since $|N|$ vanishes identically on $\gammaup$. Thus \eqref{e:poincare_1000}
can be achieved summing the last two inequalities. 
\end{proof}

\section{Differentiating $H$ and $D$}
 We compute here the derivatives of $H$ and $D$. 

\begin{proposition}\label{p:H'_maiala_corrente} If $D$ and $H$ be as in the definitions of Section \ref{sec:ff}, then
\begin{eqnarray}
D'(r) & = & -\int\phi'\left(\frac{d (x)}{r}\right)\frac{d (x)}{r^2}|DN|^2\,; \label{der_D_corrente}\\
H' (r) & = & \left(\frac{m-1}{r} + O(1)\right) H (r) + 2 E(r)\,,\label{der_H_corrente}
\end{eqnarray}
where
\[ 
E(r) := - \frac{1}{r} \int \phi' \left(\frac{d(x)}{r}\right) \sum_i N_i (x) \cdot (DN_i (x) \nabla d (x))\,.
\]
\end{proposition}
\begin{proof}
The identity \eqref{der_D_corrente} is an obvious computation. In order to compute $H'$ we first use the coarea formula on embedded manifolds to write
\begin{align}
H(r) & = -\int_0^\infty \int_{\{d=s\}} \frac 1s \phi' \left(\frac{s}{r}\right) |\nabla d (x)| |N|^2 (x)\, d\mathcal{H}^{m-1} (x)\, ds\nonumber\\
&=- \int_0^\infty \frac{\phi' (t)}{t} \underbrace{\int_{\{d= r t\}} |\nabla d (x)| |N|^2 (x)\, d\mathcal{H}^{m-1} (x)\,}_{=: h (rt)}\, dt.  \label{e:H_coarea_2}
\end{align}
In order to compute $h' (t)$ we consider that $\nu (x) = \frac{\nabla d (x)}{\abs{\nabla d (x)}}$ is orthogonal to the 
level sets of $d$ in $\mathcal{M}^+$ and it is parallel to \(\gammaup\). Thus, using  the divergence theorem on $\mathcal{M}^+$ we obtain
\begin{align*}
h(t+\varepsilon) - h(t) &= \int_{\{d=t+\varepsilon\}\cap \mathcal M^+} |N|^2 \nabla d \cdot \nu\,
d\mathcal{H}^{m-1}\\ 
&\quad- \int_{\{d=t\}\cap \mathcal M^+}|N|^2 \nabla d \cdot \nu\,
d\mathcal{H}^{m-1}
\\
&= \int_{\{t<d<t+\varepsilon\}\cap \mathcal M^+} {\rm div}\, (|N|^2 \nabla d (x))
\\
&= \int_{\{t<d<t+\varepsilon\}\cap \mathcal M^+} 2 \sum_i N_i (x)\cdot(D N_i (x) \nabla d (x))
 \\
 &+\int_{\{t<d<t+\varepsilon\}\cap \mathcal M^+} |N|^2 \Delta d (x)\, ,
\end{align*}
Dividing by $\varepsilon$, taking the limit (and using the coarea formula once again) we conclude
\begin{equation}\label{e:der_h_piccola_2}
h' (t) = \int_{\{d=t\}\cap \mathcal M^+} |\nabla d|^{-1} \left(2 \sum_i N_i \cdot(D N_i \nabla^\mathcal{M} d)\, + |N|^2 \Delta d\right)\, d\mathcal{H}^{m-1}\, .  
\end{equation}
Differentiating \eqref{e:H_coarea_2} in $r$, inserting \eqref{e:der_h_piccola_2}  and using the fact that, if \(\phi(d (x)/r)\ne 0\), then \(d (x)=O(r)\), we conclude
\begin{align}
&H' (r)\\ 
& =  \int_0^\infty \phi' (\sigma) \int_{\{d=\sigma r\}} \frac{1}{|\nabla d|} \Big(2 \sum_i N_i \cdot(DN_i \nabla d)\, + |N|^2 \Delta d)\Big)\, d\mathcal{H}^{m-1}\, d\sigma\nonumber\\
&= 2E (r) - \frac{1}{r} \int \phi' \left(\frac{d(x)}{r}\right) |N|^2 \Delta d (x)\nonumber\\
&\stackrel{\eqref{e:laplacian of the good vectorfield}}{=} 2E(r) - \frac{1}{r} \int \phi' \left(\frac{d(x)}{r}\right) |N|^2 \left(\frac{m-1}{r} + O(1)\right)\\
&= 2E (r)+\left(\frac{m-1}{r}+ O(1)\right) H(r)\, .\qedhere
\end{align}
\end{proof}

\section{First variations}
In order to derive the two key identities leading to the monotonicity of the frequency function we will use the first variations of the currents. 

\begin{lemma}\label{lem:cone-splitting}
Let \(T\), \(\Sigma\)  and \(\gammaup\) be as in Assumption~\ref{ass:cm_app}.
Then, provided \(\eps_1\) is sufficiently small, we have that
\begin{itemize}
\item[(a)] \({\mathcal C}^+\cap {\mathcal C}^-=\gammaup\);
\item[(b)] \(T\res \bB_1=T^{+}+T^{-} \) where \( T^\pm =T\res {\mathcal C}^{\pm} \);
\item [(c)] \(\|T\|(\bB_1)=\|T^+\|(\bB_1)+\|T^{-}\|(\bB_1)\);
\item[(d)] \(\partial T^+\res \bB_1= Q\a{\gammaup} \) and \(\partial T^-\res \bB_1= -(Q-1)\a{\gammaup} \) ;
\item[(e)]  For any current \(S^{\pm}\) such that \(\supp (S^{\pm}) \subset \Sigma\cap \overline{\bB_1}\) and  \(\partial S^{\pm}=\partial (T^{\pm}\res \bB_1)\) we have that \(\|T^{\pm}\|(\bB_1)\le \|S^{\pm}\|(\bB_1)\).
\end{itemize}
\end{lemma}
\begin{proof}
Statement (a) is obvious. Statement (b) is a consequence of Corollary \ref{c:cone_cut} and of Theorem \ref{thm:center_manifold}(c)\&(d). Statement (c) comes directly from (a), (b) and the fact that $\|T\|(\gammaup)=0$. Statement (e) can be inferred from (c) and (d): for instance, if $S^+$ is as in the statement then $\partial (T^-+S^+)=\partial (T\res\bB_1)$ and by minimality of $T$
\begin{align*}
\|T^+\|(\bB_1)+\|T^{-}\|(\bB_1) &=\|T\|(\bB_1)\le \|T^-+S^+\|(\bB_1)\\
&\le\|S^+\|(\bB_1)+\|T^-\|(\bB_1).
\end{align*} 
The proof of point (d) follows the same idea of the proof of Corollary \ref{c:almgren_problem}. Indeed, first remark that $\partial T^+\res(\bB_1\setminus\gammaup)=0$, thus ${\rm spt}(\partial T^+)\cap\bB_1\subset\gammaup$. Let $\boldsymbol{r}$ be a retraction of a neighborhood of $\gammaup$ onto $\gammaup$. Since $\partial T^+\res\bB_1$ is a flat chain supported in $\gammaup$, Federer's flatness theorem, cf. \cite[Section 4.1.15]{Fed}, implies that $R := \boldsymbol {r}_\sharp (\partial T^+\res\bB_1) = \partial T^+\res\bB_1$. On the other hand, since $\partial (\partial T^+\res\bB_1)\res\bB_1 =0$, we also have $\partial R\res\bB_1=0$ and we conclude from the Constancy Theorem, cf. \cite[Section 4.1.7]{Fed}, that $R = c \a{\gammaup}\res\bB_1$ for some $c\in \R$. Thus $\partial T^+ = c\a{\gammaup}\res\bB_1$.

Fix a point $p\in\gammaup\cap\bB_1$ and recall that, from Theorem \ref{thm:decay_and_uniq}  and Theorem \ref{thm:center_manifold} (e), at every $p\in\gammaup\cap\bB_1$ there is a unique tangent cone to $T^+$ and it is $T^+_p = Q \a{\pi(p)^+}$, 
where $\pi(p)$ is tangent to $T_p{\mathcal M}$, by Theorem \ref{thm:center_manifold}, and $\pi(p)^+$ is the inner half portion of $\pi(p)$, where we consider ${\mathcal M}^+$ as a manifold with boundary $\gammaup$. Hence 
\[
\lim_{r\to 0} \partial((\iota_{p,r})_\sharp T^+)=\partial(Q \a{\pi(p)^+})=Q\a{T_p\gammaup}.
\]
Since we also know that 
\[
\lim_{r\to 0}\partial((\iota_{p,r})_\sharp T^+)=\lim_{r\to 0}(\iota_{p,r})_\sharp (c\a{\gammaup}\res\bB_1)=c \a{T_p\gammaup},
\] 
then we conclude $c=Q$. A similar argument holds for $T^-$.
\end{proof}

\begin{lemma}\label{lem:variations}
Under the same assumptions and with the same notations of Lemma~\ref{lem:cone-splitting}, for all \(X\in C^1_c(\bB_1,\mathbb R^{m+n})\) which are tangent to $\gammaup$, we have that
\begin{equation}\label{eq:fv+}
\delta T^{+}(X)=-\int X^\perp(x)\cdot \vec{H}_T(x)\,d\|T^+\|(x)
\end{equation}
where $X^\perp$ is the component of $X$ orthogonal to $\Sigma$ and \(\vec{H}_T(x)\) is the mean curvature vector of \eqref{e:mean_curv}. Analogously 
\[
\delta T^{-}(X)=-\int X^\perp(x)\cdot \vec{H}_T(x)\,d\|T^-\|(x)\, .
\]
\end{lemma}
\begin{proof}
This proof follows the same ideas of Section \ref{sec:prufs}. Without loss of generality, we focus on $T^+$. 
Since $T^+$ is stationary with respect to variations which are tangential to $\gammaup$ and $\Sigma$, we have the identity
\[
\delta T^+ (X) = - \int X (x) \cdot \vec{H}_T (x)\, d\|T^+\| (x)
\]
for all $X\in C^1_c (\bB_2)$ tangent to $\gammaup$,
where $\vec{H}_T$ is defined in \eqref{e:mean_curv} (cf. for instance \cite[Lemma 9.6]{Sim}). Note next that, by the explicit formula for $\vec{H}_T$ in \eqref{e:mean_curv}, $\vec{H}_T (x)$ is orthogonal to $T_x \Sigma$, which in turn contains the tangent plane to $T$ at $x$. Thus in the integral of the right hand side we can substitute $X$ with $X^\perp$. 
\end{proof}

In what follows we  let $\bp: \bp^{-1}(\mathcal M^+)\to \mathcal M^+ $ be the retraction of a normal neighborhood of \(\mathcal M^+\) to \(\mathcal M^+\).
In this section we will use Lemma \ref{lem:variations} with two specific choices of vector fields:
\begin{itemize}
\item the {\em outer variations}, where $X_o (p) := \phi \left(\frac{d({\bf p}(p))}{r}\right) \, (p - {\bf p}(p) )$\index{aalx X_o@$X_o$}.
\item the {\em inner variations}, where $X_i (p):= -Y({\bf p} (p))$ \index{aalx X_i@$X_i$} with
\begin{equation}\label{eq:vectorfield for inner variation} 
Y = \frac12\phi\left(\frac{d}{r}\right)\frac{ \nabla d^2}{\abs{\nabla d}^2}\,. 
\end{equation}
Note that $Y$ tangent is to $\mathcal{M}$ and to \(\gammaup\).
\end{itemize}
Consider now the map $F(p) := \sum_i \a{p+ N_i (p)}$ on ${\mathcal M}^+$ and the current ${\bf T}_F$
associated to its image, cf.~\cite{DS2}. By Lemma \ref{lem:variations} \index{aale Err_4^o@${\rm Err}_4^o$} \index{aale Err_5^o@${\rm Err}_5^o$}, 
\begin{align*}
\delta {\bf T}_F(X_o) & =\underbrace{(\delta {\bf T}_F(X_o)-\delta {T}^+(X_o))}_{{\rm Err}_4^o}+\delta {T}^+(X_o)\\
&\stackrel{\eqref{eq:fv+}}{=} {\rm Err}_4^o-\underbrace{\int X_o^\perp(x)\cdot \vec{H}_T(x)\,d\|T^+\|(x)}_{{\rm Err}_5^o}\,.
\end{align*}
Since $X_i$ is also tangent to $\gammaup$, by Lemma \ref{lem:variations}, we write\index{aale  Err_4^i@${\rm Err}_4^i$} \index{aale Err_5^i@${\rm Err}_5^i$}
\begin{align*}
\delta {\bf T}_F(X_i) &=\underbrace{(\delta {\bf T}_F(X_i)-\delta {T}^+(X_i))}_{{\rm Err}_4^i}+\delta {T}^+(X_i)\\
& \stackrel{\eqref{eq:fv+}}{=} {\rm Err}_4^i-\underbrace{\int X_i^\perp(x)\cdot \vec{H}_T(x)\,d\|T^+\|(x)}_{{\rm Err}_5^i}\,.
\end{align*}
Hence
\[
\delta{\bf T}_F(X_i)= {\rm Err}_4^i+{\rm Err}_5^i\,.
\]
\subsection{Outer variation} 

The following proposition holds (for the proof, see \cite[Theorem 4.2]{DS2}).

\begin{proposition}[Expansion of outer variations]\label{p:outer}
Consider the function $\varphi:=\phi\left(\frac{d(p)}{r}\right)$ and denote by $A$ and $H_\mathcal{M}$ the second fundamental form and the mean curvature of ${\mathcal M}^+$, respectively. Then \index{aale  Err_1^o@${\rm Err}_1^o$} \index{aale  Err_2^o@${\rm Err}_2^o$} \index{aale  Err_3^o@${\rm Err}_3^o$}
\begin{align}
\delta \mathbf{T}_F (X_o) = & \int_{{\mathcal M}^+}\! \Big(\varphi \, |DN|^2 + 
\sum_i (N_i \otimes D\varphi) : D N_i\Big)\nonumber\\
& - \underbrace{Q \int_{{\mathcal M}^+}\! \varphi \langle H_{\mathcal{M}}, \etaa\circ N\rangle}_{{\rm Err}_1^o} + \sum_{j=2}^3{\rm Err}_j^o 
\end{align}
where
\begin{align}
|{\rm Err}_2^o| & \leq C \int_{{\mathcal M}^+} |\varphi| |A|^2|N|^2\label{e:outer_resto_2}\\
|{\rm Err}_3^o| & \leq C \int_{{\mathcal M}^+}\!\! \Big(|\varphi| \big(|DN|^2 |N| |A| + |DN|^4\big)\nonumber\\ 
&\qquad\qquad +
|D\varphi| \big(|DN|^3 |N| + |DN| |N|^2 |A|\big)\Big)\label{e:outer_resto_3}.
\end{align}
\end{proposition}

\subsection{Inner variation} 
We denote by $\Xi_\varepsilon$ the one-parameter family of biLipschitz homeomorphisms  of $\mathcal{M}^+$ generated by $-Y$.  
We observe  that $X_i$ is then the infinitesimal generator of the one-parameter family of biLipschitz homeomorphisms $\Phi_\varepsilon$ of ${\bf p}^{-1} (\mathcal{M})$ defined by
\[
\Xi_\eps (p):= \Psi_\eps ({\bf p} (p)) + p - {\bf p} (p)\, .
\]
Therefore, we can follow the computations of \cite[Theorem 4.3]{DS2} to prove a suitable Taylor expansion for the inner variation. In what
follows, we will denote by $D^\mathcal{M} Y$ the $(1,1)$ tensor which expresses the covariant derivative of the vector field
$Y$ (which is tangent to $\mathcal{M}$), in particular, when $Z$ is a vector field tangent to $\mathcal{M}$, $D^{\mathcal{M}}_Z Y$ is the projection onto $T \mathcal{M}$ of the standard euclidean derivative $D_Z Y$. Accordingly ${\rm div}_{\mathcal{M}} Y$ will denote the trace of $D^\mathcal{M} Y$, namely 
\[
{\rm div}_{\mathcal{M}} Y =  \sum_{i=1}^m \langle D^\mathcal{M} Y (e_i), e_i \rangle
\]
where $e_1, \ldots, e_m$ is an orthonormal frame of $T \mathcal{M}$. Note that, in particular,
\[
{\rm div}_{\mathcal{M}} Y =  \sum_{i=1}^m \langle D_{e_i} Y, e_i \rangle\, .
\] 

\begin{proposition}[Expansion of inner variations]\label{p:inner}
The following formula holds:\footnote{Recall that each $N_j$ is a map taking values in $\mathbb R^{m+n}$ and thus we understand $D N_j$ as a map from $T \mathcal{M}$ into $\mathbb R^{m+n}$. More precisely, if $N_j= (N_j^1, \ldots , N^{m+n}_j)$ is the expression of $N_j$ into its components and if $Z$ is a vector field tangent to $\mathcal{M}$, then 
\[
DN_j (Z) = (D_X N_j^1, \ldots , D_Z N_j^{m+n})\, .
\]
With $DN_j D^\mathcal{M} Y$ we then understand the following map on $T \mathcal{M}$:
\[
DN_j D^\mathcal{M} Y (Z) = DN_j (D^\mathcal{M} Y (Z)) = (D_{D^{\mathcal{M}} Y (Z)} N_j^1, \ldots 
D_{D^{\mathcal{M}} Y (Z)} N_j^{m+n})\, .
\]
Accordingly, the scalar product $D N_j : ( D N_j D^{\mathcal M} Y)$ is given by
\[
D N_j : ( D N_j D^{\mathcal M} Y) = \sum_\ell \langle D_{e_\ell} N_j, D_{D^\mathcal{M} Y (e_\ell)} N_j \rangle =
\sum_{k, \ell} D_{e_\ell} N_j^k D_{D^{\mathcal{M}} Y (e_\ell)} N_j^k\,  
\]
where $e_1, \ldots , e_m$ is an orthonormal frame on $T \mathcal{M}$.
}
\begin{align}
\delta \mathbf{T}_F (X_i) &= \int_{{\mathcal M}^+} \left( \sum_j  D N_j : ( D N_j D^{\mathcal M} Y)-\frac{|DN|^2}{2} {\rm div}_{\mathcal M}\, Y\right)\nonumber\\
&\qquad\qquad  + \sum_{j=1}^3{\rm Err}_j^i,\label{e:inner} 
\end{align}
where\index{aale  Err_1^i@${\rm Err}_1^i$} \index{aale  Err_2^i@${\rm Err}_2^i$} \index{aale  Err_3^i@${\rm Err}_3^i$}
\begin{align}
{\rm Err}_1^i & = Q \int_{{\mathcal M}^+}\big( \langle H_{\mathcal{M}}, \etaa \circ N\rangle\, {\rm div}_{\mathcal M} Y + \langle D_Y H, \etaa\circ N\rangle\big)\,,\label{e:inner_resto_1}\allowdisplaybreaks\\
|{\rm Err}_2^i| &  \leq C \int_{{\mathcal M}^+} |A|^2 \left(|DY| |N|^2  +|Y| |N|\, |DN|\right), \label{e:inner_resto_2}\allowdisplaybreaks\\
|{\rm Err}_3^i| & \leq C \int_{{\mathcal M}^+} \Big( |Y| |A| |DN|^2 \big(|N| + |DN|\big)\nonumber\\
&\qquad\qquad + |DY| \big(|A|\,|N|^2 |DN| + |DN|^4\big)\Big)\label{e:inner_resto_3}\, .
\end{align}
\end{proposition}

The proof of the previous theorem follows literally the same computations of \cite[Section 4.3]{DS2}. The only subtle point is that in the final part of that proof the integration by parts needed to handle the term $J_2$ in \cite[Eq. (4.17)]{DS2} is valid in our context because the vectorfield $Z$, on which the integration by parts is performed, vanishes on $\gammaup$.

\section{Key identities}

In this section we use the Taylor expansions of the first variations to derive the key identities which lead to the monotonicity of the frequency function. We introduce therefore the quantity
\[ 
G(r):= - \frac{1}{r^2} \int_{\mathcal{M}^+} \phi'\left(\frac{d}{r}\right)\frac{d}{\abs{\nabla d}^2}\sum_{j}|DN_j\cdot \nabla d|^2 \,. 
\]
\begin{proposition}\label{prop:quasi_fatta}
The following two inequalities hold
\begin{align}
|D(r)-E(r)| & \le \sum_{j=1}^5 |{\rm Err}_j^o| \label{e:quasi_fatta_1}\\
\left|D'(r)-\left(\frac{m-2}{r}+O(1)\right)D(r)-2G(r)\right| & \leq \frac{2}{r}\left(\sum_{j=1}^5|{\rm Err}_j^i|\right)\, . \label{e:quasi_fatta_2}
\end{align}
\end{proposition}
\begin{proof}
For the first identity it suffices to check that
\[
\int_{{\mathcal M}^+}\! \Big(\varphi \, |DN|^2 + 
\sum_i (N_i \otimes D\varphi) : D N_i\Big)=D(r)-E(r)\,,
\]
which is an obvious computation. For the second identity we need to show that
\begin{align*}
 &\int_{{\mathcal M}^+} 2\sum_j  D N_j : ( D N_j D^{\mathcal M} Y)- |DN|^2 {\rm div}_{\mathcal M}\, Y\\
&=rD'(r)-((m-2)+O(r))D(r)-2rG(r)
\end{align*}
Recalling the definition of \(Y\) in \eqref{eq:vectorfield for inner variation}, that is
\begin{equation*}
Y = \frac12\phi\left(\frac{d}{r}\right)\frac{ \nabla d^2}{\abs{\nabla d}^2}\,,
\end{equation*}
 we  easily compute, using Lemma \ref{l:good_vector_field on M} (b) (c) and \eqref{e:hessian d}
 \begin{align}
 D^{\mathcal M}Y&=\frac{d}{r}\phi'\left(\frac{d}{r}\right)\frac{ \nabla d\otimes \nabla d}{\abs{\nabla d}^2}+\frac 12\phi\left(\frac{d}{r}\right)\frac{\nabla^2d^2}{\abs{\nabla d}^2}\nonumber\\
&\qquad-\phi\left(\frac{d}{r}\right)\frac{2 (d\nabla^2d\nabla d) \otimes \nabla d)}{\abs{\nabla d}^4}\nonumber
 \\
 &=\frac{d}{r}\phi'\left(\frac{d}{r}\right)\frac{ \nabla d\otimes \nabla d}{\abs{\nabla d}^2} +\phi\left(\frac{d}{r}\right)\big(g+O(d)\big)\, ,\label{eq:derivative of vectorfield for inner variation}
 \end{align}
where we recall that $g$ is the metric induced on $\mathcal{M}$ by the Euclidean ambient manifold.
In particular
\[
\operatorname{div}_{\mathcal{M}}(Y)=\frac{d}{r}\phi'\left(\frac{d}{r}\right)+\phi\left(\frac{d}{r}\right)\left(m+O(d)\right).
\] 
Hence, using also that, on \(\{\phi\ne 0\}\), \(d=O(r)\), we obtain

\begin{eqnarray*}
& & \int_{{\mathcal M}^+} 2\sum_j  D N_j : ( D N_j D^{\mathcal M} Y)- |DN|^2 {\rm div}_{\mathcal M}\, Y\\
& = & \frac{2}{r} \int_{\mathcal{M}^+} \phi'\left(\frac{d}{r}\right)\frac{d}{\abs{\nabla^\mathcal{M}d}^2}\sum_{j}|DN_j \nabla d|^2\\ 
&  &+\int_{{\mathcal M}^+} \phi\left(\frac{d}{r}\right)(2-m+O(r))|DN|^2-\int_{{\mathcal M}^+}\phi'\left(\frac{d}{r}\right)|DN|^2\\
& = &-2r G(r)-\big((m-2)+O(r))D(r)+rD'(r),
\end{eqnarray*}

which concludes the proof.
\end{proof}

\section{Estimates on the error terms}

\subsection{Families of subregions} In order to estimate the various error terms we select an appropriate family of subregions of  $\mathscr{B}^+_r :=\{p\in \pi_0^+: d (\boldsymbol{\varphi}(p)) <r\})$ \index{aalBcal@$\mathscr{B}^+_r$}. First of all we introduce a suitable family of cubes in the Whitney decomposition:

\begin{definition} The family $\mathcal{T} \subset \mathscr{W}$ consists of :
\begin{itemize}
\item[(i)] all $L\in \sW^{\be}\cup \sW^{\bh}$ which intersect $\mathscr{B}^+_r$;
\item[(ii)] all $L\in \sW^{\be}$ which are domains of influence of some $L'\in \sW^{\bf n}$ intersecting $\mathscr{B}^+_r$, i.e., $L'\in \sW^{\bf n} (L)$ (cf. Definition \ref{d:domains}).
\end{itemize}
\end{definition}

Next, for any $L\in \mathcal{T}$ note that 
\[
{\rm sep} (L, \mathscr{B}^+_r) := \inf \{|q-p|: q\in L, p\in \mathscr{B}^+_r\} \leq 3 \sqrt{m} \ell (L)\, . 
\]
For each such $L$ we define an appropriate ``satellite'' ball $B (L)$ with the following properties:
\begin{itemize}
\item[(A)] $B (L)$ has radius comparable to $\ell (L)$ (say $\ell (L)/4)$);
\item[(B)] the concentric ball with twice the radius is contained in $\mathscr{B}^+_r$;
\item[(C)] $B (L)$ is close to $L$ (comparably to $\ell (L)$).
\end{itemize}
If \(B_{\ell (L)/2} (c(L))\subset \mathscr{B}^+_r\), then we simply set \(B(L)=B_{\ell (L)/4} (c(L))\).

If instead \(B_{\ell (L)/2} (c(L))\not\subset \mathscr{B}^+_r\), we then use the following selecting procedure.

\begin{itemize}
\item[(i)] First consider a point $q\in \partial \mathscr{B}^+_r$ at minimum distance from $L$.
\item[(ii)] Observe that, since $L\in \sW$, it is a non-boundary cube.  Thus $\dist (q, \gamma)\geq \ell (L)$ and in
particular $d (\boldsymbol{\varphi} (q)) = r$.
\item[(iii)] Let $v$ be the exterior unit normal to $\partial \mathscr{B}^+_r$ at $q$ and let 
$q_L := q - \frac{\ell (L)}{2} v$.
\item[(iv)] Recalling claim (S) in Lemma \ref{l:good_vector_field on M} and the estimates on $\boldsymbol{\varphi}$ we see that $\partial \mathscr{B}_r^+\setminus \gamma$ is locally convex and that the principal curvatures of $\partial\mathscr{B}_r^+\setminus \gamma$ can be assumed to be all smaller than $\frac{2}{r}$. Since $\ell (L) <r$, this implies that $B_{\ell(L)/2} (q_L) \subset \mathscr{B}^+_r$. We finally set 
$B(L) := B_{\ell (L)/4} (q_L)$.
\end{itemize}

\begin{definition}
Given a cube $L\in  \mathcal{T}$, the ball $B(L)$\index{aalB(L)@$B(L)$} chosen above will be called the {\em satellite ball of $L$}\index{Satellite ball}. 
\end{definition}

Note that, by simple geometric arguments and by the properties of $d$, we can assume that 
\begin{equation}\label{e:dist_5_ell}
|q_L-c (L)| \leq 5 \sqrt{m} \ell (L)\qquad\textrm{and}\qquad  \dist(L,q_L) \le 4 \sqrt{m} \ell(L).
\end{equation}

We next select a suitable countable subfamily $\mathscr{T}$ of $\mathcal{T}$ with the property that, for any pair of distinct $H, L\in \mathscr{T}$, the corresponding balls $B(L)$ and $B (H)$ are disjoint. We denote by $S$ the supremum of $\ell(L)$ for $L\in \mathcal{T}$. We start selecting a maximal subfamily $\mathscr{T}_1$ in $\mathcal{T}$ of cubes $L$ with $\ell (L) \geq S/2$ such that the corresponding balls $B(L)$ are pairwise disjoint. We then add to $\mathscr{T}_1$ a maximal subfamily $\mathscr{T}_2$ in $\mathcal{T}$ of cubes $L$ with $S/4 \leq \ell (L) \leq S/2$ such that the balls $B(L')$ corresponding to $L'\in \mathscr{T}_1 \cup \mathscr{T}_2$ are all pairwise disjoint. We proceed inductively with the selection of the family $\mathscr{T}_k\subset \mathcal{T}$ such that:
\begin{itemize}
\item[(i)] it consists of cubes with side $2^{-k-1} S \leq \ell (L) \leq 2^{-k} S$;
\item[(ii)] the balls $B(L')$ with $L'\in \mathscr{T}_1 \cup \ldots \cup \mathscr{T}_{k-1} \cup \mathscr{T}_k$ are pairwise disjoint;
\item[(iii)] $\mathscr{T}_k$ is maximal among the families satisfying (i) and (ii).
\end{itemize}
 $\mathscr{T}$ is the union of all the $\mathscr{T}_j$. A simple geometric argument and
\eqref{e:dist_5_ell} ensures that
\begin{itemize}
\item[(Cov)] If $H\in \mathcal{T}$, then there is $L\in \mathscr{T}$ such that 
the distance between $H$ and $L$ is at most $20 \sqrt{m} \ell (L)$  and even though there might be more than one $L$, we fix for each $H$ an arbitrary choice of an $L$ with such a property. 
\end{itemize}
Therefore we can partition $\mathcal{T}$ into (disjoint!) families $\mathcal{T} (L)$ with $L\in \mathscr{T}$ with the property that
for each $H\in \mathcal{T} (L)$, the distance between $H$ and $L$ is at most $20\sqrt{m} \ell(L)$ and $\ell(H) \le 2 \ell (L)$. For each $L\in \mathscr{T}$ we denote by $\sW (L)$ the family of cubes
\[
\bigcup_{H\in \mathcal{T} (L)} \sW^{\mathbf n} (H) \cup \{H\}\, .
\] 
Furthermore we denote by $\mathcal{U} (L)$ the following region in $\mathcal{M}^+$:
\[
\bigcup_{H\in \sW (L)} \mathbf{\Phi} (H)\, .
\]
From now on we fix an enumeration $\{L_i\}$ of $\mathscr{T}$ and we denote:
\begin{itemize}
\item by $\mathcal{U}_i$ the corresponding regions $\mathcal{U} (L_i)\cap \mathcal B_r^+$ \index{aalu\mathcal{U}_i@$\mathcal{U}_i$};
\item by $\mathcal{B}^i$ the regions $\mathbf{\Phi} (B (L_i))$\index{aalb\mathcal{B}^i@$\mathcal{B}^i$};
\item by $\ell_i$ the scale $\ell (L_i)$\index{aall\ell_i@$\ell_i$}.
\end{itemize}
where, here and in the following, we set 
\[
\mathcal B_r^+=\mathcal M^+\cap \{d<r\}\index{aalb\mathcal B_r^+@$\mathcal B_r^+$}\,.
\]
\subsection{Lower and upper bounds in the subregions}
First of all observe that
\begin{equation}\label{e:controllo_cutoff_0 g}
c \frac{\ell_i}{r}\leq \inf_{\bp^{-1}(\mathcal{B}^i)} \varphi  
\end{equation}
for a geometric constant $c$ (recall that $\varphi (p) = \phi \big(\frac{d (\bp (p))}{r}\big)$). In particular 
\[
\sup_{\bp^{-1}(\mathcal{U}_i)}  \varphi  - \inf_{\bp^{-1}(\mathcal{U}_i)}  \varphi  \leq C \frac{\ell_i}{r} \leq C \inf_{\bp^{-1}(\mathcal{B}^i)}  \varphi\, ,  
\]
which leads to
\begin{equation}\label{e:controllo_cutoff g}
\sup_{\bp^{-1}(\mathcal{U}_i)}  \varphi\leq C \inf_{\bp^{-1}(\mathcal{B}^i)}  \varphi \, ,
\end{equation}
where $C$ is a geometric constant. Since we have \(\bp^{-1}(\mathcal{U}_i)\cap \mathcal M^+=\mathcal U_i\) and the same for \(\mathcal B^i\), the above estimates, when restricted to \(\mathcal M^+\), become:
\begin{equation}\label{e:controllo_cutoff_0}
c \frac{\ell_i}{r}\leq \inf_{\mathcal{B}^i} \varphi  
\end{equation}
and
\begin{equation}\label{e:controllo_cutoff}
\sup_{\mathcal{U}_i}  \varphi\leq C \inf_{\mathcal{B}^i}  \varphi \,.
\end{equation}
Observe that
\[
\max \{\ell (H): H\in \mathscr{W} (L_i)\} \leq C \ell_i
\]
and
\[
\sum_{H\in \mathscr{W} (L_i)} \ell (H)^m \leq C \ell_i^m
\]
Thus, as a consequence of the estimates in Theorem \ref{thm:cm_app}  and Corollary \ref{c:cm} (b) (namely, applying the corresponding estimates in each cube in $\mathscr{W} (L_i)$ and summing the respective contributions) we achieve the following:
\begin{align}
{\rm Lip} \left(\left. N\right|_{\mathcal{U}_i}\right) &\leq C \eps_1^\sigmaexpcm \ell_i^\sigmaexpcm \label{e:bad_regions_1}\\
\|N\|_{C^0 (\mathcal{U}_i)} + \sup_{p\in \supp (T^+) \cap \bp^{-1} (\mathcal{U}_i)} |p-\bp (p)| &\leq C \eps_1^{\sfrac{1}{2m}} \ell_i^{1+\alpha_\bh}\label{e:bad_regions_2}\\
\|T^+-\mathbf{T}_F\| (\bp^{-1} (\mathcal{U}_i)) &\leq C \eps_1^{1+\sigmaexpcm} \ell_i^{m+2+\sigmaexpcm}\label{e:bad_regions_3}\\
\int_{\mathcal{U}_i} |DN|^2 &\leq C \eps_1 \ell_i^{m+2-2\alpha_\be}\label{e:bad_regions_4}\\
\int_{\mathcal{U}_i} |\etaa\circ N| &\leq C \eps_1 \ell_i^{2+m +\frac{\sigmaexpcm}{2}} + C \int_{\mathcal{U}_i} |N|^{2+\sigmaexpcm}\label{e:bad_regions_5}\, .
\end{align}
Note in particular that \eqref{e:bad_regions_5} follows from choosing $a=1$ in \eqref{e:cm_app5} and $\mathcal{V} = \mathcal{L}$. 

The second important ingredients in order to estimate the various error is the following lemma.

\begin{lemma}\label{l:stimazze}
Under the assumptions of Theorem \ref{thm:ff_estimate_current}, for a sufficiently small $r$ the following inequalities hold:
\begin{align}
\eps_1 \sum_i \ell_i^{m+2+2 \alpha_\bh}\inf_{\bp^{-1}(\mathcal{B}^i)} \varphi\;  & \leq C D(r)\label{e:stimazza1}\\
\eps_1 \sum_i \ell_i^{m+2+2 \alpha_\bh} &\le { C \int_{\mathcal B_r^+} \abs{DN}^2} \leq C (D(r) + r D'(r))\, ,\label{e:stimazza2}
\end{align}
for a geometric constant $C$. 
Moreover we have
\begin{equation}\label{e:stimazza3-Jonas}
\eps_1 \sup_i \ell_i \le C \left( r D(r) \right)^\frac{1}{m+3+\alpha_\bh} \text{ and } \eps_1\sup_i \left(\inf_{\bp^{-1}(\mathcal{B}^i)} \varphi \, \ell_i\right) \le C D(r)^\frac{1}{m+2+\alpha_\bh}\, .
\end{equation}
\end{lemma}
\begin{proof} First of all observe that every cube $L_i\in \mathscr{T}$ belongs to either $\mathscr{W}^\bh$ or to $\mathscr{W}^\be$.
 For every cube $L_i \in \mathscr{T} \cap \mathscr{W}^{\bh}$, as a consequence of  Corollary \ref{c:domains}, we must have $L_i \cap {\mathcal B}_r^+ \neq \emptyset$. Hence $\mathcal{B}^i \subset \mathcal{M} \cap \bC_{2\sqrt{m}\ell(L_i)}(p_{L_i})$ and therefore Proposition~\ref{prop:sep}(S3) applies. Recalling that $\mathcal{G} (N(x), Q\a{\etaa \circ N (x)})\leq |N|$, for every cube $L_i \in \mathscr{T} \cap \mathscr{W}^{\bh}$ we can estimate
\begin{equation}\label{e:stimazza4}
\int_{\mathcal{B}^i}  |N|^2 \geq c_0 \eps_1^{\sfrac{1}{m}} \ell_i^{m+2+2\alpha_\bh} \, .
\end{equation}
By estimate \eqref{e:split_1} in Proposition \ref{p:splitting} , for every $L_i \in \mathscr{T}\cap \mathscr{W}^\be$ we have
\begin{equation}\label{e:stimazza5}
\int_{\mathcal{B}^i} \varphi |D N|^2 \geq c_0 \eps_1 \ell_i^{m+2-2\alpha_\be}\inf_{\mathcal{B}^i} \varphi =c_0 \eps_1 \ell_i^{m+2-2\alpha_\be}\inf_{\bp^{-1}(\mathcal{B}^i)} \varphi\, . 
\end{equation}
Summing the last two inequalities over $i$, using that $\{\mathcal{B}^i\}$ are disjoint and contained in $\{d<r\}\cap \mathcal{M}^+$ and the simple observation that $2+\alpha_\bh \ge 2-2\alpha_\be$, we easily conclude
\[
\eps_1 \sum_i \ell_i^{m+2+2 \alpha_\bh} \inf_{\bp^{-1}(\mathcal{B}^i)} \varphi \leq
C_0 \int_{\mathcal B_r^+} \left(|N|^2 + \varphi |DN|^2\right)\, .
\]
Thus, \eqref{e:stimazza1} can be inferred from \eqref{e:poincare_1000}. 

Note that, analogously, for $L_i \in \mathscr{T}\cap \mathscr{W}^\be$ we have also
\begin{equation}\label{e:stimazza6}
\int_{\mathcal{B}^i}  |D N|^2 \geq c_0 \eps_1  \ell_i^{m+2-2\alpha_\be}\, .
\end{equation}
Arguing as above with \eqref{e:stimazza6} in place of \eqref{e:stimazza5} and exploiting that $2+\alpha_\bh \ge 2-2\alpha_\be$, we conclude
\[
\eps_1 \sum_i \ell_i^{m+2+2 \alpha_\bh} \leq
C_0 \int_{\mathcal B_r^+}|DN|^2\, .
\]
Since $\phi' (t) =- 2$ on $[1/2,1]$, clearly
\[
 \int_{\{r/2 < d<r\}\cap \mathcal{M}^+}|DN|^2\leq r D'(r)\, .
\]
On the other hand we trivially have 
\[
 \int_{\{d<r/2\}\cap \mathcal{M}^+}|DN|^2\leq D(r)\, .
\]
Thus, \eqref{e:stimazza2} follows easily.

Finally the second estimate of \eqref{e:stimazza3-Jonas} is a direct consequence of \eqref{e:stimazza1} and the first follows combining \eqref{e:stimazza1} with \eqref{e:controllo_cutoff_0 g}.
\end{proof}

\subsection{Estimates on the error terms} We are ready to prove the main estimates on the various error terms appearing in the inequalities of Proposition \ref{prop:quasi_fatta}. We first introduce the auxiliary term
\begin{equation}\label{e:Sigma}
S (r) := \int \phi \left(\frac{d}{r}\right) |N|^2\, .
\end{equation}

\begin{proposition}\label{p:quasi_fatta_2}
There are positive numbers $C$ and $\tau$ such that
\begin{align}
|{\rm Err}^o_1|+|{\rm Err}^o_3| + |{\rm Err}^o_4| &\leq C D(r)^{1+\tau}\, \label{e:errori_esterni_134}\\
|{\rm Err}^o_2| &\leq C S (r) \leq C r^2 D(r)\label{e:errori_esterni_2}\\
|{\rm Err}^o_5| & \leq C S(r) + C D(r)^{1+\tau} \leq C r^2 D(r) + C D(r)^{1+\tau}\label{e:errori_esterni_5}\\
|{\rm Err}^i_1|+|{\rm Err}^i_3| + |{\rm Err}^i_4| &\leq C  D(r)^\tau (D(r) + r D'(r))\label{e:errori_interni_134}\\
|{\rm Err}^i_2| &\leq C r D(r) \label{e:errori_interni_2}\\
|{\rm Err}^i_5|&\leq C r D(r) + C D(r)^\tau (D(r)+r D'(r)) .\label{e:errori_interni_5}
\end{align}
\end{proposition}
\begin{proof} Since $\sigmaexpcm$ is independent of $\alpha_\be, \alpha_\bh$ (compare Theorem \ref{thm:cm_app}), we can choose $\alpha_\be, \alpha_\bh$ such that 
\[
\frac{\sigmaexpcm}{2} \ge 4 \alpha_\bh \ge 4\alpha_{\be}\,.
\]
We let \(\tau\ll \alpha_\be\le  \alpha_\bh\le \sigmaexpcm/8\).

\medskip

{\bf Proof of \eqref{e:errori_esterni_134}.} Recalling that $\|\boldsymbol{\varphi}\|_{C^{3,\kappa}}\leq C \eps_1^{\sfrac{1}{2}}$, which in turn implies $\|H_{\mathcal{M}}\|_{C^0(\mathcal{M}^+)} \leq C \eps_1^{\sfrac{1}{2}}$, we get from 
\eqref{e:bad_regions_5} 
\begin{align*}
|{\rm Err}^o_1| &\leq C \int_{\mathcal{M}^+} \varphi |H_{\mathcal{M}^+}| |\etaa\circ N|\\
&\stackrel{ \eqref{e:controllo_cutoff}}{\leq} C \eps_1^{\sfrac{1}{2}} \sum_j \left(\sup_{\mathcal{U}_j}\varphi\, \eps_1 \ell_j^{2+m+\sfrac{\sigmaexpcm}{2}} + C \int_{\mathcal{U}_j} \varphi\, |N|^{2+\sigmaexpcm}\right)\\
&\stackrel{\eqref{e:controllo_cutoff}}{\leq} C \eps_1^{\sfrac{1}{2}} \sum_j \left(\inf_{\mathcal{B}_j} \varphi\; \eps_1 \ell_j^{2+m+\sfrac{\sigmaexpcm}{2}} + C \int_{\mathcal{U}_j} \varphi\; |N|^{2+\sigmaexpcm}\right)\\
&\stackrel{\eqref{e:bad_regions_2}}{\leq} C \eps_1^{\sfrac{1}{2}} \sum_j \left(\inf_{\mathcal{B}_j} \varphi\; \eps_1 \ell_j^{2+m+4\alpha_\bh} + C \ell_j^{8\alpha_\bh} \int_{\mathcal{U}_j} \varphi |N|^{2}\right)\\
&\stackrel{\eqref{e:stimazza1}\&\eqref{e:stimazza3-Jonas}}{\leq} C D(r)^{1+\tau} + C D(r)^\tau \int_{\mathcal{B}_r^+} \varphi  |N|^{2}\, ,
\end{align*}
where in the last line we have used also that the intersection of distinct domains $\mathcal{U}_j$ has zero measure. 
Using \eqref{e:poincare_1000} we conclude
\[
|{\rm Err}^o_1| \leq C D(r)^{1+\tau}\, .
\]
Concerning ${\rm Err}^o_3$, from Proposition \ref{p:outer} and recalling that $|D\varphi|\leq \frac{C}{r}$ we get
\begin{align*}
\left|{\rm Err}^o_3\right| &\leq  \underbrace{\int \varphi \left(|DN|^2 |N| + |DN|^4\right)}_{I_1}
+ C\underbrace{r^{-1} \int_{\mathcal{B}^+_r} |DN|^3 |N|}_{I_2}\\
&\quad  + C\underbrace{r^{-1} \int_{\mathcal{B}_r^+} |DN| |N|^2}_{I_3}\,.
\end{align*}
We estimate separately the three terms:
\begin{eqnarray*}
 	I_1  & \leq & \left(\sup_{\mathcal{B}_r^+} \abs{N} + \sup_{\mathcal{B}_r^+} |DN|^2\right) \int_{\mathcal{B}^+_r} \varphi \abs{DN}^2
	\\
	& \le & C\sup_i\left(\sup_{\mathcal{U}_i} \abs{N} +{\rm Lip} \left(N\big|_{\mathcal{U}_i}\right)\right) \int_{\mathcal{B}^+_r} \varphi \abs{DN}^2
	\\
	 &\stackrel{\eqref{e:bad_regions_1}\& \eqref{e:bad_regions_2}}{\le} & C \sup_{ i } \ell_i^{2\sigmaexpcm} \int_{\mathcal{B}^+_r} \varphi \abs{DN}^2 \le C D(r)^{1+\tau}\,.
\end{eqnarray*}
Moreover, recalling that $\sigmaexpcm \geq 4 \alpha_\be$,
\begin{eqnarray*}
I_2 & \stackrel{\eqref{e:bad_regions_1}\&\eqref{e:bad_regions_2}}{\leq} & C r^{-1} \sum_j \eps_1^{\sfrac{1}{2m}+\sigmaexpcm} \ell_j^{1+\alpha_\bh + \sigmaexpcm} \int_{\mathcal{U}_j} \abs{DN}^2\\
& \stackrel{\eqref{e:bad_regions_4}}{\leq} & C r^{-1} \sum_j \eps_1^{1+\sfrac{1}{2m}+\sigmaexpcm} \ell_j^{m+3+\alpha_\bh + \sigmaexpcm - 2\alpha_{\be}}
\\
& \stackrel{\eqref{e:controllo_cutoff_0}}{\leq} &C  \sum_j  \ell_j^{m+2+7\alpha_{\bh}} \inf_{\mathcal{B}^j} \varphi \stackrel{\eqref{e:stimazza1}\,\&\,\eqref{e:stimazza3-Jonas}}{\leq} C D (r)^{1+\tau},
\end{eqnarray*}
and
\begin{align*}
I_3 &\stackrel{\eqref{e:bad_regions_1}}{\leq}
C r^{-1} \sum_j \eps_1^{\sigmaexpcm} \ell_j^{\sigmaexpcm} \int_{\mathcal{U}_j} |N|^2
\stackrel{\eqref{e:stimazza3-Jonas}}{\leq} C  r^{-1}D (r)^{\tau}  \int_{\mathcal{B}^+_r} |N|^2\\
&\stackrel{\eqref{e:poincare_1000}}{\leq} Cr D (r)^{1+\tau}\,,
\end{align*}
provided $\tau>0$ is sufficiently small.

Recalling  that 
\[
{\rm Err}^o_4 = \delta (\mathbf{T}_F - T^+) (X^o)\, ,
\]
we can estimate
\begin{align*}
|{\rm Err}^o_4| &\leq \int_{\bp^{-1} (\mathcal{B}_r^+)} |DX^o|\, d \|\mathbf{T}_F - T^+\|\, .
\end{align*}
Since
\[
|DX_o (p)| \leq C \left(\frac{|p - \bp (p)|}{r} + \varphi (p)\right)\, ,
\]
we can estimate
\begin{eqnarray*}
|{\rm Err}^o_4| & \leq & C  \sum_j \int_{\bp^{-1} (\mathcal{U}_j)} \left(\frac{|p- \bp (p)|}{r} + \varphi (p) \right)  d \|\mathbf{T}_F - T^+\|\\
&\stackrel{\eqref{e:bad_regions_2}\&\eqref{e:bad_regions_3}}{\leq} &
C \sum_j \left(r^{-1} \eps_1^{\sfrac{1}{2m}} \ell_j^{1+\alpha_\bh} + \sup_{\bp^{-1}(\mathcal{U}_j)} \varphi\right) \eps_1^{1+\sigmaexpcm} \ell_j^{m+2+\sigmaexpcm}\\
&\stackrel{\eqref{e:controllo_cutoff_0 g}\&\eqref{e:controllo_cutoff g}}{\leq} &
C \sum_j  \inf_{\bp^{-1}(\mathcal{B}_j)} \varphi\, \eps_1^{1+\sigmaexpcm} \ell_i^{m+2+\sigmaexpcm}
\stackrel{\eqref{e:stimazza1}\&\eqref{e:stimazza3-Jonas}}{\leq} C D(r)^{1+\tau}\, .  
\end{eqnarray*}

\medskip

{\bf Proof of \eqref{e:errori_esterni_2}.} Since $\|A_{\mathcal{M}^+}\|_{C^0} \leq C \|\mathbf{\phi}\|_{C^2} \leq C \eps_1^{\sfrac{1}{2}}$, it follows easily that
\[
|{\rm Err}^o_2|\leq C S(r) \leq C \int_{\mathcal{B}_r^+} |N|^2\, .
\]
Thus the estimate follows from \eqref{e:poincare_1000}.

\medskip

{\bf Proof of \eqref{e:errori_esterni_5}.} Recall that
\[
{\rm Err}^o_5 = - \int X^\perp_o \cdot \vec{H}_T (x) \,d\|T^+\| (x)\, ,
\]
where $\vec{H}_T (x)$ is the trace of the second fundamental form $A_\Sigma$ of $\Sigma$ restricted to the tangent space $\vec{T} (x)$ to the current $T^+$ at $x$. For further use we introduce the notation $h (\vec\lambda)$ for the trace
of $A_\Sigma$ on the $m$-plane oriented by the $m$-vector $\vec{\lambda}$. In particular $\vec{H}_T (x) =
h (\vec{T} (x))$. We can therefore write
\begin{equation}\label{e:errore_5_spezzato}
|{\rm Err}^o_5| \leq \underbrace{\left|\int \langle X_o^\perp, h (\vec{T}_F)\rangle d\|\mathbf{T}_F\|\right|}_{I_1}
 + C {\|A_\Sigma\|}_0 \underbrace{\int |X_o^\perp| d\|T^+-\mathbf{T}_F\|}_{I_2}\, .
\end{equation} 
Recall that $\|A_\Sigma\|_0 \leq \varepsilon_1^{\sfrac{1}{2}}$. Since $|X^o (p)|\leq C \varphi (\bp (p))$,  the second term is estimated by $C D(r)^{1+\tau}$ by arguing as in the bound for ${\rm Err}^o_4$.  As for the first term note that 
\[
|X_o^\perp (p)| \leq \varphi (\bp (p)) |\bp_{T_p \Sigma^\perp} (p-\bp (p))|\leq C \varphi (\bp (p)) {\|A_\Sigma\|}_0 |p-\bp (p)|^2\, .
\]
Hence, using the Lipschitz bound for $N$ to pass the integration on the domain $\mathcal{B}^+_r$, we conclude
\[
I_1 \leq C \int \varphi |N|^2 = C S(r) \stackrel{\eqref{e:poincare_1000}}{\leq} C r^2 D(r)\, .
\]

\medskip
We now estimate  the error terms coming from  inner variations. First let us record here the following easy consequence of \eqref{eq:vectorfield for inner variation} and \eqref{eq:derivative of vectorfield for inner variation}:
\begin{equation}\label{eq:bounds on inner variation} \abs{Y(p)} \le \varphi(\bp(p)) \,d(\bp(p)) \quad \abs{DY}(p) \leq C \mathbf{1}_{\mathcal{B}^+_r} (\bp (p))\,.  
\end{equation}
{\bf Proof of \eqref{e:errori_interni_134}.} 
By Proposition \ref{p:inner},
\begin{eqnarray*}
|{\rm Err}^i_1| & \leq &  C \int_{\mathcal{B}^+_r} (|H_\mathcal{M}| + |DH_\mathcal{M}|) |\etaa\circ N|
\leq C \int_{\mathcal{B}^+_r} |\etaa\circ N|\\
&\stackrel{\eqref{e:bad_regions_5}}{\leq} & \sum_j \left(\eps_1 \ell_j^{m+2+\sfrac{\sigmaexpcm}{2}} + \int_{\mathcal{U}_j}  |N|^{2+\sigmaexpcm}\right)\\
&\stackrel{\eqref{e:bad_regions_2}}{\leq} &  \sum_j \left(\eps_1 \ell_j^{m+2+\sfrac{\sigmaexpcm}{2}} + \ell_j^{\sigmaexpcm} \int_{\mathcal{U}_j}  |N|^{2}\right)\\
&\stackrel{\eqref{e:stimazza2}\&\eqref{e:stimazza3-Jonas}}{\leq} & C D(r)^\tau (D(r)+r D'(r)) + C D(r)^{\tau} \int_{\mathcal{B}_r^+} |N|^2\\
& \stackrel{\eqref{e:poincare_1000}}{\leq}& C D(r)^\tau (D(r)+r D'(r))\, .
\end{eqnarray*}
Using \eqref{eq:bounds on inner variation} and Proposition \ref{p:inner},
\[
|{\rm Err}^i_3| \leq C \int_{\mathcal{B}^+_r} (|DN|^3 + |DN|^2|N| + |DN||N|^2)\,.
\]
The third integrand can be treated like $I_3$ in the estimate of ${\rm Err}^o_3
$ and thus can be bounded by $C r^2 D(r)^{1+\tau}$. As for the first two we argue as follows:
\begin{align*}
&  \int_{\mathcal{B}^+_r} (|DN|^3 + |DN|^2|N|) \stackrel{\eqref{e:bad_regions_1}\&\eqref{e:bad_regions_2}}{\leq} 
 \sum_j \eps_1^\sigmaexpcm \ell_j^\sigmaexpcm \int_{\mathcal{U}_j} |DN|^2\\
& \stackrel{\eqref{e:stimazza3-Jonas}}{\leq} C D(r)^\tau \int_{\mathcal{B}^+} |DN|^2 \leq C D(r)^\tau (D(r) + r D'(r))\, .
\end{align*} 
Concerning ${\rm Err}^i_4$, using again  \eqref{eq:bounds on inner variation}, we estimate
\begin{eqnarray}
|{\rm Err}^o_4| &\leq & C \sum_j \|\mathbf{T}_F - T^+\| (\bp^{-1} (\mathcal{U}_i))\\
&\stackrel{\eqref{e:bad_regions_3}}{\leq}& C \sum_j \eps_1^{1+\sigmaexpcm} \ell_j^{m+2+\sigmaexpcm}\nonumber\\
&\stackrel{\eqref{e:stimazza2}\&\eqref{e:stimazza3-Jonas}}{\leq} & C D(r)^{\tau} ( D(r) + r D'(r))\, .  \label{e:mass_error_1000}
\end{eqnarray}

\medskip

{\bf Proof of \eqref{e:errori_interni_2}.} By Proposition~\ref{p:inner} and  once more \eqref{eq:bounds on inner variation}, 
\begin{align*}
|{\rm Err}^i_2| &\leq C \int_{\mathcal{B}_r^+} |N|^2 + C r \int \varphi |N| |DN|\\
&\leq C \int_{\mathcal{B}_r^+} |N|^2 + r^2 \int \varphi |DN|^2 \stackrel{\eqref{e:poincare_1000}}{\leq} 
C r^2 D(r)\, .
\end{align*}

\medskip

{\bf Proof of \eqref{e:errori_interni_5}.} Arguing as for ${\rm Err}^5_o$, we write
\begin{equation}\label{e:errore_5_spezzato_2}
|{\rm Err}^i_5| \leq \underbrace{\left|\int \langle X_i^\perp, h (\vec{\mathbf{T}}_F)\rangle d\|\mathbf{T}_F\|\right|}_{J_1}
 + C {\|A_\Sigma\|}_0 \underbrace{\int |X_i^\perp| d\|T-\mathbf{T}_F\|}_{J_2}\, .
\end{equation} 
The term $J_2$ can be estimated arguing exactly as for the term $I_2$ in \eqref{e:errore_5_spezzato} and we get
$J_2\leq C r D(r)^{1+\tau}$ (recall also \eqref{eq:bounds on inner variation}). 

In order to treat the first term we proceed as in \cite[Section 4.3]{DS5}. Denote by $\nu_1, \ldots, \nu_l$
an orthonormal frame for $T_p\Sigma^\perp$ of class $C^{2,a_0}$ (cf.~\cite[Appendix A]{DS2}) and set 
$h^j_p (\vec \lambda) := - \sum_{k=1}^m \langle D_{v_k}\nu_j (p),  v_k \rangle$
whenever $v_1\wedge\ldots\wedge v_m = \vec \lambda$ is an $m$-vector of $T_p \Sigma$ (with 
$v_1, \ldots, v_m$ orthonormal). 
For the sake of simplicity, we write 
\[
h^j (p) := h^j_p (\vec{T}_F (p)) \quad \text{and}\quad
h (p) := \sum_{j=1}^l h^j (p)  \nu_j(p),
\]
\[
\hat h^j ({\bp(p)}) := h^j_{\bp(p)} (\vec{\mathcal{M}}^+ (\bp (p)))
\quad \text{and}\quad
\hat h ({\bp(p)}) := \sum_{j=1}^l \hat h^j ({\bp(p)}) \nu_j(\bp(p)).
\]
where $\vec{\mathcal{M}}(p)$ denotes the $m$-vector orienting $T_p\mathcal{M}$. Consider the exponential map $\mathbf{ex}_{\mathbf{p} (p)}: T_{\mathbf{p} (p)} \Sigma \to \Sigma$ and its inverse $\mathbf{ex}^{-1}_{\mathbf{p} (p)}$. Recall that:
\begin{itemize}
\item the geodesic distance $d_\Sigma (p, q)$ is comparable to $|p-q|$ up to a constant factor;
\item $\nu_j$ is $C^{2, a_0}$ and $\|D\nu_j\|_{C^{1, a_0}} \leq C \eps_1^{\sfrac{1}{2}}$;
\item $\mathbf{ex}_{\bp (p)}$ and $\mathbf{ex}^{-1}_{\bp (p)}$ are both $C^{2, a_0}$\\ and $\|{\rm d}\,\mathbf{ex}_{\bp (p)}\|_{C^{1, a_0}} + \|{\rm d}\,\mathbf{ex}^{-1}_{\bp (p)}\|_{C^{1, a_0}} \leq \eps_1^{\sfrac{1}{2}}$;
\item $|h^j_p| \leq  C\|A_\Sigma\|_{C^0} \leq C \eps_1^{\sfrac12}$;
\end{itemize} 
where all the constants involved are geometric.
We then conclude that
\begin{align}
 h (p) - \hat h ({\bp(p)}) &= \sum_{j} (\nu_j(p) - \nu_j(\bp(p)))  h^j(p) \nonumber\\ 
&\quad+ \sum_j \nu_j(\bp(p)) (h^j(p)- \hat{h}^j(\bp(p))) \nonumber \\
= & \sum_{j} D\nu_j(\bp(p))\cdot \mathbf{ex}^{-1}_{\bp (p)}(p)\, h^j (p)+ O(|p-\bp(p)|^2)\nonumber\\ 
&\quad + \sum_j \nu_j(\bp(p)) (h^j(p)- \hat{h}^j(\bp(p))).
\end{align}
On the other hand, $X_i(p) = Y(\bp(p))$ is tangent to $\mathcal{M}^+$ in $\bp (p)$ and hence orthogonal to
$\hat{h}(\bp (p))$ and $\langle X_i(p), \nu_j(\bp(p)) \rangle =0 $ for all $j$.
Thus using \eqref{eq:bounds on inner variation}
\begin{align}
& \langle X_i(p), h (p) \rangle  = \langle X_i(p), h (p) - \hat h ({\bp(p)}) \rangle \nonumber \\
&= \sum_j \langle Y(\bp(p)), D\nu_j(\bp(p))\cdot \mathbf{ex}^{-1}_{\bp (p)} (p)\rangle h^j (p) + O\left(r |p-\bp(p)|^2\right)\,.
\end{align}
Recalling that
$p\in \supp (\mathbf{T}_F)$, we can bound $|p - \bp (p)| \leq |N (p)|$ and therefore conclude the estimate
\begin{align}
\langle X_i(p), h (p) \rangle & = \sum_j \langle Y(\bp(p)), D\nu_j(\bp(p))\cdot \mathbf{ex}^{-1}_{\bp (p)} (p)\rangle  h^j (p)\nonumber\\
&\qquad + O \big(r |N|^2 (\bp (p)) \big)\, .\label{e:pezzo lineare}
\end{align}

We now use the area formula for multivalued maps and the Taylor expansion for the area functional in \cite[Theorem 3.2]{DS2}. Recalling that $\bp (F_i (x))= x$ we get
\begin{align*}
J_1  = \hphantom{!} & \left|\int \langle X_i, h (p) \rangle d\|\mathbf{T}_F\| \right|
= \left| \sum_{i=1}^Q\int_{\mathcal{M}^+} \langle Y, h ({F_i (x)}) \rangle \mathbf{J} F_i (x) d \mathcal{H}^m (x)\right|
\allowdisplaybreaks\\
\stackrel{\eqref{e:pezzo lineare}}{\leq} & \left|\int_{\mathcal{M}^+}\! \sum_{j=1}^l \sum_{i=1}^Q \langle Y (x), D\nu_j  (x) \cdot \mathbf{ex}^{-1}_{x} (F_i (x))\rangle { h^j  (F(x))} d\mathcal{H}^m (x)\right|\\
&\quad  + Cr\!\int\! \varphi \,(|N|^2 +|DN|^2)
\end{align*}
Using the Taylor expansion for $\mathbf{ex}^{-1}_x$ at $x$ (and recalling that $F_i (x) - x = N_i (x)$) we conclude
\begin{align*}
\Big|\sum_{i=1}^Q \mathbf{ex}_x^{-1} (F_i (x))\Big| &\leq \left| {\rm d}\, \mathbf{ex}_x^{-1} (\etaa \circ N (x))\right| + O (|N|^2)\\
&\leq C |\etaa \circ N (x)| + C |N|^2\, .
\end{align*}
Next consider that $|\langle Y, D\nu_j\cdot v \rangle | \leq C r \varphi \|A_\Sigma\|_{C^0} |v| \leq C r \varphi \, \eps_1^{\sfrac12} |v|$
for every tangent vector $v$
and $|h^j (F(x))| \leq  C\|A_\Sigma\|_{C^0} \leq \eps_1^{\sfrac12}$. We thus conclude with the estimate
\[
J_1\leq  C\, \eps_1 r \int \varphi\,|\etaa\circ N|+ C r \,\int \varphi (|N|^2 + |DN|^2)\, .
\]
Using the Poincar\'e inequality and the same argument as for ${\rm Err}^o_1$, we conclude
\[
J_1 \leq C r D(r)^{1+\tau} + C r D(r)\,.\qedhere
\]
\end{proof}

\section{Proof of Theorem \ref{thm:ff_estimate_current}}

First of all notice that, if $D(r) =0$ for some $r$, then $N \equiv Q \a{0}$ on $\mathcal{B}_r^+$. This means that no
cube of $\mathscr{W}^\be\cup \mathscr{W}^\bh$ intersects $\overline{\mathscr{B}}^+_r = \{p\in \pi_0^+: d(\boldsymbol{\varphi}(p))\le r\}$. On the other hand from Corollary \ref{c:domains} we easily conclude that no cube of $\mathscr{W}$ intersects the region $\overline{\mathscr{B}}^+_{r/2}$
(observe that no cube $L\in \mathscr{W}$ is a boundary cube and thus, if it intersects $\overline{\mathscr{B}}^+_{r/2}$, we have $\ell (L) \ll r$). In particular, $\mathcal{B}^+_{r/2}$ is contained in the contact set and thus there is a neighborhood of $0$ where $T^+$ coincides with $Q \a{\mathcal{M}^+}$. 

Thus, without loss of generality we can assume that $D(r)>0$. Notice that for the same reason we can assume that
there is a sequence of radii $r_j\downarrow 0$ such that $H (r_j) >0$. More specifically, we claim that there is a radius $r_0$ sufficiently small for which, for all $r<r_0$, $H (r)>0$ and all the estimates of the previous sections apply.
Indeed, let $]\rho, r_0[$ be a maximal interval over which $H\neq 0$. On this interval we compute the derivative of
$\log I (r)$  using \eqref{der_H_corrente}:
\begin{align}
\frac{d}{dr} \log I (r) &= \frac{1}{r} + \frac{D'(r)}{D(r)} - \frac{H'(r)}{H(r)}
= O(1)+ \frac{2-m}{r} + \frac{D'(r)}{D(r)} - \frac{2 E(r)}{H(r)}\, . \label{e:derivata_logaritmica}
\end{align}
Next, by \eqref{e:quasi_fatta_1}, \eqref{e:errori_esterni_134}, \eqref{e:errori_esterni_2} and \eqref{e:errori_esterni_5},
\begin{equation}\label{e:fame}
|D(r) - E(r)|\leq C (D(r)^{1+\tau} + C S(r)) \leq C (D(r)^{1+\tau} + r^2 D(r))\, .
\end{equation}
Note that 
\begin{align*}
D(r) &\leq \sum_j \int_{\mathcal{U}_j} |DN|^2 \stackrel{\eqref{e:bad_regions_4}}{\leq} C \sum_j \eps_1 \ell_j^{m+2-2\alpha_\be} \leq C r^{2-2\alpha_\be} \sum_j \ell_j^m\, .
\end{align*}
Recalling that all $L_j$'s are disjoint and contained in $B_{4\sqrt{m} r}$, we easily conclude that $D(r) \leq C r^{m+2-2\alpha_\be}$. In particular, \eqref{e:fame} implies
\begin{equation}\label{e:molta_fame}
D(r) (1-C r^\tau) \leq E(r) \leq D(r) (1+C r^\tau)\, .
\end{equation}
Assuming $r_0$ is sufficiently small, we infer 
\begin{equation}\label{e:molta_fame_2}
\frac{D(r)}{2} \leq E(r) \leq 2 D(r)\, .
\end{equation}
In particular, inserting \eqref{e:molta_fame}
in \eqref{e:derivata_logaritmica}, we obtain
\begin{equation}\label{e:derivata_logaritmica_2}
\frac{d}{dr} \log I (r) \geq O(1)+ \frac{2-m}{r} + \frac{D'(r)}{E (r)} - \frac{2 E(r)}{H(r)} - C \frac{D'(r) (S(r) + D(r)^{1+\tau})}{D(r)^2}\, .
\end{equation}
Using \eqref{e:quasi_fatta_2}, \eqref{e:errori_interni_134}, \eqref{e:errori_interni_2} and \eqref{e:errori_interni_5}, 
\begin{align}
\frac{d}{dr} \log I(r) &\hphantom{a}\geq  \hphantom{a} O(1) + \frac{2 G(r)}{E(r)} - \frac{2E(r)}{H(r)} -  C \frac{D'(r) (S(r) + D(r)^{1+\tau})}{D(r)^2}\nonumber\\
&\hphantom {\ge a} -
\frac{1}{ r E (r)} \sum_{j=1}^5 |{\rm Err}^i_j|\nonumber\\
&\hphantom{a}\geq  \hphantom{a} O(1)+ \frac{2 G(r)}{E(r)} - \frac{2E(r)}{H(r)} - C \frac{D'(r) (S(r) + D(r)^{1+\tau})}{D(r)^2}\nonumber\\
& \hphantom {\ge a}- C \frac{ D(r)}{E(r)}\left(1+ \frac{D(r)^\tau}{r} + \frac{D'(r)}{D(r)^{1-\tau}}\right)\nonumber\\
&\stackrel{\eqref{e:molta_fame_2}}{\geq} O(1)+ \frac{2 G(r)}{E(r)} - \frac{2E(r)}{H(r)} - C \frac{D'(r) S(r)}{D(r)^2}\nonumber\\
&\hphantom {\ge a}- C\frac{D(r)^\tau}{r} - C \frac{D'(r)}{D(r)^{1-\tau}}
 \, .
\label{e:la_richiamiamo}
\end{align}
By Cauchy--Schwartz $G(r) H(r) \geq E(r)^2$. Moreover, we have already estimated $-D(r) \geq - C r$. Inserting the latter inequalities in \eqref{e:la_richiamiamo} and integrating, we obtain
\begin{align}
\log \frac{I(r)}{I (s)} &\geq - C (r^\tau-s^\tau) - C (D(r)^\tau - D(s)^\tau) - C \int_s^r \frac{D' (\sigma)}{D(\sigma)^2} S(\sigma)\, d\sigma\nonumber\\
&\geq - C r^\tau + C \left(\frac{S(r)}{D(r)} - \frac{S(s)}{D(s)}\right) - C \int_s^r \frac{S'(\sigma)}{D(\sigma)}\, d\sigma\,,
\end{align}
for every $\rho<s<r<r_0$.
Recall that $S (\sigma) \leq C \sigma^2 D (\sigma)$ for every $\sigma\in ]\rho,r_0[$. Moreover,
\[
S' (\sigma) = - \int \frac{d}{\sigma^2} \phi' \left(\frac{d}{\sigma}\right) |N|^2 \leq C H(\sigma) \stackrel{\eqref{e:H<rD}}{\leq} C \sigma D (\sigma)\, .
\]
In particular, we conclude
\begin{equation}\label{e:semplice}
\log \frac{I(r)}{I(s)} \geq - C r^\tau\, .
\end{equation}
From the latter inequality we conclude immediately that $I(s)$ is uniformly bounded and
thus that $H(\rho) = \lim_{r\downarrow \rho} H (r)$ cannot vanish if $\rho>0$. Since $]\rho, r_0[$ is a maximal interval on which $H$ is positive, we conclude that it is positive on the whole $]0, r_0[$. 

Furthermore, it follows directly from \eqref{e:semplice} that the limit 
\[
I^+_0 := \lim_{r\downarrow 0} I^+ (r)
\]
exists. Finally, from \eqref{e:lower_bound_I} we conclude $I_0>0$. 

\chapter{Final blow-up argument}\label{chap:blowup}

In this chapter we conclude the proof of Theorem \ref{thm:main}. In particular we show that alternative (b) in Theorem \ref{thm:ff_estimate_current} cannot hold. This leaves alternative (a), which therefore shows that, under the assumptions of the theorem, the origin is in fact a regular boundary point.
On the other hand, such point was a generic collapsed point of an area-minimizing current which was later suitably rescaled and translated in order to fulfill the Assumption \ref{ass:cm_app}. 

The core of the argument is to derive a suitable contradiction to the linear theory with a blow-up of the approximating $\qhalf$-map $(N^+, N^-)$. In order to state our main theorem we introduce the following notation. 

Recall that $\mathcal{M}$ is the union of $\mathcal{M}^+$ and $\mathcal{M}^-$ and is, therefore, a $C^{1,1}$ submanifold. Moreover $\mathcal{M}$ coincides with the graph of the functions 
${\bm \varphi}^+$ and ${\bm \varphi}^-$ on the domains $B_1^+$ and $B_1^-$. In order to simplify the notation we denote by ${\bm \varphi}$ the map on $B_1$ which coincides with both on the respective domains. In particular we are ready to define suitable multivalued maps
\[
\fancyN^\pm (x) = \sum_i \a{\fancyN\,_i^\pm (x)}
\]
given by the formulas
\[
\fancyN\,_i^\pm (x) = {\bm p}_{\varkappa_0} \big(N^\pm_i (x, {\bm \varphi}^\pm (x))\big)\, ,
\]
where we recall that $\varkappa_0$ is the plane $T_0 \Sigma \cap T_0 \mathcal{M}^\perp = \{0\}\times \R^{\bar n}\times \{0\}$. 
Observe that the pair $(\fancyN^+, \fancyN^-)$ is a $\qhalf$-valued function with interface $(\gammado, 0)$. We next define
\[
\fancyD (r) = \int_{B_r^+} |D\fancyN^+|^2 +\int_{B_r^-} |D\fancyN^-|^2= \fancyD^+(r) + \fancyD^-(r) \, 
\]
and the corresponding rescaled multivalued functions
\[
\fancyN\,_r^\pm (x) := \sum_i \a{ r^{\sfrac{m}{2}-1} \fancyD (r)^{-\sfrac{1}{2}} \fancyN\,_i^\pm (rx)}\, .
\]
\begin{definition}
The domains of the rescaled functions $\fancyN\,_r^\pm$ are divided by (suitable) rescalings of $\gamma$, which in turn are converging to the $(m-1)$-dimensional plane $T_0 \gamma$. For this reason we introduce the notation $B_{r,\rho}^+$ (and $B_{r, \rho}^-$) for the intersection of the domain of $\fancyN\,_r^+$ (respectively of $\fancyN\,_r^-$) with the disk $B_\rho (0, \pi_0)$.
\end{definition}

Note that the regions $B^\pm_r$, which are subsets of the domains of the maps $\fancyN^\pm$, coincide with the sets $B^\pm_{1,r}$.
Observe that a simple consequence of the estimates in the previous chapter is that 
\begin{align}
\fancyD (r) &\leq C \varepsilon_1 r^{m+2-2\alpha_\be} \, ,\label{e:estimate_on_fancyD}\\
\Lip (\fancyN^\pm|_{B_r}) &\leq C \varepsilon_1^\sigmaexpcm r^\sigmaexpcm\, . \label{e:Lip_est_on_fancy_N}
\end{align}

We are now ready to state the key step of our final contradiction argument.

\begin{theorem}\label{thm:the_end_my_friend}\label{THM:THE_END_MY_FRIEND}
If alternative (b) in Theorem \ref{thm:ff_estimate_current} would hold in any of the two regions $\mathcal{C}^\pm$, then, up to a subsequence, the pair $(\fancyN\,_r^+, \fancyN\,_r^-)$ would converge in \(B_1\) locally strongly in $L^2$ and in energy to a $\qhalf$ ${\rm Dir}$-minimizer $(\fancyN\,^+_0, \fancyN\,^-_0)$ which collapses at the interface $(T_0 \gammado, 0)$ such that
\begin{itemize}
\item[(i)] $(\fancyN\,^+_0, \fancyN\,^-_0)$ is nontrivial;
\item[(ii)] $\etaa\circ \fancyN\,^\pm_0 \equiv 0$.
\end{itemize} 
\end{theorem}

\begin{remark}
Observe that, although the notation $\fancyN\,_0^\pm$ might suggest that the ``blow-up'' map is unique, namely independent of the sequence $\{r_k\}_k$, we do not claim such uniqueness, nor we need it for our purposes. 
\end{remark}

By convergence in energy we mean that  for every $R\in (0,1)$
\[
\lim_{k\to\infty} \left(\int_{B_R^+} |D\fancyN\,_{r_k}^+|^2 + \int_{B_R^-} |D\fancyN\,_{r_k}^-|^2\right)
= \int_{B_R^+} |D\fancyN\,_{0}^+|^2 + \int_{B_R^-} |D\fancyN\,_{0}^-|^2
\]

Since by Theorem \ref{thm:collasso} any $\qhalf$ ${\rm Dir}$ minimizer $(\fancyN\,^+_0, \fancyN\,^-_0)$ which collapses at the interface must satisfy
\[
\fancyN\,^+_0 = Q \a{\etaa \circ \fancyN\,^+_0} \quad \mbox{and}\quad \fancyN\,^-_0 = (Q-1) \a{\etaa \circ \fancyN\,^-_0}\, ,
\]
the two properties (i) and (ii) above are incompatible. In particular we conclude 

\begin{corollary}\label{cor:the_end}
Alternative (a) in Theorem \ref{thm:ff_estimate_current} must hold for both \(T\res \mathcal C^+\) and $T\res\mathcal C^-$, i.e. $0$ is a boundary regular point for the current $T$. 
\end{corollary}

\section{Asymptotics for $\fancyD (r)$}

\begin{lemma}\label{l:asymptotic}
Under the assumptions of Theorem \ref{thm:the_end_my_friend} for every  $\lambda\in (0,1)$ one has
\begin{equation}\label{e:asymptotic_D}
\infty > \limsup_{r\downarrow 0}  \frac{\fancyD (\lambda r)}{\fancyD (r)} \geq \liminf_{r\downarrow 0} \frac{\fancyD (\lambda r)}{\fancyD (r)} > 0\, .
\end{equation}
\end{lemma}

Observe that (i) in  Theorem \ref{thm:the_end_my_friend} is then a simple consequence of the above lemma and convergence in energy. 

\begin{proof}
Observe that, since $T_0 \mathcal{M} = \pi_0$ and $N^\pm$ are orthogonal to $\mathcal{M}$, we easily conclude that
\begin{equation}\label{e:diri}
\fancyD^\pm (r) = (1+ O(r)) \int_{B_r^\pm} |DN^\pm|^2\, . 
\end{equation}
Furthermore, if one among $I_0^+$ and $I_0^-$ is $+\infty$, then the corresponding energy vanishes identically. Thus, under the assumption that 
they are finite, it suffices to show 
\begin{align}
\infty > \limsup_{r\downarrow 0} \left(\int_{B_r^\pm} |DN^\pm|^2\right)^{-1} \int_{B_{\lambda r}^\pm} |DN^\pm|^2\nonumber\\
\geq
\liminf_{r\downarrow 0} \left(\int_{B_r^\pm} |DN^\pm|^2\right)^{-1} \int_{B_{\lambda r}^\pm} |DN^\pm|^2 > 0\, .\label{e:sandwich}
\end{align}
To fix ideas consider the case of $N^+$ and notice that, in the notation of the previous chapter, we must simply show 
\begin{equation}\label{e:sandwich2}
\infty > \limsup_{r\downarrow 0} D (r)^{-1} D (\lambda r)\geq \liminf_{r\downarrow 0} D (r)^{-1} D (\lambda r) > 0\, .
\end{equation}
Observe that the quantities $D$ and $H$ defined in \eqref{e:richiamo_pm} and \eqref{e:richiamo_pm_2} are integrals over (portions of) the ``right center manifold'' $\mathcal{M}^+$. Hence, from now on we use a more consistent notation for the remaining computations of this chapter, namely $D^+$ and $H^+$ (and analogously $I^+$ and $E^+$).
In order to prove the desired estimate notice first that, by Proposition \ref{p:H'_maiala_corrente},  and \eqref{e:molta_fame} we have
\[
	\frac{d}{dr} \log \left(\frac{H^+(r)}{r^{m-1}}\right) = \frac{2E^+(r)}{H^+(r)} +O(1) = \frac{2}{r}(1+ O(r^\tau)) I^+(r) +O(1)
\]
Next, by choosing $r$ sufficiently small, we can assume that
\[
\frac{I_0^+}{2} \leq (1+ O(r^\tau))I^+ (r)\leq 2I_0^+ \, .
\]
Thus, integrating the inequality above between $s$ and $t\geq s$, we conclude
\[
e^{- C (t-s)} \left(\frac{t}{s}\right)^{m-1+I_0^+} \leq \frac{H^+(t)}{H^+ (s)} \leq e^{C (t-s)} \left(\frac{t}{s}\right)^{m-1+4 I_0^+}\, .
\]
Since 
\[
\lim_{r\downarrow 0} \frac{r D^+(r)}{H^+(r)} = I_0^+\, ,
\]
we can argue as in Corollary \ref{cor:consequence of monotone frequency for H and D} (c) to conclude \eqref{e:sandwich2}. 
\end{proof}

\section{Vanishing of the average}

In this section we wish to show that 
\begin{lemma}\label{l:appiattimento}
Under the assumptions of Theorem \ref{thm:the_end_my_friend} we have
\begin{equation}\label{e:peggio}
\lim_{r\to 0}\left(\int_{B_1^+} |\etaa\circ \fancyN\,_{r}^+|+ \int_{B_1^-} |\etaa\circ \fancyN\,_{r}^-|\right) = 0\, .
\end{equation}
Indeed we have the stronger estimate
\begin{align}
\lim_{r\downarrow 0} \fancyD (r)^{-1} &r^{- (1+\tau')} \left(\int_{B_r^+} |\etaa\circ \fancyN^+|+ \int_{B_r^-} |\etaa\circ \fancyN^-|\right)\nonumber\\
 \leq\; & \lim_{r\downarrow 0} \fancyD (r)^{-(1+\tau')} r^{-1} \left(\int_{B_r^+} |\etaa\circ \fancyN^+|+ \int_{B_r^-} |\etaa\circ \fancyN^-|\right) = 0\, .\label{e:meglio}
\end{align}
 for any $\tau'$ smaller than the parameter \(\tau\) of Proposition \ref{p:quasi_fatta_2}.
\end{lemma}

Notice that (ii) in Theorem \ref{thm:the_end_my_friend} is then a trivial consequence of the lemma and of   Lemma \ref{l:asymptotic}.

\begin{proof}
In view of the same considerations used in the proof of Lemma \ref{l:asymptotic}, in order to prove \eqref{e:peggio} it suffices to show that, under the condition that alternative (b) holds, 
\begin{align}
&\lim_{r\to 0} \frac{1}{r^{m/2+1}D^+ (r)^{1/2}}\int_{B^+_{r}} |\etaa\circ N^+|\nonumber\\
= &\lim_{r\to 0} \frac{D^+(r)^{1/2}}{r^{m/2}}\frac{1}{rD(r)}\int_{B^+_{r}} |\etaa\circ N^+|=0.\label{e:eee3}
\end{align}
where we are using the notation of the previous chapter. By \eqref{e:estimate_on_fancyD} and \eqref{e:diri},
\begin{equation}\label{e:eee1}
\lim_{r\to 0} \frac{D^+ (r)^{1/2}}{r^{m/2}}=0.
\end{equation}
We now  claim that 
\begin{equation}\label{e:eee2}
\int_{B^+_r} |\etaa\circ N^+| \leq C r \left(\int_{B^+_r} |DN^+|^2 \right)^{1+\tau}\, .
\end{equation}
where \(C\) and \(\tau\) are as in Proposition \ref{p:quasi_fatta_2}. The latter inequality, together with \eqref{e:estimate_on_fancyD}, clearly implies \eqref{e:meglio}. Moreover  the combination of \eqref{e:eee1} and \eqref{e:eee2} implies \eqref{e:eee3}. Hence  the proof of the lemma will be concluded once we show \eqref{e:eee2}. To this aim, with  the notation of the previous chapter, we estimate
\[
\int_{B^+_r} |\etaa\circ N^+| \leq \sum_j \int_{\mathcal{U}_j} |\etaa\circ N^+|\,.
\]
Applying \eqref{e:cm_app5} with $a=r$ we easily conclude
\[
\int_{B^+_r} |\etaa\circ N^+|\leq C r \sum_j \eps_1 \ell_j^{m+2+\sfrac{\sigmaexpcm}{2}} + \frac{C}{r} \int_{B^+_r} |N^+|^{2+\sigmaexpcm}\, . 
\]
On the other hand, using \eqref{e:bad_regions_2}, \eqref{e:stimazza2} and \eqref{e:stimazza3-Jonas} we then conclude
\[
\int_{B^+_r} |\etaa\circ N^+|\leq C r \left( \int_{B^+_r} |DN^+|^2\right)^{1+\tau} + \frac{C}{r}  \left( \int_{B^+_r} |DN^+|^2\right)^\tau
\int_{B^+_r} |N^+|^2\, .
\] 
Combining the above estimates with the Poincar\'e inequality
\[
\int_{B^+_r} |N^+|^2\leq C r^2  \int_{B^+_r} |DN^+|^2
\]
we then conclude the proof of \eqref{e:eee2} and of the Lemma.
\end{proof}

%
%

\section{Minimality and convergence in energy}

In this section we complete the proof of Theorem \ref{thm:the_end_my_friend}.  In order to be consistent with our notation on the domains of the functions $\fancyN\,_r^\pm$, we let $B_{0,R}^\pm$ denote the intersections of the domain of
definitions of the blow-up maps $\fancyN\,_0^\pm$ with the disk $B_r (0, \pi_0)$.
By the  Rellich-Kondrakov embedding we know that we can extract a subsequence
$(\fancyN\,_{r_k}^+,\fancyN\,_{r_k}^-)$ converging locally strongly in $L^2 (B_1)$ to some $\qhalf$-map $(\fancyN\,_0^+, \fancyN\,_0^-)$. The fact that the latter collapses at the  interface $(T_0 \gammado, 0)$ comes from trace theory (cf. for instance \cite{DS1}, \cite{Jonas}). Observe that, by semicontinuity of the Dirichlet energy we have
\begin{align}
&\liminf_{k\to\infty} \left(\int_{B_{r_k, R}^+} |D\fancyN\,_{r_k}^+|^2 + \int_{B_{r_k, R}^-} |D\fancyN\,_{r_k}^-|^2\right)\nonumber\\
\geq & \int_{B_{0,R}^+} |D\fancyN\,_0^+|^2 + \int_{B_{0,R}^-} |D\fancyN\,_0^-|^2\label{e:weak_semicontinuity}
\end{align}
for every \(R\in (0,1)\).  

Assume without loss of generality that the inferior limit on the left hand side is actually a limit.
Choose now any $\qhalf$ competitor $(u^+, u^-)$ with interface $(T_0 \gammado, 0)$ which coincides with $(\fancyN\,_0^+, \fancyN\,_0^-)$ on  $B_1\setminus B_R$. 
We now want to show that, for any given positive $\eta>0$, 
\begin{align}
&\lim_{k\to\infty} \left(\int_{B_{r_k, R}^+} |D\fancyN\,_{r_k}^+|^2 + \int_{B_{r_k, R}^-} |D\fancyN\,_{r_k}^-|^2\right)\nonumber\\
\leq &\int_{B_{0, R}^+} |D u^+|^2 + \int_{B_{0, R}^-} |Du^-|^2 + \eta\, .
\label{e:upper_limit}
\end{align}
Clearly this will show both the convergence in energy (by choosing \(u^\pm=\fancyN\,^\pm_0\)) and the local minimality of \(\fancyN\,^\pm_0\). Hence  the proof of Theorem \ref{thm:the_end_my_friend} will be concluded once we show \eqref{e:upper_limit}.

Without loss of generality we can assume that $\etaa\circ u^\pm =0$. Indeed, recall that $\etaa\circ \fancyN\,_0^\pm \equiv 0$ and thus, since
\[
\int_{B_1^\pm} \abs{Du^\pm}^2 \ge \int_{B_1^\pm} \sum_i \abs{D(u_i^\pm - \etaa\circ u^{\pm})}^2\, ,
\]
$\sum_i \a{u^\pm - \etaa \circ u^\pm}$ would be a better competitor with zero average.

It is convenient to introduce the energy difference
\begin{align*}
\fancyE_k &:= \left(\int_{B_{r_k, 1}^+} |D\fancyN\,_{r_k}^+|^2 + \int_{B_{r_k, 1}^-} |D\fancyN\,_{r_k}^-|^2\right) \\
&\qquad - \left(\int_{B_{0,1}^+} |D u^+|^2 + \int_{B_{0,1}^-} |Du^-|^2\right)\, ,
\end{align*}
so that our claim reduces to 
\[
\lim_{k\to\infty} \fancyE_k \leq \eta\, .
\]
Note  also that we can assume that \(\fancyE_k\ge 0\) otherwise there is nothing to prove, in particular
\begin{align}
 & \left(\int_{B_{0,1}^+} |D u^+|^2 + \int_{B_{0,1}^-} |Du^-|^2\right)\nonumber\\
 \le & \lim_{k\to \infty} \int_{B_{r_k, 1}^+} |D\fancyN\,_{r_k}^+|^2 + \int_{B_{r_k, 1}^-} |D\fancyN\,_{r_k}^-|^2=1\,,\label{e:us}
\end{align}
where the last equality follows by the normalization of \(\fancyN\,_{r_k}^\pm\). 

Our first step is then to produce a new $\qhalf$-map $(\hat\fancyN\,_k^+, \hat\fancyN\,_k^-)$ with interface $(\gammado, 0)$ and satisfying the following four properties:
\begin{itemize}
\item[(a)] $(\hat\fancyN\,_k^+, \hat\fancyN\,_k^-)$ coincides with $(\fancyN^+, \fancyN^-)$ outside $B_{r_k}$;
\item[(b)] the Lipschitz constants ${\rm Lip} (\hat\fancyN\,_k^\pm)$ converge to $0$ as $k\to\infty$;
\item[(c)] the following inequality holds for the energy:
\begin{align}
& \int_{B_{r_k}^+} |D\hat \fancyN\,_{k}^+|^2 + \int_{B_{r_k}^-} |D\hat\fancyN\,_{k}^-|^2 \leq \int_{B_{r_k}^+} |D\fancyN^+|^2\nonumber\\
&\qquad\qquad\qquad + \int_{B_{r_k}^+} |D\fancyN^-|^2 + r_k^{2-m}\fancyD (r_k) \left(- \fancyE_k + \frac{\eta}{2}\right)\, ;\label{e:upper_limit_2}
\end{align}
\item[(d)] $|\etaa\circ \hat \fancyN\,_{k}^+|\leq C |\etaa\circ \fancyN^\pm|$;  
\end{itemize}
First, observe that by Lemma \ref{l:lip_app}, we can choose a sequence of approximants $(u_j^+, u_j^-)$ which converge in energy to
$(u^+, u^-)$ in $B_{0,1}$, satisfy $\etaa\circ u^\pm_j \equiv 0$ and with Lipschitz constant controlled by $j$,
\[
{\rm Lip}(u^\pm_j)\le j.
\]
Next, choose a sequence of diffeomorphisms $\Phi_k$ of $B_1$ which converges in $C^1$ to the identity and maps the rescalings 
\(
\gammado_{r_k} := r_k^{-1} \gammado
\)
onto $T_0 \gammado$.
We then define
\[
(u^+_{j,k},u^-_{j,k})=(u^+_j \circ \Phi_k, u^-_j \circ \Phi_k).
\]
Note that 
\begin{equation}\label{e:fe1}
\lim_{k\to \infty} \lim_{j\to \infty} \int_{B^{\pm}_{r_k,1}}|D u^\pm_{j,k}|^2=\lim_{k\to \infty} \int_{B^{\pm}_{r_k,1}}|D (u^\pm\circ \Phi_k)|^2=\int_{B^{\pm}_{0,1}}|D u^\pm|^2
\end{equation}
and
\begin{equation}\label{e:fe2}
\lim_{k\to \infty} \lim_{j\to \infty} \int_{B^{\pm}_{r_k, 1}\setminus \Phi_k^{-1} (B_{0, R}^\pm)}\mathcal G^2 (u^\pm_{j,k},\fancyN\,_{r_k}^\pm)=0\,.
\end{equation}
Using the interpolation Lemma \ref{l:interpolation}   and proceeding as in Section \ref{s:gluing} we obtain $\qhalf$-maps $(w_{j,k}^+, w_{j,k}^-)$ with the following properties for a sufficiently large $k$ and small \(\lambda\):
\begin{itemize}
\item[(a1)] $(w_{j,k}^+, w_{j,k}^-)$ coincide with \((u^\pm_{j,k},u^\pm_{j,k})\) on \(\Phi_k^{-1}(B_R (0, \pi_0))\) and with  $(\fancyN\,_{r_k}^+, \fancyN\,_{r_k}^-)$ outside $B_{s_k} (0, \pi_0)$ for some $R < s_k < 1$ such that \(\Phi_k^{-1}(B_R (0, \pi_0))\subset B_{s_k} (0, \pi_0)\);
\item[(b1)] The Lipschitz constant of $(w_{k,j}^+, w_{k,j}^-)$ is estimated as\footnote{Here we are using the simple  inequality \(\|f\|_{L^\infty(E)} \le |E|^{-1}\|f\|_{L^1(B_1)}+\diam(E)\Lip(f)\)}  
\[
\begin{split}
\Lip(w_{k,j}^\pm) &\le C\Big(  \Lip (\fancyN\,_{r_k}^\pm) +  \Lip (u^\pm_{k,j})+\frac{1}{\lambda} \sup_{B_1^\pm\setminus \Phi_k^{-1} (B_{0,R}^\pm)} \mathcal G(u_{j,k}^\pm, \fancyN\,^\pm_{r_k})\Big)
\\
&\le C\Big(  \Lip (\fancyN\,_{r_k}^\pm) +   \Lip (u^\pm_{k,j})\\
&\qquad\qquad+\frac{1}{\lambda(1-R)} \int_{B_1^\pm\setminus  \Phi_k^{-1} (B_{0,R}^\pm)} \mathcal G(u_{j,k}^\pm, \fancyN\,^\pm_{r_k})\Big);
\end{split}
\] 
\item[(c1)] The energy of $(w_{j,k}^+, w_{j,k}^-)$ can be estimated as 
\begin{align}
& \int_{B_{r_k,1}^+} |Dw_{j,k}^+|^2 + \int_{B_{r_k,1}^-} |Dw_{j,k}^-|^2\nonumber\\ 
& \leq  (1+\|\Phi_k-{\rm Id}\|_{C^1})\Big( \int_{B_{0, R}^+} |Du_{j}^+|^2 + \int_{B_{0, R}^-} |Du_{j}^-|^2\Big)\nonumber\\
&\quad + \int_{B_{r_k,1}^+\setminus B_{s_k} (0, \pi_0)} |D\fancyN\,_{r_k}^+|^2 + \int_{B_{r_k,1}^-\setminus B_{s_k} (0, \pi_0)} |D\fancyN\,_{r_k}^-|^2\nonumber\\
&\quad +C\lambda \int_{B_{r_k, 1}^+\setminus \Phi_k^{-1} (B_R (0, \pi_0))} (|Du_{j,k}^+|^2 +|D\fancyN\,_{r_k}^+|^2)\nonumber\\
& \quad +C \lambda \int_{B_{r_k, 1}^-\setminus \Phi_k^{-1} (B_R (0, \pi_0))} (|Du_{j,k}^-|^2 + |D\fancyN\,_{r_k}^-|^2)
\nonumber\\
&\quad +\frac{C}{\lambda} \int_{B_{r_k,1}^+ \setminus \Phi_k^{-1} (B_R (0, \pi_0))} \mathcal G^2(u_{j,k}^+, \fancyN\,^+_{r_k})\nonumber\\
&\qquad +\frac{C}{\lambda} \int_{B_{r_k,1}^-\setminus \Phi_k^{-1} (B_R (0, \pi_0))} \mathcal G^2(u_{j,k}^-, \fancyN\,^-_{r_k})
\nonumber\\
& \leq \int_{B_{r_k,1}^+} |D\fancyN\,_{r_k}^+|^2 + \int_{B_{r_k,1}^-} |D\fancyN\,_{r_k}^-|^2 + \frac{\eta}{4} - \fancyE_k + o_{j,k}(1)\,.
\label{e:upper_limit_3}
\end{align}
where  
\[
\lim_{j\to \infty}\lim_{k\to \infty} o_{j,k}(1) =0 
\]
and we have chosen \(\lambda\ll \eta\) (recall also \eqref{e:us}).
 
\item[(d1)] $|\etaa\circ w_k^\pm|\leq C |\etaa\circ \fancyN\,_{r_k}^\pm|$. This can be easily seen as follows: first of all we can subtract the average from $\fancyN\,_{r_k}^\pm$, and interpolate it to $0$, which is the average of the competitors $u_j^{\pm}$, hence we can interpolate between the maps $(u^+, u^-)$ and the average-free part of $(\fancyN\,_{r_k}^+, \fancyN\,_{r_k}^-)$: a simple inspection of the proof of Lemma \ref{l:interpolation}  shows that this can be done while keeping the average of the interpolation equal to $0$. Hence we can add back the average to the resulting maps in order to get $w_k^\pm$. Note that in estimating the Dirichlet energies we are using the crucial fact that the Dirichlet energy of a multivalued map equals the sum of the Dirichlet energies of its average and average-free part.
\end{itemize}
Next we set 
\[
\hat\fancyN\,_{j,k}^\pm (x) = \sum_i \a{r_k^{1-\sfrac{m}{2}} \fancyD (r_k)^{\sfrac{1}{2}} (w_{j,k}^\pm)_i (r_k^{-1} x)}\, 
\]
and 
\[
\hat\fancyN\,_{j}^\pm = \hat\fancyN\,_{j, k_j}^\pm\, 
\]
for $k_j$ appropriately large.
Observe that $(\hat\fancyN\,_j^+, \hat\fancyN\,_j^-)$ clearly satisfies property (a). Moreover, 
\[
{\rm Lip} (\hat\fancyN\,_{j,k}^\pm) \leq C \Lip (\fancyN^\pm) + C  r_k^{-\sfrac{m}{2}} \fancyD (r_k)^{\sfrac{1}{2}} j+C\eta^{-1} o_{j,k}(1)\, .
\]
In particular, taking into account \eqref{e:estimate_on_fancyD} and \eqref{e:Lip_est_on_fancy_N},
\begin{align*}
{\rm Lip} (\hat\fancyN\,_{j,k}^{\pm}) &\leq C \eta^{-1} \varepsilon_1^\sigmaexpcm r_k^\sigmaexpcm + C  \varepsilon_1^{\sfrac{1}{2}} r_k^{1-\alpha_\be} j+C \eta^{-1} r_k^{-\sfrac{m}{2}} \fancyD (r_k)^{\sfrac{1}{2}} j\nonumber\\
&\qquad\qquad+C\eta^{-1} o_{j,k}(1) .
\end{align*}
Thus, choosing first  \(j\) large and then \(k_j\) much larger, we achieve (b). Finally  \eqref{e:upper_limit_2} follows from \eqref{e:upper_limit_3}. 
\bigskip

We next define a suitable Lipschitz map $\Lambda$ between a neighborhood $U$ of the origin in $\Sigma$ onto a neighborhood of the origin in $T_0 \Sigma$. Fix therefore $z\in U \cap \Sigma$. First of all we define $x\in \pi_0 = T_0 \mathcal{M}$ as the only point such that $(x, {\bm \varphi} (x)) = {\bm p} (z)$, where \({\bm p}\) is the projection onto \(\mathcal M\). Next, we let $\varkappa_0:= T_0 \Sigma \cap T_0 \mathcal{M}^\perp$ and we define $y := {\bm p}_{\varkappa_0} (z- {\bm p} (z))$. We then set $\Lambda (z) := (x,y)\in T_0 \Sigma$
and $\Lambda^v (z) = y$.

We partition $U$ into $U^+$ and $U^-$ according on whether ${\bm p} (z)$ belongs to $\mathcal{M}^+$ or $\mathcal{M}^-$. So, we can regard $\Lambda$ as two maps $\Lambda^+$ and $\Lambda^-$ which are $C^{2,\kappa}$ on the corresponding domains and which agree on the common
boundary $U^+\cap U^- = {\bm p}^{-1} (\gammaup) \cap U$. Observe that the differentials of $\Lambda^\pm$ at the origin are the identity in both cases. Thus, using the inverse function theorem, we can find two inverse maps ${\bm \Psi}^\pm$ defined on $B_r^\pm (\pi_0) \times B_r (\varkappa_0)$. 

We are thus ready to define the competitor maps $(\hat N_k^+, \hat N_k^-)$ in the form 
\[
\hat N_k^\pm (x, {\bm \varphi} (x)) = {\bf \Psi}^\pm (x, \hat \fancyN\,_k^\pm (x)) - (x, {\bm \varphi} (x))\,,
\]
namely
\[
\hat N_k^\pm (x, {\bm \varphi} (x)) =  \sum_i \a{ {\bf \Psi}^\pm (x, (\hat \fancyN\,_k^\pm)_i (x)) - (x, {\bm \varphi} (x))}\, .
\]
Observe that
\[
\hat \fancyN\,_k^\pm (x)) = {\bm p}_{\varkappa_0} (\hat N_k (x, {\bm \varphi} (x)))\, .
\]
We thus conclude easily that:
\begin{itemize}
\item[(a2)] $(\hat N_k^+,\hat N_k^-)$ coincide with $(N^+, N^-)$ outside of $\bC_{2r_k} \cap \mathcal{M}$;
\item[(b2)] the Lipschitz constants of $\hat N_k^\pm$ on $\bC_{2r_k} \cap \mathcal{M}$ converge to $0$; 
\item[(c2)] for $k$ large enough we have the energy comparison
\begin{align}
&\int_{\bC_{2r_k}\cap \mathcal{M}^+} |D\hat{N}_k^+|^2 + \int_{\bC_{2r_k}\cap \mathcal{M}^-} |D\hat{N}_k^-|^2\nonumber\\
& \leq \int_{\bC_{2r_k}\cap \mathcal{M}^+} |DN^+|^2 + \int_{\bC_{2r_k}\cap \mathcal{M}^-} |DN^-|^2  +
\fancyD (r_k) \left(-\fancyE_k + \frac{3\eta}{4}\right)\, .\label{e:upper_limit_4}
\end{align}
\item[(d2)] $|\etaa\circ \hat N_k^\pm|\leq C |\etaa\circ N^\pm|$, since on $\bp^{-1}(B_{r_k})$ we have $0 = \etaa \circ \hat \fancyN\,_k^\pm (x)) = {\bm p}_{\varkappa_0} (\etaa \circ \hat N_k (x, {\bm \varphi} (x)))$.
\end{itemize}
Now we consider the current $S_k$ in $\bC_{2r_k}$ induced by the multi-valued map
\[
\hat F_k^\pm (x, {\bm \varphi} (x)) = \sum_i \a{ (x, {\bm \varphi} (x)) + (\hat{N}_k^\pm)_i (x, {\bm \varphi} (x))}\, 
\]
Observe that, since $S_k = {\bm T}_F$ on $\bC_{2r_k} \setminus \bC_{r_k}$, arguing as for the estimate in \eqref{e:mass_error_1000} we easily conclude that  
\begin{align*}
& \|S_k-T\| (\bC_{2r_k} \setminus \bC_{r_k})\\
 \leq & C  \left(\int_{\bC_{3r_k}\cap \mathcal{M}^+} |DN_k^+|^2 +
\int_{\bC_{4r_k}\cap \mathcal{M}^-} |DN_k^-|^2\right)^{1+\tau}\, .
\end{align*}
In turn, using Lemma \ref{l:asymptotic}, we can control the right hand side with $\fancyD (r_k)^{1+\tau}$. In particular, for a suitable $\sigma_k\in (r_k,2r_k)$ 
\[
\mathbf{M} (\partial ((S_k-T) \res \bC_{\sigma_k}))\leq \frac{C}{r_k} \fancyD (r_k)^{1+\tau}\, . 
\]
In particular, by the isoperimetric inequality we conclude the existence of a current $Z_k$ with $\partial Z_k = \partial ((S_k-T) \res \bC_{\sigma_k})$,
$\supp (Z_k) \subset \Sigma$ and such that
\begin{align*}
\mathbf{M} (Z_k) &\leq C r_k^{-m/(m-1)} \fancyD (r_k)^{m (1+\tau)/(m-1)}\\ 
&\leq C \fancyD (r_k)^{1+\tau} \left(\frac{\fancyD (r_k)^{ 1+ \tau}}{r_k^m}\right)^{\frac{1}{m-1}}\leq C \fancyD (r_k)^{1+\tau}\, ;
\end{align*}
 where we used the bound $\fancyD(r) \le C r^{m+2-2\alpha_\be}$ (compare the argument leading to \eqref{e:molta_fame_2}).
In particular, the current 
\[
\hat{T}_k = S_k \res \bC_{\sigma_k} + T\res (\mathbb R^{m+n}\setminus \bC_{\sigma_k}) + Z_k
\]
is an admissible competitor to check the minimality of $T$, since it coincides with $T$ outside a compact set
and it has boundary $\a{\gammaup}$. In particular we conclude that
\begin{equation}\label{e:upper_limit_5}
\mathbf{M} (S_k\res \bC_{\sigma_k}) \geq \mathbf{M} (T\res \bC_{\sigma_k}) - C  \fancyD (r_k)^{1+\tau}\, .
\end{equation}
Next, since \(T\) coincides with \(\bm T_{F}\) on a large set (compare with \eqref{e:mass_error_1000})  using again the same estimate as above, we conclude also
\[
\mathbf{M} (S_k\res \bC_{\sigma_k}) \geq \mathbf{M} ({\bm T}_{F^+} \res \bC_{\sigma_k}) + \mathbf{M} ({\bm T}_{F^-} \res \bC_{\sigma_k})- C  \fancyD (r_k)^{1+\tau}\, .
\]
On the other hand, since $F$ and $\hat{F}_k$ coincide outside of $\bC_{r_k}$, we can write 
\begin{align}
\mathbf{M} ({\bm T}_{\hat F_k^+} \res \bC_{r_k}) + \mathbf{M} ({\bm T}_{\hat F_k^-} \res \bC_{ r_k}) &\ge   \mathbf{M} ({\bm T}_{F^+} \res \bC_{r_k}) + \mathbf{M} ({\bm T}_{F^-} \res \bC_{r_k})\nonumber\\
&\qquad- C  \fancyD (r_k)^{1+\tau}\, .\label{e:upper_limit_6}
\end{align}
Using now the Taylor expansion in \cite[Theorem 3.2]{DS2} we easily conclude that
\begin{align*}
 &\left|\mathbf{M}({\bm T}_{F^+} \res \bC_{r_k}) - \frac{1}{2} \int_{\bC_{r_k}\cap \mathcal{M}^+} |DN^+|^2 - Q 
\mathcal{H}^m (\bC_{r_k}\cap \mathcal{M}^+)\right|\\
&\leq C \int_{\bC_{r_k}\cap \mathcal{M}^+} (|\etaa\circ N^+| + |N^+|^2 +\abs{N^+}\abs{DN^+}^2+ |DN^+|^3)\, .
\end{align*}
By the estimate on $\abs{N^+}$ and $\Lip (N^+)$, we have
\begin{align*}
&\int_{\bC_{r_k}\cap \mathcal{M}^+} \abs{N^+}\abs{DN^+}^2+|DN^+|^3\\ 
\stackrel{\eqref{e:bad_regions_1}\& \eqref{e:bad_regions_2}\&\eqref{e:stimazza3-Jonas}}{\leq} & C \left(\int_{\bC_{2r_k}\cap \mathcal{M}^+} |DN^+|^2\right)^{1+\tau} \leq C
\fancyD (r_k)^{1+\tau}\,,
\end{align*}
where in the last inequality we have also used Lemma \ref{l:asymptotic}. By the Poincar\'e inequality (and Lemma \ref{l:asymptotic})
\[
\int_{\bC_{r_k}\cap \mathcal{M}^+} |N^+|^2 \leq C r_k^2 \int_{\bC_{r_k}\cap \mathcal{M}^+} |DN^+|^2 \leq C r_k^2 \fancyD(r_k)\, .
\]
Finally, by  Lemma \ref{l:appiattimento},
\[
\int_{\bC_{r_k}\cap \mathcal{M}^+} |\etaa\circ N^+| \leq C r_k \fancyD (r_k)^{1+\tau}\, .
\]
We thus conclude
\begin{align}
&\left|\mathbf{M} ({\bm T}_{F^+} \res \bC_{2r_k}) - \frac{1}{2} \int_{\bC_{2r_k}\cap \mathcal{M}^+} |DN^+|^2 - Q 
\mathcal{H}^m (\bC_{2r_k}\cap \mathcal{M}^+)\right|\nonumber\\ 
\leq & C r_k^2 \fancyD (r_k) + C \fancyD (r_k)^{1+\tau}\, .\label{e:upper_limit_7}
\end{align}
Similarly,
\begin{align}
 &\left|\mathbf{M} ({\bm T}_{F^-} \res \bC_{2r_k}) - \frac{1}{2} \int_{\bC_{2r_k}\cap \mathcal{M}^-} |DN^-|^2 - (Q-1) 
\mathcal{H}^m (\bC_{2r_k}\cap \mathcal{M}^-)\right|\nonumber\\
\leq\; & C r_k^2 \fancyD (r_k) + C \fancyD (r_k)^{1+\tau}\, .\label{e:upper_limit_8}
\end{align}
Observe next that the similar Taylor expansions hold for $\hat F_k^\pm$ replacing $F^\pm$, namely
\begin{align}
& \left|\mathbf{M}({\bm T}_{\hat F_k^+} \res \bC_{2r_k}) - \frac{1}{2} \int_{\bC_{2r_k}\cap \mathcal{M}^+} |D\hat N_k^+|^2 - Q 
\mathcal{H}^m (\bC_{2r_k}\cap \mathcal{M}^+)\right|\nonumber\\
 \leq & C r_k^2 \fancyD (r_k) + o(1) \fancyD (r_k)\, ,\label{e:upper_limit_9}
\end{align}
and
\begin{align}
& \left|\mathbf{M} ({\bm T}_{\hat F_k^-} \res \bC_{2r_k}) - \frac{1}{2} \int_{\bC_{2r_k}\cap \mathcal{M}^-} |D\hat N_k^-|^2 - (Q-1) 
\mathcal{H}^m (\bC_{2r_k}\cap \mathcal{M}^-)\right|\nonumber\\
\leq\; & C r_k^2 \fancyD (r_k) + o(1) \fancyD (r_k)\, .\label{e:upper_limit_10}
\end{align}
Indeed:
\begin{itemize}
\item the linear term is estimated in the same way using $|\etaa\circ \hat N_k^\pm|\leq C |\etaa\circ N_k|$;
\item the quadratic term is estimated by the Poincar\'e inequality and 
\[
\int_{\bC_{r_k}\cap \mathcal{M}^+} |D\hat N_k^+|^2 + \int_{\bC_{r_k}\cap \mathcal{M}^-} |D\hat N_k^-|^2 \leq
C \fancyD (r_k)\,,
\]
since we can assume without loss of generality that  $\fancyE_k\geq -2$;
\item finally  $\abs{\hat{N}_k^+}\abs{D\hat{N}_k^+}^2+ |D\hat{N}_k^+|^3= o(1) \abs{D\hat{N}_k^+}^2$. Indeed, by (b2)  $\Lip(\hat{N}_k^+)=o(1)$ and $\sup_{x \in B^+_{2r_k}} \abs{\hat{N}^+_k(x)} \le C r_k \Lip(\hat{N}_k^+) =o(r_k)$, since $\hat{N}^+_k$ is vanishing on $\gammaup$.  
\end{itemize}
Inserting the Taylor expansions \eqref{e:upper_limit_7}--\eqref{e:upper_limit_10}, we conclude
\begin{align}
& \int_{\bC_{r_k}\cap \mathcal{M}^+} |D\hat N_k^+|^2 + \int_{\bC_{r_k}\cap \mathcal{M}^-} |D\hat N_k^-|^2\nonumber\\
\geq & \int_{\bC_{r_k}\cap \mathcal{M}^+} |D N^+|^2 + \int_{\bC_{r_k}\cap \mathcal{M}^-} |D N^-|^2 - o(1) \fancyD (r_k)\, .
\label{e:upper_limit_11}
\end{align}
Combining now \eqref{e:upper_limit_4} and \eqref{e:upper_limit_11} we achieve
\[
\fancyD (r_k) \left( - \fancyE_k +\frac{3\eta}{4}\right) \geq - o(1) \fancyD (r_k)\, .
\]
Dividing by $\fancyD (r_k)$ and choosing $k$ large enough we achieve the desired inequality $\fancyE_k \leq \eta$. 

\bibliographystyle{amsalpha}

C.~{De Lellis} and E.~Spadaro.
\newblock {{$Q$}-valued functions revisited}.
\newblock {\em Mem. Amer. Math. Soc.}, 211(991):vi+79, 2011.

\bibitem{DS3}
C.~{De Lellis} and E.~Spadaro.
\newblock {Regularity of area minimizing currents {I}: gradient {$L^p$}
  estimates}.
\newblock {\em Geom. Funct. Anal.}, 24(6):1831--1884, 2014.

\bibitem{DS2}
C.~{De Lellis} and E.~Spadaro.
\newblock {Multiple valued functions and integral currents}.
\newblock {\em Ann. Sc. Norm. Super. Pisa Cl. Sci. (5)}, 14(4):1239--1269,
  2015.

\bibitem{DS4}
C.~{De Lellis} and E.~Spadaro.
\newblock {Regularity of area minimizing currents {II}: center manifold}.
\newblock {\em Ann. of Math. (2)}, 183(2):499--575, 2016.

\bibitem{DS5}
C.~{De Lellis} and E.~Spadaro.
\newblock {Regularity of area minimizing currents {III}: blow-up}.
\newblock {\em Ann. of Math. (2)}, 183(2):577--617, 2016.

\bibitem{DSS4}
C.~{De Lellis}, E.~{Spadaro}, and L.~{Spolaor}.
\newblock {Regularity theory for $2$-dimensional almost minimal currents III:
  blowup}.
\newblock {\em ArXiv e-prints. To appear in {Jour. of Diff. Geom}}, August
  2015.

\bibitem{DSS3}
C.~{De Lellis}, E.~Spadaro, and L.~Spolaor.
\newblock {Regularity {T}heory for 2-{D}imensional {A}lmost {M}inimal
  {C}urrents {II}: {B}ranched {C}enter {M}anifold}.
\newblock {\em Ann. PDE}, 3(2):3:18, 2017.

\bibitem{DSS1}
C.~{De Lellis}, E.~Spadaro, and L.~Spolaor.
\newblock {Uniqueness of tangent cones for two-dimensional almost-minimizing
  currents}.
\newblock {\em Comm. Pure Appl. Math.}, 70(7):1402--1421, 2017.

\bibitem{DSS2}
C.~De~Lellis, E.~Spadaro, and L.~Spolaor.
\newblock Regularity theory for {$2$}-dimensional almost minimal currents {I}:
  {L}ipschitz approximation.
\newblock {\em Trans. Amer. Math. Soc.}, 370(3):1783--1801, 2018.

\bibitem{Guido}
G.~{De Philippis} and E.~Paolini.
\newblock {A short proof of the minimality of {S}imons cone}.
\newblock {\em Rend. Semin. Mat. Univ. Padova}, 121:233--241, 2009.

\bibitem{Fed}
H.~Federer.
\newblock {\em {Geometric measure theory}}.
\newblock {Die Grundlehren der mathematischen Wissenschaften, Band 153}.
  Springer-Verlag New York Inc., New York, 1969.

\bibitem{FF}
H.~Federer and W.~H. Fleming.
\newblock {Normal and integral currents}.
\newblock {\em Ann. of Math. (2)}, 72:458--520, 1960.

\bibitem{Fleming}
W.~H. Fleming.
\newblock {On the oriented {P}lateau problem}.
\newblock {\em Rend. Circ. Mat. Palermo (2)}, 11:69--90, 1962.

\bibitem{Gulliver}
R.~Gulliver.
\newblock A minimal surface with an atypical boundary branch point.
\newblock In {\em Differential geometry}, volume~52 of {\em Pitman Monogr.
  Surveys Pure Appl. Math.}, pages 211--228. Longman Sci. Tech., Harlow, 1991.

\bibitem{HS}
R.~Hardt and L.~Simon.
\newblock {Boundary regularity and embedded solutions for the oriented
  {P}lateau problem}.
\newblock {\em Ann. of Math. (2)}, 110(3):439--486, 1979.

\bibitem{Hardt}
R.~M. Hardt.
\newblock On boundary regularity for integral currents or flat chains modulo
  two minimizing the integral of an elliptic integrand.
\newblock {\em Comm. Partial Differential Equations}, 2(12):1163--1232, 1977.

\bibitem{Jonas}
J.~Hirsch.
\newblock {Boundary regularity of {D}irichlet minimizing {$Q$}-valued
  functions}.
\newblock {\em Ann. Sc. Norm. Super. Pisa Cl. Sci. (5)}, 16(4):1353--1407,
  2016.

\bibitem{Jonas2}
J.~Hirsch.
\newblock Examples of holomorphic functions vanishing to infinite order at the
  boundary.
\newblock {\em Trans. Amer. Math. Soc.}, 370(6):4249--4271, 2018.

\bibitem{HM}
J.~Hirsch and M.~Marini.
\newblock Uniqueness of tangent cones to boundary points of two-dimensional
  almost-minimizing currents, 2019.

\bibitem{DeLellisZhao}
Camillo~De Lellis and Zihui Zhao.
\newblock Dirichlet energy-minimizers with analytic boundary, 2019.

\bibitem{NV}
A.~{Naber} and D.~{Valtorta}.
\newblock {The Singular Structure and Regularity of Stationary and Minimizing
  Varifolds}.
\newblock {\em J. Eur. Math. Soc.}, 22(10):3305--3382, 2020.

\bibitem{Riesz}
F.~{Riesz} and M.~{Riesz}.
\newblock \"{U}ber die {R}andwerte einer analytischen {F}unktion, 1916.

\bibitem{Sim}
L.~Simon.
\newblock {\em {Lectures on geometric measure theory}}, volume~3 of {\em
  {Proceedings of the Centre for Mathematical Analysis, Australian National
  University}}.
\newblock Australian National University Centre for Mathematical Analysis,
  Canberra, 1983.

\bibitem{Simon2}
L.~Simon.
\newblock {Rectifiability of the singular sets of multiplicity {$1$} minimal
  surfaces and energy minimizing maps}.
\newblock In {\em {Surveys in differential geometry, {V}ol.\ {II} ({C}ambridge,
  {MA}, 1993)}}, pages 246--305. Int. Press, Cambridge, MA, 1995.

\bibitem{Simons}
J.~Simons.
\newblock {Minimal varieties in riemannian manifolds}.
\newblock {\em Ann. of Math. (2)}, 88:62--105, 1968.

\bibitem{Emanuele}
E.~N. Spadaro.
\newblock Complex varieties and higher integrability of {D}ir-minimizing
  {$Q$}-valued functions.
\newblock {\em Manuscripta Math.}, 132(3-4):415--429, 2010.

\bibitem{Spolaor}
L.~{Spolaor}.
\newblock {Almgren's type regularity for {S}emicalibrated {C}urrents}.
\newblock {\em Adv. Math.}, 350:747--815, 2019.

\bibitem{Stein}
E.~M. Stein.
\newblock {\em Singular integrals and differentiability properties of
  functions}.
\newblock Princeton Mathematical Series, No. 30. Princeton University Press,
  Princeton, N.J., 1970.

\bibitem{SW}
E.~M. Stein and G.~Weiss.
\newblock {\em {Introduction to {F}ourier analysis on {E}uclidean spaces}}.
\newblock Princeton University Press, Princeton, N.J., 1971.
\newblock Princeton Mathematical Series, No. 32.

\bibitem{White_branching}
B.~White.
\newblock Classical area minimizing surfaces with real-analytic boundaries.
\newblock {\em Acta Math.}, 179(2):295--305, 1997.

\bibitem{White97}
B.~White.
\newblock {Stratification of minimal surfaces, mean curvature flows, and
  harmonic maps}.
\newblock {\em J. Reine Angew. Math.}, 488:1--35, 1997.

\end{thebibliography}
Jr. F.~J. Almgren.
\newblock {\em {Almgren's big regularity paper}}, volume~1 of {\em {World
  Scientific Monograph Series in Mathematics}}.
\newblock World Scientific Publishing Co. Inc., River Edge, NJ, 2000.

\bibitem{BDG}
E.~Bombieri, E.~{De Giorgi}, and E.~Giusti.
\newblock {Minimal cones and the Bernstein problem}.
\newblock {\em Invent. Math.}, 7:243--268, 1969.

\bibitem{Theodora}
T.~{Bourni}.
\newblock {Allard-type boundary regularity for {$C^{1,\alpha}$} boundaries}.
\newblock {\em Adv. Calc. Var.}, 9(2):143--161, 2016.

\bibitem{Chang}
S.~X. Chang.
\newblock {Two-dimensional area minimizing integral currents are classical
  minimal surfaces}.
\newblock {\em J. Amer. Math. Soc.}, 1(4):699--778, 1988.

\bibitem{DG}
E.~{De Giorgi}.
\newblock {\em {Frontiere orientate di misura minima}}.
\newblock {Seminario di Matematica della Scuola Normale Superiore di Pisa,
  1960-61}. Editrice Tecnico Scientifica, Pisa, 1961.

\bibitem{DeGiorgi5}
E.~{De Giorgi}.
\newblock {Una estensione del teorema di {B}ernstein}.
\newblock {\em Ann. Scuola Norm. Sup. Pisa (3)}, 19:79--85, 1965.

\bibitem{DDH}
C.~De~Lellis, G.~De~Philippis, and J.~Hirsch.
\newblock Nonclassical minimizing surfaces with smooth boundary, 2019.

\bibitem{Matrix}
C.~{De Lellis}, G.~{De Philippis}, J.~Hirsch, and A.~Massaccesi.
\newblock {Boundary regularity of mass-minimizing integral currents and a
  question of Almgren}.
\newblock 2018.

\bibitem{DS1}
C.~{De Lellis} and E.~Spadaro.
\newblock {{$Q$}-valued functions revisited}.
\newblock {\em Mem. Amer. Math. Soc.}, 211(991):vi+79, 2011.

\bibitem{DS3}
C.~{De Lellis} and E.~Spadaro.
\newblock {Regularity of area minimizing currents {I}: gradient {$L^p$}
  estimates}.
\newblock {\em Geom. Funct. Anal.}, 24(6):1831--1884, 2014.

\bibitem{DS2}
C.~{De Lellis} and E.~Spadaro.
\newblock {Multiple valued functions and integral currents}.
\newblock {\em Ann. Sc. Norm. Super. Pisa Cl. Sci. (5)}, 14(4):1239--1269,
  2015.

\bibitem{DS4}
C.~{De Lellis} and E.~Spadaro.
\newblock {Regularity of area minimizing currents {II}: center manifold}.
\newblock {\em Ann. of Math. (2)}, 183(2):499--575, 2016.

\bibitem{DS5}
C.~{De Lellis} and E.~Spadaro.
\newblock {Regularity of area minimizing currents {III}: blow-up}.
\newblock {\em Ann. of Math. (2)}, 183(2):577--617, 2016.

\bibitem{DSS4}
C.~{De Lellis}, E.~{Spadaro}, and L.~{Spolaor}.
\newblock {Regularity theory for $2$-dimensional almost minimal currents III:
  blowup}.
\newblock {\em ArXiv e-prints. To appear in {Jour. of Diff. Geom}}, August
  2015.

\bibitem{DSS3}
C.~{De Lellis}, E.~Spadaro, and L.~Spolaor.
\newblock {Regularity {T}heory for 2-{D}imensional {A}lmost {M}inimal
  {C}urrents {II}: {B}ranched {C}enter {M}anifold}.
\newblock {\em Ann. PDE}, 3(2):3:18, 2017.

\bibitem{DSS1}
C.~{De Lellis}, E.~Spadaro, and L.~Spolaor.
\newblock {Uniqueness of tangent cones for two-dimensional almost-minimizing
  currents}.
\newblock {\em Comm. Pure Appl. Math.}, 70(7):1402--1421, 2017.

\bibitem{DSS2}
C.~De~Lellis, E.~Spadaro, and L.~Spolaor.
\newblock Regularity theory for {$2$}-dimensional almost minimal currents {I}:
  {L}ipschitz approximation.
\newblock {\em Trans. Amer. Math. Soc.}, 370(3):1783--1801, 2018.

\bibitem{Guido}
G.~{De Philippis} and E.~Paolini.
\newblock {A short proof of the minimality of {S}imons cone}.
\newblock {\em Rend. Semin. Mat. Univ. Padova}, 121:233--241, 2009.

\bibitem{Fed}
H.~Federer.
\newblock {\em {Geometric measure theory}}.
\newblock {Die Grundlehren der mathematischen Wissenschaften, Band 153}.
  Springer-Verlag New York Inc., New York, 1969.

\bibitem{FF}
H.~Federer and W.~H. Fleming.
\newblock {Normal and integral currents}.
\newblock {\em Ann. of Math. (2)}, 72:458--520, 1960.

\bibitem{Fleming}
W.~H. Fleming.
\newblock {On the oriented {P}lateau problem}.
\newblock {\em Rend. Circ. Mat. Palermo (2)}, 11:69--90, 1962.

\bibitem{Gulliver}
R.~Gulliver.
\newblock A minimal surface with an atypical boundary branch point.
\newblock In {\em Differential geometry}, volume~52 of {\em Pitman Monogr.
  Surveys Pure Appl. Math.}, pages 211--228. Longman Sci. Tech., Harlow, 1991.

\bibitem{HS}
R.~Hardt and L.~Simon.
\newblock {Boundary regularity and embedded solutions for the oriented
  {P}lateau problem}.
\newblock {\em Ann. of Math. (2)}, 110(3):439--486, 1979.

\bibitem{Hardt}
R.~M. Hardt.
\newblock On boundary regularity for integral currents or flat chains modulo
  two minimizing the integral of an elliptic integrand.
\newblock {\em Comm. Partial Differential Equations}, 2(12):1163--1232, 1977.

\bibitem{Jonas}
J.~Hirsch.
\newblock {Boundary regularity of {D}irichlet minimizing {$Q$}-valued
  functions}.
\newblock {\em Ann. Sc. Norm. Super. Pisa Cl. Sci. (5)}, 16(4):1353--1407,
  2016.

\bibitem{Jonas2}
J.~Hirsch.
\newblock Examples of holomorphic functions vanishing to infinite order at the
  boundary.
\newblock {\em Trans. Amer. Math. Soc.}, 370(6):4249--4271, 2018.

\bibitem{HM}
J.~Hirsch and M.~Marini.
\newblock Uniqueness of tangent cones to boundary points of two-dimensional
  almost-minimizing currents, 2019.

\bibitem{DeLellisZhao}
Camillo~De Lellis and Zihui Zhao.
\newblock Dirichlet energy-minimizers with analytic boundary, 2019.

\bibitem{NV}
A.~{Naber} and D.~{Valtorta}.
\newblock {The Singular Structure and Regularity of Stationary and Minimizing
  Varifolds}.
\newblock {\em J. Eur. Math. Soc.}, 22(10):3305--3382, 2020.

\bibitem{Riesz}
F.~{Riesz} and M.~{Riesz}.
\newblock \"{U}ber die {R}andwerte einer analytischen {F}unktion, 1916.

\bibitem{Sim}
L.~Simon.
\newblock {\em {Lectures on geometric measure theory}}, volume~3 of {\em
  {Proceedings of the Centre for Mathematical Analysis, Australian National
  University}}.
\newblock Australian National University Centre for Mathematical Analysis,
  Canberra, 1983.

\bibitem{Simon2}
L.~Simon.
\newblock {Rectifiability of the singular sets of multiplicity {$1$} minimal
  surfaces and energy minimizing maps}.
\newblock In {\em {Surveys in differential geometry, {V}ol.\ {II} ({C}ambridge,
  {MA}, 1993)}}, pages 246--305. Int. Press, Cambridge, MA, 1995.

\bibitem{Simons}
J.~Simons.
\newblock {Minimal varieties in riemannian manifolds}.
\newblock {\em Ann. of Math. (2)}, 88:62--105, 1968.

\bibitem{Emanuele}
E.~N. Spadaro.
\newblock Complex varieties and higher integrability of {D}ir-minimizing
  {$Q$}-valued functions.
\newblock {\em Manuscripta Math.}, 132(3-4):415--429, 2010.

\bibitem{Spolaor}
L.~{Spolaor}.
\newblock {Almgren's type regularity for {S}emicalibrated {C}urrents}.
\newblock {\em Adv. Math.}, 350:747--815, 2019.

\bibitem{Stein}
E.~M. Stein.
\newblock {\em Singular integrals and differentiability properties of
  functions}.
\newblock Princeton Mathematical Series, No. 30. Princeton University Press,
  Princeton, N.J., 1970.

\bibitem{SW}
E.~M. Stein and G.~Weiss.
\newblock {\em {Introduction to {F}ourier analysis on {E}uclidean spaces}}.
\newblock Princeton University Press, Princeton, N.J., 1971.
\newblock Princeton Mathematical Series, No. 32.

\bibitem{White_branching}
B.~White.
\newblock Classical area minimizing surfaces with real-analytic boundaries.
\newblock {\em Acta Math.}, 179(2):295--305, 1997.

\bibitem{White97}
B.~White.
\newblock {Stratification of minimal surfaces, mean curvature flows, and
  harmonic maps}.
\newblock {\em J. Reine Angew. Math.}, 488:1--35, 1997.

\end{thebibliography}

\printindex

\end{document}

%% file: Guido.pdf_t
\begin{picture}(0,0)%
\includegraphics{Guido.pdf}%
\end{picture}%
\setlength{\unitlength}{1243sp}%
\begingroup\makeatletter\ifx\SetFigFont\undefined%
\gdef\SetFigFont#1#2#3#4#5{%
  \reset@font\fontsize{#1}{#2pt}%
  \fontfamily{#3}\fontseries{#4}\fontshape{#5}%
  \selectfont}%
\fi\endgroup%
\begin{picture}(5416,5487)(3143,-5918)
\put(8011,-1186){\makebox(0,0)[lb]{\smash{{\SetFigFont{14}{16.8}{\rmdefault}{\mddefault}{\updefault}{\color[rgb]{0,0,0}$q$}%
}}}}
\put(7021,-2311){\makebox(0,0)[lb]{\smash{{\SetFigFont{14}{16.8}{\rmdefault}{\mddefault}{\updefault}{\color[rgb]{0,0,0}$p$}%
}}}}
\put(4636,-1501){\makebox(0,0)[lb]{\smash{{\SetFigFont{14}{16.8}{\rmdefault}{\mddefault}{\updefault}{\color[rgb]{0,0,0}$1$}%
}}}}
\put(5806,-3256){\makebox(0,0)[lb]{\smash{{\SetFigFont{14}{16.8}{\rmdefault}{\mddefault}{\updefault}{\color[rgb]{0,0,0}$2$}%
}}}}
\end{picture}%

%% file: star.pdf_t
\begin{picture}(0,0)%
\includegraphics{star.pdf}%
\end{picture}%
\setlength{\unitlength}{2072sp}%
\begingroup\makeatletter\ifx\SetFigFont\undefined%
\gdef\SetFigFont#1#2#3#4#5{%
  \reset@font\fontsize{#1}{#2pt}%
  \fontfamily{#3}\fontseries{#4}\fontshape{#5}%
  \selectfont}%
\fi\endgroup%
\begin{picture}(7347,6060)(1670,-8173)
\put(6645,-3385){\makebox(0,0)[lb]{\smash{{\SetFigFont{12}{14.4}{\rmdefault}{\mddefault}{\updefault}{\color[rgb]{0,0,0}$\ell^+_3$}%
}}}}
\put(7645,-5025){\makebox(0,0)[lb]{\smash{{\SetFigFont{12}{14.4}{\rmdefault}{\mddefault}{\updefault}{\color[rgb]{0,0,0}$\ell^+_4$}%
}}}}
\put(3895,-2825){\makebox(0,0)[lb]{\smash{{\SetFigFont{12}{14.4}{\rmdefault}{\mddefault}{\updefault}{\color[rgb]{0,0,0}$\ell^+_1=\ell^+_2$}%
}}}}
\put(1685,-7145){\makebox(0,0)[lb]{\smash{{\SetFigFont{12}{14.4}{\rmdefault}{\mddefault}{\updefault}{\color[rgb]{0,0,0}$\ell^-_1=\ell^-_2$}%
}}}}
\put(6635,-7065){\makebox(0,0)[lb]{\smash{{\SetFigFont{12}{14.4}{\rmdefault}{\mddefault}{\updefault}{\color[rgb]{0,0,0}$\ell^-_3$}%
}}}}
\end{picture}%

%% file: Qhalf.pdf_t
\begin{picture}(0,0)%
\includegraphics{Qhalf.pdf}%
\end{picture}%
\setlength{\unitlength}{2072sp}%
\begingroup\makeatletter\ifx\SetFigFont\undefined%
\gdef\SetFigFont#1#2#3#4#5{%
  \reset@font\fontsize{#1}{#2pt}%
  \fontfamily{#3}\fontseries{#4}\fontshape{#5}%
  \selectfont}%
\fi\endgroup%
\begin{picture}(6777,4348)(1789,-7294)
\put(8236,-6226){\makebox(0,0)[lb]{\smash{{\SetFigFont{12}{14.4}{\rmdefault}{\mddefault}{\updefault}{\color[rgb]{0,0,0}$\Omega$}%
}}}}
\put(4681,-6406){\makebox(0,0)[lb]{\smash{{\SetFigFont{12}{14.4}{\rmdefault}{\mddefault}{\updefault}{\color[rgb]{0,0,0}$\gamma$}%
}}}}
\put(4794,-3327){\makebox(0,0)[lb]{\smash{{\SetFigFont{12}{14.4}{\rmdefault}{\mddefault}{\updefault}{\color[rgb]{0,0,0}$\varphi$}%
}}}}
\put(8474,-3237){\makebox(0,0)[lb]{\smash{{\SetFigFont{12}{14.4}{\rmdefault}{\mddefault}{\updefault}{\color[rgb]{0,0,0}$f_2^+$}%
}}}}
\put(8174,-5107){\makebox(0,0)[lb]{\smash{{\SetFigFont{12}{14.4}{\rmdefault}{\mddefault}{\updefault}{\color[rgb]{0,0,0}$f_1^+$}%
}}}}
\put(2264,-4947){\makebox(0,0)[lb]{\smash{{\SetFigFont{12}{14.4}{\rmdefault}{\mddefault}{\updefault}{\color[rgb]{0,0,0}$f_1^-$}%
}}}}
\put(3331,-6721){\makebox(0,0)[lb]{\smash{{\SetFigFont{12}{14.4}{\rmdefault}{\mddefault}{\updefault}{\color[rgb]{0,0,0}$\Omega^-$}%
}}}}
\put(5941,-6856){\makebox(0,0)[lb]{\smash{{\SetFigFont{12}{14.4}{\rmdefault}{\mddefault}{\updefault}{\color[rgb]{0,0,0}$\Omega^+$}%
}}}}
\end{picture}%

%% file: Collapsed_int.pdf_t
\begin{picture}(0,0)%
\includegraphics{Collapsed_int.pdf}%
\end{picture}%
\setlength{\unitlength}{2072sp}%
\begingroup\makeatletter\ifx\SetFigFont\undefined%
\gdef\SetFigFont#1#2#3#4#5{%
  \reset@font\fontsize{#1}{#2pt}%
  \fontfamily{#3}\fontseries{#4}\fontshape{#5}%
  \selectfont}%
\fi\endgroup%
\begin{picture}(6774,3768)(1789,-7294)
\put(4781,-4351){\makebox(0,0)[lb]{\smash{{\SetFigFont{12}{14.4}{\rmdefault}{\mddefault}{\updefault}{\color[rgb]{0,0,0}$\varphi$}%
}}}}
\put(8461,-4261){\makebox(0,0)[lb]{\smash{{\SetFigFont{12}{14.4}{\rmdefault}{\mddefault}{\updefault}{\color[rgb]{0,0,0}$f_2^+$}%
}}}}
\put(8236,-6226){\makebox(0,0)[lb]{\smash{{\SetFigFont{12}{14.4}{\rmdefault}{\mddefault}{\updefault}{\color[rgb]{0,0,0}$\Omega$}%
}}}}
\put(4681,-6406){\makebox(0,0)[lb]{\smash{{\SetFigFont{12}{14.4}{\rmdefault}{\mddefault}{\updefault}{\color[rgb]{0,0,0}$\gamma$}%
}}}}
\put(8174,-5107){\makebox(0,0)[lb]{\smash{{\SetFigFont{12}{14.4}{\rmdefault}{\mddefault}{\updefault}{\color[rgb]{0,0,0}$f_1^+$}%
}}}}
\put(2264,-4947){\makebox(0,0)[lb]{\smash{{\SetFigFont{12}{14.4}{\rmdefault}{\mddefault}{\updefault}{\color[rgb]{0,0,0}$f_1^-$}%
}}}}
\end{picture}%

%% file: omega.pdf_t
\begin{picture}(0,0)%
\includegraphics{omega.pdf}%
\end{picture}%
\setlength{\unitlength}{1243sp}%
\begingroup\makeatletter\ifx\SetFigFont\undefined%
\gdef\SetFigFont#1#2#3#4#5{%
  \reset@font\fontsize{#1}{#2pt}%
  \fontfamily{#3}\fontseries{#4}\fontshape{#5}%
  \selectfont}%
\fi\endgroup%
\begin{picture}(9242,9014)(893,-8618)
\put(7426,-3031){\makebox(0,0)[lb]{\smash{{\SetFigFont{11}{13.2}{\rmdefault}{\mddefault}{\updefault}{\color[rgb]{0,0,0}$\partial \Omega$}%
}}}}
\put(5131,-4021){\makebox(0,0)[lb]{\smash{{\SetFigFont{11}{13.2}{\rmdefault}{\mddefault}{\updefault}{\color[rgb]{0,0,0}$0$}%
}}}}
\put(7021,-5911){\makebox(0,0)[lb]{\smash{{\SetFigFont{11}{13.2}{\rmdefault}{\mddefault}{\updefault}{\color[rgb]{0,0,0}$\Omega$}%
}}}}
\end{picture}%

%% file: currentZ.pdf_t
\begin{picture}(0,0)%
\includegraphics{currentZ.pdf}%
\end{picture}%
\setlength{\unitlength}{1243sp}%
\begingroup\makeatletter\ifx\SetFigFont\undefined%
\gdef\SetFigFont#1#2#3#4#5{%
  \reset@font\fontsize{#1}{#2pt}%
  \fontfamily{#3}\fontseries{#4}\fontshape{#5}%
  \selectfont}%
\fi\endgroup%
\begin{picture}(12084,6329)(1249,-7371)
\put(4906,-2986){\makebox(0,0)[lb]{\smash{{\SetFigFont{11}{13.2}{\rmdefault}{\mddefault}{\updefault}{\color[rgb]{0,0,0}$Z$}%
}}}}
\put(8641,-3076){\makebox(0,0)[lb]{\smash{{\SetFigFont{11}{13.2}{\rmdefault}{\mddefault}{\updefault}{\color[rgb]{0,0,0}$\gamma$}%
}}}}
\put(7336,-7036){\makebox(0,0)[lb]{\smash{{\SetFigFont{11}{13.2}{\rmdefault}{\mddefault}{\updefault}{\color[rgb]{0,0,0}$x_m$}%
}}}}
\end{picture}%

%% file: cones.pdf_t
\begin{picture}(0,0)%
\includegraphics{cones.pdf}%
\end{picture}%
\setlength{\unitlength}{1657sp}%
\begingroup\makeatletter\ifx\SetFigFont\undefined%
\gdef\SetFigFont#1#2#3#4#5{%
  \reset@font\fontsize{#1}{#2pt}%
  \fontfamily{#3}\fontseries{#4}\fontshape{#5}%
  \selectfont}%
\fi\endgroup%
\begin{picture}(6438,7581)(1768,-8059)
\put(7651,-3796){\makebox(0,0)[lb]{\smash{{\SetFigFont{14}{16.8}{\rmdefault}{\mddefault}{\updefault}{\color[rgb]{0,0,0}$\pi (q)$}%
}}}}
\put(8191,-4741){\makebox(0,0)[lb]{\smash{{\SetFigFont{14}{16.8}{\rmdefault}{\mddefault}{\updefault}{\color[rgb]{0,0,0}$\pi$}%
}}}}
\end{picture}%

%% file: whitney.pdf_t
\begin{picture}(0,0)%
\includegraphics{whitney.pdf}%
\end{picture}%
\setlength{\unitlength}{1657sp}%
\begingroup\makeatletter\ifx\SetFigFont\undefined%
\gdef\SetFigFont#1#2#3#4#5{%
  \reset@font\fontsize{#1}{#2pt}%
  \fontfamily{#3}\fontseries{#4}\fontshape{#5}%
  \selectfont}%
\fi\endgroup%
\begin{picture}(9270,6186)(1543,-7024)
\put(5176,-3886){\makebox(0,0)[lb]{\smash{{\SetFigFont{14}{16.8}{\rmdefault}{\mddefault}{\updefault}{\color[rgb]{0,0,0}....}%
}}}}
\put(6931,-2716){\makebox(0,0)[lb]{\smash{{\SetFigFont{14}{16.8}{\rmdefault}{\mddefault}{\updefault}{\color[rgb]{0,0,0}....}%
}}}}
\put(8956,-1726){\makebox(0,0)[lb]{\smash{{\SetFigFont{14}{16.8}{\rmdefault}{\mddefault}{\updefault}{\color[rgb]{0,0,0}....}%
}}}}
\put(2746,-6541){\makebox(0,0)[lb]{\smash{{\SetFigFont{14}{16.8}{\rmdefault}{\mddefault}{\updefault}{\color[rgb]{0,0,0}....}%
}}}}
\put(3691,-5236){\makebox(0,0)[lb]{\smash{{\SetFigFont{14}{16.8}{\rmdefault}{\mddefault}{\updefault}{\color[rgb]{0,0,0}....}%
}}}}
\put(8146,-5731){\makebox(0,0)[lb]{\smash{{\SetFigFont{14}{16.8}{\rmdefault}{\mddefault}{\updefault}{\color[rgb]{0,0,0}$c_L$}%
}}}}
\put(9361,-6586){\makebox(0,0)[lb]{\smash{{\SetFigFont{14}{16.8}{\rmdefault}{\mddefault}{\updefault}{\color[rgb]{0,0,0}$L$}%
}}}}
\put(10486,-1366){\makebox(0,0)[lb]{\smash{{\SetFigFont{14}{16.8}{\rmdefault}{\mddefault}{\updefault}{\color[rgb]{0,0,0}$\gamma$}%
}}}}
\end{picture}%

%% file: heights.pdf_t
\begin{picture}(0,0)%
\includegraphics{heights.pdf}%
\end{picture}%
\setlength{\unitlength}{2072sp}%
\begingroup\makeatletter\ifx\SetFigFont\undefined%
\gdef\SetFigFont#1#2#3#4#5{%
  \reset@font\fontsize{#1}{#2pt}%
  \fontfamily{#3}\fontseries{#4}\fontshape{#5}%
  \selectfont}%
\fi\endgroup%
\begin{picture}(11158,5894)(879,-6383)
\put(1666,-1096){\makebox(0,0)[lb]{\smash{{\SetFigFont{12}{14.4}{\rmdefault}{\mddefault}{\updefault}{\color[rgb]{0,0,0}$\bB_{H_{j+1}}$}%
}}}}
\put(2071,-6091){\makebox(0,0)[lb]{\smash{{\SetFigFont{12}{14.4}{\rmdefault}{\mddefault}{\updefault}{\color[rgb]{0,0,0}$\bC_{j+1}$}%
}}}}
\put(6391,-2896){\makebox(0,0)[lb]{\smash{{\SetFigFont{12}{14.4}{\rmdefault}{\mddefault}{\updefault}{\color[rgb]{0,0,0}$\pi_{H_{j+1}}$}%
}}}}
\put(11566,-3481){\makebox(0,0)[lb]{\smash{{\SetFigFont{12}{14.4}{\rmdefault}{\mddefault}{\updefault}{\color[rgb]{0,0,0}$\pi_0$}%
}}}}
\end{picture}%